\documentclass{article}
\usepackage[T1]{fontenc}
\usepackage{amsmath, amsthm, amssymb, fullpage}

\newtheorem{proposition}{Proposition}[subsection]
\newtheorem{corollary}[proposition]{Corollary}
\newtheorem{lemma}[proposition]{Lemma}
\newtheorem{theorem}[proposition]{Theorem}
\theoremstyle{definition}
\newtheorem{definition}[proposition]{Definition}

\usepackage{xcolor}

\usepackage{tikz}
\usepackage{tikz-cd}

\theoremstyle{remark}

\newtheorem{re}[proposition]{Remark}
\newtheorem*{re*}{Remark}
\newtheorem{example}[proposition]{Example}

\def\D{\Delta}
\def\bA{\mathbb{A}}
\def\bC{\mathbb{C}}

\def\bF{\mathbb{F}}
\def\bG{\mathbb{G}}
\def\bP{\mathbb{P}}
\def\bZ{\mathbb{Z}}
\def\cA{\mathcal{A}}
\def\cB{\mathcal{B}}
\def\cC{\mathcal{C}}

\def\cO{\mathcal{O}}

\def\cW{\mathcal{W}}

\def\cZ{\mathcal{Z}}
\def\fA{\mathfrak{A}}
\def\fE{\mathfrak{E}}

\def\fW{\mathfrak{W}}

\def\cE{\mathcal{E}}

\newcommand{\bfc}{\mathbf{c}}

\newcommand{\bfct}{\mathbf{c}_{\mathbf{t}}}
\newcommand{\uw}{\underline{\omega}}
\newcommand{\bL}{\mathbb{L}}
\newcommand{\bT}{\mathbb{T}}
\newcommand{\bLa}{\mathbb{L}^{\alpha}}
\newcommand{\bTa}{\mathbb{T}^{\alpha}}
\newcommand{\bLL}{\mathbb{L}^{\mathrm{log}}}
\newcommand{\bTT}{\mathbb{T}^{\mathrm{log}}}
\newcommand{\Ta}{T^*_{\alpha}}
\newcommand{\Tad}{T_{\alpha}}

\newcommand{\bTs}{\mathbb{T}^{\mathrm{slog}}}

\def\cL{\mathcal{L}}
\def\cLA{\mathcal{L}_{\mathcal{A}}}
\def\fC{\mathfrak{C}}

\def\fCgn{\mathfrak{C}_{g,n}}
\def\fMgn{\mathfrak{M}_{g,n}}
\def\bfCgn{\overline{\mathcal{C}}_{g,n}}
\def\bfMgn{\overline{\mathcal{M}}_{g,n}}
\newcommand{\RM}{R\overline{\mathcal{M}}_{g,n}(X,\beta)}
\newcommand{\MP}{\overline{\mathcal{M}}_{g,n}(\mathbb{P}^r,\beta)}
\newcommand{\MPp}{\overline{\mathcal{M}}_{g,n}(\mathbb{P}^r,\beta)^E}
\newcommand{\RMP}{R\overline{\mathcal{M}}_{g,n}(\mathbb{P}^r,\beta)}
\newcommand{\RH}{R\underline{Hom}}
\newcommand{\RW}{R\underline{Sec}}
\newcommand{\Hom}{\underline{Hom}}
\newcommand{\Weil}{\underline{Sec}}
\newcommand{\HB}{\mathcal{H}_{g,n}}
\newcommand{\RHB}{R\mathcal{H}_{g,n}}
\newcommand{\RHBM}{R\mathcal{H}_{R\mathcal{W}}}
\newcommand{\RHBP}{R\mathfrak{H}_{g,n}}
\newcommand{\RHBPM}{R\mathfrak{H}_{R\mathfrak{W}}}
\newcommand{\w}{\omega_{\overline{\mathcal{C}}_{g,n}/\overline{\mathcal{M}}_{g,n}}}
\newcommand{\W}{\omega_{\mathfrak{C}_{g,n}/\mathfrak{M}_{g,n}}}

\renewcommand{\dim}{\mathrm{dim}}
\renewcommand{\log}{\mathrm{log}}

\newcommand{\V}{\mathbf{V}}
\newcommand{\Tot}{\mathbf{Tot}}
\newcommand{\LogDStk}{\mathbf{LogDStk}}
\newcommand{\DStk}{\mathbf{DStk}}
\newcommand{\DSchm}{\mathbf{DSchm}}
\newcommand{\DPStk}{\mathbf{DPStk}}
\newcommand{\DRing}{\mathbf{DRing}}
\newcommand{\QCoh}{\mathbf{QCoh}}
\newcommand{\LL}{\mathbf{L}}
\newcommand{\RR}{\mathbf{R}}
\newcommand{\DR}{\mathbf{DR}}

\newcommand{\op}{\mathrm{op}}
\newcommand{\cl}{\mathrm{cl}}
\newcommand{\cofib}{\mathrm{cofib}}
\newcommand{\fib}{\mathrm{fib}}
\newcommand{\coker}{\mathrm{coker}}
\renewcommand{\ker}{\mathrm{ker}}
\newcommand{\id}{\mathrm{id}}
\newcommand{\pt}{\mathbf{pt}}
\newcommand{\red}{\mathrm{red}}
\newcommand{\tred}{\mathrm{tred}}
\newcommand{\tw}{\mathrm{tw}}
\newcommand{\vdim}{\mathrm{vdim}}

\newcommand{\vir}{\mathrm{vir}}
\newcommand{\im}{\mathrm{im}}
\newcommand{\ev}{\mathrm{ev}}
\newcommand{\rk}{\mathrm{rk}}
\newcommand{\Ob}{\mathrm{Ob}}
\newcommand{\st}{\mathrm{st}}

\DeclareMathOperator{\map}{\mathbf{Map}}
\DeclareMathOperator{\Sym}{Sym}
\DeclareMathOperator{\Spec}{Spec}
\DeclareMathOperator{\Proj}{\mathrm{Proj}}

\newcommand{\longhookrightarrow}{\lhook\joinrel\longrightarrow}

\usepackage{hyperref}

\title{Deligne--Faltings Logarithmic Derived Geometry with a View towards Logarithmic GLSM}
\author{Felix Janda and Ruoxi Li}

\begin{document}

\maketitle

\abstract{
  Many perfect obstruction theories and virtual cycles in enumerative geometry are shadows from derived geometry.
  In this paper, we give a general construction of reduced perfect obstruction theories arising from a geometric reduction of the derived stack $X$ with respect to a $(-1)$-shifted log function $w_\alpha$.
  Notably, this construction leads to a decomposition of the virtual cycle $[X]^\vir$ into the reduced virtual cycle $[X]^\red$ and a reduced boundary cycle $[\Delta_X]^\red$.
  Our construction is motivated by future applications to the moduli stack of UMD logarithmic R-maps in log GLSM \cite{CJR21}.
  Along the way, we develop logarithmic analogs of several constructions in derived geometry, in particular proving a logarithmic Lagrangian intersection theorem.
}

\setcounter{tocdepth}{2}
\tableofcontents

\setcounter{section}{-1}
\section{Introduction}

An essential difficulty in defining invariants in enumerative geometry is that the relevant moduli stacks $X_\cl$ of maps or sheaves are rarely smooth or even equidimensional, and hence their ordinary fundamental classes do not provide the expected-dimensional cycles needed for enumerative invariants.
However, according to Kontsevich's hidden smoothness principle \cite{Kontsevich95}, singular moduli stacks are often the shadows of (quasi-)smooth derived enhancements.
This led to the development of \emph{perfect obstruction theories} and \emph{virtual fundamental classes/cycles} of Li--Tian \cite{LiTi98} and Behrend--Fantechi \cite{BeFa97}, and has been one of the major drivers for the development of derived geometry \cite{CFKa09, STV15, Schurg13, Pridham13, STG16}.
In particular, if a Deligne--Mumford stack $X_\cl$ admits a quasi-smooth derived thickening $X$, applying the cotangent complex functor to the closed embedding $X_\cl \hookrightarrow X$ yields a perfect obstruction theory on $X_\cl$ in the sense of \cite{BeFa97}, see \cite{STV15}.

A different challenge to defining enumerative invariants is that integrating against the (virtual) fundamental class requires the moduli stack $U_\cl$ to be proper.
There are several ways to resolve this.
One approach is to use the cosection localized virtual cycles of Kiem--Li \cite{KiLi13}:
Given a cosection $\sigma$ of the obstruction sheaf of $U_\cl$, this allows us to replace the virtual cycle $[U_\cl]^\vir$ by a \emph{cosection-localized virtual cycle} $[U_\cl]^\vir_\sigma$, which is supported on a closed \emph{degeneracy locus} $U_\cl(\sigma) \subset U_\cl$.
There are many cases when $U_\cl(\sigma)$ is proper even when $U_\cl$ is not, and in such a situation, we can define enumerative invariants as integrals against $[U_\cl]^\vir_\sigma$.
From a derived geometry viewpoint, cosections are related to \emph{$(-1)$-shifted differential $1$-forms}, which arise naturally by differentiating \emph{$(-1)$-shifted functions}, see \cite[Section~6]{AKLPR22P}\footnote{Note that we refer to a specific arXiv version, not the latest or published version.} and \cite{OhTh24, KiPa25}.

A second way to resolve non-properness is to find a compactification $X_\cl$ of $U_\cl$, and to extend the perfect obstruction theory from $U_\cl$ to $X_\cl$.
We may then define enumerative invariants by integrating against $[X]^\vir$.
Usually, there are many ways to compactify $U_\cl$, and these may lead to different versions of enumerative invariants.
Logarithmic geometry often leads to modular compactifications of moduli stacks.

Chang--Li's moduli stack $U_\cl$ of stable maps with $p$-fields \cite{ChLi12} is neither smooth nor proper.\footnote{See \cite{CJW21, Picciotto21, OhTh24, KiPa25} for more recent treatments of stable (quasi-)maps with $p$-fields.}
Despite this, Chang--Li apply cosection localization to define enumerative invariants of $U_\cl$, and they prove that (up to a sign) those invariants agree with Gromov--Witten invariants of quintic threefolds. The log GLSM program \cite{CJRS22, CJR21, CJR-tropical, CJR-effective, CJR-localization} instead constructs a compactification $X_{\cl}$ of UMD (uniform minimal degeneracy) logarithmic R-maps, and constructs \emph{two} perfect obstruction theories (and hence virtual cycles) for $X_{\cl}$: the \emph{canonical perfect obstruction theory} and the \emph{reduced perfect obstruction theory}.
The canonical perfect obstruction theory is defined via the deformation theory of logarithmic maps.
The definition of the reduced perfect obstruction theory is more subtle.
Its construction motivates this work and is discussed further below.
As proven in \cite[Theorem~1.9]{CJR21}, the reduced virtual cycle of UMD log R-maps recovers the cosection-localized virtual cycle of Chang--Li, and hence thev Gromov--Witten invariants of interest.
This allows performing computations of invariants via virtual localization on the moduli of UMD log R-maps, which was not possible for the moduli of stable maps with $p$-fields or the moduli of stable maps to quintic threefolds.
In particular, this has been applied to major higher genus computations and structural predictions for Gromov--Witten invariants \cite{GJR17p,GJR18P}.

The reduced perfect obstruction theories considered in the context of the enumerative geometry of K3 surfaces \cite{OkPa10, MaPa13, MPT10}\footnote{See also \cite{STV15} and \cite[Section~6]{AKLPR22P} for the derived geometric point of view.} turn a vanishing canonical virtual cycle into an interesting reduced virtual cycle by removing a trivial factor from the obstruction sheaf.
In particular, this construction increases the virtual dimension by one.
The reduced perfect obstruction theory for the moduli of UMD log R-maps is different --- it is obtained from the canonical perfect obstruction theory by using a \emph{twisted cosection} to remove a two-term ``boundary'' complex from the canonical perfect obstruction theory, and virtual dimension of the canonical and reduced perfect obstruction theories are the same.

The motivation for this paper is to build a general derived-geometric framework for reducing perfect obstruction theories along a logarithmic boundary, using twisted cosections.
In the following sections, we outline the theory.
In particular we will discuss the Decomposition Formula $[X_\cl]^\vir = [X_\cl]^\red + [(\D_{X_\cl})_\cl]^\red$, Theorem~\ref{thm:intro-decompose-virt-cycle}, of $[X_\cl]^\vir$ into the reduced virtual cycle $[X_\cl]^\red$ and the reduced boundary virtual cycle $[(\D_{X_\cl})_\cl]^\red$.

\subsection{Deligne--Faltings Logarithmic Derived Geometry}

In Section \ref{sec:2}, we develop aspects of Deligne--Faltings rank one derived log geometry related to twisted cosections from Log GLSM \cite{CJR21}.
We follow \cite{Olssonlogstk03} and \cite{Pridham13} in spirit, but we work in a technically much simpler setting, while retaining interesting geometric features, such as the canonical recursive structure carried by every derived log stack, see Definition \ref{def:can-recur-str}.
We refer the reader to \cite{Rognes09, Bhatt12, Pridham13, STG16, zhang26} for other (more sophisticated) approaches to derived log geometry.

Set $\cA:=[\bA^1/\bG_m]$. Following \cite{Olssonlogstk03}, a \emph{Deligne-Faltings log structure of rank $m$} on a derived stack $X$ is a morphism $\alpha_X: X \to \cA^m$.
Since we will mainly consider Deligne--Faltings log structures of rank $1$, we refer to a pair $(X, X \overset{\alpha_X}{\longrightarrow} \cA)$ as a \emph{derived log stack}. Consequently, we define the $\infty$-category $\LogDStk$ of derived log stacks to be the $\infty$-category $\DStk_{/\cA}$ of derived stacks over $\cA$. Hence, a morphism of derived log stacks is a morphism over $\cA$, and thus naturally preserves the log structure\footnote{Such a morphism in classical log geometry is called a strict morphism.}. Although the definition of morphisms of derived log stacks is somewhat restrictive, it has the advantage that homotopy fibered products are naturally defined.\footnote{When considering non-strict morphisms of classical log schemes, taking fibered products is somewhat subtle.} Following \cite{Olssonlogcotangent05}, the log cotangent complex $\bLL_{X}$ of $(X, \alpha_X)$ is the relative cotangent complex $\bL_{X/\cA}$.

A \emph{$(-1)$-shifted log function} on a derived log scheme is a $(-1)$-shifted function allowed to have poles along the boundary; see Definition \ref{def:shift-log-function-diff-form} and Remark \ref{rmk:understanding-of-shift-log-function}.
This will be a central notion, and most of derived log geometry developed here will be built around it.
In particular, if we restrict the $(-1)$-shifted log function to the open locus away from the boundary, then we obtain an ordinary $(-1)$-shifted function (see Remark \ref{rmk:generic-fib-shift-univ-line-bund}).

Taking the differential, a $(-1)$-shifted log function $w_\alpha$ on a derived log Artin stack $(X, \alpha_X)$ induces an exact $(-1)$-shifted log differential $1$-form
$$dw_\alpha: \cO_X \to \bLL_X \otimes^{\LL}_{\cO_{X}} \cO_X(\D_X) [-1].$$
Importantly, there is a close relation between $(-1)$-shifted log differential forms and log cosections, discussed below.
Following \cite{CS24}, we call $\bLa_X := \bLL_X \otimes^{\LL}_{\cO_{X}} \cO_X(\D_X)$ the twisted log cotangent complex of $(X, \alpha_X)$.\footnote{It is tempting to call $\bLa_X$ the log cotangent complex, as one might expect that an exact $(-1)$-shifted log differential $1$-form is a section of the $(-1)$-shifted log cotangent complex.} The linear derived stack $\Ta[-1] X := \V_X\left((\bLa_X[-1])^\vee\right)$ is called the $(-1)$-shifted log cotangent bundle\footnote{An analogous $0$-shifted log cotangent bundle appears, for example, in \cite{BDMOP19,MRWW24}.}. By abuse of notation, we also use $dw_{\alpha}$ to denote the morphism $X \to \Ta[-1]X$ induced by $dw_\alpha: \cO_X \to \bLa_X[-1].$ Then the \emph{derived log critical locus} of $w_\alpha$ is defined via the following homotopy Cartesian diagram in $\LogDStk$:
$$
    \begin{tikzcd}
{X(dw_\alpha)} \arrow[r] \arrow[d] \ar[dr, phantom, "\lrcorner_{\text{\tiny h}}", very near start] & {X} \arrow[d,"dw_\alpha"] \\
{X} \arrow[r, "0"]           & {\Ta[-1]X}          
    \end{tikzcd}
$$
Guided by Pantev--To\"en--Vaqui\'e--Vezzosi's work \cite{Vezzosi20, PTVV13} on shifted symplectic structures, we expect (see Remark \ref{rmk:log-cotangent-complex-der-log-crit-loci}) the following quasi-isomorphism
$$
\bTT_{X(dw_\alpha)} \simeq \bLa_{X(dw_\alpha)}[-2],
$$
which is established by the following theorem, building on \cite{Calaque19,CS24}.
\begin{theorem}[Logarithmic Lagrangian Intersection Theorem, see Theorem~\ref{thm:log-lag-intersect}]
\label{thm:intro-log-lag-intersect}
    Let $w_\alpha: X \to L_{\cA}[-n]$ be a $(-n)$-shifted log function on a derived log Artin stack $X$.
    Then $dw_\alpha: X \to \Ta[-n]X$ is a log Lagrangian, and the derived log critical locus $X(dw_\alpha)$ carries a $(-n-1)$-shifted log symplectic structure.
\end{theorem}
Theorem \ref{thm:intro-log-lag-intersect} is a log version of the Lagrangian Intersection Theorem \cite[Corollary~2.11]{PTVV13} (or the combination of \cite[Theorem~2.9]{PTVV13} and \cite[Theorem~2.22]{Calaque19}).
The shifted log symplectic structure introduced in Section \ref{sec:shift-log-symp} is obtained by specializing the twisted symplectic structure in \cite{CS24} to the case where the base stack is $\cA$, and the line bundle is the universal line bundle $L_\cA$ over $\cA$. One may also view the shifted log symplectic geometry in Section \ref{sec:shift-log-symp} as a derived analogue of $b^m$-symplectic geometry (see \cite{BDMOP19}). It would also be interesting to explore the relation among shifted symplectic structures \cite{PTVV13}, twisted symplectic structures \cite{CS24}, and shifted log symplectic structures (see the end of Section \ref{sec:shift-log-symp}).

\subsection{Logarithmic Cosections and Virtual Cycles}

A $(-1)$-shifted log function $w_\alpha$ on a quasi-smooth derived log Deligne--Mumford stack $(X, \alpha_X)$ naturally induces a log cosection
$$
\sigma^\log: \Ob^\log_{X_\cl} \to \cO_{X_\cl}(\D_{X_\cl})
$$
of the log obstruction sheaf $\Ob^\log_{X_\cl}$ (see Definition \ref{def:log-cosection}). Moreover, the derived degeneracy locus of the log cosection $\sigma^\log$ induced by $w_\alpha$ is defined to be the derived log critical locus $X(dw_\alpha)$. As mentioned in Remark \ref{rmk:shift-log-function}, the definition of log cosections is directly inspired by the twisted cosection in Log GLSM \cite{CJR21}.

From a $(-1)$-shifted log function $w_\alpha$ and its induced log cosection $\sigma^\log$, we may construct a log-meromorphic cosection $\widetilde{\sigma^\log}$ (see Definition \ref{def:log-mero-cosection}), and a reduced cosection $\sigma^\red$ (see Definition \ref{def:red-cosection}).

A log-meromorphic cosection $\widetilde{\sigma^\log}$ is a (meromorphic) cosection in the sense of \cite{KiLi13}, whose degeneracy locus supports the usual virtual cycle $[X_\cl]^\vir$ (see Proposition \ref{prop:log-meromorphic-cosection-localized-virt-cycle}).

The construction of the reduced cosection $\sigma^\red$ is the most crucial part of the paper. Given a $(-1)$-shifted log function $w_\alpha$ on a quasi-smooth derived log Deligne--Mumford stack $(X, \alpha_X)$, we have the following commutative diagram, where each square is a homotopy Cartesian diagram in $\LogDStk$.
$$
\begin{tikzcd}
{\D_X^\red} \arrow[r, hook] \arrow[d] \arrow[rd, "\lrcorner_{\text{\tiny h}}", phantom, very near start] & {X^\tred} \arrow[r] \arrow[d] \arrow[rd, "\lrcorner_{\text{\tiny h}}", phantom, very near start] & {X^\red} \arrow[r] \arrow[d, "w_{\alpha}^\red"'] \arrow[rd, "\lrcorner_{\text{\tiny h}}", phantom, very near start] & {X} \arrow[d, "w_{\alpha}"'] \\
{B\bG_m} \arrow[r, hook]           & {\cA} \arrow[r, "0"]           & {\mathbb{A}^1_{\mathcal{A}}[-1]} \arrow[r, "\tilde{s}"]           & {\mathcal{L}_{\mathcal{A}}[-1]}          
\end{tikzcd}
$$
Here, the morphism $\tilde{s}$ is induced by multiplication with the coordinate on $\cA$.
We call $X^\red$ the $w_\alpha$-reduction of $X$, and $\D_X^\red$ the $w_\alpha$-reduction of $\D_X$. The intermediate object $X^\tred$ leading to $\D_X^\red$ plays a role in the proof of the Decomposition Theorem \ref{thm:intro-decompose-virt-cycle}. The reduced cosection $\sigma^\red$ is naturally induced by the $(-1)$-shifted function $w_{\alpha}^\red$ on $X^\red$.

Note that $X^\red$ and $\D_X^\red$ are not necessarily quasi-smooth. However, as in \cite{STV15, AKLPR22P, ParkNotes25} (see Example \ref{eg:origin-red-POT}), we can overcome this difficulty and obtain the following theorem.

\begin{theorem}[Reduced Virtual/Boundary Cycles, see Corollary \ref{cor:red-POT}, Remark \ref{rmk:X-red-being-qsmth}, Corollary \ref{cor:red-divisor-POT}, and Remark \ref{rmk:what-does-log-cosection-surj-imply}]
\label{thm:intro-red-virt-or-bound-cycle}
    Let $w_{\alpha}: X \to L_{\cA}[-1]$ be a $(-1)$-shifted log function on a separated quasi-smooth derived log Deligne--Mumford stack $(X, \alpha_X)$. If the log cosection $\sigma^\log: \Ob^{\log}_{X_\cl} \to \cO_{X_\cl}(\D_{X_\cl})$ associated with $w_{\alpha}$ is surjective along $(\D_{X_\cl})_\cl = (\D_{X})_\cl$,\footnote{$\D_{X_\cl}$ is not necessarily classical with $X_\cl$ being classical; see the beginning of Section \ref{sec:3}.} then $X^\red$ and $\D_X^\red$ are quasi-smooth, and therefore induce perfect obstruction theories and the associated virtual cycles.
\end{theorem}
The perfect obstruction theory induced by $X^\red$ is called the reduced perfect obstruction theory on $X_\cl$, and the associated virtual cycle, denoted by $[X_\cl]^\red$, is called the reduced virtual cycle. Likewise, the perfect obstruction theory induced by $\D_X^\red$ is called the reduced perfect obstruction theory on $(\D_{X_\cl})_\cl = (\D_{X})_\cl$, and the associated virtual cycle, denoted by $[(\D_{X_\cl})_\cl]^\red$ or $[(\D_X)_\cl]^\red$, is called the reduced boundary cycle. 

See Appendix \ref{append:red-POT-in-general} for a more general construction of reduced perfect obstruction theories, which includes both the reduced perfect obstruction theory discuss above and the reduced perfect obstruction theory as considered in \cite{OkPa10, MaPa13, MPT10} as examples.

Combining the results related to log cosections, log-meromorphic cosections, and reduced cosections, we prove the following theorem, which may be viewed as a derived-geometric version of \cite[Theorem~6.5]{CJR21}.
\begin{theorem}[Decomposition Theorem for Virtual Cycles, see Theorem \ref{thm:decompose-virt-cycle}]
\label{thm:intro-decompose-virt-cycle}
    Let $w_{\alpha}: X \to L_{\cA}[-1]$ be a $(-1)$-shifted log function on a separated quasi-smooth derived log Deligne--Mumford stack $(X, \alpha_X)$. If the log cosection $\sigma^\log: \Ob^{\log}_{X_\cl} \to \cO_{X_\cl}(\D_{X_\cl})$ associated with $w_{\alpha}$ is surjective along $(\D_{X_\cl})_\cl = (\D_{X})_\cl$, then we have a relation of virtual cycles
$$
[X_\cl]^\vir = [X_\cl]^\red + i_*[(\D_{X_\cl})_\cl]^\red.
$$
    Here $i: (\D_{X_\cl})_\cl \hookrightarrow X_\cl$ is the natural closed embedding.
\end{theorem}

The decomposition theorem for virtual cycles (Theorem \ref{thm:intro-decompose-virt-cycle}) was the starting point for this paper.
We refer the reader to Example \ref{eg:toy-model-decompose-virt-cycle} and Example \ref{eg:local-model-decompose-virt-cycle} for an illustration of the main ideas behind the theorem and related constructions.

Compared with \cite[Theorem~6.5]{CJR21}, our decomposition theorem for virtual cycles does not require the existence of an auxiliary stack (see {\cite[Remark~6.2]{CJR21}}), thereby allowing for a broader range of potential applications.
In Example \ref{eg:der-self-intersect-P1-in-P2} and Section \ref{ss-sec:log-comp-der-hodg-bund}, we give concrete decompositions of virtual cycles that arise from the theorem.

\subsection{A Prototype for Derived Logarithmic GLSM}

This work culminates with a first step toward the derived geometry behind the reduced perfect obstruction theory on the moduli of UMD log R-maps.
We will consider a locally free sheaf $\cE$ on projective space $\bP^r$ together with a regular section $s \in H^0(\cE)$ that defines a smooth complete intersection.
Let $R\cW^E$ be the corresponding derived moduli of stable maps to projective space $\bP^r$ with $p$-fields, see \cite{KMMP25, OhTh24, KiPa25}.
Instead of constructing a derived version of the moduli of UMD R-maps compactification, we construct a simpler logarithmic compactification via a \emph{derived projective completion}.
This allows us to remain within the framework of Deligne--Faltings of rank-$1$ log geometry --- a similar approach is taken in the work of Oh \cite{Oh22}.

Following \cite{Hekking21} and \cite{Jiang25}, we show in Proposition \ref{prop:log-compact-derived-proj-completion} that derived projective completions are natural log compactifications. As a consequence, we obtain the following:
\begin{theorem}[A Log Compactification for a Special Class of Derived Weil Restrictions, see Corollary~\ref{cor:log-compact-derived-weil-restriction}]
\label{thm:intro-log-compact-derived-weil-restriction}
Assume that $p: C \to M$ is a flat proper quasi-smooth morphism of derived $1$-Artin stacks that is representable by derived schemes. Let $\cE^{\vee}$ be a perfect complex on $C$ of Tor-amplitude in $[-n,0]$ for some $n \in \bZ_{\ge 0}$, and $E := \Spec_{C} \Sym_{C} (\cE^{\vee})$ be a derived vector bundle over $C$. Then the derived Weil restriction $p_*E$ is equivalent to the derived vector bundle $\Spec_{M} \Sym_{M} \left( \left( \RR p_* \cE \right)^{\vee} \right)$, which admits a log compactification $$\bP_M(p_*E \oplus 1) := \Proj_{M} \Sym_{M} \left( \left( \RR p_* \cE \right)^{\vee} \oplus \cO_{M} \right).$$
\end{theorem}

As discussed in Section~\ref{sec:4}, such derived projective completions lead to examples of $(-1)$-shifted log functions and the decomposition theorem for virtual cycles.
In particular, the derived moduli $R\cW^E$ of stable maps with $p$-fields to $\bP^r$ is a derived Weil restriction satisfying the assumptions of Theorem~\ref{thm:intro-log-compact-derived-weil-restriction}.
Let $\overline{R\cW^E}^\log$ be the resulting derived projective completion.
The next theorem then provides a suitable $(-1)$-shifted log function:
\begin{theorem}[A Prototype for Derived Logarithmic GLSM, see Section \ref{sec:log-compact-der-moduli-p-fields}]
\label{thm:intro-prototype-der-log-GLSM}
There exists a relative $(-1)$-shifted log function
$$
\begin{tikzcd}
{\overline{R\cW^E}^\log} \arrow[rr, "\ell_{/\fMgn}^\log"] \arrow[rd] &    & {L_\cA[-1] \times^h \fMgn} \arrow[ld] \\
                         & {\cA \times^h \fMgn} &,
\end{tikzcd}
$$
associated to the section $s$.
Furthermore, the restriction $\ell_{/\fMgn}^\log\big|_{R\cW^E}$ is a relative $(-1)$-shifted function
$$
\begin{tikzcd}
{R\cW^E} \arrow[rr, "\ell_{/\fMgn}"] \arrow[rd] &    & {\bA^1[-1] \times^h \fMgn} \arrow[ld] \\
                         & {\fMgn} &              
\end{tikzcd}
$$
on the derived moduli $R\cW^E$ of stable maps with $p$-fields that agrees with the $(-1)$-shifted function in \cite{OhTh24}.
In particular, it induces precisely the (relative) cosection $\rho$ constructed in \cite[Section~4.4]{CJW21}.
\end{theorem}

\smallskip

We conclude by listing a few properties we expect to hold for this prototype for derived log GLSM:
\begin{enumerate}
    \item The classical degeneracy locus of the log cosection associated with $\ell_{/\fMgn}^\log$ is $\overline{\mathcal{M}}_{g,n}(Z(s)_\cl,\beta)$. In particular, the log cosection is surjective along the boundary, and hence $(\overline{R\cW^E}^\log)^\red$ is quasi-smooth.

    \item The virtual cycle associated with the reduced perfect obstruction theory induced by $(\overline{R\cW^E}^\log)^\red$ recovers the cosection localized virtual cycle associated with $\ell_{/\fMgn}$, and hence (up to a sign) recovers the virtual cycle associated with the perfect obstruction theory induced by $R\overline{\mathcal{M}}_{g,n}(Z(s)_\cl,\beta)$, which is the virtual cycle used in defining the Gromov-Witten invariants of $Z(s)_\cl$.

    \item Moreover, the reduced perfect obstruction theory induced by $(\overline{R\cW^E}^\log)^\red$ is $\bG_m$-equivariant, and hence virtual (torus) localization applies. This would lead to a virtual (torus) localization formula for the virtual cycle used in defining Gromov-Witten invariants of $Z(s)_\cl$.
\end{enumerate}
    
It would be interesting to extend the prototype from the stable maps with $p$-fields setting to FJRW theory and more general hybrid model GLSMs.

\subsection{Plan of the Paper}

In Section \ref{sec:1}, we collect preliminary material used throughout the paper. Readers familiar with this material may skip this section.
In Section \ref{sec:2}, we develop a version of derived log geometry centered around shifted log functions that will be used in Sections \ref{sec:3} and \ref{sec:4}. In particular, we prove Theorem \ref{thm:intro-log-lag-intersect} and Theorem \ref{thm:intro-log-compact-derived-weil-restriction}.
In Section \ref{sec:3}, we introduce log cosections and use them for the construction of reduced virtual cycles and reduced boundary cycles (Theorem \ref{thm:intro-red-virt-or-bound-cycle}) and the decomposition of virtual cycles (Theorem \ref{thm:intro-decompose-virt-cycle}).
In Section \ref{sec:4}, we first construct a $(-1)$-shifted log function on log-compactified derived Hodge bundles, and then construct a $(-1)$-shifted log function on log-compactified derived moduli of stable maps with $p$-fields (Theorem \ref{thm:intro-prototype-der-log-GLSM}).

There are two appendices.
In Appendix \ref{append:1}, we give a general construction of reduced perfect obstruction theories, while
in Appendix \ref{append:2}, we compare the duals of log obstruction sheaves and the duals of obstruction sheaves.

\subsection*{Conventions and Notations}

We refer to \cite{Fulton98, Vistoli89, Kresch99} for background on intersection theory, and assume that all Chow groups have rational coefficients.
We refer to \cite{LurieThesis04, Toen09, ToenSurvey14, KhanNotes23} for background on derived algebraic geometry.

We work over the field $k = \bC$ of complex numbers. For simplicity, we write $\pt = \Spec k$, and let $\mathcal{S}$ denote the $\infty$-category of $\infty$-groupoids.
Since we are working over $\bC$, connective commutative differential graded algebras, simplicial commutative rings (or equivalently animated commutative rings), and connective $\mathbb{E}_{\infty}$-rings are all equivalent. We will refer to them simply as CDGAs, and denote the $\infty$-category of CDGAs by $\DRing$.
We are using cohomological conventions unless otherwise stated. In particular, we are using cohomological conventions for Tor-amplitude, following {\cite[Definition~1.3]{STV15}}.

\paragraph{}
A derived pre-stack is an $\infty$-functor from the $\infty$-category of CDGAs to the $\infty$-category of $\infty$-groupoids. A derived stack is a derived pre-stack satisfying etale descent. The $\infty$-category of derived pre-stacks is denoted by $\DPStk$, and the $\infty$-category of derived stacks is denoted by $\DStk$.

As in \cite{PTVV13}, all derived Artin stacks\footnote{See \cite[Definition~4.4]{Toen09}.} are assumed to be (homotopically) locally of finite presentation.
A derived Deligne--Mumford stack is a derived Artin stack whose classical truncation is a classical Deligne--Mumford stack.
A derived scheme is a derived Artin stack whose classical truncation is a classical scheme.

By a classical ($1$-)Artin stack, we mean a classical ($1$-)stack in the sense of \cite{Artin74}. Under this convention, the classical truncation of a derived Artin stack is not necessarily a classical (1-)Artin stack. We denote the smooth classical ($1$-)Artin stack $[\bA^1/\bG_m]$ by $\cA$.

Every smooth derived scheme is a smooth classical scheme (see \cite[Lemma~2.1.2(b)]{ArGa15}). We will refer to it as a smooth (classical) scheme.

\paragraph{}
Let $X$ be a derived Artin stack. Its classical truncation is denoted by $X_\cl$. The stable $\infty$-category of quasi-coherent complexes on $X$ is denoted by $\QCoh(X)$; the cotangent complex of $X$ is denoted by $\bL_X \in \QCoh(X)$. Since we assume that derived Artin stacks are (homotopically) locally of finite presentation, $\bL_X$ is a perfect complex, i.e. a dualizable object in $\QCoh(X)$. The dual $\bT_X := \bL_X^\vee$ of $\bL_X$ is called the tangent complex of $X$.

For a perfect complex $\cE$ on $X$, there is a derived Artin stack
\begin{align*}
    \V_X(\cE) : \DSchm_{/X}^{\op} &\longrightarrow \mathcal{S} \\
    (u: T \to X) &\longmapsto \map_{\QCoh(T)}\left(\LL u^* \cE, \cO_{T} \right)
\end{align*}
over $X$. $\V_X(\cE)$ is called the linear derived stack over $X$ associated with $\cE$. For simplicity, we use $E$ to denote $\V_X(\cE^\vee)$. Note that we have $\bL_{E/X} \simeq \cE^\vee|_{E}$. Given a section $s: \cO_X \to \cE$, there is a morphism $X \to E$, also denoted by $s$. The derived zero locus $X(s)$ of the section $s: \cO_X \to \cE$ is defined via the following homotopy Cartesian diagram.
$$
\begin{tikzcd}
{X(s)} \arrow[r] \arrow[d] \ar[dr, phantom, "\lrcorner_{\text{\tiny h}}", very near start] & {X} \arrow[d,"s"] \\
{X} \arrow[r, "0"]           & {E}          
\end{tikzcd}
$$
See Appendix \ref{append:cotang-complex-der-zero-loci} for the description of $\bL_{X(s)}$.
If $\cE^\vee$ is moreover Tor-amplitude $\le 0$, then there is an equivalence
$
E := \V_X(\cE^\vee) \cong \Spec_X \Sym_X \left( \cE^\vee \right).
$

\paragraph{}
Let $X$, $Y$, and $Z$ be derived stacks. We denote the homotopy fibered product of $X$ and $Y$ over $Z$ by $X \times^h_Z Y$. Its classical truncation is the classical fibered product $X_\cl \times_{Z_\cl} Y_\cl$ of $X_\cl$ and $Y_\cl$ over $Z_\cl$. 

In general, the three fibered products $$X \times^h_Z Y, \quad X_\cl \times_{Z_\cl} Y_\cl, \quad X_\cl \times^h_{Z_\cl} Y_\cl$$ are different. Note that the classical truncation of $X_\cl \times^h_{Z_\cl} Y_\cl$ is also $X_\cl \times_{Z_\cl} Y_\cl$. If the morphism $Y_\cl \to Z_\cl$ is moreover flat, then the natural morphism $X_\cl \times_{Z_\cl} Y_\cl \to X_\cl \times^h_{Z_\cl} Y_\cl$ is an equivalence, i.e. these two fibered products agree.

\subsection*{Acknowledgments}

The authors thank Nachiketa Adhikari, Dima Arinkin, Younghan Bae, Qile Chen, Jeremiah Heller, Leo Herr, Sheldon Katz, John Pardon, James Pascaleff, and Rachel Webb for helpful discussions. The second author thanks Timmy Mengzhudong Feng for many discussions on derived algebraic geometry, he thanks Zijian Han help suggestions on an earlier draft of Sections~\ref{sec:basic-setup-der-log}--\ref{sec:shift-log-symp}, and he thanks Zhiyu Liu and Nikolay Sheshko for pointing out useful references. He also thanks Damien Calaque and Pavel Safronov for addressing his questions about \cite{CS24}.

The authors thank the Simons Center for Geometry and Physics for its hospitality, where part of this work was conducted.

The authors were partially supported by NSF grant DMS-2422291.
This work was supported in part by a grant of access to OpenAI models through the ChatGPT for Academic Researchers program.

\paragraph{AI disclosure}

The authors used the generative AI tools, ChatGPT and Gemini, for English-language polishing, literature discovery, and to perform a final proof-read of the paper before posting. The authors take full responsibility for the mathematical content and claims in this work.

\section{Preliminaries}\label{sec:1}

This section collects preliminary material used throughout the paper, which should be known to experts.

Section \ref{sec:POT} discusses the relation between perfect obstruction theories and quasi-smoothness (in particular, Remark \ref{re:functor-of-POT} and Example \ref{eg:local-model-POT}). Such a relation is the foundation of the entire paper.
Section \ref{sec:shift-fun-shift-diff-form} concerns shifted functions and exact shifted differential $1$-forms. By analogy, we will introduce shifted log functions and exact shifted log differential $1$-forms in Section \ref{sec:shift-log-fun-shift-log-diff-form}. This motivates us to study shifted log symplectic structure in Section \ref{sec:shift-log-symp}.
Section \ref{sec:cosection-localization} is related to Section \ref{sec:3}. For instance, the definitions of log cosections and their degeneracy loci in Section \ref{sec:3} are inspired by Definition \ref{def:cosection} and Definition \ref{def:degeneracy-loci}, respectively; Example \ref{eg:local-model-decompose-virt-cycle} builds on Example \ref{eg:local-model-cosection-localization}.
Section \ref{sec:der-weil-res} reviews the basic theory of derived Weil restrictions needed for Corollary \ref{cor:log-compact-derived-weil-restriction} and Section \ref{sec:4}.

We hope this section convinces the reader that some things that are technical in the classical setting turn out to be quite natural in the derived geometry. For instance, see Proposition \ref{prop:rel-qsmth-rel-POT} and Proposition \ref{prop:rel-cosection-rel-POT}.

We refer the reader to \cite{OhNotes25, ParkNotes25} for general discussions of modern viewpoints and recent developments in virtual techniques in enumerative geometry (also known as virtual geometry).

\subsection{Perfect Obstruction Theories}\label{sec:POT}

In short, this subsection is to make precise the slogan that perfect obstruction theories are shadows of quasi-smoothness.

We refer the reader to \cite[Section~2]{ArGa15}\cite[Section~2]{KhRy25} for more discussions on quasi-smoothness, and to \cite[Proposition~3.2.18]{LurieThesis04}\cite[Section~8.7]{KhanNotes23} for more discussions on homotopical $n$-smoothness.

\begin{proposition}[{\cite[Proposition~1.2]{STV15}}]
\label{prop:cotangent-complex-obstruction-theory}

    Let $X$ be a derived geometric stack\footnote{\cite[Corollary~2.2.3.3]{ToVe08} guarantees the existence of cotangent complexes of derived geometric stacks.}, and $X_{\cl}$ be the classical truncation of $X$. Then the natural closed immersion $\tau: X_{\cl} \hookrightarrow X$ induces a morphism $\bL\tau: \bL_X|_{X_{\cl}} \to \bL_{X_{\cl}}$,
    whose cofiber $\bL_{X_{\cl}/X}$ has vanishing cohomologies in degrees $\ge -1$.
\end{proposition}

\begin{definition}[{\cite[Definition~8.7.3]{KhanNotes23}}]

A morphism $f: X \to Y$ between derived stacks is \emph{homotopically $n$-smooth} if its classical truncation $f_\cl: X_\cl \to Y_\cl$ is (classically) locally of finite presentation and its cotangent complex $\bL_f$ is a perfect complex of Tor-amplitude\footnote{Recall that we are using cohomological conventions.} $\ge -n$ for some $n \in \bZ_{\ge 0}$. A morphism $f: X \to Y$ between derived stacks is \emph{homotopically smooth} if it is homotopically $n$-smooth for some $n \in \bZ_{\ge 0}$. A morphism $f: X \to Y$ between derived stacks is \emph{quasi-smooth} if it is homotopically $1$-smooth.

A derived stack $X$ is \emph{homotopically $n$-smooth} if the morphism $X \to \pt$ is homotopically $n$-smooth. A derived stack $X$ is \emph{homotopically smooth} if it is homotopically $n$-smooth for some $n \in \bZ_{\ge0}$. A derived stack $X$ is \emph{quasi-smooth} if it is homotopically $1$-smooth.
\end{definition}

\begin{re}
    A derived Artin stack that is (homotopically) of finite presentation is homotopically smooth.
\end{re}

\begin{re}\label{re:property-of-n-smth}
    Homotopically $n$-smooth morphisms are stable under base change and composition.
\end{re}

\begin{re}
    A morphism between homotopically $n$-smooth derived Deligne--Mumford stacks is homotopically $(n+1)$-smooth.
\end{re}

\begin{corollary}[{\cite[Corollary~1.4]{STV15}}]
\label{cor:qsmooth-POT}
    
    Let $X$ be a quasi-smooth derived Deligne--Mumford stack, and $X_{\cl}$ be the classical truncation of $X$. Then the natural closed embedding $\tau: X_{\cl} \hookrightarrow X$ induces a perfect obstruction theory $\bL\tau : \bL_{X}|_{X_{\cl}} \to \bL_{X_{\cl}}$ on $X_{\cl}$ in the sense of Behrend--Fantechi \cite{BeFa97}.
\end{corollary}

\begin{re}\label{re:functor-of-POT}
    The takeaway of Corollary \ref{cor:qsmooth-POT} is that the functoriality of perfect obstruction theories follows directly from the functoriality of cotangent complexes (see \cite[Definition~5.8]{BeFa97} and \cite[Section~1.2]{STV15}).
\end{re}

\begin{re}\label{rmk:et-local-q-smth-derived-DM-stack}
    A derived stack $X$ is a derived Deligne--Mumford stack if and only if there exist an etale surjective morphism $W \to X$ with $W$ being a derived scheme. If $X$ is moreover quasi-smooth, then $W$ is also quasi-smooth. Hence, by \cite[Corollary~2.1.6]{ArGa15} (see also {\cite[Proposition~2.3.14]{KhRy25}}), $W$ is Zariski-locally a derived zero locus $Z(s)$ of a section $s:=(f_1, \dots, f_m): \cO_{\bA^n} \to \cO_{\bA^n}^{\oplus m}$, defined via the following homotopy Cartesian diagram. 
$$
\begin{tikzcd}
{Z(s)} \ar[dr, phantom, "\lrcorner_{\text{\tiny h}}", very near start] \arrow[r] \arrow[d] & {\bA^n} \arrow[d, shift left, "s"] \\
{\bA^n} \arrow[r,"0"]  & {\bA^m_{\bA^n} = \bA^{n+m}}        
\end{tikzcd}
$$
    To sum up, if $X$ is a quasi-smooth derived Deligne--Mumford stack, then $X$ is etale-locally a derived zero locus $Z(s)$ of a section $s:=(f_1, \dots, f_m): \cO_{\bA^n} \to \cO_{\bA^n}^{\oplus m}$.
\end{re}

\begin{example}[Local Model for Perfect Obstruction Theory, {\cite[Section~2]{GrPa99}}]
\label{eg:local-model-POT}
    
    Let $X$ be a smooth (classical) scheme of dimension $d$, and $\cE$ be a locally free sheaf of rank $r$ on $X$. Let $s$ denote a section $\cO_X \to \cE$.
    Abusing notation, we also use $s$ to denote the morphism $X \to E:=\Spec_X \Sym_X (\cE^{\vee})$ induced by $s: \cO_X \to \cE$.
    
    The derived zero locus $X(s)$ is defined via the homotopy Cartesian diagram
    $$
    \begin{tikzcd}
{X(s)} \arrow[r] \arrow[d] \ar[dr, phantom, "\lrcorner_{\text{\tiny h}}", very near start] & {X} \arrow[d,"s"] \\
{X} \arrow[r, "0"']           & {E}          .
    \end{tikzcd}
    $$
    
    By Proposition \ref{prop:cotangent-complex-derived-zero-loci}, there is a quasi-isomorphism
    $$\bL_{X(s)} \simeq [d\bar{s}: \cE^{\vee}|_{X(s)} \to \Omega^1_X|_{X(s)}].$$
    Note that $\bL_{X(s)}$ is a perfect complex of Tor-amplitude in $[-1,0]$, so $X(s)$ is quasi-smooth.
    
    Let $X(s)_{\cl}$ denote the classical truncation of $X(s)$. There is a natural closed embedding $\tau: X(s)_{\cl} \hookrightarrow X(s)$. Applying the cotangent complex functor to $\tau$ induces a perfect obstruction theory
    $$\bL\tau: \bL_{X(s)}|_{X(s)_{\cl}} \simeq [d\bar{s}|_{X(s)_{\cl}}: \cE^{\vee}|_{X(s)_{\cl}} \to \Omega^1_X|_{X(s)_{\cl}}] \to \bL_{X(s)_{\cl}}$$ on $X(s)_{\cl}$ due to Corollary \ref{cor:qsmooth-POT}.
    
     Moreover, one can see that the virtual cycle $[X(s)_{\cl}]^{\text{vir}}$ induced by the perfect obstruction theory $\bL\tau$ coincides with both the refined Euler class (also known as localized top Chern class in \cite{Fulton98}) $e(E, s)$ of $E$ with respect to $s$, and the intersection product $[X \cdot_{E} X]$ arising from the classical Cartesian diagram
    $$
    \begin{tikzcd}
{X(s)_{\cl}} \arrow[r] \arrow[d] \ar[dr, phantom, "\lrcorner", very near start] & {X} \arrow[d,"s"] \\
{X} \arrow[r, "0"']           & {E}       .   
    \end{tikzcd}
    $$
    Note that $\dim\ [X(s)_{\cl}]^{\text{vir}} = \vdim\ X(s) = d-r$. In particular, if $s$ is regular, then $\tau: X(s)_{\cl} \overset{\cong}{\longrightarrow} X(s)$ is an equivalence of derived schemes. $\blacksquare$
    
\end{example}

\begin{proposition}[Passing from Relative to Absolute, see {\cite[Proposition~A.1(1)]{BrLe00}}]
\label{prop:rel-qsmth-rel-POT}

    Let $f: X \to Y$ be a quasi-smooth morphism from a derived Deligne--Mumford stack $X$ to a smooth classical $1$-Artin stack $Y$, and $X_{\cl}$ be the classical truncation of $X$. Then 
    
    $(1)$ $X$ is quasi-smooth;
    
    $(2)$ the natural closed embedding $\tau: X_{\cl} \hookrightarrow X$ induces a (relative) perfect obstruction theory $$\bL\tau_{/Y}: \bL_{X/Y}|_{X_{\cl}} \to \bL_{X_{\cl}/Y}$$ on $X_{\cl}$ in the sense of Behrend--Fantechi \cite[Section~7]{BeFa97};

    $(3)$ the virtual cycles associated with the (absolute) perfect obstruction theory $\bL\tau : \bL_{X}|_{X_{\cl}} \to \bL_{X_{\cl}}$ and the (relative) perfect obstruction theory $\bL\tau_{/Y}: \bL_{X/Y}|_{X_{\cl}} \to \bL_{X_{\cl}/Y}$ coincide.

\end{proposition}

\begin{proof}
    $(1)$ is due to Remark \ref{re:property-of-n-smth}. $(2)$ is due to Proposition \ref{prop:cotangent-complex-obstruction-theory}. 

    Now we proceed to prove $(3)$. For the morphism $f: X \to Y$, there is a commutative diagram as follows.
$$
\begin{tikzcd}
{X_\cl} \arrow[r, "f_\cl"] \arrow[d, hook, "\tau"'] & {Y_\cl = Y} \arrow[d, "\id"] \\
{X} \arrow[r, "f"]           & {Y}        
\end{tikzcd}
$$
    Applying the cotangent complex functor to the above diagram, we obtain the following commutative diagram, where rows are (co)fiber sequences.
$$
\begin{tikzcd}
{\bL_{Y}|_{X_\cl}} \arrow[r] \arrow[d, "\simeq"] & {\bL_X|_{X_\cl}} \arrow[d, "\bL\tau"] \arrow[r] & {\bL_{X/Y}|_{X_\cl}} \arrow[r, "+1"] \arrow[d, "\bL\tau_{/Y}"] & {} \\
{\bL_{Y}|_{X_\cl}} \arrow[r]           & {\bL_{X_\cl}} \arrow[r]           & {\bL_{X_\cl/Y}} \arrow[r, "+1"]           & {}          
\end{tikzcd}
$$
    As a result, $Y$ being smooth implies that two perfect obstruction theories $\bL\tau$ and $\bL\tau_{/Y}$ induce the same virtual cycle by {\cite[Proposition~A.1(1)]{BrLe00}}. 
\end{proof}

\begin{re}
    Proposition \ref{prop:rel-qsmth-rel-POT} is likely to be insufficient for applications of derived log geometry to enumerative geometry, since the base stacks are usually only equidimensional (or log smooth) rather than smooth. However, it suffices for the purpose of this paper, because for the derived log geometry used here (see Section \ref{sec:2}), the base stack $\cA:= [\bA^1/\bG_m]$ is smooth.
\end{re}

\begin{re}
    See \cite[Section~5]{BeFa97} for the original construction of virtual cycles via absolute perfect obstruction theories; see \cite[Section~7]{BeFa97} for the original construction of virtual cycles via relative perfect obstruction theories. Other relevant references include \cite{Manolache12,KiPa21,Khan19,AKLPR25}.
\end{re}

\subsection{Shifted Functions and Exact Shifted Differential 1-Forms}\label{sec:shift-fun-shift-diff-form}

We refer the reader to \cite[Section~3.3]{ToVe25} and \cite{OhTh24, ParkNotes25, KiPa25} for other discussions on shifted functions and shifted differential $1$-forms.

\begin{definition}
    Let $n \in \bZ_{\ge 0}$. $k[n]$ is the complex concentrated in degree $-n$ with $k$ being the only non-zero term. The \emph{$(-n)$-shifted affine line $\bA^1[-n]$} is defined to be $\Spec \Sym (k[n])$.
\end{definition}

\begin{re}
    When $n=0$, $\bA^1[0]$ is just the affine line $\bA^1$. When $n \ge 1$, $\bA^1[-n]$ fits into the homotopy Cartesian diagram
    $$
\begin{tikzcd}
{\bA^1[-n]} \ar[dr, phantom, "\lrcorner_{\text{\tiny h}}", very near start] \arrow[r] \arrow[d] & {\pt} \arrow[d] \\
{\pt } \arrow[r]  & {\bA^1[-n+1]}   .
\end{tikzcd}
    $$
    In particular, $\bA^1[-1]$ fits into the homotopy Cartesian diagram
$$
\begin{tikzcd}
{\bA^1[-1]} \ar[dr, phantom, "\lrcorner_{\text{\tiny h}}", very near start] \arrow[r] \arrow[d] & {\pt} \arrow[d, "0"'] \\
{\pt } \arrow[r,"0"]  & {\bA^1}   .
\end{tikzcd}
$$
    In other words, $\bA^1[-1]$ is the derived self-intersection of the origin in $\bA^1$.
\end{re}

\begin{re}
    Note that when $n \ge 1$, the classical truncation of $\bA^1[-n]$ is $\pt$.
\end{re}

\begin{definition}\label{def:shift-fun-shift-diff}
    For a derived Artin stack $X$, a \emph{$(-n)$-shifted function on $X$} is a morphism $w: X \to \bA^1[-n]$. Applying the cotangent complex functor to $w$ induces a morphism $$\bL w: \bL_{\bA^1[-n]}|_X \simeq \cO_X[n] \to \bL_X.$$ We then define the \emph{exact $(-n)$-shifted differential $1$-form associated with $w$} to be $$dw:= \bL w[-n]: \cO_X \to \bL_X[-n].$$
\end{definition}

\begin{re}\label{rmk:shifted-function-and-homotopy}
    Note that, for $n \ge 1$, a $(-n)$-shifted function encodes subtle homotopical data that does not appear in classical algebraic geometry. For simplicity, we consider a $(-1)$-shifted function $w: X \to \bA^1[-1]$ on a derived scheme $X$. Algebraically speaking, there are equivalences of $\infty$-groupoids
\begin{align*}
    &\map_{\DStk}(X,\bA^1[-1]) \cong \map_{\DRing}(\Sym(k[1]),\Gamma(X, \cO_X)) \\
    \cong& \map_{\QCoh(\pt)}(k[1],\Gamma(X, \cO_X)) \cong \map_{\QCoh(X)}(\cO_X[1],\cO_X),
\end{align*}
    so that $w \in \pi_0\left( \map_{\DStk}(X,\bA^1[-1]) \right)$ corresponds to an element in $\pi_{1}(\Gamma(X, \cO_X))$\footnote{We are using homological conventions here. If using cohomological conventions, we will write $H^{-1}(\Gamma(X, \cO_X))$ instead.}. Geometrically speaking, $w: X \to \bA^1[-1]$ is determined by a homotopy $\zeta$ appearing in the following commutative diagram, where the square diagram is a homotopy Cartesian diagram.
$$
\begin{tikzcd}
	X \\
	& \bA^1[-1] & \pt \\
	& \pt & \bA^1
	\arrow[""{name=0, anchor=center, inner sep=0}, shift right=-2, bend right=-10, from=1-1, to=2-3]
	\arrow[""{name=1, anchor=center, inner sep=0}, shift right=2, bend right=10, from=1-1, to=3-2]
	\arrow[from=2-2, to=2-3]
	\arrow[from=2-2, to=3-2]
	\arrow["\lrcorner_{\text{\tiny h}}"{anchor=center, pos=0.125}, draw=none, from=2-2, to=3-3]
	\arrow[from=2-3, to=3-3]
	\arrow[from=3-2, to=3-3]
	\arrow[shorten <=10pt, shorten >=5pt, shift right=3, Rightarrow, from=0, to=1, "\zeta"'{pos=0.6}]
\end{tikzcd}
$$
\end{re}

\begin{definition}\label{rmk:rel-shifted-function}
    Let $n \in \bZ_{\ge 0}$, and $S$ be a derived Artin stack. The 
    \emph{relative $(-n)$-shifted affine line $\bA^1_S[-n]$ over $S$} is defined to be $\Spec_S \Sym_S (\cO_S[n])$. For a derived Artin stack $X$ over $S$, a \emph{relative $(-n)$-shifted function on $X$ over $S$} is a morphism $\widetilde{w} :X \to \bA^1_S[-n]$ over $S$, i.e. a commutative diagram
$$
\begin{tikzcd}
{X} \arrow[rr, "\widetilde{w}"] \arrow[rd] &    & {\bA^1_S[-n]} \arrow[ld] \\
                         & {S} &              .
\end{tikzcd}
$$
    Applying the relative cotangent complex functor to $\widetilde{w}$ induces a morphism $$\bL\widetilde{w}_{/S}: \bL_{\bA_S^1[-n]/S}|_X \simeq \cO_X[n] \to \bL_{X/S}.$$
    We then define the \emph{relative exact $(-n)$-shifted differential $1$-form associated with $\widetilde{w}$} to be $$d\widetilde{w}:= \bL \widetilde{w}_{/S}[-n]: \cO_X \to \bL_{X/S}[-n].$$
\end{definition}

We conclude with an example of an affine derived scheme whose $0$- and $(-1)$-shifted functions may be described explicitly in terms of algebraic data.

\begin{example}[A Specific Affine Derived Scheme]
\label{eg:aff-derived-zero-locus}

    Let $\bA^3 := \Spec k[x_1,x_2,x_3]$, and $s := (f_1,f_2) \in k[x_1,x_2,x_3]^{\oplus 2}$. Note that $s$ is equivalent to a section $\cO_{\bA^3} \to \cO_{\bA^3}^{\oplus 2}$, also denoted as $s$. The derived zero locus $Z(s)$ of the section $s: \cO_{\bA^3} \to \cO_{\bA^3}^{\oplus 2}$ is defined via the following homotopy Cartesian diagram.
$$
\begin{tikzcd}
{Z(s)} \ar[dr, phantom, "\lrcorner_{\text{\tiny h}}", very near start] \arrow[r] \arrow[d] & {\bA^3} \arrow[d, "s"] \\
{\bA^3 } \arrow[r,"0"]  & {\bA^2_{\bA^3}}        
\end{tikzcd}
$$
    
    By Proposition \ref{prop:cotangent-complex-derived-zero-loci}, there is a quasi-isomorphism
    $$\bL_{Z(s)} \simeq \cofib \left(d\bar{s}: \cO_{Z(s)}^{\oplus2} \to \cO_{Z(s)}^{\oplus3} \right).$$
    Note that $\bL_{Z(s)}$ has Tor-amplitude in $[-1,0]$.

    Using the Koszul resolution, we conclude that $\cO_{Z(s)}$ is the CDGA $k[x_1,x_2,x_3;\nu_1,\nu_2]$, whose underlying complex is the Koszul complex\footnote{See \cite[Section~2.2]{Vezzosi20}.}
$$k[x_1,x_2,x_3]\langle \nu_1\nu_2\rangle \to k[x_1,x_2,x_3]\langle \nu_1, \nu_2\rangle \to k[x_1,x_2,x_3], $$
    where $\nu_1 \mapsto f_1$,  $\nu_2 \mapsto f_2$, and $\nu_1\nu_2 \mapsto f_1\nu_2 - \nu_1f_2$ (determined by the Leibniz rule). Note the cohomological degrees $|x_1|=|x_2|=|x_3|=0$ and $|\nu_1|=|\nu_2|=-1$. Note that the classical truncation $Z(s)_\cl$ of $Z(s)$ is the classical zero locus of functions $s := (f_1,f_2)$ on $\bA^3$, since $$H^0(k[x_1,x_2,x_3;\nu_1,\nu_2]) = k[x_1,x_2,x_3]/(f_1,f_2).$$

    Thus, $\bL_{Z(s)}$ is quasi-isomorphic to the total complex of the bi-complex
$$
\cO_{Z(s)}\langle d\nu_1,d\nu_2\rangle \overset{d\bar{s}}{\longrightarrow} \cO_{Z(s)}\langle dx_1,dx_2,dx_3\rangle,
$$
    which can be explicitly written as
\[\begin{tikzcd}[column sep=6em]
	{k[x_1,x_2,x_3]\langle d\nu_1,d\nu_2\rangle} & {k[x_1,x_2,x_3]\langle dx_1,dx_2,dx_3\rangle} \\
	{k[x_1,x_2,x_3]\langle \nu_1d\nu_1,\nu_1d\nu_2,\nu_2d\nu_1,\nu_2d\nu_2\rangle} & {k[x_1,x_2,x_3]\langle \nu_1dx_1,\nu_1dx_2,\nu_1dx_3,\nu_2dx_1,\nu_2dx_2,\nu_2dx_3\rangle} \\
	{k[x_1,x_2,x_3]\langle \nu_1\nu_2d\nu_1,\nu_1\nu_2d\nu_2\rangle} & {k[x_1,x_2,x_3]\langle \nu_1\nu_2dx_1,\nu_1\nu_2dx_2,\nu_1\nu_2dx_3\rangle} .
    \arrow[from=1-1, to=1-2, "d\nu_j \mapsto df_j"]
	\arrow[from=2-1, to=1-1, "\nu_id\nu_j \mapsto f_id\nu_j"']
	\arrow[from=2-2, to=1-2, "\nu_idx_j \mapsto f_idx_j"']
	\arrow[from=2-1, to=2-2, "\nu_id\nu_j \mapsto \nu_idf_j"]
	\arrow[from=3-1, to=2-1, "\nu_1\nu_2d\nu_j \mapsto (f_1\nu_2-\nu_1f_2)d\nu_j"']
	\arrow[from=3-1, to=3-2, "\nu_1\nu_2d\nu_j \mapsto \nu_1\nu_2df_j"]
	\arrow[from=3-2, to=2-2, "\nu_1\nu_2dx_j \mapsto (f_1\nu_2-\nu_1f_2)dx_j"']
\end{tikzcd}\]
    Note the cohomological degrees $|dx_1| = |dx_2| = |dx_3| = 0$ and $|d\nu_1| = |d\nu_2| = -1$.

    Now consider a function $g: Z(s) \to \bA^1 = \Spec k[z_0]$ with $|z_0|=0$. This is equivalent to a morphism of CDGAs
\[\begin{tikzcd}
	{k[z_0]} & {k[x_1,x_2,x_3]} \\
	0 & {k[x_1,x_2,x_3]\langle \nu_1,\nu_2 \rangle} \\
	0 & {k[x_1,x_2,x_3]\langle \nu_1\nu_2 \rangle} .
	\arrow[from=1-1, to=1-2]
	\arrow[from=2-1, to=1-1]
	\arrow[from=2-2, to=1-2]
	\arrow[from=2-1, to=2-2]
	\arrow[from=3-1, to=2-1]
	\arrow[from=3-1, to=3-2]
	\arrow[from=3-2, to=2-2]
\end{tikzcd}\]
    In other words, $g: Z(s) \to \bA^1$ is determined by the image of $z_0$.
    Abusing notation, let $g$ denote the image of $z_0$. Note that $g$ is a representative of an element of $H^{0}\left(\cO_{Z(s)}\right) = \cO_{Z(s)_\cl}$.

    Applying the cotangent complex to $g: Z(s) \to \bA^1$, we obtain an exact differential $1$-form $$\bL_{\bA^1}|_{Z(s)} \simeq \cO_{Z(s)} \to \bL_{Z(s)}$$ that is determined by
\[
1 \mapsto dg=\sum_{i=1}^3 \frac{\partial g}{\partial x_i} \, dx_i \in k[x_1,x_2,x_3]\langle dx_1,dx_2,dx_3\rangle .
\]
    Abusing notation, let $dg$ denote the exact differential $1$-form $\cO_{Z(s)} \to \bL_{Z(s)}$. Note that $dg$ is a representative of an element in $H^0(\bL_{Z(s)}) = \Omega^1_{Z(s)_\cl}$. $\blacksquare$
    
    Next consider a $(-1)$-shifted function $w: Z(s) \to \bA^1[-1] \cong \Spec k[z_{-1}]$ with $|z_{-1}| = -1$. This is equivalent to a morphism of CDGAs
\[\begin{tikzcd}
	{k} & {k[x_1,x_2,x_3]} \\
	k\langle z_{-1} \rangle & {k[x_1,x_2,x_3]\langle \nu_1,\nu_2 \rangle} \\
	0 & {k[x_1,x_2,x_3]\langle \nu_1\nu_2 \rangle} .
	\arrow[from=1-1, to=1-2]
	\arrow[from=2-1, to=1-1, "0"]
	\arrow[from=2-2, to=1-2]
	\arrow[from=2-1, to=2-2]
	\arrow[from=3-1, to=2-1]
	\arrow[from=3-1, to=3-2]
	\arrow[from=3-2, to=2-2]
\end{tikzcd}\]
    In other words, $w: Z(s) \to \bA^1[-1]$ is determined by the image of $z_{-1}$.
    Abusing notation, let $w = h_1\nu_1 + h_2\nu_2 \in k[x_1,x_2,x_3]\langle \nu_1,\nu_2 \rangle$ denote the image of $z_{-1}$, where $h_1,h_2 \in k[x_1,x_2,x_3]$. Note that $w$ is a representative of an element in $H^{-1}\left(\cO_{Z(s)}\right)$, because $h_1f_1 + h_2f_2 = 0$.

    Applying the cotangent complex to $w: Z(s) \to \bA^1[-1]$, we obtain a morphism $$\bL_{\bA^1[-1]}|_{Z(s)} \simeq \cO_{Z(s)}[1] \to \bL_{Z(s)},$$
    which is equivalent to an exact $(-1)$-shifted differential $1$-form
    $$\cO_{Z(s)} \to \bL_{Z(s)}[-1]$$
    determined by
\begin{align*}
    1 \mapsto dw&=d(h_1\nu_1) + d(h_2\nu_2) \\ 
    &=\nu_1(dh_1) + h_1d\nu_1 + \nu_2(dh_2) + h_2d\nu_2 \\
    &=\nu_1\left(\sum_{i=1}^3 \frac{\partial h_1}{\partial x_i} \, dx_i\right) + \nu_2\left(\sum_{i=1}^3 \frac{\partial h_2}{\partial x_i} \, dx_i\right) + h_1d\nu_1 + h_2d\nu_2 \\
    & \in k[x_1,x_2,x_3]\langle \nu_1dx_1,\nu_1dx_2,\nu_1dx_3,\nu_2dx_1,\nu_2dx_2,\nu_2dx_3\rangle \oplus k[x_1,x_2,x_3]\langle d\nu_1,d\nu_2\rangle .
\end{align*}
    Abusing notation, let $dw$ denote the exact $(-1)$-shifted differential $1$-form $\cO_{Z(s)} \to \bL_{Z(s)}[-1]$. Note that $dw$ is a representative of an element in $H^{-1}(\bL_{Z(s)})$, because $d(h_1f_1 + h_2f_2) = 0$. $\blacksquare$

\end{example}

\begin{re}
    $Z(s)$ also fits into the following homotopy Cartesian diagram.
$$
\begin{tikzcd}
{Z(s)} \ar[dr, phantom, "\lrcorner_{\text{\tiny h}}", very near start] \arrow[r] \arrow[d] & {\bA^3} \arrow[d, "s"] \\
{\pt } \arrow[r,"0"]  & {\bA^2}        
\end{tikzcd}
$$
We then immediately see that
$$
\bL_{Z(s)} \simeq \bL_{\bA^3/\bA^2}|_{Z(s)} \simeq \cofib \left(\bL s: \bL_{\bA^2}|_{Z(s)} \simeq \cO_{Z(s)}^{\oplus2} \to \bL_{\bA^3}|_{Z(s)} \simeq \cO_{Z(s)}^{\oplus3} \right).
$$
\end{re}

\begin{re}
    If $f_1,f_2$ form a regular sequence in $k[x_1,x_2,x_3]$, then $Z(s) \cong Z(s)_\cl$ is automatically a classical scheme and, in particular, an l.c.i.
\end{re}

\subsection{Cosections and Cosection Localizations}\label{sec:cosection-localization}

See also \cite[Section~6]{AKLPR22P}\footnote{Note that we refer to a specific arXiv version, not the latest or published version.} and \cite{OhTh24, ParkNotes25, KiPa25} for discussions of cosections and cosection localizations from the viewpoint of derived geometry.

In this subsection, $X$ is assumed to be a quasi-smooth derived Deligne--Mumford stack.

\begin{definition}\label{def:cosection}
    Let $w : X \to \bA^1[-1]$ be a $(-1)$-shifted function on $X$. The induced morphism
    $$
    \bT w|_{X_\cl}: \bT_{X}|_{X_\cl} \to  \bT_{\bA^1[-1]}|_{X_\cl} \simeq \cO_{X_\cl}[-1]
    $$
    is a \emph{cosection of the perfect obstruction theory $\bL_{X}|_{X_\cl} \to \bL_{X_\cl}$} in the sense of \cite[Appendix~B]{CJW21}. Moreover, $H^1\left(\bT_X|_{X_{\cl}}\right)$ is the \emph{obstruction sheaf $\Ob_{X_\cl}$} in the sense of \cite{KiLi13}, and the morphism
    $$
    \sigma:= H^1\left(\bT w|_{X_{\cl}}\right): \Ob_{X_\cl} = H^1(\bT_X|_{X_{\cl}}) \to \cO_{X_{\cl}}
    $$
    is a \emph{cosection of the obstruction sheaf $\Ob_{X_\cl}$} in the sense of \cite{KiLi13}.
\end{definition}

\begin{re}
    $T^*[-1]X := \Spec_X \Sym_X \left(\bT_X[1]\right) \cong \Spec_X \Sym_X \left(\left(\bL_X[-1]\right)^{\vee}\right)$ is called the \emph{$(-1)$-shifted cotangent bundle on $X$}. Note that the classical truncation $(T^*[-1]X)_\cl$ of $T^*[-1]X$ is $$\Spec_{X_\cl} \Sym_{X_\cl} \left(H^0\left(\bT_X[1]|_{X_\cl}\right)\right) = \Spec_{X_\cl} \Sym_{X_\cl} \left(H^1\left(\bT_X|_{X_\cl}\right)\right) = \Spec_{X_\cl} \Sym_{X_\cl} \left( \Ob_{X_\cl} \right).$$
\end{re}

\begin{re}\label{rmk:class-shift-diff-form-is-cosection}
    Recall that $dw:= \bL w[-1]: \cO_X \to \bL_X[-1]$ is the exact $(-1)$-shifted differential $1$-form associated with $w$. Note that the morphism
    $$
    dw^\vee|_{X_\cl} : \bT_X|_{X_\cl}[1] \to \cO_{X_\cl}
    $$
    induces the cosection $\bT w|_{X_\cl}: \bT_{X}|_{X_\cl} \to  \cO_{X_\cl}[-1]$ after applying the shift functor $[-1]$. Moreover, the induced morphism
    $$
    H^0\left(dw^\vee|_{X_\cl}\right): \Ob_{X_\cl} = H^1(\bT_X|_{X_{\cl}}) \to \cO_{X_{\cl}}
    $$
    is precisely the cosection $\sigma: \Ob_{X_\cl} \to \cO_{X_{\cl}}$. Abusing notation, we also use $dw$ to denote the morphism
    $$
    X \to T^*[-1]X
    $$
    induced by $dw: \cO_X \to \bL_X[-1]$, and use $\sigma^\vee$ to denote the morphism
    $$
    X_\cl \to \Spec_{X_\cl} \Sym_{X_\cl} \left( \Ob_{X_\cl} \right)
    $$
    induced by the cosection $\sigma: \Ob_{X_\cl} \to \cO_{X_\cl}$. Then 
    $
    H^0\left(dw^\vee|_{X_\cl}\right) = \sigma: \Ob_{X_\cl} \to \cO_{X_{\cl}}
    $
    implies that 
    $$
    (dw)_\cl = \sigma^\vee: X_\cl \to (T^*[-1]X)_\cl = \Spec_{X_\cl} \Sym_{X_\cl} \left( \Ob_{X_\cl} \right).
    $$
\end{re}

\begin{definition}\label{def:degeneracy-loci}
    Let $w : X \to \bA^1[-1]$ be a $(-1)$-shifted function on $X$. Following the notation in Remark \ref{rmk:class-shift-diff-form-is-cosection}, we define the \emph{derived degeneracy locus of the cosection $\sigma: \Ob_{X_\cl} \to \cO_{X_{\cl}}$ associated with $w$} to be the derived zero locus $X(dw)$ of $dw: \cO_X \to \bL_X[-1]$, which is defined via the following homotopy Cartesian diagram.
$$
    \begin{tikzcd}
{X(dw)} \arrow[r] \arrow[d] \ar[dr, phantom, "\lrcorner_{\text{\tiny h}}", very near start] & {X} \arrow[d,"dw"] \\
{X} \arrow[r, "0"]           & {T^*[-1]X}          
    \end{tikzcd}
$$
    Then the classical truncation $X(dw)_\cl$ of $X(dw)$ is the \emph{classical degeneracy locus of the cosection $\sigma: \Ob_{X_\cl} \to \cO_{X_{\cl}}$} in the sense of \cite{KiLi13}, which fits into the following classical Cartesian diagram.
$$
    \begin{tikzcd}
{X(dw)_\cl} \arrow[r] \arrow[d] \ar[dr, phantom, "\lrcorner", very near start] & {X_\cl} \arrow[d,"(dw)_\cl = \sigma^\vee"] \\
{X_\cl} \arrow[r, "0"]           & {(T^*[-1]X)_\cl}          
    \end{tikzcd}
$$
\end{definition}

\begin{re}
    $X(dw)$ is the derived zero locus of the section $dw: \cO_X \to \bL_X[-1]$, which is equivalent to the derived zero locus $X(dw^\vee)$ of the cosection $dw^\vee: \bT_X[1] \to \cO_X$ (see \cite[Definition~4.13]{Jiang25}).
\end{re}

\begin{re}\label{rmk:virt-Lag-cycle-cosection-localized-virt-cycle}
    $X(dw)$ is also called the \emph{derived critical locus of the $(-1)$-shifted function $w: X \to \bA^1[-1]$}. By the Lagrangian intersection theorem (see \cite[Section~2.2]{PTVV13}), $X(dw)$ carries a $(-2)$-shifted symplectic structure\footnote{In particular, $X(dw)$ is not quasi-smooth in general.}. However, \cite{KiPa25} shows that the virtual Lagrangian cycle induced by the $(-2)$-shifted symplectic structure on $X(dw)$ recovers the cosection localized virtual cycle induced by the cosection $\sigma: \Ob_{X_{\cl}} \to \cO_{X_\cl}$.
\end{re}

\begin{re}\label{rmk:nachiketa}
  We refer the reader to \cite[Section~2]{AdSh24} for explicit computation of the underlying complexes of affine derived critical loci, and to \cite[Section~2]{OhTh24} for explicit computation of the classical truncation of derived critical loci.
\end{re}

\begin{example}[Local Model for Cosection Localization, {\cite[Section~2]{KiLi13}}{\cite[Section~3.4.1]{OhNotes25}}; {\cite[Section~4.2.1]{ParkNotes25}}]
\label{eg:local-model-cosection-localization}

Consider the following diagram.\footnote{Within this example, vector bundles and locally free sheaves share the same notation, e.g. $E := \Spec_X \Sym_X (E^\vee)$.}
$$
\begin{tikzcd}
                   & {E} \arrow[r, "\sigma"] \arrow[d]      & {\cO_X} \\
{Z(s)_\cl} \arrow[r, hook] & {X} \arrow[u, bend right, "s"'] &   
\end{tikzcd}
$$
Here, $X$ is a connected smooth (classical) scheme of dimension $d$; $E$ is a vector bundle of rank $r$; $Z(s)_\cl$ is the classical zero locus of a section $s: X \to E$. We assume that $\sigma \circ s = 0$ and $\sigma \ne 0$.

By Example \ref{eg:local-model-POT}, we have
$$
[Z(s)_\cl]^\vir = e(E,s) = [X \cdot_{E} X] \in A_{d-r}\left( Z(s)_\cl\right).
$$
Here, $[Z(s)_\cl]^\vir$ is the virtual cycle; $e(E,s)$ is the refined Euler class of $E$ with respect to $s$; $[X \cdot_{E} X]$ is the intersection product.

We define $Z(\sigma)_\cl$ via the following classical Cartesian diagram.
$$
\begin{tikzcd}
{Z(\sigma)_\cl} \arrow[r] \arrow[d] \ar[dr, phantom, "\lrcorner", very near start] & {X} \arrow[d,"\sigma^\vee"] \\
{X} \arrow[r, "0"']           & {E^\vee}          
\end{tikzcd}
$$
Then we take the blow-up $b: \widetilde{X} := Bl_{Z(\sigma)_\cl}X \to X$ of $X$ along $Z(\sigma)_\cl$ satisfying the classical Cartesian diagram
$$
\begin{tikzcd}
{D} \arrow[r, hook] \arrow[d] \ar[dr, phantom, "\lrcorner", very near start] & {\widetilde{X}} \arrow[d,"b"] \\
{Z(\sigma)_\cl} \arrow[r, hook]           & {X}          
\end{tikzcd}
$$
with $D$ being the exceptional divisor on $\widetilde{X}$.
Note that $b^*\sigma: b^*E \to \cO_{\widetilde{X}}$ vanishes along $D$. Hence, $b^*\sigma$ admits a factorization
$$
b^*E \overset{\overline{\sigma}}{\longrightarrow} \cO_{\widetilde{X}}(-D) \longhookrightarrow \cO_{\widetilde{X}}
$$
such that $\overline{\sigma}: b^*E \to \cO_{\widetilde{X}}(-D)$ is surjective. Moreover, $\sigma \circ s = 0$ implies that $\overline{\sigma} \circ b^*s = 0$.
Thus, there is a commutative diagram below, where the row is a short exact sequence of locally free sheaves.
$$
\begin{tikzcd}
{0} \arrow[r] & {K} \arrow[r] & {b^*E} \arrow[r, "\overline{\sigma}"]                               & {\cO_{\widetilde{X}}(-D)} \arrow[r] & {0} \\
             &              & {\cO_{\widetilde{X}}} \arrow[lu, dashed, "b^*s"] \arrow[u, "b^*s"'] \arrow[ru, "0"'] &              &   
\end{tikzcd}
$$
We define the cosection localized virtual cycle to be $$[Z(s)_\cl]^\vir_{\sigma} := b_*\left( -[D] \cap e(K,b^*s) \right) \in A_{d-r}\left( Z(s)_\cl \cap Z(\sigma)_\cl \right).$$

For now, let us assume that $\sigma: E \to \cO_X$ is surjective, and hence there is a short exact sequence of locally free sheaves as follows.
$$
0 \longrightarrow E(\sigma) \longrightarrow E \overset{\sigma}{\longrightarrow} \cO_X \to 0
$$

Thanks to MacPherson's graph construction (see \cite[Remark~5.1.1]{Fulton98} and \cite[Remark~4.1.7]{ParkNotes25}), the assumption $\sigma \circ s = 0$ implies that
$$
0: C_{Z(s)_\cl/X} \hookrightarrow E|_{Z(s)_\cl} \hookrightarrow E \overset{\sigma}{\to} \bA^1,
$$
where $C_{Z(s)_\cl/X}$ is the normal cone for $Z(s)_\cl \hookrightarrow X$. As a result, there is a closed embedding $C_{Z(s)_\cl/X} \hookrightarrow E(\sigma)|_{Z(s)_\cl}$. Consequently, we have
\begin{align*}
    [Z(s)_\cl]^\vir &= [X \cdot_{E} X] := 0^!_{E|_{Z(s)_\cl}}(C_{Z(s)_\cl/X}) \\
    &=e(\cO_X|_{Z(s)_\cl}) \cap 0^!_{E(\sigma)|_{Z(s)_\cl}}(C_{Z(s)_\cl/X}) \\
    &=0 .
\end{align*}
Note that the second line follows from the classical Cartesian diagram below and the excess intersection formula \cite[Theorem~6.3]{Fulton98}; the third line follows from the fact that the Euler class of a trivial vector bundle is trivial.
$$
\begin{tikzcd}
{Z(s)_\cl} \ar[drr, phantom, "\lrcorner", very near start] \arrow[rr] \arrow[d] & & {C_{Z(s)_\cl/X}} \arrow[d, hook] \\
{Z(s)_\cl} \ar[drr, phantom, "\lrcorner", very near start] \arrow[rr, "0_{E(\sigma)|_{Z(s)_\cl}}"] \arrow[d] & & {E(\sigma)|_{Z(s)_\cl}} \arrow[d, hook] \\
{Z(s)_\cl} \arrow[rr, "0_{E|_{Z(s)_\cl}}"]      &     & {E|_{Z(s)_\cl}}          
\end{tikzcd}
$$
To summarize, if $\sigma: E \to \cO_X$ is surjective, i.e. $Z(\sigma)_\cl = \emptyset$, then we have $[Z(s)_\cl]^\vir = 0$. This shows that $[Z(s)_\cl]^\vir$ vanishes away from $Z(s)_\cl \cap Z(\sigma)_\cl$ (or equivalently, $[Z(s)_\cl]^\vir$ is supported on $Z(s)_\cl \cap Z(\sigma)_\cl$).

In general (dropping the assumption that $\sigma: E \to \cO_X$ is surjective), there is a right exact sequence
$$
A_{d-r}\left( Z(s)_\cl \cap Z(\sigma)_\cl \right) \overset{i_*}{\longrightarrow} A_{d-r}\left( Z(s)_\cl\right) \overset{j^*}{\longrightarrow} A_{d-r}\left( Z(s)_\cl \backslash Z(\sigma)_\cl \right) \longrightarrow 0
$$
such that $j^*[Z(s)_\cl]^\vir = 0$. Here, $i: Z(s)_\cl \cap Z(\sigma)_\cl \hookrightarrow Z(s)_\cl$ is the closed embedding; $j: Z(s)_\cl \backslash Z(\sigma)_\cl \hookrightarrow Z(s)_\cl$ is the open embedding. Thus, there exists a cycle in $A_{d-r}\left( Z(s)_\cl \cap Z(\sigma)_\cl \right)$ that pushes forward to $[Z(s)_\cl]^\vir \in A_{d-r}\left( Z(s)_\cl\right)$. Moreover, the computation below shows that the cosection localized virtual cycle $[Z(s)_\cl]^\vir_{\sigma}$ is such a cycle.
\begin{align*}
    i_*[Z(s)_\cl]^\vir_{\sigma} &= i_*b_*\left( -[D] \cap e(K,b^*s) \right) \\
    &= b_*\overline{i}_*\left( -[D] \cap e(K,b^*s) \right) \\
    &= b_*\left( e(b^*E,b^*s) \right) \\
    &= e(E,s) \\
    &= [Z(s)_\cl]^\vir
\end{align*}
Note that the second line follows from the classical Cartesian diagram below;
$$
\begin{tikzcd}
{D \cap Z(b^*s)_\cl} \arrow[r, hook, "\overline{i}"] \arrow[d,"b"] \ar[dr, phantom, "\lrcorner", very near start] & {Z(b^*s)_\cl} \arrow[d,"b"] \\
{Z(\sigma)_\cl \cap Z(s)_\cl} \arrow[r, hook, "i"]           & {Z(s)_\cl}          
\end{tikzcd}
$$
the third line follows from MacPherson's graph construction (see \cite[Remark~5.1.1]{Fulton98} and \cite[Remark~4.1.7]{ParkNotes25}), the classical Cartesian diagram below, and the excess intersection formula \cite[Theorem~6.3]{Fulton98};
$$
\begin{tikzcd}
{Z(b^*s)_\cl} \ar[drr, phantom, "\lrcorner", very near start] \arrow[rr] \arrow[d] & & {C_{Z(b^*s)_\cl/\widetilde{X}}} \arrow[d, hook] \\
{Z(b^*s)_\cl} \ar[drr, phantom, "\lrcorner", very near start] \arrow[rr] \arrow[d] & & {K|_{Z(b^*s)_\cl}} \arrow[d, hook] \\
{Z(b^*s)_\cl} \arrow[rr]      &     & {b^*E|_{Z(b^*s)_\cl}}          
\end{tikzcd}
$$
the fourth line follows from \cite[Proposition~14.1(d)(ii)]{Fulton98}.
$\blacksquare$
\end{example}

\begin{re}
    Intuitively, $\sigma: E \to \cO_X$ plays the role of a cosection of the obstruction sheaf.
\end{re}

\begin{proposition}[Passing from Relative to Absolute, see {\cite[Proposition~B.0.1]{CJW21}}]
\label{prop:rel-cosection-rel-POT}

    Let $f: X \to Y$ be a quasi-smooth morphism from a separated\footnote{A derived Deligne--Mumford stack is separated if its classical truncation is separated.} derived Deligne--Mumford stack $X$ to a smooth classical $1$-Artin stack $Y$ with a separated diagonal, and $X_{\cl}$ be the classical truncation of $X$. We further assume that there is a relative $(-1)$-shifted function 
$$
\begin{tikzcd}
{X} \arrow[rr, "\widetilde{w}"] \arrow[rd, "f"'] &    & {\bA^1_Y[-1]} \arrow[ld] \\
                         & {Y} &              
\end{tikzcd}
$$
    on $X$ over $Y$, which naturally induces a $(-1)$-shifted function $w: X \to \bA^1[-1]$ on $X$. Then we have the following results.
    
    $(1)$ The relative $(-1)$-shifted function $\widetilde{w}$ induces a (relative) cosection $$\bT\widetilde{w}_{/Y}|_{X_\cl}: \bT_{X/Y}|_{X_\cl} \to  \bT_{\bA^1_Y[-1]/Y}|_{X_\cl} \simeq \cO_{X_\cl}[-1]$$ of the relative perfect obstruction theory $\bL_{X/Y}|_{X_\cl} \to \bL_{X_\cl/Y}$ in the sense of \cite[Appendix~B]{CJW21}. Moreover, $H^1(\bT_{X/Y}|_{X_\cl})$ is a relative obstruction sheaf $\Ob_{X_\cl/Y}$ in the sense of \cite{KiLi13}, and the morphism
$$
\widetilde{\sigma}:= H^1(\bT\widetilde{w}_{/Y}|_{X_\cl}): \Ob_{X_\cl/Y} = H^1(\bT_{X/Y}|_{X_\cl}) \to \cO_{X_\cl}
$$
    is a (relative) cosection of the relative obstruction sheaf $\Ob_{X_\cl/Y}$ in the sense of \cite[Section~4]{CJW21}.
    
    $(2)$ Recall that $d\widetilde{w}:= \bL\widetilde{w}_{/Y}[-1]: \bL_{\bA^1_Y[-1]/Y}|_X[-1] \simeq \cO_X \to \bL_{X/Y}[-1]$ is the relative exact $(-1)$-shifted differential $1$-form associated with $\widetilde{w}$. The classical truncation $X(dw)_\cl$ of the derived zero locus $X(d\widetilde{w})$ of $d\widetilde{w}: \cO_X \to \bL_{X/Y}[-1]$ is the classical degeneracy locus of the (relative) cosection $\widetilde{\sigma}: \Ob_{X_\cl/Y} \to \cO_{X_\cl}$ in the sense of \cite[Section~4]{CJW21}.
    
    $(3)$ Following the notation in Definition \ref{def:degeneracy-loci}, the classical degeneracy locus of the (absolute) cosection $\sigma$ associated with the $(-1)$-shifted function $w$ and the degeneracy locus of the (relative) cosection $\widetilde{\sigma}$ associated with the relative $(-1)$-shifted function $\widetilde{w}$ coincide.
    
    $(4)$ The cosection localized virtual cycle induced by the (absolute) cosection $\sigma$ associated with the $(-1)$-shifted function $w$ and the cosection localized virtual cycle induced by the (relative) cosection $\widetilde{\sigma}$ associated with the relative $(-1)$-shifted function $\widetilde{w}$ coincide.
\end{proposition}

\begin{proof}

    We omit the proof of $(1)$, $(2)$ and $(3)$, but leave a few comments. $(1)$ and $(2)$ rely on Proposition \ref{prop:rel-qsmth-rel-POT}. $(3)$ essentially follows from definitions. We now proceed to prove $(4)$.
    
    There is a natural closed embedding $\tau: X_{\cl} \hookrightarrow X$. As in the proof of Proposition \ref{prop:rel-qsmth-rel-POT}, we obtain the following commutative diagram, where rows are (co)fiber sequences.
$$
\begin{tikzcd}
{} &{\bT_X|_{X_\cl}} \arrow[l, "+1"']   & {\bT_{X/Y}|_{X_\cl}} \arrow[l]  & {\bT_{Y}|_{X_\cl}[-1]} \arrow[l] \\
{} &{\bT_{X_\cl}} \arrow[u, "\bT\tau"]  \arrow[l, "+1"']          & {\bT_{X_\cl/Y}}  \arrow[l]   \arrow[u, "\bT\tau_{/Y}"]        & {\bT_{Y}|_{X_\cl}[-1]} \arrow[l] \arrow[u, "\simeq"]  
\end{tikzcd}
$$

    Under the additional assumption that a relative $(-1)$-shifted function exists, we have the following commutative diagram, where the square diagram is a homotopy Cartesian diagram.
$$
\begin{tikzcd}
{X_\cl} \arrow[r, hook, "\tau"] \arrow[rrd, "f_\cl"'] & {X} \arrow[r, "\widetilde{w}"] \arrow[rd, "f"] \arrow[rr, bend left, shift left, "w"'] & {\bA_Y^1[-1]} \ar[dr, phantom, "\lrcorner_{\text{\tiny h}}", very near start] \arrow[r] \arrow[d] & {\bA^1[-1]} \arrow[d] \\
                   &                                               & {Y} \arrow[r]           & {\pt}          
\end{tikzcd}
$$
    Note that $w \circ \tau: X_\cl \to \bA^1[-1]$ is a zero $(-1)$-shifted function on $X_\cl$, and $\widetilde{w} \circ \tau: X_\cl \to \bA_Y^1[-1]$ is a zero relative $(-1)$-shifted function on $X_\cl$ over $Y$. Hence, we see that 
\begin{align*}
    \bT w \circ \bT\tau &: \bT_{X_\cl} \to \bT_X|_{X_\cl} \to (\cO_{X}[-1])|_{X_\cl} \simeq \cO_{X_\cl}[-1] \\
    \bT\widetilde{w}_{/Y} \circ \bT\tau_{/Y} &: \bT_{X_\cl/Y} \to \bT_{X/Y}|_{X_\cl} \to (\cO_{X}[-1])|_{X_\cl} \simeq \cO_{X_\cl}[-1]
\end{align*}
    are null-homotopic.

    To sum up, we have the following commutative diagram, where rows are (co)fiber sequences, and vertical compositions are null-homotopic.
$$
\begin{tikzcd}
{} & {\cO_{X_\cl}[-1]} \arrow[l, "+1"']           & {\cO_{X_\cl}[-1]} \arrow[l, "\simeq"']           & 0 \arrow[l]            \\
{} & {\bT_X|_{X_\cl}} \arrow[u, "\bT w"] \arrow[l, "+1"'] & {\bT_{X/Y}|_{X_\cl}} \arrow[u, "\bT\widetilde{w}_{/Y}"] \arrow[l] & {\bT_{Y}|_{X_\cl}[-1]} \arrow[u] \arrow[l] \\
{} & {\bT_{X_\cl}} \arrow[l, "+1"'] \arrow[u, "\bT\tau"] & {\bT_{X_\cl/Y}} \arrow[l] \arrow[u, "\bT\tau_{/Y}"] & {\bT_{Y}|_{X_\cl}[-1]} \arrow[l] \arrow[u, "\simeq"]
\end{tikzcd}
$$

    The rest of the proof is exactly the same as in the proof of \cite[Proposition~B.0.1]{CJW21} starting at \cite[Page~47]{CJW21}.
\end{proof}

\begin{re}
    Under the assumptions of the proposition, $f: X \to Y$ is automatically a separated morphism by the cancellation law.
\end{re}

\begin{re}
    By \cite[Lemma~2.2]{BeFa97}, $\bT w \circ \bT\tau$ is null-homotopic if and only if $h^1/h^0\left(\bT w \circ \bT\tau\right)=0$.
\end{re}

\begin{re}
    We call $X(d\widetilde{w})$ the \emph{derived degeneracy locus of the (relative) cosection $\widetilde{\sigma}: \Ob_{X_\cl/Y} \to \cO_{X_\cl}$}.
\end{re}

\begin{re}
    See \cite{KiLi13} for the construction of cosection localized virtual cycles via absolute cosections; see \cite[Appendix~B]{CJW21}\cite[Section~6]{AKLPR22P}\footnote{Note that we refer to a specific arXiv version, not the latest or published version.}\cite{KiPa25} for the construction of cosection localized virtual cycles via relative cosections. Other relevant references include \cite{ChKiLi17}, \cite{KiPa21}, and \cite{KiPa23}.
\end{re}

\subsection{Derived Weil Restrictions}\label{sec:der-weil-res}

This subsection is only needed for Corollary \ref{cor:log-compact-derived-weil-restriction} and Section \ref{sec:4}.

We assume the reader has basic knowledge of (derived) mapping (pre-)stacks.

We refer to \cite{KMMP25}\cite{CS24}\cite{KiPa25}\cite{HKR25} for recent works on derived Weil restrictions.

Let $X \longrightarrow Y \overset{p}{\longrightarrow} Z$ be morphisms of derived pre-stacks. The derived Weil restriction $\RW_Z(X/Y)$ (or $p_*X$ for simplicity) of $X$ along $p: Y \to Z$ is the derived pre-stack satisfying the universal property
$$
\map_{\DPStk_{/Z}}\left( T, p_*X \right) \cong \map_{\DPStk_{/Y}}\left( p^*T, X \right),
$$
where $p^*T := T \times^h_Z Y$. Moreover, the universal family of $p_*X$ is the commutative diagram as follows.
$$
\begin{tikzcd}
{p_*X} \arrow[d] &{p_*X \times^h_Z Y} \ar[dl, phantom, "_{\text{\tiny h}}\llcorner", very near start] \arrow[r, "\ev"] \arrow[d] \arrow[l, "\pi"'] & {X} \ar[dl] \\
{Z} &{Y} \arrow[l, "p"]                    &   
\end{tikzcd}
$$

\begin{re}\label{rmk:Weil-Restriction-is-homotopy-fib-prod}
    By definition, the derived Weil restriction $\RW_Z(X/Y)$ fits into the following homotopy Cartesian diagram
$$
\begin{tikzcd}
{\RW_Z(X/Y)} \ar[dr, phantom, "\lrcorner_{\text{\tiny h}}", very near start] \arrow[r] \arrow[d] & {Z} \arrow[d, "\{\id_{Y}\}"] \\
{\RH_{Z}\left( Y, X \right)} \arrow[r, "q"]  & {\RH_{Z}\left( Y, Y \right)} .
\end{tikzcd}
$$
\end{re}

\begin{re}\label{rmk:Weil-Restriction-is-algebraic}(Algebraicity of Derived Weil Restrictions) 
To determine the algebraicity of derived Weil restrictions, it suffices to determine the algebraicity of the derived mapping pre-stacks appearing in the homotopy Cartesian diagram in Remark \ref{rmk:Weil-Restriction-is-homotopy-fib-prod}. See \cite[Theorem~2.2.6.11]{ToVe08}\cite{OlssonHomStk06}\cite{HLP23}\cite{HaRy19}\cite{HKR25} for more discussions.
\end{re}
    
\begin{re}(\cite[Section~2.1]{ToVe25})
\label{rmk:alg-weil-res}
    By taking $p: Y \to Z$ to be the universal curve $\fCgn \to \fMgn$ over the classical $1$-Artin stack $\fMgn$ of pre-stable curves, $p$ induces an adjunction between $\infty$-categories
$$
p^*: \DStk_{/\fMgn} \;\rightleftarrows\;  \DStk_{/\fCgn}:p_*,
$$
where $p^*$ is defined by the homotopy fibered product.
\end{re}

\begin{example}(Derived Mapping Pre-Stacks are Derived Weil Restrictions)
\label{eg:der-map-stk-is-weil-res}

Let $p: Y \to Z$ be a morphism of derived pre-stacks, $W$ be a derived pre-stack.
By definition, we have the following homotopy Cartesian diagram
$$
\begin{tikzcd}
{\RH_{Z}\left( Y, Z \times^h W \right)} \ar[dr, phantom, "\lrcorner_{\text{\tiny h}}", very near start] \arrow[r] \arrow[d] & {Z} \arrow[d, "\{\id_{Y}\}"] \\
{\RH_{Z}\left( Y, Y \times^h W \right)} \arrow[r]  & {\RH_{Z}\left( Y, Y \right)} .
\end{tikzcd}
$$
By Remark \ref{rmk:Weil-Restriction-is-homotopy-fib-prod} and setting $X := Y \times^h W$, we see that $$p_*\left(Y \times^h W\right) := \RW_{Z} \left( Y \times^h W / Y \right) \cong \RH_{Z}\left( Y, Z \times^h W \right).$$ More generally, if $V$ is a derived pre-stack over $Z$, then we have $$p_*p^*V \cong \RH_{Z}(Y,V).$$
$\blacksquare$
\end{example}

\begin{re}\label{rmk:class-truncation-Weil-Restriction}(Classical Truncations of Derived Mapping Pre-Stacks and Derived Weil Restrictions) If the morphism $p: Y \to Z$ is flat, then the classical truncations of derived mapping pre-stacks agree with classical mapping pre-stacks in the following sense:
$$
\left(\RH_{Z}(Y,V)\right)_\cl = \Hom_{Z_\cl}(Y_\cl,V_\cl).
$$
Consequently, if the morphism $p: Y \to Z$ is flat, then by Remark \ref{rmk:Weil-Restriction-is-homotopy-fib-prod}, the classical truncations of derived Weil restrictions agree with classical Weil restrictions in the following sense:
$$
\left(\RW_Z(X/Y)\right)_\cl = \Weil_{Z_\cl}(X_\cl/Y_\cl).
$$
\end{re}

\begin{proposition}[{\cite[Remark~1.2.6]{KMMP25}}{\cite[Lemma~2.3]{HKR25}}; {\cite[Lemma~A.1.3]{CJW21}}]
\label{prop:weil-res-base-change}

    Let $p^\prime: Y^\prime \to Z^\prime$ be the base change of $p: Y \to Z$ along a morphism $Z^\prime \to Z$. Then there is an equivalence of derived pre-stacks
    $$
    \RW_{Z}\left(X / Y\right) \times^h_{Z} Z^\prime\cong \RW_{Z^\prime}\left(X \times^h_Y Y^\prime / Y^\prime\right).
    $$
\end{proposition}

\begin{proposition}[{\cite[Proposition~1.2.3]{KMMP25}}{\cite[Lemma~A.3]{KiPa25}}{\cite[Proposition~2.19]{HKR25}}; {\cite[Proposition~2.2]{ChLi12}}{\cite[Lemma~3.9]{Herr25}}]
\label{prop:weil-res-vect-bundle}

    Let $p: Y \to Z$ be the universal curve $\fCgn \to \fMgn$ over the classical $1$-Artin stack $\fMgn$ of pre-stable curves. Let $E^{\vee}$ be a perfect complex on $Y$ that is of Tor-amplitude in $[-n,0]$ for some $n \in \bZ_{\ge 0}$, and $\cE := \Spec_{Y} \Sym_{Y} (E^{\vee}) \cong \V_Y(E^\vee)$ be the derived vector bundle over $Y$ associated to $E$. Then we have an equivalence of derived stacks
    $$
    p_*\cE := \RW_Z(\cE/Y) \cong \Spec_{Z} \Sym_{Z} \left( \left( \RR p_* E \right)^{\vee} \right).
    $$
    
\end{proposition}

\begin{re}\label{rmk:pushforward-preserve-perf-complex}
    To the best of our knowledge, there are two major results that address when the derived pushforward preserves perfect complexes. The first is {\cite[Theorem~0.3]{Toen12}} (or \cite{LiNe07}), where $p: Y \to Z$ is assumed to be a morphism of classical schemes that is proper and l.c.i. The second is \cite[Theorem~6.1.3.2]{LurieSAG}, where $p: Y \to Z$ is assumed to be a morphism of spectral Deligne--Mumford stacks that is proper, (homotopically) locally almost of finite presentation, and of finite Tor-amplitude. In particular, by \cite[Lemma~1.15]{Khan22}, if $p: Y \to Z$ is a proper quasi-smooth morphism of derived Deligne--Mumford stacks, then $\RR p_*$ preserves perfect complexes.
\end{re}

\begin{re}
    A log compactification of $p_*\cE$ is discussed in Corollary \ref{cor:log-compact-derived-weil-restriction}, which is important in Section \ref{sec:4}.
\end{re}

\begin{proposition}[{\cite[Proposition~1.1.5]{KMMP25}}{\cite[Lemma~A.2]{KiPa25}}; {\cite[Lemma~A.1.2]{CJW21}}]
\label{prop:two-weil-res-equiv}

    Consider a sequence of morphisms $X^\prime \longrightarrow X \longrightarrow Y \overset{p}{\longrightarrow} Z$ of derived Artin stacks, where $p: Y \to Z$ is the universal curve $\fCgn \to \fMgn$ over the classical $1$-Artin stack $\fMgn$ of pre-stable curves, and $X^\prime \to Y$ and $X \to Y$ are smooth morphisms. Then we have an equivalence of derived stacks
    $$
    \RW_Z\left(X^{\prime}/Y\right) \cong \RW_{\RW_Z\left(X/Y\right)}\left(  X^{\prime} \times^h_X  \left( \RW_Z\left(X/Y\right) \times^h_Z Y \right)  \Big/  \RW_Z\left(X/Y\right) \times^h_Z Y  \right).
    $$
    Note that the morphism $\RW_Z\left(X/Y\right) \times^h_Z Y \to X$ used in this equivalence is the universal evaluation map $$\ev: \RW_Z\left(X/Y\right) \times^h_Z Y \to X.$$
\end{proposition}

\begin{re}
    The construction of $(-1)$-shifted (log) functions on (log compactified) derived moduli of stable maps with $p$-fields in Section \ref{sec:4} relies on Proposition \ref{prop:two-weil-res-equiv}.
\end{re}

\begin{proposition}[{\cite[Proposition~19.1.4.3]{LurieSAG}}{\cite[Proposition~B.3.7]{Rozenblyum21}}{\cite[Theorem~3.3.1]{KMMP22}}\footnote{Note that we refer to a specific arXiv version, not the published version.}{\cite[Appendix~A.(112)]{KiPa25}}{\cite[Section~2.8]{HKR25}}; {\cite[Theorem~1.1]{Webb22}}{\cite[Lemma~3.10]{Herr25}}]
\label{prop:weil-res-tangent-complex}

    Assume that $\RW_Z\left(X/Y\right)$ is a derived Artin stack. Then there is a quasi-isomorphism
    $$
    \bT_{\RW_Z\left(X/Y\right)/Z} \simeq \RR\pi_* \LL \ev^* \bT_{X/Y},
    $$
    where $\ev: \RW_Z\left(X/Y\right) \times^h_Z Y \to X$ is the universal evaluation map and $\pi: \RW_Z\left(X/Y\right) \times^h_Z Y \to \RW_Z\left(X/Y\right)$ is the natural projection map.
\end{proposition}

\begin{re}
    Note that $$\bL_{\RW_Z\left(X/Y\right)/Z} \simeq \bT_{\RW_Z\left(X/Y\right)/Z}^{\vee} \simeq (\RR\pi_* \LL \ev^* \bT_{X/Y})^{\vee} \not\simeq \RR\pi_* \LL \ev^* \bL_{X/Y}.$$
\end{re}

\begin{re}\label{rmk:tangent-complex-der-map-stk}
In particular, for a derived mapping stack $\RH_{Z}(Y,V)$, there is a quasi-isomorphism
$$
\bT_{\RH_{Z}(Y,V)/Z} \simeq \RR\pi_* \LL \ev^* \bT_{V/Z},
$$
where $\ev: \RH_{Z}(Y,V) \times^h_Z Y \to V$ is the universal evaluation map and $\pi: \RH_{Z}(Y,V) \times^h_Z Y \to \RH_{Z}(Y,V)$ is the natural projection map. See {\cite[Proposition~4.3.1]{MaRo18}} for the computation of tangent complexes of derived mapping stacks. See {\cite[Proposition~19.1.4.1]{LurieSAG}}{\cite[Proposition~5.1.10]{HLP23}}{\cite[Proposition~B.3.5]{Rozenblyum21}}{\cite[Theorem~10.1.6]{KhanNotes23}} for the computation of cotangent complexes of the derived mapping (pre-)stacks.
\end{re}

\begin{re}
    See \cite[Lemma~1.15]{PTVV13}{\cite[Theorem~3.1]{BMS19}} for proof of (co)tangent complexes satisfying descent.
\end{re}

\begin{example}[Gromov--Witten Theory and Derived Moduli of Stable Maps, \cite{Behrend97}; \cite{STV15}]
\label{eg:Gromov-Witten}
    
    Let $X$ be a smooth projective variety. Let $\fMgn$ denote the smooth classical $1$-Artin stack of pre-stable curves of genus $g$ with $n$ marked points, and $\fCgn \to \fMgn$ denote the universal curve over $\fMgn$.

    By Remark \ref{rmk:alg-weil-res} and Remark \ref{eg:der-map-stk-is-weil-res}, $\RH_{\fMgn}\left( \fCgn, X \times \fMgn  \right)$ is a derived mapping stack. Since $\fCgn \to \fMgn$ is flat, the classical truncation $\left(\RH_{\fMgn}\left( \fCgn, X \times \fMgn  \right)\right)_\cl$ of $\RH_{\fMgn}\left( \fCgn, X \times \fMgn  \right)$ is the classical mapping stack $\Hom_{\fMgn}\left( \fCgn, X \times \fMgn  \right)$ by Remark \ref{rmk:class-truncation-Weil-Restriction}. Moreover, as the stability condition is an open condition, by {\cite[Proposition~2.1]{STV15}}, there exists a derived moduli $\RM$ of stable maps to $X$ such that it is an open sub-stack of $\RH_{\fMgn}\left( \fCgn, X \times \fMgn  \right)$.

    Note that, by construction, the classical truncation $\left(\RM\right)_{\cl}$ of $\RM$ is the classical Deligne--Mumford stack $\overline{\mathcal{M}}_{g,n}(X,\beta)$ of stable maps to $X$. Hence, $\RM$ is a derived Deligne--Mumford stack.

    Consider the following commutative diagram, where two square diagrams are homotopy Cartesian diagrams.
$$
\begin{tikzcd}
{\RM \times^h_{\fMgn}\fCgn}  \ar[dr, phantom, "\lrcorner_{\text{\tiny h}}", very near start] \arrow[rrr, shift left=1, bend left=15, "f^\st"'] \arrow[r, hook, "j \times id"] \arrow[d, "\widetilde{\pi}^\st"] &{\RH_{\fMgn}\left( \fCgn, X \times \fMgn  \right) \times^h_{\fMgn} \fCgn}  \arrow[r, "\ev"] \arrow[d, "\widetilde{\pi}"] & { X \times \fMgn} \arrow[r] \ar[dr, phantom, "\lrcorner_{\text{\tiny h}}", very near start] \arrow[d] & { X } \arrow[d] \\
{\RM} \arrow[r, hook, "j"] &{\RH_{\fMgn}\left( \fCgn, X \times \fMgn  \right)}                     & { \fMgn} \arrow[r]           & {\pt }      
\end{tikzcd}
$$
    Here, $j$ is an open embedding.
    Then we have quasi-isomorphisms
\begin{align*}
    \bT_{\RM/\fMgn} \simeq& \LL j^*\bT_{\RH_{\fMgn}\left( \fCgn, X \times \fMgn  \right)/\fMgn}\\
    \simeq&\LL j^*\RR\widetilde{\pi}_*\LL \ev^*(\bT_{X \times \fMgn / \fMgn}) \\
    \simeq&\RR\widetilde{\pi}^\st_*\LL(j \times id)^*\LL \ev^*\left(\bT_X|_{X \times \fMgn}\right) \\
    \simeq&\RR\widetilde{\pi}^\st_*\LL (f^\st)^*\bT_X.
\end{align*}
    The second quasi-isomorphism is due to Remark \ref{rmk:tangent-complex-der-map-stk}; the third quasi-isomorphism is due to the base change formula.

    Since $\fCgn \to \fMgn$ is proper l.c.i, we see that $\widetilde{\pi}^\st_*$ is proper quasi-smooth, and hence the morphism $\RR\widetilde{\pi}^\st_*$ preserves perfect complexes by Remark \ref{rmk:pushforward-preserve-perf-complex}. Thus, $\bT_{\RM/\fMgn}$ is a perfect complex.

    Since $\fCgn \to \fMgn$ is a proper morphism of relative dimension $1$, $$(\widetilde{\pi}^\st)_\cl: \overline{\mathcal{M}}_{g,n}(X,\beta) \times_{\fMgn}\fCgn \to \overline{\mathcal{M}}_{g,n}(X,\beta)$$ is a proper morphism of relative dimension $1$, and hence is finite cohomological dimension at most $1$. Thus, by {\cite[Definition~A.1.4]{HLP23}}, $$\widetilde{\pi}^\st: \RM \times^h_{\fMgn}\fCgn \to \RM$$ has finite cohomological dimension at most $1$. Therefore, we can conclude that $$\bL_{\RM/\fMgn} \simeq \left(\RR\bar{\pi}_*\LL f^*T_X\right)^{\vee}$$ is a perfect complex of Tor-amplitude in $[-1,0]$ by {\cite[Lemma~A.1.6]{HLP23}}. Consequently, $\RM$ is quasi-smooth over $\fMgn$. Since $\fMgn$ is smooth, we see that $\RM$ is quasi-smooth.
    
     Finally, by Proposition \ref{prop:rel-qsmth-rel-POT}, the natural closed embedding
    $$\tau:\overline{\mathcal{M}}_{g,n}(X,\beta) \hookrightarrow \RM$$
    induces a (relative) perfect obstruction theory 
    $$\bL\tau_{/\fMgn}: \bL_{\RM/\fMgn}|_{\overline{\mathcal{M}}_{g,n}(X,\beta)} \to \bL_{\overline{\mathcal{M}}_{g,n}(X,\beta)/\fMgn}.$$ 
    
    To check that $\bL\tau_{/\fMgn}$ is indeed the desired (relative) perfect obstruction theory for Gromov-Witten theory as in \cite{Behrend97}, it suffices to check that the commutative diagram
$$
\begin{tikzcd}
{\overline{\mathcal{M}}_{g,n}(X,\beta) \times_{\fMgn}\fCgn} \ar[dr, phantom, "\lrcorner_{\text{\tiny h}}", very near start] \arrow[r, hook] \arrow[d] & {\RM \times^h_{\fMgn}\fCgn} \arrow[d, ""] \\
{\overline{\mathcal{M}}_{g,n}(X,\beta)} \arrow[r, hook]  & {\RM}   
\end{tikzcd}
$$
    is a homotopy Cartesian diagram. This follows from the fact that $\fCgn \to \fMgn$ is a flat morphism.
    $\blacksquare$
    
\end{example}

\section{Deligne--Faltings Logarithmic Derived Geometry}\label{sec:2}

The derived log geometry developed here is motivated by the twisted cosections in Log GLSM \cite{CJR21} (see Remark \ref{rmk:shift-log-function}). It follows the spirit of \cite{Olssonlogstk03} and \cite{Pridham13}, but it is technically much simpler. In our view, it is probably the simplest workable framework for derived log geometry that still captures certain interesting geometric features. See \cite{Rognes09}, \cite{Bhatt12}, \cite{Pridham13}, \cite{STG16}, and \cite{zhang26} for other (more sophisticated) approaches to derived log geometry.

\subsection{Basic Setup}\label{sec:basic-setup-der-log}

In this subsection, we set up the derived log geometry used throughout the paper. We introduce Deligne--Faltings log structures (of rank $1$) on derived stacks via $\cA := [\bA^1 / \bG_m]$, inspired by \cite{Olssonlogstk03}. We then define the $\infty$-category $\LogDStk$ of derived log stacks to be the $\infty$-category $\DStk_{/\cA}$ of derived stacks over $\cA$. In this way, our version of derived log geometry is embedded into the usual framework of derived algebraic geometry.

Deligne--Faltings log structures (of rank $1$) are naturally associated to virtual Cartier divisors (or equivalently generalized Cartier divisors) on derived stacks. We refer to \cite{KhRy25} for references on virtual Cartier divisors (or equivalently, generalized Cartier divisors) in derived algebraic geometry.

Throughout the paper, $\cA := [\bA^1 / \bG_m]$ is the derived moduli of generalized Cartier divisors, or equivalently the derived moduli of virtual Cartier divisors (see \cite{KhRy25}); the natural closed embedding $B\bG_m \hookrightarrow \cA$ is the universal Cartier divisor on $\cA$; the natural morphism $s_{\cA}: \cO_{\cA}(-B\bG_m) \hookrightarrow \cO_{\cA}$ is the universal section.

\begin{definition}
    A \emph{Deligne--Faltings logarithmic structure of rank $1$ on a derived stack $X$} is a morphism $\alpha_X: X \to \cA$. We refer to a pair $(X,\alpha_X)$ as a \emph{derived log stack}\footnote{We only consider Deligne--Faltings logarithmic structure of rank $1$ in this paper.}, $X$ as the \emph{underlying derived stack}, and $\alpha_X: X \to \cA$ as the \emph{anchor morphism}.
\end{definition}

\begin{re}\label{rmk:der-log-structure-rk-n}
  Similarly, we can define a \emph{Deligne--Faltings logarithmic structure of rank $m$ on a derived stack $X$} to be a morphism $X \to \cA^m$.
\end{re}

\begin{re}(\cite{KhRy25})
    For a derived log stack $(X, \alpha_X)$, its virtual Cartier divisor $\D_X \hookrightarrow X$ is the pullback of $B\bG_m \hookrightarrow \cA$ along its anchor morphism $\alpha_X$; its generalized Cartier divisor $\cO_X(-\D_X) \to \cO_X$ is the pullback of $\cO_{\cA}(-B\bG_m) \hookrightarrow \cO_{\cA}$ along its anchor morphism $\alpha_X$.
\end{re}

\begin{re}\label{rmk:triv-log-der-stk}
    For a derived stack $X$, there is always a morphism $X \to \cA$ factoring through the open embedding $\pt = [\bG_m/\bG_m] \hookrightarrow \cA$. We denote this morphism $X \to \cA$ by $1_X$, and the corresponding derived log stack by $(X, 1_X)$. Note that $\D_X$ is empty in this case.
\end{re}

\begin{definition}
    A \emph{derived log Artin stack} $(X, \alpha_X)$ is a derived log stack whose underlying derived stack $X$ is a derived Artin stack. A \emph{derived log Deligne--Mumford stack} $(X, \alpha_X)$ is a derived log stack whose underlying derived stack $X$ is a derived Deligne--Mumford stack. A \emph{derived log scheme} $(X, \alpha_X)$ is a derived log stack whose underlying derived stack $X$ is a derived scheme. 
\end{definition}

\begin{definition}
    A derived log stack $(X, \alpha_X)$ is \emph{log quasi-smooth} if the anchor morphism $\alpha_X: X \to \cA$ is quasi-smooth. A derived log stack $(X, \alpha_X)$ is \emph{log smooth} if the anchor morphism $\alpha_X: X \to \cA$ is smooth. 
\end{definition}

\begin{re}
    Since $\cA$ is smooth, a derived log Deligne--Mumford stack $(X, \alpha_X)$ is log quasi-smooth if and only if its underlying derived Deligne--Mumford stack $X$ is quasi-smooth. For simplicity, we will refer to a log quasi-smooth derived log Deligne--Mumford stack simply as a \emph{quasi-smooth derived log Deligne--Mumford stack}.
\end{re}

\begin{re} Assuming $\pt$ carries the trivial log structure, then a classical log scheme of Deligne--Faltings type of rank $1$ is log smooth (i.e. its underlying scheme is smooth and the associated divisor is smooth) if and only if it admits a smooth morphism to $\cA$.
\end{re}

\begin{definition}
    A \emph{morphism $f: (X,\alpha_X) \to (Y, \alpha_Y)$ of derived log stacks} is a morphism of derived stacks over $\cA$. In other words, $f: (X,\alpha_X) \to (Y, \alpha_Y)$ is a commutative diagram of derived stacks
$$
\begin{tikzcd}
{X} \arrow[rd, "\alpha_X"'] \arrow[rr] &    & {Y} \arrow[ld, "\alpha_Y"] \\
                         & {\cA} &              .
\end{tikzcd}
$$
Here, the morphism $X \to Y$ that fits into the above commutative diagram is referred to as the \emph{underlying morphism of $f: (X,\alpha_X) \to (Y, \alpha_Y)$}. Abusing notation, we also use $f$ to denote the underlying morphism $X \to Y$.
\end{definition}

\begin{re}\label{re:strict-morph}
    The morphism of derived log stacks defined above should be referred to as a \emph{strict morphism}, because its pullback preserves the log structure. However, we do not consider non-strict morphisms in the paper.
\end{re}

\begin{definition}\label{def:log-cotangent-complex}
    We define the \emph{log cotangent complex $\bLL_X$ of a derived log Artin stack $(X, \alpha)$} to be the relative cotangent complex $\bL_{\alpha_X} = \bL_{X/\cA}$ of the anchor morphism $\alpha_X$. Such a definition forces us to define the \emph{relative log cotangent complex $\bLL_f$ of a morphism $f:  (X,\alpha_X) \to (Y, \alpha_Y)$ of derived log Artin stacks} to be the relative cotangent complex $\bL_{f} = \bL_{X/Y}$ of the underlying morphism $f: X \to Y$, i.e. $\bLL_f = \bL_f$.
\end{definition}

\begin{re}
    For a morphism $f: (X,\alpha_X) \to (Y, \alpha_Y)$ of derived log Artin stacks, there is a (co)fiber sequence
    $$
    \LL f^*\bLL_Y \longrightarrow \bLL_{X} \longrightarrow \bLL_{f} \overset{+1}{\longrightarrow}   .
    $$
\end{re}

\begin{re} The definition of the (relative) log cotangent complex presented here is standard in the sense of Olsson's log cotangent complexes, although not very conventional. In the classical setting, there exists a classical algebraic stack $\mathcal{L}og_S$ that parametrizes classical log schemes over a classical log scheme $S$ (see \cite{Olssonlogstk03}). In other words, a classical log scheme $X$ over $S$ is equivalent to a (representable) morphism $X \to \mathcal{L}og_S$ of classical algebraic stacks. Moreover, the Olsson's log cotangent complex of a classical log scheme $X$ is defined as the relative cotangent complex of the associated morphism $X \to \mathcal{L}og_S$ of classical algebraic stacks (see \cite{Olssonlogcotangent05}). In particular, if $X$ is a classical log scheme of Deligne--Faltings type of rank $1$ over $\pt$, assuming $\pt$ is equipped with a trivial log structure, then the associated morphism $X \to \mathcal{L}og_{\pt}$ of classical algebraic stacks admits a factorization
$$
\begin{tikzcd}
{X} \arrow[rr] \arrow[rd] &    & {\cA} \arrow[ld, hook] \\
                         & {\mathcal{L}og_{\pt}} &       .       
\end{tikzcd}
$$
Note that $\cA \hookrightarrow \mathcal{L}og_{\pt}$ is an open embedding so that there is a quasi-isomorphism $$\bLL_X:=\bL_{X/\mathcal{L}og_\pt} \overset{\simeq}{\longrightarrow} \bL_{X/\cA} .$$
\end{re}

\begin{re}
    There are four natural properties (see {\cite[Section~7]{Olssonlogcotangent05}}) one would expect log cotangent complexes to satisfy. But by Bauer's arguments in {\cite[Section~7]{Olssonlogcotangent05}}, it is impossible to satisfy all four simultaneously in the classical setting. Fortunately, our (relative) log cotangent complexes satisfy all four properties by definition, at the cost of not considering non-strict morphisms (see Remark \ref{re:strict-morph}). 
\end{re}

\begin{re}\label{rmk:compare-triv-log-der-stk}
    For a derived log stack $(X, 1_X)$, there is a quasi-isomorphism $\bLL_X \simeq \bL_X$, where $\bLL_X$ is the log cotangent complex of $(X, 1_X)$ and $\bL_X$ is the cotangent complex of the underlying derived stack $X$.
\end{re}

\begin{re}
    Under our conventions, cotangent complexes of derived Artin stacks are perfect complexes. Hence, relative cotangent complexes of morphisms between derived Artin stacks are perfect complexes. Consequently, log cotangent complexes of derived log Artin stacks and relative log cotangent complexes of morphisms of derived log Artin stacks are also perfect complexes. As a result, we define the \emph{log tangent complex $\bTT_X$ of a derived log Artin stack $(X,\alpha_X)$} to be $(\bLL_X)^{\vee}$, and the \emph{relative log tangent complex} $\bTT_f$ of a morphism $f: (X,\alpha_X) \to (Y, \alpha_Y)$ of derived log Artin stacks to be $(\bLL_f)^{\vee}$.
\end{re}

\begin{example}[Justification of the Definition of Log Cotangent Complexes]
\label{eg:log-cotangent-complex-aff-derived-log-schm}

This is a continuation of the Example \ref{eg:aff-derived-zero-locus}. 

There is a virtual Cartier divisor $\D_{Z(s)} \hookrightarrow Z(s)$ on $Z(s)$ (see \cite{KhRy25}) with $$\D_{Z(s)} := \Spec k[x_1,x_2,x_3;\nu_1,\nu_2,\nu_3].$$ The underlying complex of the CDGA $k[x_1,x_2,x_3;\nu_1,\nu_2,\nu_3]$ is 
$$k[x_1,x_2,x_3]\langle \nu_1\nu_2\nu_3 \rangle \to k[x_1,x_2,x_3]\langle \nu_1\nu_2, \nu_1\nu_3, \nu_2\nu_3\rangle \to k[x_1,x_2,x_3]\langle \nu_1, \nu_2,\nu_3\rangle \to k[x_1,x_2,x_3],$$
where $\nu_1 \mapsto f_1$,  $\nu_2 \mapsto f_2$, $\nu_3 \mapsto x_3$, and the rest of the maps are extended by the Leibniz rule. Note the cohomological degrees $|x_1|=|x_2|=|x_3|=0$ and $|\nu_1|=|\nu_2|=|\nu_3|=-1$. Note that $$H^0(k[x_1,x_2,x_3;\nu_1,\nu_2, \nu_3]) \simeq k[x_1,x_2,x_3]/(f_1,f_2, x_3).$$ 

As a result, we see that $\left( Z(s), \alpha_{Z(s)} \right)$ is a derived log scheme, where $\alpha_{Z(s)}: Z(s) \to \cA$ is the morphism associated to the virtual Cartier divisor $\D_{Z(s)} \hookrightarrow Z(s)$ on $Z(s)$. Note that we have the following commutative diagrams.
$$
\begin{tikzcd}
{\D_{Z(s)}} \arrow[r] \arrow[d, hook] \ar[dr, phantom, "\lrcorner_{\text{\tiny h}}", very near start] & {\pt} \arrow[r] \arrow[d, hook] \ar[dr, phantom, "\lrcorner_{\text{\tiny h}}", very near start] & {B\bG_m} \arrow[d, hook] \\
{Z(s)} \arrow[r, "x_3"]  \arrow[rr, bend right=20, "\alpha_{Z(s)}"']         & {\bA^1} \arrow[r]           & {\cA}          
\end{tikzcd}
$$

Applying the cotangent complex functor to $\alpha_{Z(s)}: Z(s) \to \cA$, we obtain a (co)fiber sequence
$$\bL_{\cA}|_{Z(s)} \overset{\bL \alpha_{Z(s)}}{\longrightarrow} \bL_{Z(s)} \longrightarrow \bLL_{Z(s)} \overset{+1}\longrightarrow .$$
In other words, we have $\bLL_{Z(s)}  \simeq \cofib\left(\bL\alpha_{Z(s)}: \bL_{\cA}|_{Z(s)} \to \bL_{Z(s)}\right)$. By Remark \ref{rmk:cotangent-complex-cA}, $$\bL_{\cA} \simeq [\cO_{\cA}\langle d\mu \rangle \overset{d\mu \mapsto \mu d\lambda}{\longrightarrow} \cO_{\cA}\langle d\lambda \rangle]$$ is a perfect complex of Tor-amplitude in $[0,1]$, with $|d\mu|=0$ and $|d\lambda|=1$. As a result, $\bL_{\cA}|_{Z(s)} \overset{\bL \alpha_{Z(s)}}{\longrightarrow} \bL_{Z(s)}$ can be explicitly written as
\begin{equation}
\begin{tikzcd}\label{diag:log-cotangent-complex}
{\cO_{Z(s)}\langle d\lambda \rangle} \arrow[r]           & {0}           \\
{\cO_{Z(s)}\langle d\mu \rangle} \arrow[u, "d\mu \mapsto x_3d\lambda"] \arrow[r, "d\mu \mapsto dx_3"] & {\cO_{Z(s)}\langle dx_1,dx_2,dx_3\rangle} \arrow[u] \\
{0} \arrow[u] \arrow[r] & {\cO_{Z(s)}\langle d\nu_1,d\nu_2\rangle} \arrow[u,"d\bar{s}"]  .
\end{tikzcd}
\end{equation}
We then observe the relation $x_3 d\lambda \sim dx_3$ in $\bLL_{Z(s)}$, which implies that $d\lambda$ corresponds to $d\log (x_3)$ in $\bLL_{Z(s)}$. This is exactly how $d \log$ terms arise in the classical log cotangent bundle, justifying the definition of log cotangent complexes.

Then consider the following commutative diagram, where square diagrams are homotopy Cartesian diagrams.
\[\begin{tikzcd}
	\D_{Z(s)_{cl}}  & \D_{Z(s)}  & B\bG_m  \\
	Z(s)_{\cl}  & Z(s)  & \cA
	\arrow[from=1-1, to=1-2, hook]
	\arrow[from=1-1, to=2-1, hook]
	\arrow[phantom, "\lrcorner_{\text{\tiny h}}", very near start, from=1-1, to=2-2]
	\arrow[from=1-2, to=1-3]
	\arrow[from=1-2, to=2-2, hook]
	\arrow[phantom, "\lrcorner_{\text{\tiny h}}", very near start, from=1-2, to=2-3]
	\arrow[from=1-3, to=2-3, hook]
	\arrow[from=2-1, to=2-2, hook, "\tau"]
	\arrow[from=2-2, to=2-3, "\alpha_{Z(s)}"]
\end{tikzcd}\]
    Note that $\D_{Z(s)_{cl}}$ is not necessarily a classical scheme, although $Z(s)_{cl}$ is the classical truncation of $Z(s)$.
    
    We additionally assume that $x_3$ is a non-zero-divisor in $\cO_{Z(s)_\cl} = k[x_1,x_2,x_3]/(f_1,f_2)$, and hence $\D_{Z(s)_{cl}}$ is necessarily a classical scheme $\Spec k[x_1,x_2,x_3]/(f_1,f_2, x_3)$.
    Under this additional assumption, we may conclude that $\bLL_{Z(s)}$ is quasi-isomorphic to (the total complex of the bi-complex)
$$
\cO_{Z(s)}\langle d\nu_1,d\nu_2\rangle \overset{d\bar{s}}{\longrightarrow}  \cO_{Z(s)}\langle dx_1,dx_2,d\log (x_3)\rangle .
$$
    See Appendix \ref{append:2} for related discussions.
$\blacksquare$

\end{example}

\begin{re}\label{rmk:cotangent-complex-cA}
    The cotangent complex $\bL_{\cA}$ of $\cA$ is quasi-isomorphic to
    $$ \left[s_{\cA}: \cO_{\cA}(-B\bG_m) \hookrightarrow \cO_{\cA} \right] \simeq [\cO_{\cA}\langle d\mu \rangle \overset{d\mu \mapsto \mu d\lambda}{\longrightarrow} \cO_{\cA}\langle d\lambda \rangle],$$
    which is a perfect complex of Tor-amplitude in $[0,1]$.

    Set $\bA^1:= \Spec k[x]$. Consider the action morphism
\begin{align*}
    \psi: \bG_m \times \bA^1 &\longrightarrow \bA^1 \\
    (t,x) &\longmapsto tx .
\end{align*}
    Note that $\psi$ admits a section $(e,\id_{\bA^1}): \bA^1 \to \bG_m \times \bA^1$, where $e: \pt \to \bG_m$ is the identity section of $\bG_m$. Let $\mathfrak{g}$ denote $\LL e^* (\bT_{\bG_m})$, which is the tangent space of $\bG_m$ at $e$.

    Applying the cotangent complex functor to $\psi$ and then pulling back to $\bA^1$ along $(e,\id_{\bA^1})$, we obtain a morphism
$$
d\psi|_{\bA^1}:=\LL(e,\id_{\bA^1})^*\left(\bL\psi\right): \bL_{\bA^1} \simeq \Omega^1_{\bA^1} \to \bL_{\bG_m \times \bA^1}|_{\bA^1} \simeq \left( \mathfrak{g}^{\vee} \boxtimes \cO_{\bA^1} \right) \oplus \Omega^1_{\bA^1} ,
$$
    which can be explicitly written as
\begin{align*}
        d\psi|_{\bA^1}: \cO_{\bA^1}\langle dx \rangle &\longrightarrow \cO_{\bA^1}\langle dt \rangle \oplus \cO_{\bA^1}\langle dx \rangle \\
        dx &\longmapsto \left( d\left( tx \right) \right)|_{t=1} = \left( xdt + tdx \right)|_{t=1} =xdt + dx  .
\end{align*}
    Hence, we see that $d\psi|_{\bA^1} \simeq d\bar{\psi} \oplus \id$, where $d\bar{\psi} : \cO_{\bA^1}\langle dx \rangle \to \cO_{\bA^1}\langle dt \rangle : dx \mapsto xdt$. Following Remark \ref{rmk:cotangent-complx-der-quot-stk}, we conclude that
\begin{align*}
    \bL_{\cA}|_{\bA^1} &\simeq [d\bar{\psi}: \cO_{\bA^1}\langle dx \rangle \overset{dx \mapsto xdt}{\longrightarrow} \cO_{\bA^1}\langle dt \rangle] \\
    &\simeq [\cO_{\bA^1} \overset{\cdot x}{\longrightarrow} \cO_{\bA^1}],
\end{align*}
    which is a perfect complex of Tor-amplitude in $[0,1]$.
    Note that $d\bar{\psi}$ is a $\bG_m$-equivariant morphism, hence $\bL_{\cA}|_{\bA^1} \in \QCoh^{\bG_m}(\bA^1)$. By descent, we have $\QCoh^{\bG_m}(\bA^1) \cong \QCoh(\cA)$. Therefore, we conclude that
\begin{align*}
    \bL_{\cA} &\simeq [\cO_{\cA}\langle d\mu \rangle \overset{d\mu \mapsto \mu d\lambda}{\longrightarrow} \cO_{\cA}\langle d\lambda \rangle] \\
    &\simeq [s_{\cA}: \cO_{\cA}(-B\bG_m) \hookrightarrow \cO_{\cA}].
\end{align*}
    
\end{re}

\begin{re}
    Diagram (\ref{diag:log-cotangent-complex}) is in fact a triple complex. See \href{https://stacks.math.columbia.edu/tag/08BI}{Remark 08BI} for a quick discussion.
\end{re}

We summarize the terminology introduced so far in this subsection. The \emph{$\infty$-category $\LogDStk$ of derived log stacks} is defined to be the $\infty$-category $\DStk_{/\cA}$ of derived stacks over $\cA$. Thus, a derived log stack is an object in $\LogDStk$, and a morphism of derived log stacks is a morphism in $\LogDStk$. Moreover, the log cotangent complex functor $\bLL_{-}$ is simply the relative cotangent complex functor $\bL_{-/\cA}$.

We conclude this subsection by generalizing the preceding terminology to the relative setting.

\begin{definition}\label{def:rel-log-der-stk}
    Let $S$ be a derived Artin stack. We define the \emph{$\infty$-category $\LogDStk_{/S}$ of relative derived log stacks over $S$} to be the $\infty$-category $\DStk_{/\cA \times^h S}$ of derived stacks over $\cA \times^h S$. Thus, a \emph{relative derived log stack over $S$} is an object in $\LogDStk_{/S}$, and a \emph{morphism of relative derived log stacks over $S$} is a morphism in $\LogDStk_{/S}$. Moreover, the \emph{relative log cotangent complex functor $\bL^{S-\log}_{-}$ over $S$} is simply the relative cotangent complex functor $\bL_{-/\cA \times^h S}$.
\end{definition}

\begin{re}
    Let $(X, X \to \cA \times^h S)$ be a relative derived log stack over $S$. If $\bL^{S-\log}_{X}$ is a perfect complex, then we define the \emph{relative log tangent complex $\bT^{S-\log}_{X}$ of $(X, X \to \cA \times^h S)$} to be $\left(\bL^{S-\log}_{X}\right)^\vee$.
\end{re}

\begin{re}
    The \emph{relative virtual Cartier divisor $\D_{X/S}$ on a relative derived log stack $(X, X \to \cA \times^h S)$ over $S$} is defined via the following homotopy Cartesian diagram.
    $$
\begin{tikzcd}
{\D_{X/S}} \ar[dr, phantom, "\lrcorner_{\text{\tiny h}}", very near start] \arrow[r, hook] \arrow[d] & {X} \arrow[d] \\
{B\bG_m \times^h S } \arrow[r, hook]  & {\cA \times^h S}   .
\end{tikzcd}
    $$
\end{re}

\subsection{Shifted Logarithmic Functions and Exact Shifted Logarithmic Differential $1$-Forms}\label{sec:shift-log-fun-shift-log-diff-form}

Thanks to the basic setup in Section \ref{sec:basic-setup-der-log}, we are able to introduce shifted log functions, (twisted) log (co)tangent complexes, shifted log cotangent bundles, and derived log critical loci in this section without much effort. Moreover, the construction of derived log critical loci motivates us to find the definition of shifted log symplectic structures and the proof of the log Lagrangian intersection theorem (see Theorem \ref{thm:log-lag-intersect}), which are developed in Section \ref{sec:shift-log-symp}.

\begin{definition}
    Set $\cLA := \left(\cO_{\cA}(-B\bG_m)\right)^{\vee}$ and hence $\cLA^{\vee} := \cO_{\cA}(-B\bG_m)$. Let $n \in \bZ_{\ge 0}$. The \emph{$(-n)$-shifted universal line bundle $L_\cA[-n]$ over $\cA$} is defined to be $\Spec_{\cA} \Sym_{\cA} \left(\left(\cLA[-n]\right)^{\vee}\right)$.
\end{definition}

\begin{re}
    When $n=0$, $L_{\cA} := L_{\cA}[0]$ is literally the universal line bundle over $\cA$. When $n \ge 1$, $L_{\cA}[-n]$ fits into the homotopy Cartesian diagram
$$
\begin{tikzcd}
{L_{\cA}[-n]} \ar[dr, phantom, "\lrcorner_{\text{\tiny h}}", very near start] \arrow[r] \arrow[d] & {\cA} \arrow[d] \\
{\cA } \arrow[r]  & {L_{\cA}[-n+1]}   .
\end{tikzcd}
$$
    In particular, $L_{\cA}[-1]$ fits into the homotopy Cartesian diagram
$$
\begin{tikzcd}
{L_{\cA}[-1]} \ar[dr, phantom, "\lrcorner_{\text{\tiny h}}", very near start] \arrow[r] \arrow[d] & {\cA} \arrow[d,"0"'] \\
{\cA } \arrow[r, "0"]  & {L_{\cA}}   .
\end{tikzcd}
$$    
    In other words, $L_{\cA}[-1]$ is the derived self-intersection of the zero section in $L_{\cA}$.
\end{re}

\begin{re}
    Note that when $n \ge 1$, the classical truncation $\left(L_{\cA}[-n]\right)_{\cl}$ of $L_{\cA}[-n]$ is $\cA$.
\end{re}

\begin{re}
    There is a natural morphism $L_{\cA}[-n] \to \cA$, which gives $L_{\cA}[-n]$ a log structure. For simplicity, we denote this derived log stack by $\left(L_\cA[-n], L_\cA[-n] \to \cA\right)$.
\end{re}

\begin{definition}\label{def:shift-log-function-diff-form}
    Let $n \in \bZ_{\ge 0}$. For a derived log Artin stack $(X, \alpha_X)$, a \emph{$(-n)$-shifted log function on $(X, \alpha_X)$} is a morphism $$w_\alpha: (X, \alpha_X) \to \left(L_\cA[-n], L_\cA[-n] \to \cA\right)$$ of derived log stacks, i.e. a commutative diagram of derived stacks
$$
\begin{tikzcd}
{X} \arrow[rd, "\alpha_X"'] \arrow[rr, "w_\alpha"] &    & {L_\cA[-n]} \arrow[ld] \\
                         & {\cA} &              .
\end{tikzcd}
$$
    Applying the log cotangent complex functor to $w_\alpha$ induces a morphism $$\bLL w_\alpha: \bLL_{L_\cA[-n]}|_X \simeq \cO_X(-\D_X)[n] \to \bLL_X.$$ We then define the \emph{exact $(-n)$-shifted log differential $1$-form associated with $w_\alpha$} to be $$dw_\alpha:= \bLL w_\alpha \otimes^{\LL}_{\cO_{X}} \cO_X(\D_X)[-n]: \cO_X \to \bLL_X \otimes^{\LL}_{\cO_{X}} \cO_X(\D_X) [-n].$$
\end{definition}

\begin{re}
    Based on Remark \ref{rmk:der-log-structure-rk-n}, there is a natural definition of a $(-n)$-shifted log function on a derived stack equipped with a Deligne--Faltings logarithmic structure of rank $m$. We leave the precise definition to the reader.
\end{re}

\begin{re}\label{rmk:understanding-of-shift-log-function}
    Let us consider a $(-1)$-shifted log function $w_\alpha: (X, \alpha_X) \to \left(L_\cA[-1], L_\cA[-1] \to \cA\right)$ on a derived log scheme $(X, \alpha_X)$. There are equivalences of $\infty$-groupoids
\begin{align*}
    &\map_{\LogDStk}(X,L_\cA[-1]) \cong \map_{\DStk_{/\cA}}(X,L_\cA[-1]) \\
    \cong& \map_{\QCoh(X)}(\LL\alpha_X^*(\cL_{\cA}^{\vee}[1]),\cO_X) \\
    \cong& \map_{\QCoh(X)}(\cO_X(-\D_X)[1],\cO_X) \cong \map_{\QCoh(X)}(\cO_X,\cO_X(\D_X)[-1]).
\end{align*}
    Here, the second line follows from Definition \ref{def:def-linear-der-stk} and Remark \ref{rmk:der-vect-bund}. Hence, $w_\alpha \in \pi_0\left( \map_{\LogDStk}(X,L_\cA[-1]) \right)$ corresponds to an element in $\pi_{1}(\Gamma(X, \cO_X(\D_X)))$\footnote{We are using homological conventions here. If using cohomological conventions, we will write $H^{-1}(\Gamma(X, \cO_X(\D_X)))$ instead.}. Intuitively, we may roughly regard a $(-1)$-shifted log function on $(X, \alpha_X)$ as a $(-1)$-shifted function on $X$ with a pole along $\D_X$. Alternatively, a $(-1)$-shifted log function $w_\alpha$ is determined by a homotopy $\zeta$ appearing in the following commutative diagram, where the square diagram is a homotopy Cartesian diagram.
$$
\begin{tikzcd}
	X \\
	& L_\cA[-1] & \cA \\
	& \cA & L_\cA
	\arrow[""{name=0, anchor=center, inner sep=0}, shift right=-2, bend right=-10, from=1-1, to=2-3, "\alpha_X"]
	\arrow[""{name=1, anchor=center, inner sep=0}, shift right=2, bend right=10, from=1-1, to=3-2, "\alpha_X"']
	\arrow[from=2-2, to=2-3]
	\arrow[from=2-2, to=3-2]
	\arrow["\lrcorner_{\text{\tiny h}}"{anchor=center, pos=0.125}, draw=none, from=2-2, to=3-3]
	\arrow[from=2-3, to=3-3, "0"]
	\arrow[from=3-2, to=3-3, "0"]
	\arrow[shorten <=10pt, shorten >=5pt, shift right=3, Rightarrow, from=0, to=1, "\zeta"'{pos=0.6}]
\end{tikzcd}
$$
\end{re}

\begin{re}\label{rmk:shift-log-function}
    Inspired by Section \ref{sec:cosection-localization}, we can use $(-1)$-shifted log functions on quasi-smooth derived log Deligne--Mumford stacks to construct log cosections (Definition \ref{def:log-cosection}), which are motivated by the twisted cosections in Log GLSM \cite{CJR21}. This construction is the starting point for the derived log geometry developed here. 
\end{re}

\begin{definition}\label{def:tw-log-cotangent-complex}
    Inspired by Definition \ref{def:shift-log-function-diff-form} and \cite{CS24}, we define the \emph{twisted log cotangent complex $\bLa_X$ of a derived log Artin stack $(X, \alpha_X)$} to be $\bL_{\alpha_X} \otimes^{\LL}_{\cO_{X}} \LL\alpha_X^* \cL_{\cA} \simeq \bLL_X \otimes^{\LL}_{\cO_{X}} \cO_X(\D_X)$; the \emph{twisted relative log cotangent complex $\bLa_f$ of a morphism $f: (X, \alpha_X) \to (Y, \alpha_Y)$ of derived log Artin stacks} to be $\bLL_f \otimes^{\LL}_{\cO_X} \cO_X(\D_X)$. The \emph{twisted log tangent complex $\bTa_X$ of a derived log Artin stack $(X, \alpha_X)$} is defined to be $(\bLa_X)^{\vee}$; the \emph{twisted relative log tangent complex $\bTa_f$ of a morphism $f: (X, \alpha_X) \to (Y, \alpha_Y)$ of derived log stacks} is defined to be $(\bLa_f)^{\vee}$.
\end{definition}

\begin{re}
    The twisted log (co)tangent complex is functorial by definition. In particular, for a morphism $f: (X,\alpha_X) \to (Y, \alpha_Y)$ of derived log Artin stacks, there is a (co)fiber sequence
    $$
    \LL f^*\bLa_Y \longrightarrow \bLa_{X} \longrightarrow \bLa_{f} \overset{+1}{\longrightarrow}   .
    $$
\end{re}

\begin{re}
    For a derived log stack $(X, 1_X)$, there is a quasi-isomorphism $\bLa_X \simeq \bL_X$, where $\bLa_X$ is the twisted log cotangent complex of $(X, 1_X)$ and $\bL_X$ is the cotangent complex of the underlying derived stack $X$.
\end{re}

\begin{definition}\label{def:shift-log-co-tangent-bund}
    The \emph{$(-n)$-shifted log cotangent bundle $\Ta[-n]X$ over a derived log Artin stack $(X, \alpha_X)$} is defined to be $\V_X\left((\bLa_X [-n])^{\vee}\right)$ (see Definition \ref{def:def-linear-der-stk} for definition of $\V_X$); the \emph{$(-n)$-shifted log tangent bundle $\Tad[-n]X$ over a derived log Artin stack $(X, \alpha_X)$} is defined to be $\V_X\left((\bTa_X [-n])^{\vee}\right)$.
\end{definition}

\begin{re}\label{rmk:shift-log-diff-form-is-section-tw-log-cotangent-bund}
    The exact $(-n)$-shifted log differential $1$-form
    $dw_\alpha$ associated with a $(-n)$-shifted log function $w_\alpha: (X, \alpha_X) \to \left(L_\cA[-n], L_\cA[-n] \to \cA\right)$ is precisely the section
    $$
    \bLa w_\alpha [-n]: \cO_X \to \bLa_X[-n].
    $$
    Abusing notation, we also use $dw_\alpha$ to denote the morphism
    $$
    X \to \Ta[-n]X
    $$
    induced by $dw_\alpha: \cO_X \to \bLa_X[-n]$. Moreover, we have the following commutative diagram.
$$
\begin{tikzcd}
{X} \arrow[rr, "dw_\alpha"] \arrow[rd, "\alpha_X"'] \arrow[rrd, "\id"'] &    & {\Ta[-n]X} \arrow[d]             \\
                         & {\cA} & {X} \arrow[l, "\alpha_X"] 
\end{tikzcd}
$$
    This implies that $(\Ta[-n]X, \Ta[-n]X \longrightarrow X \overset{\alpha_X}{\longrightarrow} X)$ is a derived log stack, and $$dw_\alpha: (X, \alpha_X) \to (\Ta[-n]X, \Ta[-n]X \longrightarrow X \overset{\alpha_X}{\longrightarrow} \cA)$$ is a morphism of derived log stacks.
\end{re}

\begin{definition}\label{def:der-log-crit-locus}
    Let $w_\alpha: (X, \alpha_X) \to \left(L_\cA[-n], L_\cA[-n] \to \cA\right)$ be a $(-n)$-shifted log function on a derived log Artin stack $(X, \alpha_X)$. Following the notation in Remark \ref{rmk:shift-log-diff-form-is-section-tw-log-cotangent-bund}, we define the \emph{derived log critical locus of a $(-n)$-shifted log function $w_\alpha$ on $(X, \alpha_X)$} to be the derived zero locus $X(dw_\alpha)$ of $dw_\alpha: \cO_X \to \bLa_X[-n]$, which is defined via the following homotopy Cartesian diagram in $\LogDStk$.
$$
    \begin{tikzcd}
{X(dw_\alpha)} \arrow[r] \arrow[d] \ar[dr, phantom, "\lrcorner_{\text{\tiny h}}", very near start] & {X} \arrow[d,"dw_\alpha"] \\
{X} \arrow[r, "0"]           & {\Ta[-n]X}          
    \end{tikzcd}
$$
\end{definition}

\begin{re}
    $\left( X(dw_\alpha), X(dw_\alpha) \longrightarrow X \overset{\alpha_X}{\longrightarrow} \cA \right)$ is a derived log stack.
\end{re}

\begin{re}
    If $dw_\alpha \simeq 0 : \cO_X \to \bLa_X[-n]$, then $X(dw_\alpha) \cong \Ta[-n-1]X$.
\end{re}

\begin{re}\label{rmk:log-cotangent-complex-der-log-crit-loci}
    By Proposition \ref{prop:cotangent-complex-derived-zero-loci}, there is a (co)fiber sequence
\begin{equation}\label{align:cotangent-complex-der-log-crit-loci}
    \bLL_X|_{X(dw_\alpha)} \longrightarrow \bLL_{X(dw_\alpha)} \longrightarrow (\bTa_X[n+1])|_{X(dw_\alpha)} \overset{+1}{\longrightarrow }
\end{equation}
Taking dual of Sequence (\ref{align:cotangent-complex-der-log-crit-loci}), we obtain a (co)fiber sequence
$$
(\bLa_X[-n-1])|_{X(dw_\alpha)} \longrightarrow \bTT_{X(dw_\alpha)} \longrightarrow \bTT_X|_{X(dw_\alpha)} \overset{+1}{\longrightarrow}   ;
$$
applying $- \otimes_{\cO_X}^{\LL} \cO_X(\D_X)[-n-1]$ to Sequence (\ref{align:cotangent-complex-der-log-crit-loci}), we obtain another (co)fiber sequence
$$
(\bLa_X[-n-1])|_{X(dw_\alpha)} \longrightarrow \bLa_{X(dw_\alpha)}[-n-1] \longrightarrow \bTT_X|_{X(dw_\alpha)} \overset{+1}{\longrightarrow}.
$$
Thus, we expect that there is a quasi-isomorphism $$\bTT_{X(dw_\alpha)} \simeq \bLa_{X(dw_\alpha)}[-n-1].$$
Such an expectation is what motivates us to find the definition of shifted log symplectic structures and the proof of the log Lagrangian intersection theorem, so that $\left( X(dw_\alpha), X(dw_\alpha) \longrightarrow X \overset{\alpha_X}{\longrightarrow} \cA \right)$ carries a $(-n-1)$-shifted log symplectic structure. 
\end{re}

\begin{example}[The Classical Truncation of Derived Log Critical Loci of $(-1)$-Shifted Log Functions on a Specific Derived Log Affine Scheme]
\label{eg:classical-truncation-affine-derived-log-critical-locus}

This is a continuation of Example \ref{eg:log-cotangent-complex-aff-derived-log-schm}.

we continue to assume that $x_3$ is a non-zero-divisor in $\cO_{Z(s)_\cl} = k[x_1,x_2,x_3]/(f_1,f_2)$, and hence $\D_{Z(s)_{cl}}$ is necessarily a classical scheme $\Spec k[x_1,x_2,x_3]/(f_1,f_2, x_3)$. Moreover, $\bLL_{Z(s)}$ is quasi-isomorphic to (the total complex of the bi-complex)
$$
\cO_{Z(s)}\langle d\nu_1,d\nu_2\rangle \overset{d\bar{s}}{\longrightarrow}  \cO_{Z(s)}\langle dx_1,dx_2,d\log (x_3)\rangle .
$$

Note that $$\cO_{Z(s)}(\D_{Z(s)}) \simeq \Big[k[x_1,x_2,x_3]\langle \frac{\nu_1\nu_2}{x_3} \rangle \to k[x_1,x_2,x_3]\langle \frac{\nu_1}{x_3}, \frac{\nu_2}{x_3} \rangle \to k[x_1,x_2,x_3]\langle \frac{1}{x_3} \rangle \Big].$$
Set $w_\alpha := h_1\frac{\nu_1}{x_3} + h_2\frac{\nu_2}{x_3} \in k[x_1,x_2,x_3]\langle \frac{\nu_1}{x_3}, \frac{\nu_2}{x_3} \rangle$ with $h_1, h_2 \in k[x_1,x_2,x_3]$ such that $$h_1\frac{\nu_1}{x_3} + h_2\frac{\nu_2}{x_3} \mapsto h_1\frac{f_1}{x_3} + h_2\frac{f_2}{x_3}  = 0 \in k[x_1,x_2,x_3]\langle \frac{1}{x_3} \rangle.$$ In other words, $w_\alpha$ is a representative of an element in $H^{-1}\left(\cO_{Z(s)}(\D_{Z(s)})\right)$. Hence, by Remark \ref{rmk:understanding-of-shift-log-function}, $w_\alpha$ corresponds to a $(-1)$-shifted log function $\left( Z(s), \alpha_{Z(s)} \right) \to (L_\cA[-1], L_\cA[-1] \to \cA)$, also denoted by $w_\alpha$.

Then the exact $(-1)$-shifted log differential $1$-form
$$dw_\alpha: \cO_{Z(s)} \to \bLa_{Z(s)}[-1]$$
is determined by 
\begin{align*}
    1 \mapsto dw_\alpha =& d \left( h_1\frac{\nu_1}{x_3} \right) + d \left( h_2\frac{\nu_2}{x_3} \right) = \frac{x_3d(h_1\nu_1)-h_1\nu_1dx_3}{x^2_3} + \frac{x_3d(h_2\nu_2)-h_2\nu_2dx_3}{x^2_3} \\
    =& \frac{x_3h_1d\nu_1 + x_3\nu_1dh_1 - h_1\nu_1dx_3}{x^2_3} + \frac{x_3h_2d\nu_2 + x_3\nu_2dh_2 - h_2\nu_2dx_3}{x^2_3} \\
    =& h_1\frac{d\nu_1}{x_3} + h_2\frac{d\nu_2}{x_3} + \frac{\nu_1}{x_3}dh_1 + \frac{\nu_2}{x_3}dh_2 - h_1\frac{\nu_1d\log (x_3)}{x_3} - h_2\frac{\nu_2d\log (x_3)}{x_3}.
\end{align*}
Note that $dw_\alpha$ is a representative of $H^{-1}\left( \bLa_{Z(s)} \right)$. This is due to the equation $h_1f_1 + h_2f_2 =0 $ and the fact that $\bLa_{Z(s)}$ is quasi-isomorphic to (the total complex of the bi-complex)
$$
\cO_{Z(s)}(\D_{Z(s)})\langle d\nu_1,d\nu_2\rangle \overset{\frac{d\bar{s}}{x_3}}{\longrightarrow}  \cO_{Z(s)}(\D_{Z(s)})\langle dx_1,dx_2,d\log (x_3)\rangle ,
$$
which can be explicitly written as
\[\begin{tikzcd}[column sep=6em]
	{k[x_1,x_2,x_3]\langle \frac{d\nu_1}{x_3},\frac{d\nu_2}{x_3}\rangle} & {k[x_1,x_2,x_3]\langle \frac{dx_1}{x_3},\frac{dx_2}{x_3},\frac{d \log (x_3)}{x_3}\rangle} \\
	{k[x_1,x_2,x_3]\langle \frac{\nu_1d\nu_1}{x_3},\frac{\nu_1d\nu_2}{x_3},\frac{\nu_2d\nu_1}{x_3},\frac{\nu_2d\nu_2}{x_3}\rangle} & {k[x_1,x_2,x_3]\langle \frac{\nu_1dx_1}{x_3},\frac{\nu_1dx_2}{x_3},\frac{\nu_1d\log (x_3)}{x_3},\frac{\nu_2dx_1}{x_3},\frac{\nu_2dx_2}{x_3},\frac{\nu_2d\log (x_3)}{x_3} \rangle} \\
	{k[x_1,x_2,x_3]\langle \frac{\nu_1\nu_2d\nu_1}{x_3},\frac{\nu_1\nu_2d\nu_2}{x_3} \rangle} & {k[x_1,x_2,x_3]\langle \frac{\nu_1\nu_2dx_1}{x_3},\frac{\nu_1\nu_2dx_2}{x_3},\frac{\nu_1\nu_2d\log (x_3)}{x_3} \rangle} .
    \arrow[from=1-1, to=1-2, "\frac{d\nu_j}{x_3} \mapsto \frac{df_j}{x_3}"]
	\arrow[from=2-1, to=1-1, "\frac{\nu_id\nu_j}{x_3} \mapsto \frac{f_id\nu_j}{x_3}"']
	\arrow[from=2-2, to=1-2, "\frac{\nu_idx_j}{x_3} \mapsto \frac{f_idx_j}{x_3}"']
	\arrow[from=2-1, to=2-2, "\frac{\nu_id\nu_j}{x_3} \mapsto \frac{\nu_idf_j}{x_3}"]
	\arrow[from=3-1, to=2-1, "\frac{\nu_1\nu_2d\nu_j}{x_3} \mapsto \frac{(f_1\nu_2-\nu_1f_2)d\nu_j}{x_3}"']
	\arrow[from=3-1, to=3-2, "\frac{\nu_1\nu_2d\nu_j}{x_3} \mapsto \frac{\nu_1\nu_2df_j}{x_3}"]
	\arrow[from=3-2, to=2-2, "\frac{\nu_1\nu_2dx_j}{x_3} \mapsto \frac{(f_1\nu_2-\nu_1f_2)dx_j}{x_3}"']
\end{tikzcd}\]

Abusing notation, we use $dw_\alpha$ to denote the morphism
$$Z(s) \to \Ta[-1]Z(s) \cong \Spec_{Z(s)} \Sym_{Z(s)} \left( \bTa_{Z(s)}[1] \right)$$
induced by $dw_\alpha: \cO_{Z(s)} \to \bLa_{Z(s)}[-1]$. Then the classical truncation $(dw_\alpha)_{\cl}$ is the morphism
$$
Z(s)_\cl \to \left(\Ta[-1]Z(s)\right)_{\cl} = \Spec_{Z(s)_\cl} \Sym_{Z(s)_\cl} \left( H^0\left( \bTa_{Z(s)}[1]|_{Z(s)_\cl} \right) \right)
$$
induced by $h_1\frac{d\nu_1}{x_3}+h_2\frac{d\nu_2}{x_3}$ (see \cite[Section~2]{OhTh24}).

As a result, the classical truncation $(Z(s)(dw_\alpha))_{\cl}$ of the derived log critical locus $Z(s)(dw_\alpha)$ fits into the following classical Cartesian diagram.
$$
\begin{tikzcd}
{\left(Z(s)(dw_\alpha)\right)_{\cl}} \ar[dr, phantom, "\lrcorner", very near start] \arrow[r] \arrow[d] & {Z(s)_{\cl}} \arrow[d, "(dw_\alpha)_{\cl}"] \\
{Z(s)_{\cl}} \arrow[r, "0"]  & {\left(\Ta[-1]Z(s)\right)_{\cl}}
\end{tikzcd}
$$
This implies that $\left(Z(s)(dw_\alpha)\right)_{\cl}$ is isomorphic to $\left(Z(s)_{\cl}(h_1,h_2)\right)_\cl = \Spec k[x_1,x_2,x_3]/(f_1, f_2, h_1,h_2)$, the classical zero locus of the functions $h_1$ and $h_2$ on $Z(s)_{\cl} = \Spec k[x_1,x_2,x_3]/(f_1, f_2)$. $\blacksquare$
\end{example}

\begin{re}(Continuation of the Remark \ref{rmk:shift-log-function})
    The classical truncation of the derived log critical locus $X(dw_\alpha)$ of a $(-1)$-shifted log function $w_\alpha$ on a quasi-smooth derived log Deligne--Mumford stack $(X, \alpha_X)$ is the degeneracy locus of the log cosection induced by $w_\alpha$. See Section \ref{sec:log-cosection} for details.
\end{re}

\begin{definition}
    Let $n \in \bZ_{\ge 0}$. For a relative derived log Artin stack $(X, X \to \cA \times^h S)$ over $S$, a \emph{relative $(-n)$-shifted log function on $(X, X \to \cA \times^h S)$ over $S$} is a morphism $$w_\alpha: (X, X \to \cA \times^h S) \to \left(L_\cA[-n] \times^h S, L_\cA[-n] \times^h S \to \cA \times^h S \right)$$ of relative derived log stacks over $S$, i.e. a commutative diagram of derived stacks
$$
\begin{tikzcd}
{X} \arrow[rd] \arrow[rr] &    & {L_\cA[-n] \times^h S} \arrow[ld] \\
                         & {\cA \times^h S} &              .
\end{tikzcd}
$$
\end{definition}

\begin{re}
    Note that $\bL^{S-\log}_{L_\cA[-n] \times^h S} \simeq \bL_{L_\cA[-n] \times^h S/\cA \times^h S} \simeq \bL_{L_\cA[-n] /\cA}|_{L_\cA[-n] \times^h S} \simeq \bLL_{L_\cA[-n]}|_{L_\cA[-n] \times^h S}$.
\end{re}

The remaining terminologies generalize to the relative setting immediately.

\subsection{Shifted Logarithmic Symplectic Structures}\label{sec:shift-log-symp}

In this subsection, we introduce shifted log symplectic structures by specializing twisted symplectic structures in \cite{CS24} to the case where the base stack is $\cA = [\bA^1/\bG_m]$ and the invertible sheaf is $\cL_{\cA} := \cO_{\cA}(B\bG_m)$. Most of the results in this subsection follow directly as corollaries of corresponding results in \cite{CS24}, except for the logarithmic Lagrangian intersection theorem (Theorem \ref{thm:log-lag-intersect}).

One may view this subsection as a derived analogue of $b^m$-symplectic geometry (see \cite{BDMOP19}).

We refer to \cite{PTVV13}, \cite{Calaque19}, and \cite{PaYo25} for references on shifted symplectic structures. We refer to \cite{CS24} for references on twisted symplectic structures. We also note that {\cite[Appendix~A]{KiPa25}} contains some discussions of twisted symplectic structures.

For simplicity, if there is no confusion involved, we denote $X$ as a derived log Artin stack $(X, \alpha_X)$, and $f: X \to Y$ as a morphism $f: (X, \alpha_X) \to (Y, \alpha_Y)$ of derived log Artin stacks throughout this subsection.

\begin{definition}
    We follow the notations given in {\cite[Section~2]{CS24}}.
The \emph{space of $(-n)$-shifted log differential $p$-forms on $X$} is defined as $$\fA^p(X,-n):= \map_{QCoh(\cA)^{gr}}\left( \cO_{\cA}(p)[-p] , \DR_{\cA}(X)^{\not\epsilon} \otimes^{\LL}_{\cO_{\cA}} L_{\cA}[-n] \right).$$
The \emph{space of closed $(-n)$-shifted log differential $p$-forms on $X$} is defined as $$\fA^{p,cl}(X,-n):= \map_{QCoh(\cA)^{gr,\epsilon}}\left( \cO_{\cA}(p)[-p] , \DR_{\cA}(X) \otimes^{\LL}_{\cO_{\cA}} L_{\cA}[-n] \right).$$
\end{definition}

\begin{re}
    \begin{align*}
        \fA^{p}(-n)&: \LogDStk^{op}:=\DStk_{\cA}^{op} \to \mathcal{S}: X \mapsto \fA^{p}(X,-n) \\
        \fA^{p,cl}(-n)&: \LogDStk^{op}:=\DStk_{\cA}^{op} \to \mathcal{S}: X \mapsto \fA^{p, cl}(X,-n)
    \end{align*}
     are derived prestacks over $\cA$.
\end{re}

\begin{re}\label{rmk:shift-diff-form}
Thanks to {\cite[lemma~1.2]{CS24}} and {\cite[Theorem~2.7]{CS24}}, we have
    \begin{align*}
        \fA^p(X,-n) :=& \map_{QCoh(\cA)^{gr}}\left( \cO_{\cA}(p)[-p] , \DR_{\cA}(X)^{\not\epsilon} \otimes^{\LL}_{\cO_{\cA}} L_{\cA}[-n] \right) \\
        \cong& \map_{QCoh(\cA)^{gr}}\left( \cO_{\cA}(p)[-p] , \left( \RR\alpha_{X_*} \Sym_X \left(\bLL_X[-1]\right) \right) \otimes^{\LL}_{\cO_{\cA}} L_{\cA}[-n] \right) \\
        \cong& \map_{QCoh(\cA)^{gr}}\left( \cO_{\cA}(p)[-p] ,  \RR\alpha_{X_*} \left( \Sym_X \left(\bLL_X[-1]\right) \otimes^{\LL}_{\cO_{X}} \LL\alpha_X^* L_{\cA}[-n] \right) \right) \\
        \cong& \map_{QCoh(X)^{gr}}\left( \cO_{X}(p)[-p] , \Sym_X \left(\bLL_X[-1]\right) \otimes^{\LL}_{\cO_X} \cO_X(\D_X)[-n] \right) \\
        \cong& \map_{QCoh(X)}\left( \cO_{X} , \wedge^p \bLL_X \otimes^{\LL}_{\cO_X} \cO_X(\D_X)[-n] \right)
    \end{align*}
\end{re}

\begin{definition}
    A \emph{$(-n)$-shifted log differential $p$-forms on $X$} is a morphism $X \to \fA^{p}(-n)$ of derived prestacks over $\cA$. A \emph{closed $(-n)$-shifted log differential $p$-forms on $X$} is a morphism $X \to \fA^{p,cl}(-n)$ of derived prestacks over $\cA$.
\end{definition}

\begin{re}
    An exact $(-n)$-shifted log differential $1$-form is a closed $(-n)$-shifted log differential $1$-form for $n \in \bZ_{\ge 0}$.
\end{re}

\begin{re}
    There is a forgetful morphism $\fA^{p,cl}(-n) \to \fA^{p}(-n)$ induced by the forgetful morphism $\DR_{\cA}(-) \to \DR_{\cA}(-)^{\not \epsilon}$. As a result, every closed $(-n)$-shifted log differential $p$-form induces a $(-n)$-shifted log differential $p$-form.
\end{re}

Let $\omega_X$ denote a closed $(-n)$-shifted log differential $2$-form on $X$. Its induced (underlying) $(-n)$-shifted log differential $2$-form is denoted as $\uw_X$, which is equivalent to a morphism $$\cO_X \to \wedge^2 \bLL_X \otimes^{\LL}_{\cO_X} \cO_X(\D_X) [-n].$$ 
Taking the dual, we obtain the morphism $$\wedge^2 \bTT_X \otimes^{\LL}_{\cO_X} \cO_X(-\D_X) [n] \to \cO_X.$$ 
Then by adjunction, we obtain the morphism $$\uw_X^{\flat}: \bTT_X \to \bLL_X \otimes^{\LL}_{\cO_X} \cO_X(\D_X) [-n] \simeq \bLa_X[-n].$$

\begin{definition}
A closed $(-n)$-shifted log differential $2$-form $\omega_X$ on $X$ is \emph{non-degenerate} if its induced morphism $\uw_X^{\flat}: \bTT_X \to \bLa_X[-n]$ is a quasi-isomorphism. A \emph{$(-n)$-shifted log symplectic form on $X$} is a non-degenerate closed $(-n)$-shifted log differential $2$-form $\omega_X$. 

A derived log stack $X$ is \emph{$(-n)$-shifted log pre-symplectic} if it is equipped with a closed $(-n)$-shifted log differential $2$-form $\omega_X$. A derived log stack $X$ is \emph{$(-n)$-shifted log symplectic} if it is equipped with a $(-n)$-shifted log symplectic form $\omega_X$.
\end{definition}

Thanks to {\cite[Corollary~2.36(1)]{CS24}}, we have the following proposition.

\begin{proposition}\label{prop:log-AKSZ}(Logarithmic AKSZ Construction)

    $(1)$ Let $X$ be a $(-n)$-shifted log symplectic derived log stack, and $Y$ be a derived log stack equipped with an $\LL\alpha_Y^* (\cO_{\cA})[d]$-orientation (see \cite{CS24}). Assuming that $\RH_{\cA}(Y, X)$ is a derived stack, then $\RH_{\cA}(Y, X)$ is a $(-n-d)$-shifted log symplectic derived log stack.

    $(2)$ Let $X$ be a derived log stack equipped with a relative $(-n)$-shifted symplectic structure over $\cA$, and $Y$ be a derived log stack equipped with an $\LL\alpha_Y^* (\cL_{\cA}^\vee)[d]$-orientation (see \cite{CS24}). Assuming that $\RH_{\cA}(Y, X)$ is a derived stack, then $\RH_{\cA}(Y, X)$ is a $(-n-d)$-shifted log symplectic derived log stack.
\end{proposition}

Thanks to {\cite[Corollary~2.36(1)]{CS24}}, {\cite[Proposition~5.1]{CS24}} and {\cite[Corollary~6.3]{HKR25}}, we have the following proposition.

\begin{proposition}(A Special AKSZ Construction)

    Assume $X$ is a $(-n)$-shifted log symplectic derived log stack. Then $\RH_{\cA}(B\bG_m, X)$ is a derived stack equipped with a relative $(-n+1)$-shifted symplectic structure over $\cA$.
\end{proposition}

\begin{re}
    $B\bG_m$ is naturally a derived log stack by setting its anchor morphism to be the natural closed embedding $B\bG_m \hookrightarrow \cA$.
\end{re}

\begin{re}
     Note that any morphisms $B\bG_m \to X$ of derived log stacks factor through the virtual Cartier divisor $\D_X \hookrightarrow X$ on $X$. In particular, we have the following commutative diagram.
    $$
\begin{tikzcd}
{\D_{\RH_{\cA}\left( B\bG_m, X \right)}:= \RH_{\cA}\left( B\bG_m, X \right) \times^h_{\cA} B\bG_m} \arrow[d, hook] \arrow[dr] \arrow[r, "
\ev"] &{X}               \\
{\RH_{\cA}\left( B\bG_m, X \right)}                                 &{\D_X} \arrow[u, hook]             
\end{tikzcd}
    $$
\end{re}

\begin{re}
    $\RH_{\cA}(B\bG_m, X)$ is a derived (log) stack which turns out to be shifted symplectic instead of shifted log symplectic, and moreover its shift increases instead of decreases (most constructions we know decrease the shift for shifted symplectic structures\footnote{See the paragraph above \cite[Notations~and~conventions]{PTVV13}}). 
\end{re}

\

For the rest of this subsection, $X$ is assumed to be $(-n)$-shifted log pre-symplectic equipped with a closed $(-n)$-shifted log differential $2$-form $\omega_X$, whose underlying $(-n)$-shifted log differential $2$-form is denoted as $\uw_X$.

\begin{definition}
A \emph{log isotropic structure on a morphism $l: L \to X$ of derived log Artin stacks} is a commutative diagram
$$
\begin{tikzcd}
{L} \arrow[r, "l"] \arrow[d] & {X} \arrow[d, "\omega_X"] \\
{\cA} \arrow[r,"0"]           & {\fA^{2,cl}(-n)}        ,  
\end{tikzcd}
$$
or equivalently a path $\gamma : 0 \to \omega_X|_L $ in $ \fA^{2,cl}(L,-n)$, which induces a commutative diagram
$$
\begin{tikzcd}
{\bTT_L} \arrow[r] \arrow[rd] & {\bTT_X|_L} \arrow[r,"\uw_X^{\flat}|_L"]  & {\bLa_X[-n]|_L} \arrow[rd] &    \\
                        & {0} \arrow[rr] &               & {\bLa_L[-n]} .
\end{tikzcd}
$$
If $X$ is moreover $(-n)$-shifted log symplectic, i.e. $\uw_X^\flat: \bTT_X \to \bLa_X[-n]$ is a quasi-isomorphism, then a log isotropic structure $\gamma$ on a morphism $l: L \to X$ of derived log Artin stacks is \emph{non-degenerate} if the induced commutative diagram
$$
\begin{tikzcd}
{\bTT_L} \arrow[r] \arrow[d] & {\bTT_X|_L \simeq \bLa_X[-n]|_L} \arrow[d] \\
{0} \arrow[r]           & {\bLa_L[-n]}          
\end{tikzcd}
$$
is (co)Cartesian. A morphism $l: L \to X$ of derived log Artin stacks is a \emph{log Lagrangian} if it is equipped with a non-degenerate isotropic structure $\gamma$.
\end{definition}

\begin{re}
    For a non-degenerate isotropic structure $\gamma$, it induces a quasi-isomorphism $$\underline{\gamma}^{\flat}: \bTT_L[1] \overset{\simeq}{\longrightarrow} \bLa_l[-n].$$
\end{re}

\begin{definition}
We follow the notations given in {\cite[Section~2]{CS24}}.
The space of $(-n)$-shifted relative log differential $p$-forms on a morphism $\pi: X \to B$ of derived log Artin stacks is defined as $$\fA^{p}(X/B,-n):= \map_{QCoh(B)^{gr}}\left( \cO_{B}(p)[-p] , \DR_{B}(X)^{\not \epsilon} \otimes^{\LL}_{\cO_{B}} \cO_B(\D_B)[-n] \right).$$
The space of closed $(-n)$-shifted relative log differential $p$-forms on a morphism $\pi: X \to B$ of derived log Artin stacks is defined as $$\fA^{p,cl}(X/B,-n):= \map_{QCoh(B)^{gr,\epsilon}}\left( \cO_{B}(p)[-p] , \DR_{B}(X) \otimes^{\LL}_{\cO_{B}} \cO_B(\D_B)[-n] \right).$$
\end{definition}

\begin{re}
    Thanks to {\cite[Lemma~1.2]{CS24}} and {\cite[Theorem~2.7]{CS24}}, by similar arguments in Remark \ref{rmk:shift-diff-form}, we have
    $$ \fA^{p}(X/B,-n) \cong \map_{QCoh(X)}\left( \cO_{X} , \wedge^p \bLL_{\pi} \otimes^{\LL}_{\cO_X} \cO_X(\D_X)[-n] \right).$$
    Recall that $\bLL_{\pi} \simeq \bL_{\pi}$.
\end{re}

\begin{re}
    There is a natural morphism $\fA^{p,cl}(X,-n) \to \fA^{p,cl}(X/B,-n)$ induced by the natural morphism $\bLL_X \to \bLL_{\pi}$. 
\end{re}

\begin{definition}
Let $\omega_{X/B}$ denote the image of $\omega_X$ under the natural morphism $$\fA^{2,cl}(X,-n) \to \fA^{2,cl}(X/B,-n).$$
A \emph{log isotropic fibration structure on a morphism ${\pi}: X \to B$ of derived log Artin stacks} is a path $\eta : 0 \to \omega_{X/B}$ in $\fA^{2,cl}(X/B,-n)$, which induces a commutative diagram
$$
\begin{tikzcd}
{\bTT_{\pi}} \arrow[r] \arrow[rd] & {\bTT_X} \arrow[r,"\uw_X^{\flat}"]  & {\bLa_X[-n]} \arrow[rd] &    \\
                        & {0} \arrow[rr] &               & {\bLa_{\pi}[-n]} .
\end{tikzcd}
$$
A log isotropic fibration structure $\eta$ on a morphism ${\pi}: X \to B$ of derived log Artin stack is \emph{non-degenerate} if the induced commutative diagram
$$
\begin{tikzcd}
{\bTT_{\pi}} \arrow[r] \arrow[d] & {\bTT_X} \arrow[d] \\
{0} \arrow[r]           & {\bLa_{\pi}[-n]}          
\end{tikzcd}
$$
is (co)Cartesian. A morphism ${\pi}: X \to B$ of derived log Artin stacks is a \emph{log Lagrangian fibration} if it is equipped with a non-degenerate isotropic fibration structure $\eta$.
\end{definition}

\begin{re}
    If a log isotropic fibration structure $\eta$ is non-degenerate, then the closed $(-n)$-shifted log differential $2$-form $\omega_X$ is also non-degenerate.
\end{re}

Thanks to {\cite[Theorem~3.6]{CS24}}, we have the following proposition.

\begin{proposition}\label{prop:shifted-log-symp-lag}Let $X$ be a derived log Artin stack.

$1)$ $\Ta[-n]X$ (see Definition \ref{def:shift-log-co-tangent-bund}) is $(-n)$-shifted log symplectic.

$2)$ The natural projection morphism $\pi:\ \Ta[-n]X \to X $ is a log Lagrangian fibration.
\end{proposition}

\begin{re}\label{rmk:taut-form-on-shift-log-cotang-bund}

    Assume there is a morphism $f: Y\to X$, then we have
    $$
    \map_{\DStk_{/X}}\left( Y,  \Ta[-n]X\right) \cong \map_{\QCoh(Y)}\left( \LL f^*\bTa_X[n] , \cO_Y \right) \cong \map_{\QCoh(Y)}\left( \cO_Y,\LL f^*\bLa_X[-n]  \right).
    $$
    In particular, for the natural projection morphism $\pi:  \Ta[-n]X \to X$, we have
    $$
    \map_{\DStk_{/X}}\left( \Ta[-n]X,  \Ta[-n]X\right) \cong \map_{\QCoh(\Ta[-n]X)}\left( \cO_{\Ta[-n]X}, \pi^*\bLa_X[-n]  \right) .
    $$
Hence, $\id_{\Ta[-n]X}$ induces a tautological morphism $\theta_{\Ta[-n]X}:\cO_{\Ta[-n]X} \to \pi^*\bLa_X[-n]$, which induces a $(-n)$-shifted log differential $1$-form 
$$
    \lambda_{\Ta[-n]X}: \cO_{\Ta[-n]X} \overset{\theta_{\Ta[-n]X}}{\longrightarrow} \LL \pi^*\bLa_X[-n] \to \bLa_{\Ta[-n]X}[-n]
$$
on $\Ta[-n]X$. As a result, there is a canonical closed $(-n)$-shifted log differential $2$-form $\omega_{\Ta[-n]X} := d_{dR} \lambda_{\Ta[-n]X}$ on $\Ta[-n]X$. For Proposition \ref{prop:shifted-log-symp-lag}, the statement that $\Ta[-n]X$ is $(-n)$-shifted log symplectic means precisely that $\omega_{\Ta[-n]X}$ is a $(-n)$-shifted log symplectic form.

    Note that there is a (co)fiber sequence
$$
\LL \pi^*\bLa_X[-n] \longrightarrow \bLa_{\Ta[-n]X}[-n] \longrightarrow \bLa_{\pi}[-n] \overset{+1}{\longrightarrow},
$$
under which  $$\theta_{\Ta[-n]X} \mapsto \lambda_{\Ta[-n]X} \mapsto \im(\lambda_{\Ta[-n]X})$$ and $\im(\lambda_{\Ta[-n]X})$ is homotopic to zero. This implies that there is a log isotropic fibration
$$
\eta: 0 \to \omega_{\Ta[-n]X/X}
$$on the natural projection morphism $\pi: \Ta[-n]X \to X$. For Proposition \ref{prop:shifted-log-symp-lag}, the statement that $\pi: \Ta[-n]X \to X$ is a log Lagrangian fibration means precisely that $\eta$ is non-degenerate.
\end{re}

\begin{theorem}[Logarithmic Lagrangian Intersection Theorem]
\label{thm:log-lag-intersect}

    Let $w_\alpha: X \to L_{\cA}[-n]$ be a $(-n)$-shifted log function (see Definition \ref{def:shift-log-function-diff-form}) on a derived log Artin stack $X$.
    Then $dw_\alpha: X \to \Ta[-n]X$ (see Remark \ref{rmk:shift-log-diff-form-is-section-tw-log-cotangent-bund}) is a log Lagrangian, and the derived log critical locus $X(dw_\alpha)$ (see Definition \ref{def:der-log-crit-locus}) is $(-n-1)$-shifted log symplectic.
\end{theorem}

\begin{proof}
Following Remark \ref{rmk:taut-form-on-shift-log-cotang-bund}, let $\lambda_{\Ta[-n]X}$ denote the tautological $(-n)$-shifted log differential $1$-form on $\Ta[-n]X$, and $\omega_{\Ta[-n]X}$ denote  the canonical $(-n)$-shifted log symplectic form on $\Ta[-n]X$. Let $\pi: \Ta[-n]X \to X$ denote the natural projection morphism.

First, we prove that $dw_\alpha: X \to \Ta[-n]X$ is a log Lagrangian, which is basically a modification of the proof of {\cite[Theorem~2.4(2)]{Calaque19}} and {\cite[Theorem~2.22]{Calaque19}}.

By definition, $dw_\alpha: X \to \Ta[-n]X$ corresponds to an exact $(-n)$-shifted log differential $1$-form 
$$
    dw_\alpha: \cO_{X} \to \bLa_X[-n]
$$
on $X$.  Tautologically, we see that
$$
(dw_\alpha)^* \lambda_{\Ta[-n]X} : \cO_{X} \to \bLa_X[-n] \to (dw_\alpha)^*\bLa_{\Ta[-n]X}[-n] \simeq \bLa_X[-n] \oplus \bTT_X
$$
is homotopic to 
$$
\cO_{X} \overset{dw_\alpha}{\longrightarrow} \bLa_X[-n] \to \bLa_X[-n] \oplus \bTT_X   .
$$
As a result, there is a log isotropic structure
$$
\gamma_{dw_\alpha}: 0 \to d_{dR}dw_\alpha \to d_{dR} (dw_\alpha)^*  \lambda_{\Ta[-n]X} \to (dw_\alpha)^* d_{dR} \lambda_{\Ta[-n]X} =: (dw_\alpha)^* \omega_{\Ta[-n]X}
$$
on the morphism $dw_\alpha: X \to \Ta[-n]X$. This log isotropic structure $\gamma_{dw_\alpha}:0 \to (dw_\alpha)^* \omega_{\Ta[-n]X}$ induces a commutative diagram
\begin{equation}\label{diag:1-prf-log-lag}
\begin{tikzcd}
{\bTT_X} \arrow[r] \arrow[d] & {(dw_\alpha)^*\bTT_{\Ta[-n]X} \simeq (dw_\alpha)^*\bLa_{\Ta[-n]X}[-n]} \arrow[d] &    \\
                         {0} \arrow[r]    & {\bLa_X[-n]} .
\end{tikzcd}
\end{equation}
Note that the quasi-isomorphism $(dw_\alpha)^*\bTT_{\Ta[-n]X} \simeq (dw_\alpha)^*\bLa_{\Ta[-n]X}[-n]$ follows from Proposition \ref{prop:shifted-log-symp-lag}.

On the other hand, applying the log tangent complex functor to $dw_\alpha: X \to \Ta[-n]X$, we obtain a (co)fiber sequence
$$
\bTT_X \to (dw_\alpha)^*\bTT_{\Ta[-n]X} \to \bTT_{dw_\alpha}[1] \simeq (dw_\alpha)^*\bTT_{\pi} \simeq \bLa_X[-n] \overset{+1}{\longrightarrow} ,
$$
which is equivalent to a homotopy (co)Cartesian diagram
$$
\begin{tikzcd}
{\bTT_X} \arrow[r] \arrow[d] & {(dw_\alpha)^*\bTT_{\Ta[-n]X}} \arrow[d] &    \\
                         {0} \arrow[r]    & {\bLa_X[-n]} .
\end{tikzcd}
$$
By the universal property of homotopy (co)Cartesian diagram and Diagram (\ref{diag:1-prf-log-lag}), there is a commutative diagram
$$
\begin{tikzcd}
{\bTT_X} \arrow[rd, dashed, "\phi^{log}_{X, dw_\alpha}"] \arrow[rdd, bend right=15] \arrow[rrd, bend left=17] &                        &              \\
                                              & {\bTT_X} \arrow[r] \arrow[d] & {(dw_\alpha)^*\bTT_{\Ta[-n]X}} \arrow[d] \\
                                              & {0} \arrow[r]           & {\bLa_X[-n]}          ,
\end{tikzcd}
$$
where the square diagram is a homotopy (co)Cartesian diagram. Thus, we can see that if $\phi^{log}_{X, dw_\alpha}: \bTT_X \to \bTT_X$ is a quasi-isomorphism, then $\gamma_{dw_\alpha}$ is non-degenerate by definition, and hence $dw_\alpha: X \to \Ta[-n]X$ is a log Lagrangian.

By arguments similar to the proof of \cite[Lemma~2.6]{Calaque19}, given a morphism $f: Y \to X$ of derived log Artin stacks and a $(-n)$-shifted log function $w_\alpha$ on $X$, there is a commutative diagram
$$
\begin{tikzcd}[column sep=6em]
{\bTT_Y} \arrow[r, "\phi^{log}_{Y, \LL f^*(dw_\alpha)}"] \arrow[d] & {\bTT_Y} \arrow[d] \\
{\LL f^*\bTT_X} \arrow[r, "\LL f^*(\phi^{log}_{X, dw_\alpha})"]           & {\LL f^*\bTT_X}         . 
\end{tikzcd}
$$
Consequently, by inductive arguments similar to the proof of {\cite[Corollary~2.7]{Calaque19}}, to show that $\phi^{log}_{X, dw_\alpha}: \bTT_X \to \bTT_X$ is a quasi-isomorphism, it suffices to show the base case that $\phi^{log}_{X, dw_\alpha}: \bTT_X \to \bTT_X$ is a quasi-isomorphism when $X$ is a derived log affine scheme (homotopically) locally of finite presentation. The base case is true by standard computation, so we conclude that $dw_\alpha: X \to \Ta[-n]X$ is a log Lagrangian when $X$ is a derived log Artin stack.

Note that if $w_\alpha \simeq 0$ and hence $dw_\alpha \cong 0: X \to \Ta[-n]X$, there is a log isotropic structure $\gamma_0: 0 \to 0^*\omega_{\Ta[-n]X}$ that is non-degenerate. In other words, $0: X \to \Ta[-n]X$ is a log Lagrangian.

Next, we prove that $X(dw_\alpha)$ is $(-n-1)$-shifted log symplectic, which basically follows from the proof of {\cite[Theorem~2.9]{PTVV13}}.

Consider the homotopy Cartesian diagram
\begin{equation}\label{diag:prf-log-lag-intersect}
\begin{tikzcd}
{X(dw_\alpha)} \ar[dr, phantom, "\lrcorner_{\text{\tiny h}}", very near start] \arrow[r, "i_0"] \arrow[d, "i_{dw_\alpha}"'] & {X} \arrow[d, "dw_\alpha"'] \\
{X} \arrow[r, "0"]  & {\Ta[-n]X}           .
\end{tikzcd}
\end{equation}

The two morphisms
\begin{align*}
    &q_1: X(dw_\alpha) \overset{i_0}{\longrightarrow} X \overset{dw_\alpha}{\longrightarrow} \Ta[-n]X \\
    &q_2: X(dw_\alpha) \overset{i_{dw_\alpha}}{\longrightarrow} X \overset{0}{\longrightarrow} \Ta[-n]X
\end{align*}
come equipped with a natural homotopy $u: q_1 \Rightarrow q_2$. This $u$ gives rise to a homotopy between the induced morphisms on the spaces of closed $(-n)$-shifted log differential $2$-forms
$$
u^*: q_1^* \to q_2^*: \fA^{2,cl}(\Ta[-n]X,-n) \to \fA^{2,cl}(X(dw_\alpha),-n).
$$
Moreover, $\gamma_{dw_\alpha}$ (respectively, $\gamma_0$) defines a path 
$$
h: 0 \to q_1^*(\omega_{\Ta[-n]X}) \quad \left(k: 0 \to q_2^*(\omega_{\Ta[-n]X}) \right)
$$
in the space $\fA^{2,cl}(X(dw_\alpha),-n)$. By concatenation of 
\begin{align*}
    h:&\  0 \to q_1^*(\omega_{\Ta[-n]X}) \\
    u^*(\omega_{\Ta[-n]X}):&\  q_1^*(\omega_{\Ta[-n]X}) \to q_2^*(\omega_{\Ta[-n]X}) \\
    k^{-1}:&\  q_2^*(\omega_{\Ta[-n]X}) \to 0
\end{align*}
we get a loop at $0$ in the space $\fA^{2,cl}(X(dw_\alpha),-n)$, therefore a well-defined element
$$
\omega_{X(dw_\alpha)}\in \pi_1\left(\fA^{2,cl}\left(X(dw_\alpha), -n\right)\right)\footnote{We are using homological conventions here.} \simeq \pi_0\left(\fA^{2,cl}\left(X(dw_\alpha), -n-1\right)\right)
$$
It remains to show that the closed $(-n-1)$-shifted log differential $2$-form $\omega_{X(dw_\alpha)}$ is non-degenerate. 

Set $q:= q_1: X(dw_\alpha) \to \Ta[-n]X$ (or equivalently $q:= q_2$). 

By Lemma \ref{lem:cotangent-complex-homotopy-cartesian-diagram}, applying the log cotangent complex functor to Diagram (\ref{diag:prf-log-lag-intersect}), we obtain a homotopy (co)Cartesian diagram
$$
\begin{tikzcd}
{\bTT_{X(dw_\alpha)}} \arrow[r] \arrow[d] & {\LL i_{0}^*(\bTT_X)} \arrow[d] \\
{\LL i_{dw_\alpha}^*(\bTT_X)} \arrow[r]           & {\LL q^*(\bTT_{\Ta[-n]X})}          ,
\end{tikzcd}
$$
which is equivalent to a (co)fiber sequence
$$
\bTT_{X(dw_\alpha)} \longrightarrow \LL i_{dw_\alpha}^*(\bTT_X) \oplus \LL i_{0}^*(\bTT_X) \longrightarrow \LL q^*(\bTT_{\Ta[-n]X}) \overset{+1}{\longrightarrow}.
$$

Applying the log tangent complex to $q: X(dw_\alpha) \to \Ta[-n]X$, we obtain a (co)fiber sequence
$$
\LL q^*(\bLL_{\Ta[-n]X}) \longrightarrow \bLL_{X(dw_\alpha)} \longrightarrow \bLL_q \simeq \LL i_{dw_\alpha}^*(\bLL_0) \oplus \LL i_{0}^*(\bLL_{dw_\alpha}) \overset{+1}{\longrightarrow},
$$
which induces a (co)fiber sequence
$$
\bLa_{X(dw_\alpha)}[-n-1]  \longrightarrow \LL i_{dw_\alpha}^*(\bLa_0)[-n-1] \oplus \LL i_{0}^*(\bLa_{dw_\alpha})[-n-1] \to \LL q^*(\bLa_{\Ta[-n]X})[-n]   .
$$

By construction, there is a commutative diagram in $\QCoh\left(X(dw_\alpha)\right)$ with rows being (co)fiber sequences
$$
\begin{tikzcd}
{\bTT_{X(dw_\alpha)}} \arrow[r] \arrow[d, "\underline{\omega}_{X(dw_\alpha)}^{\flat}"'] & {\LL i_{dw_\alpha}^*(\bTT_X) \oplus \LL i_{0}^*(\bTT_X)} \arrow[d, "\LL i_{dw_\alpha}^* (\underline{\gamma}_{0}^{\flat}) \oplus \LL i_0^* (\underline{\gamma}_{dw_\alpha}^{\flat}) "'] \arrow[r] & {\LL q^*(\bTT_{\Ta[-n]X})} \arrow[r, "+1"] \arrow[d, "\LL q^* (\underline{\omega}_{\Ta[-n]X}^{\flat})"'] & {} \\
{\bLa_{X(dw_\alpha)}[-n-1]} \arrow[r]           & {\LL i_{dw_\alpha}^*(\bLa_0)[-n-1] \oplus \LL i_{0}^*(\bLa_{dw_\alpha})[-n-1]} \arrow[r]           & {\LL q^*(\bLa_{\Ta[-n]X})[-n]} \arrow[r, "+1"]           & {}        .
\end{tikzcd}
$$
Note that $\underline{\gamma}_{dw_\alpha}^{\flat}$ and $\underline{\gamma}_{0}^{\flat}$ are quasi-isomorphisms because $\gamma_{dw_\alpha}$ and $\gamma_0$ are non-degenerate, which follows from the fact that $dw_\alpha: X \to \Ta[-n]X$ and $0: X \to \Ta[-n]X$ are log Lagrangians; $\underline{\omega}_{\Ta[-n]X}^{\flat}$ is a quasi-isomorphism because $\omega_{\Ta[-n]X}$ is non-degenerate, which follows from Proposition \ref{prop:shifted-log-symp-lag}. Thus, we conclude that $\underline{\omega}_{X(dw_\alpha)}^{\flat}$ is a quasi-isomorphism, and hence $\omega_{X(dw_\alpha)}$ is a $(-n-1)$-shifted log symplectic form on $X(dw_\alpha)$.

\end{proof}

We conclude this subsection by proposing a series of questions. What are the relationships among shifted symplectic structures \cite{PTVV13}, twisted symplectic structures \cite{CS24}, and shifted log symplectic structures? For example, the open complement of the virtual Cartier divisor of a $(-n)$-shifted log symplectic derived log stack is a $(-n)$-shifted symplectic derived stack. Conversely, given a shifted symplectic derived stack, can one construct a log compactification that carries a shifted log symplectic structure? Given that shifted log symplectic structures are special cases of twisted symplectic structures and the base stack for shifted log structures is $\cA:= [\bA^1 / \bG_m]$, are shifted log symplectic structures in some sense universal among twisted symplectic structures?

\subsection{Derived Projective Completions and Logarithmic Compactifications}\label{sec:der-proj-complete-log-compact}

In this subsection, we show that derived projective completions (following \cite{Hekking21} and \cite{Jiang25}) are naturally log compactifications (Proposition \ref{prop:log-compact-derived-proj-completion}). As a consequence, we obtain a log compactification for a special class of derived Weil restrictions (Corollary \ref{cor:log-compact-derived-weil-restriction}), which lays the foundation for a prototype for derived Logarithmic GLSM (see Section \ref{sec:4}).

In addition to {\cite[Appendix~A]{AHPS23}}, we refer to \cite{Hekking21} for a detailed discussion of derived projective spectra and to {\cite[Section~4]{Jiang25}} for an alternative treatment.

Let $X$ be a derived stack, and $\cE^{\vee}$ be a perfect complex on $X$ of Tor-amplitude in $[-n,0]$ for some $n \in \bZ_{\ge 0}$. Then the natural morphism $$E := \Spec_X \Sym_X (\cE^{\vee}) \to X$$ is representable by derived schemes and  homotopically $n$-smooth (see \cite{ToVa07} and \cite{AnGe14}). $E$ is called a derived vector bundle.

Set $\cB := \Sym_X(\cE^{\vee})$ and $\cB_n:= \Sym_X^n (\cE^{\vee})$. In particular, $\cB_1 = \cE^{\vee}$. $\cB = \Sym_X(\cB_1)$ is a quasi-coherent $\cO_X$-graded algebra generated in grading\footnote{We use the term grading instead of degree.} $1$  (see {\cite[Appendix~A]{AHPS23}}). We define $\cB_+ := \bigoplus_{n > 0} \cB_n$ as the irrelevant ideal of $\cB$, which comes with a (co)fiber sequence
$$
\cB_+ \longrightarrow \cB \longrightarrow \cB_0 \overset{+1}{\longrightarrow}.
$$
The closed embedding $\Spec_X \cB_0 \to \Spec_X \cB$ is written as $V(\cB_+) \to \Spec_X \cB$. Note that the $\bG_{m,X}$-action on $\Spec_X \cB$ restricts to a $\bG_{m,X}$-action on $(\Spec_X \cB) \setminus V(\cB_+)$. The \emph{relative projective spectrum $\Proj_X \cB$ over $X$} is defined as the derived quotient stack $$\left[ \left( \left(\Spec_X \cB\right) \setminus V\left(\cB_+\right) \right) / \bG_{m,X} \right]$$ over $X$. The natural morphism $$\bP_X(E):= \Proj_X \cB = \Proj_X \Sym_X (\cE^{\vee}) \to X$$ is representable by derived schemes due to {\cite[Proposition~A.6.(iii)]{AHPS23}} and {\cite[Proposition~A.6.(iv)]{AHPS23}}.

\begin{re}
\label{rmk:proj-sym-being-empty}

    By definition, there is a natural open embedding $\bP_X(E) \hookrightarrow \left[E/\bG_{m,X}\right]$.
\end{re}

Let $U$ denote $(\Spec_X \cB) \setminus V(\cB_+)$ for simplicity. There is a commutative diagram as follows.
$$
\begin{tikzcd}
{U} \arrow[r] \arrow[d, hook] & {\bP_X(E)} \arrow[d, "\bar{\pi}"] \\
{E} \arrow[r, "\pi"']           & {X}          
\end{tikzcd}
$$

\begin{theorem}[Euler (co)fiber sequence, {\cite[Theorem~4.27]{Jiang25}}]
\label{thm:derived-Euler-seq}
    
    Under the assumptions above, there is a (co)fiber sequence
    $$
    \bL_{\bP_X(E)/X} \longrightarrow \LL \bar{\pi}^* (\cE^{\vee}) \otimes^{\LL}_{\cO_{\bP_X(E)}} \cO_{\bP_X(E)}(-1) \longrightarrow \cO_{\bP_X(E)} \overset{+1}{\longrightarrow}
    $$
    in $\QCoh\left(\bP_X(E)\right)$.
\end{theorem}

\begin{re}\label{rmk:Jeremiah}
    One can give an alternative proof of Theorem \ref{thm:derived-Euler-seq} using Remark \ref{rmk:cotangent-complx-der-quot-stk} (see Example \ref{eg:cotangent-cmplx-der-proj-space}) and induction on the Tor-amplitude of a perfect complex.
\end{re}

\begin{example}
\label{eg:cotangent-cmplx-der-proj-space}

    Set $\cE^{\vee} := \cO_X^{\oplus 2}\oplus (\cO_X[1])$. Then $E \cong (\bA^2 \times^h \bA^1[-1]) \times^h X$, whose classical truncation $E_\cl$ is $\bA^2 \times X_{\cl}$. Hence, the classical truncation $(\bP_X(E))_\cl$ of $\bP_X(E)$ is $\bP^1 \times X_{\cl}$. 

    We write $\cE^{\vee} = \cO_X \langle x_1 \rangle \oplus \cO_X \langle x_2 \rangle \oplus \cO_X \langle \nu \rangle = \cO_X \langle x_1, x_2; \nu \rangle$ with cohomological degrees $|x_1|=|x_2|=0$ and $|\nu|=-1$. For the $\cO_X$-graded algebra $\cB := \Sym_X\left(\cE^{\vee}\right) = \cO_X[x_1,x_2 ; \nu]$, the generators $x_1, x_2$ and $\nu$ all have grading $1$. 
    
    Consider the action morphism\footnote{Note that $\bG_m$ is a linearly reductive group scheme and its higher group cohomology vanishes.}
    \begin{align*}
        \psi_X: \mathbb{G}_{m,X} \times^h_X U &\longrightarrow U \\
        \left(t,\left(x_1, x_2, \nu \right)\right) &\longmapsto \left(tx_1, tx_2, t\nu \right) .
    \end{align*}
    Note that $\psi_X$ admits a section $(e_X, \id_U): U \cong X \times^h_X U \to \mathbb{G}_{m,X} \times^h_X U$, where $e_X: X \to \bG_{m,X}$ is the identity section of $\bG_{m,X}$. Let $\mathfrak{g}$ denote the tangent space of $\bG_m$ at the identity.
    
    Applying the relative cotangent complex functor to $\psi_X$ and then pulling back to $U$ along $(e_X, \id_U)$, we obtain a morphism
    $$
    d\psi_X|_{U}:=\LL(e_X,\id_{U})^*\left(\bL\psi_X\right): \bL_{U/X} \simeq \cE^{\vee}|_U \to \bL_{\bG_{m,X} \times^h_X U/X}|_{U} \simeq \left( \mathfrak{g}^{\vee} \boxtimes \cO_{U} \right) \oplus \left(\cE^{\vee}|_U \right) ,
    $$
    which can be explicitly written as
    \begin{align*}
        d\psi_X|_{U}: \cO_{U}\langle dx_1, dx_2; d\nu \rangle &\longrightarrow \cO_{U}\langle dt \rangle \oplus \cO_{U}\langle dx_1, dx_2, d\nu \rangle \\
        dx_j &\longmapsto \left( d\left(tx_j\right) \right)|_{t=1} = \left( x_jdt + tdx_j \right)|_{t=1} =x_jdt + dx_j \\
        d\nu &\longmapsto \left( d\left(t\nu\right) \right)|_{t=1} = \left( \nu dt + td\nu \right)|_{t=1} =\nu dt + d\nu.
\end{align*}
    Thus, by setting
\begin{align*}
    d\bar{\psi}_X : \cO_{U}\langle dx_1, dx_2; d\nu \rangle \longrightarrow& \cO_{U}\langle dt \rangle \\
    dx_j \longmapsto& x_jdt \\
    d\nu \longmapsto& \nu dt ,
\end{align*}
    we have $d\psi_X|_{U} \simeq d\bar{\psi}_X \oplus \id$. Following Remark \ref{rmk:cotangent-complx-der-quot-stk}, we conclude that
\begin{align*}
    \bL_{[U/\bG_{m,X}]/X}|_{U} &\simeq \fib \left( d\bar{\psi}_X: \cO_{U}\langle dx_1, dx_2; d\nu \rangle \longrightarrow \cO_{U}\langle dt \rangle \right)\\
    &\simeq \fib \left( \cO_U \oplus \cO_U \oplus \left(\cO_U [1]\right) \overset{\cdot x_1 \oplus \cdot x_2 \oplus \cdot \nu}{\longrightarrow} \cO_{U}  \right),
\end{align*}
    which is a perfect complex of Tor-amplitude in $[-1,1]$ on $U$.  Note that the morphism $\cO_{U}[1] \overset{\cdot \nu}{\longrightarrow} \cO_{U}$ is well-defined, because $\nu$ is literally a $(-1)$-shifted function on $U$ (See Remark \ref{rmk:shifted-function-and-homotopy} and Example \ref{eg:aff-derived-zero-locus}).

    Since $x_1, x_2$ and $\nu$ all have grading $1$, we conclude that
$$
    \bL_{\bP_X(E)/X} \simeq \fib \left( \cO_{\bP_X(E)}(-1) \oplus \cO_{\bP_X(E)}(-1) \oplus \left(\cO_{\bP_X(E)}(-1)[1]\right) \to \cO_{\bP_X(E)} \right),
$$
    which is a perfect complex of Tor-amplitude in $[-1,1]$ on $\bP_X(E)$. By {\cite[Proposition~A.6(iii)]{AHPS23}} and {\cite[Proposition~A.6(iv)]{AHPS23}}, the natural morphism $\bP_X(E) \to X$ is representable by derived schemes. This forces $\bL_{\bP_X(E)/X}$ to be a perfect complex of Tor-amplitude in $[-1,0]$ on $\bP_X(E)$.

    Further assume that $X := \pt$. Then there is a natural closed embedding $$\tau: \bP^1 \cong \left(\bP_X(E)\right)_{\cl} \hookrightarrow \bP_X(E)$$ that induces a (co)fiber sequence
$$
    \bL_{\bP_X(E)}|_{\bP^1} \simeq \fib \left( \cO_{\bP^1}(-1)^{\oplus 2} \oplus \left(\cO_{\bP^1}(-1)[1]\right) \to \cO_{\bP^1} \right) \overset{\bL\tau}{\longrightarrow} \bL_{\bP^1} \simeq \Omega^1_{\bP^1} \longrightarrow \bL_{\bP^1/\bP_X(E)} \overset{+1}{\longrightarrow}.
$$
    By Proposition \ref{prop:cotangent-complex-obstruction-theory} (or Corollary \ref{cor:qsmooth-POT}), $\bL \tau$ induces a perfect obstruction theory on $\bP^1$. In particular, there is an isomorphism
$$
H^0(\bL\tau): H^0(\bL_{\bP_X(E)}|_{\bP^1}) \overset{=}{\longrightarrow} H^0(\bL_{\bP^1}) = \Omega^1_{\bP^1}.
$$
    Note that the morphism $\cO_{\bP^1}(-1)^{\oplus 2} \oplus \left(\cO_{\bP^1}(-1)[1]\right) \to \cO_{\bP^1}$ can be explicitly written as
$$
\begin{tikzcd}
{\cO_{\bP^1}(-1)^{\oplus 2}} \arrow[r, "\cdot x_1 \oplus \cdot x_2"]           & {\cO_{\bP^1}}           \\
{\cO_{\bP^1}(-1)} \arrow[u, "0"] \arrow[r] & {0} \arrow[u] .
\end{tikzcd}
$$
    Hence, $\bL_{\bP_X(E)}|_{\bP^1}$ is quasi-isomorphic to the complex
$
    \big[ \cO_{\bP^1}(-1) \overset{0}{\longrightarrow} \cO_{\bP^1}(-1)^{\oplus 2} \overset{\cdot x_1 \oplus \cdot x_2}{\longrightarrow} \cO_{\bP^1} \big].
$
    Note that $\cO_{\bP^1}(-1)^{\oplus 2} \overset{\cdot x_1 \oplus \cdot x_2}{\longrightarrow} \cO_{\bP^1}$ is surjective, so we have
$$
H^1(\bL_{\bP_X(E)}|_{\bP^1}) = \coker \left( \cO_{\bP^1}(-1)^{\oplus 2} \overset{\cdot x_1 \oplus \cdot x_2}{\longrightarrow} \cO_{\bP^1} \right) = 0.
$$
    This forces $\bL_{\bP_X(E)}|_{\bP^1}$ to be a perfect complex of Tor-amplitude in $[-1,0]$, which agrees with what we discussed above. Moreover, we have
$$
\Omega^1_{\bP^1} =H^0(\bL_{\bP_X(E)}|_{\bP^1}) = \ker \left( \cO_{\bP^1}(-1)^{\oplus 2} \to \cO_{\bP^1} \right),
$$
    recovering the usual Euler sequence for $\bP^1$. $\blacksquare$
\end{example}

\begin{corollary}\label{cor:der-proj-vect-bund}

    Let $X$ be a derived stack, and $E^{\vee}$ be a perfect complex on $X$ of Tor-amplitude in $[-n,0]$ for some $n \in \bZ_{\ge 0}$. Then the natural morphism $$\bP_X(E) := \Proj_X \Sym_X (\cE^{\vee}) \to X$$ is representable by derived schemes and homotopically $n$-smooth.
\end{corollary}

\begin{proof}
    It directly follows from \cite[Proposition~A.6(iii)]{AHPS23}, \cite[Proposition~A.6(iv)]{AHPS23}, and the Euler (co)fiber sequence (Theorem \ref{thm:derived-Euler-seq}).
\end{proof}

\begin{definition}[\cite{Jiang25}]\label{def:der-proj-complete}

    $\bP_X(E) := \Proj_X \Sym_X (\cE^{\vee})$ is called the \emph{derived projectivization of a derived vector bundle $E := \Spec_X \Sym_X (\cE^{\vee})$ over $X$}. 
    $\bP_X(E \oplus 1) := \Proj_X \Sym_X \left( \cE^{\vee} \oplus \cO_X \right)$ is called the \emph{derived projective completion of a derived vector bundle $E := \Spec_X \Sym_X (\cE^{\vee})$ over $X$}.
\end{definition}

\begin{lemma}\label{lem:1-prop-log-compact-proj-completion}
    The natural morphism $\bP(\cE \oplus 1) \to X$ is proper.
\end{lemma}

\begin{proof}
    It directly follows from \cite[Proposition~A.6(v)]{AHPS23} and \cite[Remark~5.1.2.2]{LurieSAG}.
\end{proof}

\begin{lemma}\label{lem:2-prop-log-compact-proj-completion}
    The natural morphism $E \hookrightarrow \bP(E \oplus 1)$ is an open embedding.
\end{lemma}

\begin{proof}
    Let $t_{\text{\tiny $\infty$}}$ denote the coordinate associated with $\cO_X$ in $\Sym_X \left( \cE^{\vee} \oplus \cO_X \right)$. Note that $t_{\text{\tiny $\infty$}}$ is a grading $1$ homogeneous element, and has cohomological degree $|t_{\text{\tiny $\infty$}}|=0$. Let $\left( \Sym_X \left( \cE^{\vee} \oplus \cO_X \right) \right)_{(t_{\text{\tiny $\infty$}})}$ denote the grading-zero part of the homogeneous localization\footnote{It is called degree-zero part of the homogeneous localization in \cite{Hekking21}.} $\left(\Sym_X \left( \cE^{\vee} \oplus \cO_X \right)\right)_{t_{\text{\tiny $\infty$}}}$ of $\Sym_X \left( \cE^{\vee} \oplus \cO_X \right)$ at $t_{\text{\tiny $\infty$}}$. 

    There is a morphism $\Sym_X \left( \cE^{\vee} \oplus \cO_X \right) \to \Sym_X \left( \cE^{\vee}\right)$ that sends $t_{\text{\tiny $\infty$}}$ to $1$. Hence, we obtain a morphism $$\left( \Sym_X \left( \cE^{\vee} \oplus \cO_X \right) \right)_{(t_{\text{\tiny $\infty$}})} \to \left(\Sym_X \left( \cE^{\vee} \oplus \cO_X \right)\right)_{t_{\text{\tiny $\infty$}}} \to \Sym_X \left( \cE^{\vee}\right)$$ by the universal property of homogeneous localization {\cite[Definition~2.7.2]{Hekking21}}. We claim that the morphism $\left( \Sym_X \left( \cE^{\vee} \oplus \cO_X \right) \right)_{(t_{\text{\tiny $\infty$}})} \to \Sym_X \left( \cE^{\vee}\right)$ is an equivalence of $\cO_X$-algebras, and hence proving the lemma.
    
    To prove the claim, it suffices to show that the morphism $\left( \Sym_X \left( \cE^{\vee} \oplus \cO_X \right) \right)_{(t_{\text{\tiny $\infty$}})} \to \Sym_X \left( \cE^{\vee}\right)$ of underlying complexes is a quasi-isomorphism. Thus, it suffices to show isomorphisms
    $$H^{-n}\left(\left( \Sym_X \left( \cE^{\vee} \oplus \cO_X \right) \right)_{(t_{\text{\tiny $\infty$}})}\right) =  H^{-n}\left(\Sym_X \left( \cE^{\vee}\right)\right)$$ for $n \in \bZ_{\ge 0}$. By explicitly writing out the underlying complex of $\left( \Sym_X \left( \cE^{\vee} \oplus \cO_X \right) \right)_{(t_{\text{\tiny $\infty$}})}$, one verifies that these isomorphisms hold, completing the proof.
\end{proof}

\begin{re}
    For a graded algebra $\cB:=\Sym_X \left( E^{\vee} \oplus \cO_X \right)$,  we have isomorphisms
    $$ H^{-n}\left( \cB_{(t_{\text{\tiny $\infty$}})} \right) := H^{-n} \left(\left( \cB_{t_{\text{\tiny $\infty$}}}\right)_0\right) = \left(H^{-n} \left(\cB_{t_{\text{\tiny $\infty$}}}\right)\right)_0 = \left(\left(H^{-n} \left( \cB \right)\right)_{t_{\text{\tiny $\infty$}}}\right)_0 $$
    for $n \in \bZ_{\ge 0}$ by {\cite[Lemma~2.5.2]{Hekking21}} and \cite[Proposition~4.1.18]{LurieDAGV}.
\end{re}

\begin{lemma}\label{lem:3-prop-log-compact-proj-completion}
    The natural morphism $\bP_X(E) \hookrightarrow \bP_X(E \oplus 1)$ is a quasi-smooth closed embedding of virtual codimension $1$, and hence a virtual Cartier divisor.
\end{lemma}

\begin{proof}
    It directly follows from \cite[Corollary~4.32]{Jiang25}, which essentially follows from the Euler (co)fiber sequence (Theorem \ref{thm:derived-Euler-seq}).
\end{proof}

\begin{proposition}[Derived Projective Completions Are Log Compactifications]
\label{prop:log-compact-derived-proj-completion}

Let $X$ be a derived stack, and $\cE^{\vee}$ be a perfect complex on $X$ of Tor-amplitude in $[-n,0]$ for some $n \in \bZ_{\ge 0}$. Then the derived vector bundle $E := \Spec_X \Sym_X (\cE^{\vee})$ over $X$ admits a log compactification $\bP_X(E \oplus 1) := \Proj_X \Sym_X \left( \cE^{\vee} \oplus \cO_X \right)$ in the following sense:

$(1)$ the natural morphism $\bP_X(E \oplus 1) \to X$ is proper;

$(2)$ the natural morphism $E \hookrightarrow \bP_X(E \oplus 1)$ is an open embedding;

$(3)$ the natural morphism $\bP_X(E) := \Proj_X \Sym_X \left( \cE^{\vee} \right) \hookrightarrow \bP_X(E \oplus 1)$ is a quasi-smooth closed embedding of virtual codimension $1$, and hence a virtual Cartier divisor.

In particular, $(3)$ implies that there is a morphism $$\bP_X(E \oplus 1) \to \cA \times^h X$$ such that $\left( \bP_X(E \oplus 1), \bP_X(E \oplus 1) \to \cA \times^h X \right)$ is a relative derived log stack over $X$.
\end{proposition}

\begin{proof}
    $(1)$ is proved by Lemma \ref{lem:1-prop-log-compact-proj-completion}; $(2)$ is proved by Lemma \ref{lem:2-prop-log-compact-proj-completion}; $(3)$ is proved by Lemma \ref{lem:3-prop-log-compact-proj-completion}.    
\end{proof}

\begin{re}
    Note that in Proposition \ref{prop:log-compact-derived-proj-completion} we do not assume that $X$ is equipped with a log structure.
\end{re}

\begin{re}
     Suppose now $X$ is additionally equipped with a log structure, i.e. an anchor morphism $X \to \cA$. Then the derived vector bundle $E$ over $X$ is naturally equipped with a log structure via the composition $E \to X \to \cA$. Similarly, $\bP_X(E \oplus 1)$ is naturally equipped with a log structure via the composition $\bP_X(E \oplus 1) \to X \to \cA$, which is, in general, different from the log structure given in Proposition \ref{prop:log-compact-derived-proj-completion}.
\end{re}

\begin{example}[{\cite{Hekking21, Jiang25}}]
Let $Y \hookrightarrow X$ be a quasi-smooth closed embedding of derived stacks (e.g. a virtual Cartier divisor or a regular closed embedding). Then $\bL_{Y/X}$ is a perfect complex of Tor-amplitude in $[-1,-1]$ (see \cite{KhRy25}). This implies that both $\bL_{Y/X}[-1]$ and $\bT_{Y/X}[1]$ are perfect complexes of Tor-amplitude $[0,0]$. Thus, the normal bundle for $Y \hookrightarrow X$ is the derived vector bundle $$N_{Y/X}:= \Spec_{Y} \Sym_{Y} \left( \bL_{Y/X}[-1] \right)$$ over $Y$; the conormal bundle for $Y \hookrightarrow X$ is the derived vector bundle $$N^*_{Y/X}:= \Spec_{Y} \Sym_{Y} \left( \bT_{Y/X}[1] \right)$$ over $Y$. Consequently, the derived projective completion of the normal bundle $N_{Y/X}$ is defined as $$\bP\left( N_{Y/X}\oplus 1 \right) := \Proj_{Y} \Sym_{Y} \left( \bL_{Y/X}[-1] \oplus \cO_{Y} \right);$$ the derived projective completion of the conormal bundle is defined as $$\bP\left( N^*_{Y/X}\oplus 1 \right) := \Proj_{Y} \Sym_{Y} \left( \bT_{Y/X}[1] \oplus \cO_{Y} \right).$$ Both constructions give log compactifications by Proposition \ref{prop:log-compact-derived-proj-completion}. $\blacksquare$
    
\end{example}

\begin{corollary}[A Log Compactification for a Special Class of Derived Weil Restrictions]
\label{cor:log-compact-derived-weil-restriction}

Assume that $p: C \to M$ is a flat proper quasi-smooth morphism of derived $1$-Artin stacks that is representable by derived schemes. Let $\cE^{\vee}$ be a perfect complex on $C$ of Tor-amplitude in $[-n,0]$ for some $n \in \bZ_{\ge 0}$, and $E := \Spec_{C} \Sym_{C} (\cE^{\vee})$ be a derived vector bundle over $C$. Then the derived Weil restriction $p_*E$ is equivalent to the derived vector bundle $\Spec_{M} \Sym_{M} \left( \left( \RR p_* \cE \right)^{\vee} \right)$, which admits a log compactification $$\bP_M(p_*E \oplus 1) := \Proj_{M} \Sym_{M} \left( \left( \RR p_* \cE \right)^{\vee} \oplus \cO_{M} \right).$$
\end{corollary}

\begin{proof}

    Since $p: C \to M$ is a proper morphism that is representable by derived schemes, $p: C \to M$ is formally proper by \cite[Corollary~2.3.5]{HLP23} and \cite[Corollary~4.3.2]{HLP23}. Moreover, $p: C \to M$ is flat, so $p: C \to M$ is of Tor-amplitude $0$. Thus, $\RH_{M}( C, E )$ and $\RH_{M}( C, C )$ are derived $1$-Artin stacks over $M$ by {\cite[Theorem~5.1.1]{HLP23}}. Therefore, the derived Weil restriction $p_*E$ is a derived $1$-Artin stack over $M$ by Remark \ref{rmk:Weil-Restriction-is-homotopy-fib-prod} and Remark \ref{rmk:Weil-Restriction-is-algebraic}.

    Since $p: C \to M$ is a proper quasi-smooth morphism that is representable by derived schemes, $\RR p_* \cE$ is a perfect complex on $M$ by \cite[Theorem~6.1.3.2]{LurieSAG} and \cite[Lemma~1.15]{Khan22} (see Remark \ref{rmk:pushforward-preserve-perf-complex}). Then, $p_*E \cong \Spec_{M} \Sym_{M} \left( \left( \RR p_* \cE \right)^{\vee} \right)$ is a derived vector bundle over $M$ following the proof of {\cite[Proposition~1.2.3]{KMMP25}}. 
    
    Now apply Proposition \ref{prop:log-compact-derived-proj-completion} and the proof is complete.
\end{proof}

\subsection{The Canonical Recursive Structure on Derived Logarithmic Stacks}\label{sec:can-recurs}

This subsection is independent of the remainder of the paper. We include it to show an interesting feature of the derived log geometry developed here: every derived log stack carries the canonical recursive structure due to derived self-intersections. An interesting byproduct of such canonical recursive structure is that it allows us to produce an infinite sequence of perfect obstruction theories from a single perfect obstruction theory (see Remark \ref{rmk:infin-seq-POT-case-0} and Remark \ref{rmk:infin-seq-POT-case--1}). Another interesting perspective provided by the canonical recursive structure is that we may now view $\pt$ equipped with non-trivial log structure in a derived way (see Remark \ref{rmk:subtlety-log-structure} and Remark \ref{rmk:subtlety-log-structure-exp}).

\begin{definition}
    Let $n, d \in \bZ_{\ge 0}$. The \emph{$(-n)$-shifted $d$-dimensional affine space $\bA^d[-n]$} is defined to be $\Spec \Sym (k^{\oplus d}[n])$.
\end{definition}

\begin{re}
    For $n \ge 1$, the classical truncation of $\bA^d[-n]$ is $\pt$.
\end{re}

\begin{proposition}\label{prop:base-case-canonical-recursive-structure}
    There is a commutative diagram as follows, where each square diagram is a homotopy Cartesian diagram.
$$
\begin{tikzcd}
{\cdots} \arrow[r, hook] \arrow[d, hook] \ar[dr, phantom, "\lrcorner_{\text{\tiny h}}", very near start] & {\bA^3[-1]} \arrow[r, hook] \arrow[d, hook] \ar[dr, phantom, "\lrcorner_{\text{\tiny h}}", very near start] & {\bA^2[-1]} \arrow[r, hook] \arrow[d, hook] \ar[dr, phantom, "\lrcorner_{\text{\tiny h}}", very near start] & {\bA^1[-1]} \arrow[r, hook] \arrow[d, hook] \ar[dr, phantom, "\lrcorner_{\text{\tiny h}}", very near start] & {\pt} \arrow[r] \arrow[d, hook, "0"] \ar[dr, phantom, "\lrcorner_{\text{\tiny h}}", very near start] & {B\bG_m} \arrow[d, hook] \\
{\cdots} \arrow[r, hook]                 & {\bA^2[-1]} \arrow[r, hook]                 & {\bA^1[-1]} \arrow[r, hook]                 & {\pt} \arrow[r, hook, "0"]                 & {\bA^1} \arrow[r]                 & {\cA}                
\end{tikzcd}
$$
\end{proposition}

\begin{proof}
$$
\begin{tikzcd}
{\cdots} \ar[dr, phantom, "\lrcorner_{\text{\tiny h}}", very near start] \arrow[d, hook] \arrow[r, hook] & {\cdots} \ar[dr, phantom, "\lrcorner_{\text{\tiny h}}", very near start] \arrow[d, hook] \arrow[r, hook] & {\cdots} \ar[dr, phantom, "\lrcorner_{\text{\tiny h}}", very near start] \arrow[d, hook] \arrow[r, hook] & {\cdots} \ar[dr, phantom, "\lrcorner_{\text{\tiny h}}", very near start] \arrow[r, hook] \arrow[d, hook] & {\cdots}  \arrow[d, hook] \\
{\bA^3[-1]} \ar[dr, phantom, "\lrcorner_{\text{\tiny h}}", very near start] \arrow[d, hook] \arrow[r, hook] & {\bA^2[-1]} \ar[dr, phantom, "\lrcorner_{\text{\tiny h}}", very near start] \arrow[d, hook] \arrow[r, hook] & {\bA^1[-1]} \ar[dr, phantom, "\lrcorner_{\text{\tiny h}}", very near start] \arrow[d, hook] \arrow[r, hook] & {\pt} \ar[dr, phantom, "\lrcorner_{\text{\tiny h}}", very near start] \arrow[d, hook, "0"] \arrow[r, hook, "0"] & {\cdots}  \arrow[d, hook] \\
{\bA^2[-1]} \ar[dr, phantom, "\lrcorner_{\text{\tiny h}}", very near start] \arrow[d, hook] \arrow[r, hook] & {\bA^1[-1]} \ar[dr, phantom, "\lrcorner_{\text{\tiny h}}", very near start] \arrow[d, hook] \arrow[r, hook] & {\pt} \ar[dr, phantom, "\lrcorner_{\text{\tiny h}}", very near start] \arrow[d, hook, "0"] \arrow[r, hook, "0"] & {\bA^1} \ar[dr, phantom, "\lrcorner_{\text{\tiny h}}", very near start] \arrow[d, hook] \arrow[r, hook] & {\cdots} \arrow[d, hook] \\
{\bA^1[-1]} \ar[dr, phantom, "\lrcorner_{\text{\tiny h}}", very near start] \arrow[r, hook] \arrow[d, hook] & {\pt} \ar[dr, phantom, "\lrcorner_{\text{\tiny h}}", very near start] \arrow[r, hook, "0"] \arrow[d, hook, "0"] & {\bA^1} \ar[dr, phantom, "\lrcorner_{\text{\tiny h}}", very near start] \arrow[r, hook] \arrow[d, hook] & {\bA^2} \ar[dr, phantom, "\lrcorner_{\text{\tiny h}}", very near start] \arrow[r, hook] \arrow[d, hook] & {\cdots} \arrow[d, hook] \\
{\pt} \arrow[r, hook, "0"]                 & {\bA^1} \arrow[r, hook]                 & {\bA^2} \arrow[r, hook]                 & {\bA^3} \arrow[r, hook]                 & {\cdots}                
\end{tikzcd}
$$
\end{proof}

\begin{re}
    By Proposition \ref{prop:base-case-canonical-recursive-structure}, there is an infinite sequence of virtual Cartier divisors
$${\cdots} \hookrightarrow \bA^3[-1] \hookrightarrow \bA^2[-1] \hookrightarrow \bA^1[-1] \hookrightarrow \pt \overset{0}{\hookrightarrow} \bA^1.$$
\end{re}

\begin{corollary}\label{cor:fund-canonical-resursive-structure}
There is a commutative diagram as follows, where each square diagram is a homotopy Cartesian diagram.
$$
\begin{tikzcd}
{\cdots} \arrow[r, hook] \arrow[d, hook] \ar[dr, phantom, "\lrcorner_{\text{\tiny h}}", very near start] & {\left[\bA^3[-1]/\bG_m\right]} \arrow[r, hook] \arrow[d, hook] \ar[dr, phantom, "\lrcorner_{\text{\tiny h}}", very near start] & {\left[\bA^2[-1]/\bG_m\right]} \arrow[r, hook] \arrow[d, hook] \ar[dr, phantom, "\lrcorner_{\text{\tiny h}}", very near start] & {\left[\bA^1[-1]/\bG_m\right]} \arrow[r, hook] \arrow[d, hook] \ar[dr, phantom, "\lrcorner_{\text{\tiny h}}", very near start] & {B\bG_m} \arrow[d, hook] \\
{\cdots} \arrow[r, hook]                 & {\left[\bA^2[-1]/\bG_m\right]} \arrow[r, hook]                 & {\left[\bA^1[-1]/\bG_m\right]} \arrow[r, hook]                 & {B\bG_m} \arrow[r, hook]                 & {\cA}             
\end{tikzcd}
$$
\end{corollary}

\begin{proof}
    This is due to Proposition \ref{prop:base-case-canonical-recursive-structure} and the fact that
    $$
    \left[ X / \bG_m \right] \times^h_{\left[ Z / \bG_m \right]} \left[ Y / \bG_m \right] \cong \left[ X \times^h_Z Y / \bG_m \right].
    $$
\end{proof}

\begin{re}
By Corollary \ref{cor:fund-canonical-resursive-structure}, there is an infinite sequence of virtual Cartier divisors
$${\cdots} \hookrightarrow \left[\bA^3[-1]/\bG_m\right] \hookrightarrow \left[\bA^2[-1]/\bG_m\right] \hookrightarrow \left[\bA^1[-1]/\bG_m\right] \hookrightarrow B\bG_m \hookrightarrow \cA.$$
\end{re}

\begin{re}\label{rmk:der-log-structure-on-virt-Cart-div}
    By Corollary \ref{cor:fund-canonical-resursive-structure}, we see that the virtual Cartier divisor $$\left[\bA^1[-1]/\bG_m\right] \hookrightarrow B\bG_m$$ on $B\bG_m$ yields a derived log stack $$\left(B\bG_m, B\bG_m \hookrightarrow \cA\right),$$
    and the virtual Cartier divisor $$\left[\bA^2[-1]/\bG_m\right] \hookrightarrow \left[\bA^1[-1]/\bG_m\right]$$ on $\left[\bA^1[-1]/\bG_m\right]$ yields a derived log stack $$\left(\left[\bA^1[-1]/\bG_m\right], \left[\bA^1[-1]/\bG_m\right] \hookrightarrow B\bG_m \hookrightarrow \cA\right).$$
    For $d \ge 2$, the virtual Cartier divisor $$\left[\bA^{d+1}[-1]/\bG_m\right] \hookrightarrow \left[\bA^d[-1]/\bG_m\right]$$ on $\left[\bA^d[-1]/\bG_m\right]$ yields a derived log stack $$\left( \left[\bA^d[-1]/\bG_m\right], \left[\bA^d[-1]/\bG_m\right] \hookrightarrow \left[\bA^{d-1}[-1]/\bG_m\right]  \hookrightarrow  {\cdots}    \hookrightarrow  B\bG_m \hookrightarrow \cA \right).$$
\end{re}

\begin{re}
    For $d \ge 1$, the classical truncation of $\left[\bA^d[-1]/\bG_m\right]$ is $B\bG_m$.
\end{re}

\begin{definition}\label{def:can-recur-str}
The \emph{canonical recursive structure} on a derived log stack $(X, \alpha_X)$ is an infinite sequence of virtual Cartier divisors
$$
{\cdots} \hookrightarrow \D_X^{(3)} \hookrightarrow \D_X^{(2)} \hookrightarrow \D_X^{(1)} \hookrightarrow \D_X \hookrightarrow X
$$
obtained by the following commutative diagram, where each square diagram is a homotopy Cartesian diagram.
$$
\begin{tikzcd}
{\cdots} \arrow[r, hook] \arrow[d] \ar[dr, phantom, "\lrcorner_{\text{\tiny h}}", very near start] & {\D_X^{(3)}} \arrow[r, hook] \arrow[d] \ar[dr, phantom, "\lrcorner_{\text{\tiny h}}", very near start] & {\D_X^{(2)}} \arrow[r, hook] \arrow[d] \ar[dr, phantom, "\lrcorner_{\text{\tiny h}}", very near start] & {\D_X^{(1)}} \arrow[r, hook] \arrow[d] \ar[dr, phantom, "\lrcorner_{\text{\tiny h}}", very near start] & {\D_X} \arrow[r, hook] \arrow[d] \ar[dr, phantom, "\lrcorner_{\text{\tiny h}}", very near start] & {X} \arrow[d, "\alpha_X"] \\
{\cdots} \arrow[r, hook]                 & {\left[\bA^3[-1]/\bG_m\right]} \arrow[r, hook]                 & {\left[\bA^2[-1]/\bG_m\right]} \arrow[r, hook]                 & {\left[\bA^1[-1]/\bG_m\right]} \arrow[r, hook]                 & {B\bG_m} \arrow[r, hook]                 & {\cA}                
\end{tikzcd}
$$
For convenience, set $\D_X^{(0)} := \D_X$.
\end{definition}

\begin{re}
    For $d \ge 0$, by Corollary \ref{cor:fund-canonical-resursive-structure}, we have the following commutative diagram, where each square diagram is a homotopy Cartesian diagram.
    $$
\begin{tikzcd}
{\D_X^{(d+1)}} \ar[dr, phantom, "\lrcorner_{\text{\tiny h}}", very near start] \arrow[r] \arrow[d, hook] & {\left[\bA^{d+1}[-1]/\bG_m\right]} \ar[dr, phantom, "\lrcorner_{\text{\tiny h}}", very near start] \arrow[r, hook] \arrow[d, hook] & {\cdots} \ar[dr, phantom, "\lrcorner_{\text{\tiny h}}", very near start] \arrow[r, hook] \arrow[d, hook] & {\left[\bA^1[-1]/\bG_m\right]} \ar[dr, phantom, "\lrcorner_{\text{\tiny h}}", very near start] \arrow[r, hook] \arrow[d, hook] & {B\bG_m} \arrow[d, hook] \\
{\D_X^{(d)}} \arrow[r]           & {\left[\bA^{d}[-1]/\bG_m\right]} \arrow[r, hook]           & {\cdots} \arrow[r, hook]           & {B\bG_m} \arrow[r, hook]           & {\cA}          
\end{tikzcd}    
    $$       
    As a result, similar to Remark \ref{rmk:der-log-structure-on-virt-Cart-div}, the virtual Cartier divisor
    $$
    \D_X^{(d+1)} \hookrightarrow \D_X^{(d)}
    $$
    on $\D_X^{(d)}$ yields a derived log stack
    $$\left( \D_X^{(d)}, \D_X^{(d)} \hookrightarrow \D_X^{(d-1)}  \hookrightarrow  {\cdots}  \hookrightarrow \D_X  \hookrightarrow  X \overset{\alpha_X}{\longrightarrow} \cA \right).$$
\end{re}

\begin{re}\label{rmk:infin-seq-POT-case-0}
    For $d \ge 0$, all classical truncations $\left( \D_X^{(d)} \right)_\cl$ of $\D_X^{(d)}$ agree. Note that if $X$ is quasi-smooth, then $\D_X^{(d)}$ is also quasi-smooth. Thus, assuming $X$ is a quasi-smooth derived Deligne--Mumford stack, the family $\left\{ \D_X^{(d)} \right\}_{d \ge 0}$ induces an infinite sequence of perfect obstruction theories on $\left( \D_X \right)_\cl$.
\end{re}

\begin{re}\label{rmk:infin-seq-POT-case--1}
    If the anchor morphism $\alpha_X: X \to \cA$ factors through the closed embedding $B\bG_m \hookrightarrow \cA$, then all classical truncations $\left( \D_X^{(d)} \right)_\cl$ of $\D_X^{(d)}$ agree with the classical truncation $X_\cl$ of $X$. In this case, we set $\D_X^{(-1)} := X$. Thus, if moreover $X$ is a quasi-smooth derived Deligne--Mumford stack, then the family $\left\{ \D_X^{(d)} \right\}_{d \ge -1}$ induces an infinite sequence of perfect obstruction theories on $X_\cl$.
\end{re}

For the rest of the subsection, we aim to compute the (log) cotangent complex of $\D_X^{(d)}$ to gain a better understanding of the derived (log) structure on $\D_X^{(d)}$, for $d \ge 1$.

Applying the cotangent complex functor to the morphism $\D_X^{(d)} \to \cA$, we obtain a (co)fiber sequence
$$
\bL_{\cA}|_{\D_X^{(d)}} \longrightarrow \bL_{\D_X^{(d)}} \longrightarrow \bLL_{\D_X^{(d)}} \overset{+1}{\longrightarrow}.
$$
Note that $\bL_{\cA}$ is known (see Remark \ref{rmk:cotangent-complex-cA}). Hence, to compute $\bLL_{\D_X^{(d)}}$, we need to compute $\bL_{\D_X^{(d)}}$.

Applying the cotangent complex functor to the morphism $\D_X^{(d)} \hookrightarrow X$, we obtain a (co)fiber sequence
$$
\bL_{X}|_{\D_X^{(d)}} \longrightarrow \bL_{\D_X^{(d)}} \longrightarrow \bL_{\D_X^{(d)}/X} \overset{+1}{\longrightarrow}.
$$
Note that $\bL_{\D_X^{(d)}/X} \simeq \bL_{\left[\bA^d[-1]/\bG_m\right] / \cA}|_{\D_X^{(d)}} \simeq \bLL_{\left[\bA^d[-1]/\bG_m\right]}|_{\D_X^{(d)}}$. As a result, to compute $\bL_{\D_X^{(d)}}$, we need to compute $\bLL_{\left[\bA^d[-1]/\bG_m\right]}$.

Applying the cotangent complex functor to the morphism $\left[\bA^d[-1]/\bG_m\right] \hookrightarrow \cA$, we obtain a (co)fiber sequence
$$
\bL_{\cA}|_{\left[\bA^d[-1]/\bG_m\right]} \longrightarrow \bL_{\left[\bA^d[-1]/\bG_m\right]} \longrightarrow \bLL_{\left[\bA^d[-1]/\bG_m\right]} \overset{+1}{\longrightarrow}.
$$
Hence, to compute $\bLL_{\left[\bA^d[-1]/\bG_m\right]}$, we need to compute $\bL_{\left[\bA^d[-1]/\bG_m\right]}$, which is carried out in Example \ref{eg:1st-cotangent-complex-genuine-der-quot-stack} and Example \ref{eg:2nd-cotangent-complex-genuine-der-quot-stack}.

\begin{example}
\label{eg:1st-cotangent-complex-genuine-der-quot-stack}

Set $\bA^1[-1]:= \Spec \Sym (k[1]) \cong \Spec k[\nu] \cong \Spec \left( \left[ k \langle \nu \rangle \overset{0}{\longrightarrow} k \right] \right)$ with cohomological degree $|\nu|=-1$. Consider the action morphism\footnote{Note that $\bG_m$ is a linearly reductive group scheme, so its higher group cohomology vanishes.}
\begin{align*}
    \psi: \bG_m \times^h \bA^1[-1] &\longrightarrow \bA^1[-1] \\
    (t,\nu) &\longmapsto t\nu .
\end{align*}
Note that $\psi$ admits a section $(e,\id_{\bA^1}): \bA^1[-1] \to \bG_m \times^h \bA^1[-1]$, where $e: \pt \to \bG_m$ is the identity section of $\bG_m$. Let $\mathfrak{g}$ denote $\LL e^* (\bT_{\bG_m})$, which is the tangent space of $\bG_m$ at $e$.

Applying the cotangent complex functor to $\psi$ and then pulling back to $\bA^1[-1]$ along $(e,\id_{\bA^1[-1]})$, we obtain a morphism
$$
d\psi|_{\bA^1[-1]}:=\LL(e,\id_{\bA^1[-1]})^*\left(d\psi\right): \bL_{\bA^1[-1]} \to \left( \mathfrak{g}^{\vee} \boxtimes \cO_{\bA^1[-1]} \right) \oplus \bL_{\bA^1[-1]} ,
$$
which can be explicitly written as
\begin{align*}
        d\psi|_{\bA^1[-1]}: \cO_{\bA^1[-1]}\langle d\nu \rangle &\longrightarrow \cO_{\bA^1[-1]}\langle dt \rangle \oplus \cO_{\bA^1[-1]}\langle d\nu \rangle \\
        d\nu &\longmapsto \left(d(t\nu)\right)|_{t=1} = \left( \nu dt + td\nu \right)|_{t=1} = \nu dt + d\nu .
\end{align*}
     Hence, we see that $d\psi|_{\bA^1[-1]} \simeq d\bar{\psi} \oplus \id$, where $d\bar{\psi} : \cO_{\bA^1[-1]}\langle d\nu \rangle \to \cO_{\bA^1[-1]}\langle dt \rangle : d\nu \mapsto \nu dt$. Following Remark \ref{rmk:cotangent-complx-der-quot-stk}, we conclude that
\begin{align*}
    \bL_{[\bA^1[-1]/\bG_m]}|_{\bA^1[-1]} &\simeq  \fib \left( d\bar{\psi}: \cO_{\bA^1[-1]}\langle d\nu \rangle \overset{d\nu \mapsto \nu dt}{\longrightarrow} \cO_{\bA^1[-1]}\langle dt \rangle \right) \\
    &\simeq \fib \left( \cO_{\bA^1[-1]}[1] \overset{\cdot \nu}{\longrightarrow} \cO_{\bA^1[-1]} \right),
\end{align*}
    which is a perfect complex of Tor-amplitude in $[-1,1]$ on $\bA^1[-1]$. Note that the morphism $\cO_{\bA^1[-1]}[1] \overset{\cdot \nu}{\longrightarrow} \cO_{\bA^1[-1]}$ is well-defined, because $\nu$ is literally a $(-1)$-shifted function on $\bA^1[-1]$ (See Remark \ref{rmk:shifted-function-and-homotopy} and Example \ref{eg:aff-derived-zero-locus}).
    
    Note that $d\bar{\psi}$ is a $\bG_m$-equivariant morphism, hence $\bL_{[\bA^1[-1]/\bG_m]}|_{\bA^1[-1]} \in \QCoh^{\bG_m}(\bA^1[-1])$. By descent, we have $\QCoh^{\bG_m}(\bA^1[-1]) \cong \QCoh([\bA^1[-1]/\bG_m])$. Thus, $\bL_{[\bA^1[-1]/\bG_m]}|_{\bA^1[-1]} \in \QCoh^{\bG_m}(\bA^1[-1])$ corresponds to $\bL_{[\bA^1[-1]/\bG_m]} \in \QCoh([\bA^1[-1]/\bG_m])$. $\blacksquare$
\end{example}

\begin{re}
The morphism $\cO_{\bA^1[-1]}[1] \overset{\cdot \nu}{\longrightarrow} \cO_{\bA^1[-1]}$ can be explicitly written as
$$
\begin{tikzcd}
{0} \arrow[r]                                & {k}                                \\
{k} \arrow[u] \arrow[r, "1 \mapsto \nu"]                      & {k \langle \nu \rangle} \arrow[u]                      \\
{k \langle \nu \rangle } \arrow[u] \arrow[r]                      & {0} \arrow[u]              
\end{tikzcd}
$$
Hence, $\bL_{[\bA^1[-1]/\bG_m]}|_{\bA^1[-1]}$ is quasi-isomorphic to 
$$ k \langle \nu \rangle \overset{0}{\longrightarrow} k \overset{\cdot \nu}{\longrightarrow} k \langle \nu \rangle \overset{0}{\longrightarrow} k.$$
Moreover, pulling back $\bL_{[\bA^1[-1]/\bG_m]}|_{\bA^1[-1]}$ to $\pt$, we have
$$
\bL_{[\bA^1[-1]/\bG_m]}|_{\pt} \simeq \left[0 \overset{0}{\longrightarrow} k \overset{0}{\longrightarrow} 0 \overset{0}{\longrightarrow} k \right],
$$
which is a perfect complex of Tor-amplitude in $[-1,1]$ on $\pt$.
\end{re}

\begin{example}
\label{eg:2nd-cotangent-complex-genuine-der-quot-stack}

    Let $n,d \in \bZ_{\ge 0}$. Set $\bA^d[-n]:= \Spec \Sym (k^{\oplus d}[n]) \cong \Spec k[\nu_1, \dots, \nu_d]$ with cohomological degrees $|\nu_j|= -n$ for $1 \le j \le d$. Consider the action morphism
    \begin{align*}
    \psi: \bG_m \times \bA^d[-n] &\longrightarrow \bA^d[-n] \\
    (t,\nu_j) &\longmapsto t\nu_j .
\end{align*}
    Then, by similar arguments in Example \ref{eg:1st-cotangent-complex-genuine-der-quot-stack}, we have
\begin{align*}
    \bL_{[\bA^d[-n]/\bG_m]}|_{\bA^d[-n]} &\simeq  \fib \left( d\bar{\psi}: \bigoplus_{j=1}^d \cO_{\bA^d[-n]}\langle d\nu_j \rangle \overset{d\nu_j \mapsto \nu_jdt}{\longrightarrow} \cO_{\bA^d[-n]}\langle dt \rangle \right) \\
    &\simeq \fib \left( \bigoplus_{j=1}^d (\cO_{\bA^d[-n]}[n]) \overset{\oplus \cdot \nu_j}{\longrightarrow} \cO_{\bA^d[-n]} \right),
\end{align*}
    which is a perfect complex of Tor-amplitude in $[-n,1]$ on $\bA^d[-n]$. Note that the morphism $\cO_{\bA^d[-n]}[n] \overset{\cdot \nu_j}{\longrightarrow} \cO_{\bA^d[-n]}$ is well-defined, because each $\nu_j$ is a $(-n)$-shifted function on $\bA^d[-n]$ (See Remark \ref{rmk:shifted-function-and-homotopy}).
    
    Note that $d\bar{\psi}$ is a $\bG_m$-equivariant morphism, hence $\bL_{[\bA^d[-n]/\bG_m]}|_{\bA^d[-n]} \in \QCoh^{\bG_m}(\bA^d[-n])$. By descent, we have $\QCoh^{\bG_m}(\bA^d[-n]) \cong \QCoh([\bA^d[-n]/\bG_m])$. Thus, $\bL_{[\bA^d[-n]/\bG_m]}|_{\bA^d[-n]} \in \QCoh^{\bG_m}(\bA^d[-n])$ corresponds to $\bL_{[\bA^d[-n]/\bG_m]} \in \QCoh^{\bG_m}([\bA^d[-n]/\bG_m])$. $\blacksquare$
\end{example}

\begin{re}
    The cotangent complex $\bL_{\bA^d[-n]}$ of $\bA^d[-n]$ is quasi-isomorphic to
    $
    \bigoplus_{j=1}^d \cO_{\bA^d[-n]}\langle d\nu_j \rangle,
    $
    which is a perfect complex of Tor-amplitude in $[-n,-n]$.
\end{re}

We conclude this subsection by computing log cotangent complexes of $\D_X^{(d)}$ for $X = \bA^1$.

We may interpret Proposition \ref{prop:base-case-canonical-recursive-structure} as giving an explicit description of the canonical recursive structure on the derived log scheme $(\bA^1, \alpha_{\bA^1})$, where $\alpha_{\bA^1}: \bA^1 \to \cA$ is a smooth atlas for the smooth classical $1$-Artin stack $\cA$. Note that $(\bA^1, \alpha_{\bA^1})$ is log smooth.

Following arguments similar to Example \ref{eg:log-cotangent-complex-aff-derived-log-schm}, we obtain a quasi-isomorphism
$$
\bLL_{\bA^1} := \bL_{\alpha_{\bA^1}} \simeq \cO_{\bA^1}\langle d\log x \rangle,
$$
which is a perfect complex of Tor-amplitude in $[0,0]$.

\begin{re}
    Since $\bA^1$ is smooth, the derived log scheme $(\bA^1, \alpha_{\bA^1})$ is in fact a classical log scheme of Deligne--Faltings type of rank $1$ over $\pt$, where $\pt$ is equipped with the trivial log structure. Moreover, $\bLL_{\bA^1}$ is precisely the classical log cotangent bundle of this classical log scheme.
\end{re}

From the canonical recursive structure on $(\bA^1, \alpha_{\bA^1})$, there is a derived log scheme $(\pt, \alpha_{\pt})$, where $\alpha_{\pt}: \pt \to \cA$ is the morphism associated with the virtual Cartier divisor $i_1: \bA^1[-1] \hookrightarrow \pt$ on $\pt$. Following arguments similar to Example \ref{eg:log-cotangent-complex-aff-derived-log-schm}, there is a quasi-isomorphism
$$
\bLL_{\pt} := \bL_{\alpha_{\pt}} \simeq \bL_{\cA}|_{\pt}[1] \simeq \left[ k\langle d\mu \rangle \overset{0}{\longrightarrow} k\langle d\lambda \rangle \right],
$$
which is a perfect complex of Tor-amplitude in $[-1,0]$. Note the cohomological degrees $|d\mu|=-1$ and $|d\lambda|=0$ after applying $[1]$ to $\bL_{\cA}|_{\pt}$.

\begin{re}\label{rmk:subtlety-log-structure}
    Note the subtlety lies in the log structure on $(\pt, \alpha_{\pt})$. While the associated virtual Cartier divisor $i_1: \bA^1[-1] \hookrightarrow \pt$ on $\pt$ genuinely lives in the derived world, the derived log scheme $(\pt, \alpha_{\pt})$ is in fact a classical log scheme. The morphism $\alpha_{\pt}: \pt \to \cA$ admits a factorization
$$
\pt \to B\bG_m \hookrightarrow \cA,
$$
    which makes $\pt$ into a classical log scheme equipped with the unique non-trivial Deligne--Faltings log structure of rank $1$.
\end{re}

\begin{re}
    There is another derived log scheme $(\pt, 1_\pt)$ (see Remark \ref{rmk:triv-log-der-stk}) with $1_{\pt}: \pt \to \cA$ being the open embedding
$$
\pt = [\bG_m/\bG_m] \hookrightarrow \cA,
$$
    which makes $\pt$ into a classical log scheme equipped with the trivial log structure. Note that the virtual Cartier divisor of $(\pt, 1_{\pt})$ is $\emptyset \to \pt$. One can immediately see that
$$
\bLL_{\pt} := \bL_{1_{\pt}} \simeq 0 \simeq \bL_{\pt}.
$$
\end{re}

From the canonical recursive structure on $(\bA^1, \alpha_{\bA^1})$, there is a derived log scheme $(\bA^1[-1], \alpha_{\bA^1[-1]})$, where $\alpha_{\bA^1[-1]}: \bA^1[-1]:= \Spec \Sym (k[1]) \cong \Spec k[\nu] \to \cA$ is the morphism associated with the virtual Cartier divisor $i_2: \bA^2[-1] \hookrightarrow \bA^1[-1]$ on $\bA^1[-1]$. For morphism $\alpha_{\bA^1[-1]}: \bA^1[-1] \to \cA$, there is a (co)fiber sequence
$$
\bL_{\cA}|_{\bA^1[-1]} \overset{\bL\alpha_{\bA^1[-1]}}{\longrightarrow} \bL_{\bA^1[-1]} \longrightarrow  \bLL_{\bA^1[-1]} \overset{+1}{\longrightarrow}.
$$
By Proposition \ref{prop:base-case-canonical-recursive-structure}, the morphism $\alpha_{\bA^1[-1]}: \bA^1[-1] \to \cA$ admits a factorization
$$
\bA^1[-1] \overset{i_1}{\hookrightarrow} \pt \overset{0}{\hookrightarrow} \bA^1 \to \cA.
$$
Hence, we see that $\bL\alpha_{\bA^1[-1]} \simeq 0$. Consequently, we conclude that there is a quasi-isomorphism
$$
\bLL_{\bA^1[-1]} := \bL_{\alpha_{\bA^1[-1]}} \simeq \bL_{\cA}|_{\bA^1[-1]}[1] \oplus \bL_{\bA^1[-1]} \simeq \left[ \cO_{\bA^1[-1]}\langle d\nu \rangle \oplus \cO_{\bA^1[-1]}\langle d\mu \rangle \overset{0}{\longrightarrow} \cO_{\bA^1[-1]}\langle d\lambda \rangle \right],
$$
which is a perfect complex of Tor-amplitude $[-1,0]$. Note the cohomological degrees $|d\mu|=-1$ and $|d\lambda|=0$ after applying the shift functor $[1]$ to $\bL_{\cA}|_{\bA^1[-1]}$. Note that $|d\nu|=-1$.

Therefore, for a derived log scheme $(\bA^d[-1], \alpha_{\bA^d[-1]})$ whose log structure is induced from the canonical recursive structure on $(\bA^1, \alpha_{\bA^1})$,  we conclude, by arguments similar to above, that the log cotangent complex $\bLL_{\bA^d[-1]}$ is quasi-isomorphic to
$$
\bL_{\cA}|_{\bA^d[-1]}[1] \oplus \bL_{\bA^d[-1]} \simeq \left[ \left(\bigoplus_{j=1}^d \cO_{\bA^d[-1]}\langle d\nu_j \rangle\right) \oplus \cO_{\bA^d[-1]}\langle d\mu \rangle \overset{0}{\longrightarrow} \cO_{\bA^d[-1]}\langle d\lambda \rangle \right],
$$
which is a perfect complex of Tor-amplitude in $[-1,0]$. Note the cohomological degrees $|d\mu|=-1$ and $|d\lambda|=0$ after applying the shift functor $[1]$ to $\bL_{\cA}|_{\bA^1[-1]}$.

\begin{re}\label{rmk:subtlety-log-structure-exp}
    The composition
    $$
    \bA^d[-1] \hookrightarrow \bA^{d-1}[-1] \hookrightarrow \cdots \hookrightarrow \bA^1[-1] \hookrightarrow \pt
    $$
    in Proposition \ref{prop:base-case-canonical-recursive-structure} makes $\pt$ into a classical log scheme equipped with the unique non-trivial Deligne-Faltings log structure of rank $d$.
\end{re}

\section{Logarithmic Cosections and Virtual Cycles}\label{sec:3}

In this section, we introduce logarithmic cosections (Section \ref{sec:log-cosection}) and use them to construct reduced virtual cycles (Section \ref{sec:red-vir-cycle}) and reduced boundary cycles (Section \ref{sec:red-bound-cycle}).
Moreover, using logarithmic cosections, we prove a decomposition theorem for virtual cycles (Section \ref{sec:decompose-vir-cycle}). As mentioned in Remark \ref{rmk:shift-log-function}, the definition of logarithmic cosections is motivated by the twisted cosection in Log GLSM \cite{CJR21}.
Besides logarithmic cosections, we also introduce log-meromorphic cosections (Section \ref{sec:log-cosection}), and reduced cosections (Section \ref{sec:red-vir-cycle}).

The decomposition theorem for virtual cycles (Theorem \ref{thm:decompose-virt-cycle}) in Section \ref{sec:decompose-vir-cycle} is the starting point of the whole paper, and several potential applications we have in mind are based on it (or a more general version). The main advantage of our approach, compared with \cite[Theorem~6.5]{CJR21}, is that it removes the need for an auxiliary stack (see {\cite[Remark~6.2]{CJR21}}), thereby allowing for a broader range of potential applications. The reader might find Example \ref{eg:toy-model-decompose-virt-cycle}, Example \ref{eg:local-model-decompose-virt-cycle}, and Example \ref{eg:der-self-intersect-P1-in-P2} helpful.

This section is heavily influenced by \cite[Section~6]{CJR21}. In addition, Section \ref{sec:red-vir-cycle} and Section \ref{sec:red-bound-cycle} are inspired by \cite[Section~4]{STV15} and \cite[Section~6]{AKLPR22P}\footnote{Note that we refer to a specific arXiv version, not the latest or published version.}.

\begin{re}\label{rmk:sect-3-to-sect-4}
    Thanks to Proposition \ref{prop:rel-qsmth-rel-POT} and Proposition \ref{prop:rel-cosection-rel-POT}, everything in this section immediately generalizes to the relative setting in Section \ref{sec:4} by replacing $\cA$ with $\cA \times S$, where $S$ is a smooth classical $1$-Artin stack with separated diagonal (e.g.\ the smooth classical $1$-Artin stack $\fMgn$ of prestable curves, see the proof of \cite[Lemma~A.1]{GrSi13}).
\end{re}

\

Throughout this section, we fix the following convention.

Let $X$ be a quasi-smooth derived Deligne--Mumford stack equipped with a morphism $\alpha_X: X \to \cA$. In other words, $(X, \alpha_X)$ is a quasi-smooth derived log Deligne--Mumford stack, where $X$ is the underlying derived stack and $\alpha_X: X \to \cA$ is the anchor morphism. Note that $X$ being quasi-smooth implies that the anchor morphism $\alpha_X: X \to \cA$ is also quasi-smooth.

For simplicity, let $w_{\alpha}: X \to L_{\cA}[-1]$ denote a $(-1)$-shifted log function on $(X, \alpha_X)$. 

Consider the following commutative diagram, where each square diagram is a homotopy Cartesian diagram.
\[\begin{tikzcd}
	\D_{X_{cl}}  & \D_X  & B\bG_m  \\
	X_{\cl}  & X  & \cA
	\arrow[from=1-1, to=1-2, hook]
	\arrow[from=1-1, to=2-1, hook]
	\arrow[phantom, "\lrcorner_{\text{\tiny h}}", very near start, from=1-1, to=2-2]
	\arrow[from=1-2, to=1-3]
	\arrow[from=1-2, to=2-2, hook]
	\arrow[phantom, "\lrcorner_{\text{\tiny h}}", very near start, from=1-2, to=2-3]
	\arrow[from=1-3, to=2-3, hook]
	\arrow[from=2-1, to=2-2, hook, "\tau"]
	\arrow[from=2-2, to=2-3, "\alpha_X"]
\end{tikzcd}\]
Note that $\D_{X_{cl}}$ is not necessarily a classical Deligne--Mumford stack, although $\cO_{X_{cl}}(\D_{X_{\cl}})$ is always an invertible sheaf on $X_\cl$. We also have $(\D_{X_{cl}})_\cl = (\D_{X})_\cl$. For simplicity, set
\begin{align*}
    s_X := \LL\alpha_X^*(s_{\cA}) &: \cO_{X}(-\D_{X}) \to \cO_{X}, \\
    s_{X_{\cl}} := \LL\tau^*\LL\alpha_X^*(s_{\cA}) &: \cO_{X_{cl}}(-\D_{X_{\cl}}) \to \cO_{X_{\cl}}.
\end{align*}
    Recall that the natural inclusion $s_{\cA}: \cO_{\cA}(-B\bG_m) \hookrightarrow \cO_{\cA}$ is the universal section (see Section \ref{sec:basic-setup-der-log}).

Consider the following commutative diagram, where each square diagram is a homotopy Cartesian diagram.
\[\begin{tikzcd}
	U_{cl}   & U  & \pt \cong [\bG_m/\bG_m]  \\
	X_{\cl}  & X  & \cA
	\arrow[from=1-1, to=1-2, hook]
	\arrow[from=1-1, to=2-1, hook]
	\arrow[phantom, "\lrcorner_{\text{\tiny h}}", very near start, from=1-1, to=2-2]
	\arrow[from=1-2, to=1-3]
	\arrow[from=1-2, to=2-2, hook]
	\arrow[phantom, "\lrcorner_{\text{\tiny h}}", very near start, from=1-2, to=2-3]
	\arrow[from=1-3, to=2-3, hook]
	\arrow[from=2-1, to=2-2, hook, "\tau"]
	\arrow[from=2-2, to=2-3, "\alpha_X"]
\end{tikzcd}\]
    Note that $\pt \hookrightarrow \cA$ is an open embedding, and hence a flat morphism. Thus, $U_\cl$ is a classical Deligne--Mumford stack that is an open substack of $X_\cl$. Moreover, we have $U_\cl = X_{\cl} \setminus (\D_{X_{cl}})_\cl$.

\subsection{Logarithmic Cosections}\label{sec:log-cosection}

\begin{definition}\label{def:log-POT}
    The \emph{log perfect obstruction theory associated with $(X, \alpha_X)$} is defined to be
    $$
    \bLL\tau : \bLL_X|_{X_\cl} \to \bLL_{X_\cl}.
    $$
\end{definition}

\begin{re}
    Recall the log cotangent complex $\bLL_X := \bL_{\alpha_X} = \bL_{X/\cA}$ (see Definition \ref{def:log-cotangent-complex}).
\end{re}

\begin{re}\label{rmk:log-POT-is-rel-POT}
    Since $X$ is quasi-smooth and $\cA$ is smooth, $\bLL\tau$ is a relative perfect obstruction theory on $X_{\cl}$ in the sense of Behrend--Fantechi \cite[Section~7]{BeFa97} by Proposition \ref{prop:rel-qsmth-rel-POT}.
\end{re}

\begin{re}\label{rmk:log-POT-virt-cycle-is-usual-virt-cycle}
    By Corollary \ref{cor:qsmooth-POT}, there is a perfect obstruction theory $\bL\tau : \bL_X|_{X_\cl} \to \bL_{X_\cl}$. By Proposition \ref{prop:rel-qsmth-rel-POT}, the two perfect obstruction theories $\bL\tau$ and $\bLL\tau$ induce the same virtual cycle. For simplicity, we use $[X_{\cl}]^{\vir}$ to denote the virtual cycle induced by either $\bL\tau$ or $\bLL\tau$.
\end{re}

\begin{definition}\label{def:log-cosection}
    Let $w_{\alpha} : X \to L_{\cA}[-1]$ be a $(-1)$-shifted log function on $(X, \alpha_X)$. The morphism
    $$
    \bTT w_{\alpha}|_{X_\cl}: \bTT_{X}|_{X_\cl} \to  \bTT_{L_{\cA}[-1]}|_{X_\cl} \simeq \cO_{X_\cl}(\D_{X_\cl})[-1]
    $$
    is called the \emph{log cosection of the log perfect obstruction theory $\bLL\tau : \bLL_{X}|_{X_\cl} \to \bLL_{X_\cl}$ associated with $w_{\alpha}$}. Moreover, $H^1\left(\bTT_X|_{X_{\cl}}\right)$ is called the \emph{log obstruction sheaf $\Ob^\log_{X_\cl}$}, and the morphism
    $$
    \sigma^\log:= H^1\left(\bTT w_\alpha|_{X_{\cl}}\right): \Ob^\log_{X_\cl} = H^1(\bTT_X|_{X_{\cl}}) \to \cO_{X_\cl}(\D_{X_\cl})
    $$
    is called the \emph{log cosection of the log obstruction sheaf $\Ob^\log_{X_\cl}$ associated with $w_{\alpha}$}.
\end{definition}

\begin{re}
    By Proposition \ref{prop:rel-cosection-rel-POT}, log obstruction sheaves are relative obstruction sheaves in the sense of \cite{KiLi13}.
\end{re}

\begin{re}
    Log cosections are not (meromorphic) cosections in the sense of \cite{KiLi13}.
\end{re}

\begin{re}\label{rmk:class-shift-log-diff-form-is-log-cosection}
    Recall that $dw_\alpha:= \bLa w_\alpha[-1]: \cO_X \to \bLa_X[-1] \simeq \bLL_X \otimes^{\LL}_{\cO_{X}} \cO_X(\D_X) [-1]$ is the exact $(-1)$-shifted log differential $1$-form associated with $w_\alpha$ (see Definition \ref{def:shift-log-function-diff-form}). Note that the morphism
    $$
    (dw_\alpha)^\vee|_{X_\cl}: \bTa_X|_{X_\cl}[1] \simeq \bTT_X|_{X_\cl} \otimes^{\LL}_{\cO_{X_\cl}} \cO_{X_\cl}(-\D_{X_\cl}) [1] \to \cO_{X_\cl}
    $$
    induces the log cosection $\bTT w_{\alpha}|_{X_\cl}: \bTT_{X}|_{X_\cl} \to \cO_{X_\cl}(\D_{X_\cl})[-1]$ after applying the functor $- \otimes^{\LL}_{\cO_{X_\cl}} \cO_{X_\cl}(\D_{X_\cl}) [-1]$. Moreover, the morphism
    $$
    \sigma^{\log\tw}:=H^0\left((dw_\alpha)^\vee|_{X_\cl}\right) : H^0\left(\bTa_X|_{X_\cl}[1]\right) = H^1\left(\bTa_X|_{X_\cl}\right) = \Ob_{X_{cl}}^{\log} \otimes_{\cO_{X_{\cl}}} \cO_{X_{\cl}}(-\D_{X_{\cl}}) \to \cO_{X_\cl}
    $$
    induces the log cosection $\sigma^\log: \Ob^\log_{X_\cl} \to \cO_{X_\cl}(\D_{X_\cl})$ applying the functor $- \otimes_{\cO_{X_\cl}} \cO_{X_\cl}(\D_{X_\cl})$. Abusing notation, we also use $dw_\alpha$ to denote the morphism
    $$
    X \to \Ta[-1]X \cong \Spec_X \Sym_X \left(\bTa_X[1]\right)
    $$
    induced by $dw_\alpha: \cO_X \to \bLa_X[-1]$, and use $(\sigma^{\log\tw})^\vee$ to denote the morphism
    $$
    X_\cl \to \Spec_{X_\cl} \Sym_{X_\cl} \left(  \Ob^\log_{X_\cl} \otimes_{\cO_{X_\cl}} \cO_{X_\cl}(-\D_{X_\cl}) \right)
    $$
    induced by $\sigma^{\log\tw}= \sigma^\log \otimes_{\cO_{X_\cl}} \cO_{X_\cl}(-\D_{X_\cl}) : \Ob^\log_{X_\cl} \otimes_{\cO_{X_\cl}} \cO_{X_\cl}(-\D_{X_\cl}) \to \cO_{X_\cl}$. Then 
    $
    H^0\left((dw_\alpha)^\vee|_{X_\cl}\right) = \sigma^{\log\tw}: \Ob^\log_{X_\cl} \otimes_{\cO_{X_\cl}} \cO_{X_\cl}(-\D_{X_\cl}) \to \cO_{X_\cl}
    $
    implies that 
    $$
    (dw_\alpha)_\cl = (\sigma^{\log\tw})^\vee: X_\cl \to (\Ta[-1]X)_\cl = \Spec_{X_\cl} \Sym_{X_\cl} \left(  \Ob^\log_{X_\cl} \otimes_{\cO_{X_\cl}} \cO_{X_\cl}(-\D_{X_\cl}) \right).
    $$
\end{re}

\begin{definition}\label{def:def-locus-log-cosection}
    Let $w_{\alpha} : X \to L_{\cA}[-1]$ be a $(-1)$-shifted log function on $(X, \alpha_X)$. Following the notation in Remark \ref{rmk:class-shift-log-diff-form-is-log-cosection}, we define the \emph{derived degeneracy locus of the log cosection $\sigma^\log: \Ob^\log_{X_\cl} \to \cO_{X_{\cl}}(\D_{X_\cl})$ associated with $w_{\alpha}$} to be the derived zero locus $X(dw_{\alpha})$ of $dw_{\alpha}: \cO_X \to \bLa_X[-1]$, which is defined via the following homotopy Cartesian diagram in $\LogDStk$.
$$
    \begin{tikzcd}
{X(dw_\alpha)} \arrow[r] \arrow[d] \ar[dr, phantom, "\lrcorner_{\text{\tiny h}}", very near start] & {X} \arrow[d,"dw_\alpha"] \\
{X} \arrow[r, "0"]           & {\Ta[-1]X}          
    \end{tikzcd}
$$
    Moreover, we define the \emph{classical degeneracy locus of the log cosection $\sigma^\log: \Ob^\log_{X_\cl} \to \cO_{X_{\cl}}(\D_{X_\cl})$} to be the classical truncation $X(dw_{\alpha})_\cl$ of $X(dw_{\alpha})$, which fits into the following classical Cartesian diagram.
$$
    \begin{tikzcd}
{X(dw_{\alpha})_\cl} \arrow[r] \arrow[d] \ar[dr, phantom, "\lrcorner", very near start] & {X_\cl} \arrow[d,"(dw_{\alpha})_\cl = \left(\sigma^{\log\tw}\right)^\vee"] \\
{X_\cl} \arrow[r, "0"]           & {(\Ta[-1]X)_\cl}          
    \end{tikzcd}
$$
\end{definition}

\begin{re}\label{rmk:degeneracy-locus-log-cosection-is degeneracy-locus-log-tw-cosection}
    $X(dw_{\alpha})$ is the derived zero locus of the section $dw_{\alpha}: \cO_X \to \bLa_X[-1]$, which is equivalent to the derived zero locus $X((dw_{\alpha})^\vee)$ of the cosection $(dw_{\alpha})^\vee: \bTa_X[1] \to \cO_X$ (see \cite[Definition~4.13]{Jiang25}). Hence, the classical truncation $X(dw_{\alpha})_\cl$ of $X(dw_{\alpha})$ is the locus where the morphism $$\sigma^{\log\tw}: \Ob_{X_{cl}}^{\log} \otimes_{\cO_{X_\cl}} \cO_{X_{\cl}}(-\D_{X_{\cl}}) \to \cO_{X_{\cl}}$$ is zero. Moreover, we see that $X(dw_{\alpha})_\cl$ is the locus where the log cosection $$\sigma^{\log}: \Ob_{X_{cl}}^{\log} \to \cO_{X_{\cl}}(\D_{X_{\cl}})$$ is zero, justifying the name of the classical degeneracy locus of the log cosection $\sigma^{\log}$. For simplicity, let $X_{\cl}(\sigma^{\log})$ denote the classical degeneracy locus of the log cosection $\sigma^{\log}$.
\end{re}

\begin{re}\label{rmk:log-virt-Lag-cycle}(Continuation of the Remark \ref{rmk:virt-Lag-cycle-cosection-localized-virt-cycle})
    $X(dw_{\alpha})$ is also called the derived log critical locus of the $(-1)$-shifted log function $w_{\alpha}$ on $(X, \alpha_X)$ (see Definition \ref{def:der-log-crit-locus}). By the log Lagrangian intersection theorem (Theorem \ref{thm:log-lag-intersect}), $X(dw_{\alpha})$ carries a $(-2)$-shifted log symplectic structure.
\end{re}

\begin{definition}\label{def:log-mero-cosection}
Let $w_{\alpha} : X \to L_{\cA}[-1]$ be a $(-1)$-shifted log function on $(X, \alpha_X)$. We define the \emph{log-meromorphic cosection $\widetilde{\sigma^\log}$ associated with the log cosection $\sigma^{\log}$ induced by $w_{\alpha}$} to be a meromorphic cosection (in the sense of \cite{KiLi13})
$$
\Ob^\log_{X_{cl}} \dashrightarrow \cO_{X_{\cl}}
$$
such that it is undefined along $(\D_{X_{cl}})_\cl$ and 
$$
\widetilde{\sigma^\log}|_{U_\cl} = \sigma^{\log}|_{U_\cl} : \Ob^\log_{X_{cl}}|_{U_\cl} \to \cO_{X_{\cl}}(\D_{X_{\cl}})|_{U_\cl} = \cO_{U_\cl}.
$$
Let $X_{\cl}(\widetilde{\sigma^\log})$ denote the degeneracy locus of $\widetilde{\sigma^\log}$ in the sense of \cite[Remark~2.2]{KiLi13}.
\end{definition}

\begin{re}\label{rmk:log-mero-cosection-ind-by-shifted-function}
    Note that $\bTT_X|_U \simeq \bTT_U \simeq \bT_U \simeq \bT_X|_U$, and hence $\Ob^\log_{X_{cl}}|_{U_\cl} = \Ob_{U_\cl} = \Ob_{X_\cl}|_{U_\cl}$ is the obstruction sheaf in the sense of \cite{KiLi13}.
\end{re}

\begin{proposition}\label{prop:log-meromorphic-cosection-localized-virt-cycle}
    Assume that $X_\cl$ is separated. Let $w_{\alpha} : X \to L_{\cA}[-1]$ be a $(-1)$-shifted log function on $(X, \alpha_X)$. The log-meromorphic cosection $$\widetilde{\sigma^\log}: \Ob^\log_{X_{cl}} \dashrightarrow \cO_{X_{\cl}}$$ associated with the log cosection $\sigma^\log$ induced by $w_{\alpha}$ gives rise to a cosection localized virtual cycle $$ [X_{\cl}]^\vir_{\widetilde{\sigma^\log}} \in A_*\left(X_{\cl}(\widetilde{\sigma^\log})\right) $$ such that $$ \iota_* \left([X_{\cl}]^\vir_{\widetilde{\sigma^\log}}\right) = [X_{\cl}]^{\vir}.$$ Here, $\iota: X_{\cl}(\widetilde{\sigma^\log}) \hookrightarrow X_{\cl}$ is the natural closed embedding, and $[X_{\cl}]^{\vir}$ denotes the virtual cycle induced by either the perfect obstruction theory $\bL\tau$ or the log perfect obstruction theory $\bLL\tau$ (see Remark \ref{rmk:log-POT-virt-cycle-is-usual-virt-cycle}).
\end{proposition}

\begin{proof}
    It directly follows from Remark \ref{rmk:log-mero-cosection-ind-by-shifted-function} and \cite[Theorem~1.1]{KiLi13}.
\end{proof}

\begin{re}\label{rmk:decompose-virt-cycle}
    If the log cosection $\sigma^{\log}$ is surjective along $(\D_{X_{cl}})_\cl$, i.e. the intersection of $X_{\cl}(\sigma^{\log}):= X(dw_{\alpha})_\cl$ and $(\D_{X_{cl}})_\cl$ is empty, then $X_{\cl}(\widetilde{\sigma^\log})$ is (set-theoretically) the disjoint union of $X_{\cl}(\sigma^{\log})$ and $(\D_{X_{cl}})_\cl$. Such an observation leads to the decomposition theorem for virtual cycles in Section \ref{sec:decompose-vir-cycle}.
\end{re}

\subsection{Reduced Virtual Cycles via Logarithmic Cosections}\label{sec:red-vir-cycle}

Throughout this subsection, we work with the duals of (relative) perfect obstruction theories in the sense of \cite{BeFa97}, and refer to them simply as perfect obstruction theories, to align with the conventions in \cite[Section~6]{CJR21}.

Let $\tilde{s}: \mathbb{A}^1_{\mathcal{A}}[-1] \to L_{\mathcal{A}}[-1]$ denote the morphism
$$
\Spec_{\cA}\Sym_{\cA}\left( s_{\cA}[1] \right): \Spec_{\cA}\Sym_{\cA}\left( \cO_{\cA}[1] \right) \to \Spec_{\cA}\Sym_{\cA}\left( \cO_{\cA}(-B\bG_m)[1] \right).
$$
\begin{re}\label{rmk:tangent-complex-s}
    By standard computation, the relative tangent complex $\bT_{\tilde{s}}$ of the morphism $\tilde{s}$ is quasi-isomorphic to
$$
\fib \Big( s_\cA^{\vee}[-1]: \cO_\cA[-1] \hookrightarrow \cO_\cA(B\bG_m)[-1] \Big),
$$
    which is a perfect complex of Tor-amplitude in $[1,2]$. Hence, $\tilde{s}$ is homotopically $2$-smooth.
\end{re}

\begin{definition}\label{def:red-POT}
    Let $w_{\alpha} : X \to L_{\cA}[-1]$ be a $(-1)$-shifted log function on $(X, \alpha_X)$. The $w_\alpha$-reduction $X^{\red}$ of $X$ is defined via the following homotopy Cartesian diagram in $\LogDStk$.
\begin{equation}\label{diagram:X^red}
\begin{tikzcd}
X^{\red} \arrow[d, "\theta"] \arrow[r, "w_{\alpha}^\red"] \arrow[rd, "\lrcorner_{\text{\tiny h}}", phantom, very near start] & {\mathbb{A}^1_{\mathcal{A}}[-1]} \arrow[d, "\tilde{s}"] \\
X \arrow[r, "w_{\alpha}"]                 & {\mathcal{L}_{\mathcal{A}}[-1]}  \end{tikzcd}
\end{equation}
\end{definition}

\begin{re}
    $X^\red$ is homotopically $2$-smooth by Remark \ref{rmk:tangent-complex-s}.
\end{re}

\begin{re}\label{rmk:class-X-red-is-class-X}
    The classical truncation $X^\red_\cl$ of $X^\red$ fits into the following classical Cartesian diagram.
    $$
\begin{tikzcd}
X^\red_\cl \arrow[d, "\theta_\cl"] \arrow[rr, "(w_{\alpha}^\red)_\cl"] \arrow[rrd, "\lrcorner", phantom, very near start] & &{(\mathbb{A}^1_{\mathcal{A}}[-1])_\cl = \cA} \arrow[d, "\tilde{s}_\cl"] \\
X_\cl \arrow[rr, "(w_{\alpha})_\cl"]               & & {(\mathcal{L}_{\mathcal{A}}[-1])_\cl = \cA}     
\end{tikzcd}
    $$
    Note that $\tilde{s}_\cl = \id: \cA \to \cA$ and $(w_{\alpha})_\cl = (\alpha_X)_\cl : X_\cl \to \cA$. As a result, we see that $\theta_{\cl}$ is precisely the identity morphism, i.e. $X^\red_\cl$ and $X_\cl$ are exactly the same classical Deligne--Mumford stack. Moreover, we have
$$
(w_{\alpha}^\red)_\cl = (\alpha_X)_\cl : X^\red_\cl = X_\cl \to (\mathbb{A}^1_{\mathcal{A}}[-1])_\cl = \cA.
$$
\end{re}

\begin{re}\label{rmk:X-fact-through-X-red}
For the morphism $\theta: X^{\red} \to X$ in Diagram (\ref{diagram:X^red}), there is a commutative diagram as follows.
\[\begin{tikzcd}
	X_\cl = X^{\red}_\cl & X^{\red} \\
	X_\cl & X
	\arrow[from=1-1, to=1-2, hook, "\tau^{\red}"]
	\arrow[from=1-1, to=2-1, "\theta_\cl = id"]
	\arrow[from=1-2, to=2-2, "\theta"]
	\arrow[from=2-1, to=2-2, hook, "\tau"]
\end{tikzcd}\]
For simplicity, we also use $X_\cl$ to denote the classical truncation $X^{\red}_\cl$ of $X^\red$. Hence, we have the following commutative diagram.
\begin{equation}\label{diagram:X^red-factor}
\begin{tikzcd}
	{} & X^{\red} \\
	X_\cl & X
	\arrow[from=2-1, to=1-2, hook, "\tau^{\red}"]
	\arrow[from=1-2, to=2-2, "\theta"]
	\arrow[from=2-1, to=2-2, hook, "\tau"]
\end{tikzcd}
\end{equation}
\end{re}

\begin{lemma}\label{lem:tangent-complex-factor}
    There is a commutative diagram as follows.
$$
\begin{tikzcd}
{\bTT_{X_\cl}} \arrow[rr, "\bTT \tau"] \arrow[rd, "\bTT \tau^{\red}"'] &               & {\bTT_{X}|_{X_\cl}} \\
                         & {\bTT_{X^\red}|_{X_\cl}} \arrow[ru, "(\bTT \theta)|_{X_\cl}"'] &   
\end{tikzcd}
$$
\end{lemma}

\begin{proof}
    It directly follows from applying the log tangent complex functor to Diagram (\ref{diagram:X^red-factor}).
\end{proof}

\begin{re}
    $\bTT\tau: \bTT_{X_\cl} \to \bTT_{X}|_{X_\cl}$ is a relative perfect obstruction theory on $X_\cl$ in the sense of \cite{BeFa97} (see Remark \ref{rmk:log-POT-is-rel-POT}).
\end{re}

\begin{proposition}
\label{prop:red-POT-log-GLSM}
    There is a (co)fiber sequence as follows.
    $$\bTT_{X^\red}|_{X_\cl} \overset{(\bTT \theta)|_{X_\cl}}{\longrightarrow} \bTT_{X}|_{X_\cl} \longrightarrow \bF \overset{+1}{\longrightarrow} $$
    Here, $\bF := \left[s_{X_\cl}^{\vee}: \cO_{X_\cl} \to \cO_{X_\cl}(\D_{X_\cl})\right]$ is a perfect complex of Tor-amplitude in $[0,1]$.
\end{proposition}

\begin{proof}
    For the morphism $\theta: X^{\red} \to X$ in Diagram (\ref{diagram:X^red}), there is a (co)fiber sequence
$$
\bTT_{X^\red}|_{X_\cl} \overset{(\bTT \theta)|_{X_\cl}}{\longrightarrow} \bTT_{X}|_{X_\cl} \longrightarrow (\bTT_{\theta}[1])|_{X_\cl} \overset{+1}{\longrightarrow}.
$$
    By Diagram (\ref{diagram:X^red}), we have $(\bTT_{\theta}[1])|_{X_\cl} \simeq (\bTT_{\tilde{s}}|_{X_\cl})[1] \simeq (\bT_{\tilde{s}}|_{X_\cl})[1]$. By Remark \ref{rmk:tangent-complex-s}, we have $$(\bT_{\tilde{s}}|_{X_\cl})[1] \simeq \left[ s_{X_\cl}^{\vee}: \cO_{X_\cl} \to \cO_{X_\cl}(\D_{X_\cl}) \right],$$ which is a perfect complex of Tor-amplitude in $[0,1]$.
\end{proof}

\begin{re}
    Proposition \ref{prop:red-POT-log-GLSM} may be interpreted as saying that $X^\red$ is a natural derived geometric enhancement of the reduced perfect obstruction theory on $X_\cl$ constructed in \cite[Section~6.2]{CJR21}.
\end{re}

\begin{re}\label{rmk:criterion-red-POT-is-POT}
    Note that $X^\red$ is homotopically $2$-smooth. Hence, to show that $$\bTT\tau^{\red}: \bTT_{X_\cl} \to \bTT_{X^\red}|_{X_\cl}$$ is a relative perfect obstruction theory on $X_{\cl}$ in the sense of \cite{BeFa97}, it suffices to show that $H^2( \bTT_{X^\red}|_{X_\cl} ) =0$ by Proposition \ref{prop:cotangent-complex-obstruction-theory}.
\end{re}

\begin{corollary}\label{lem:POT-coindcide-outside-divisor}
    $\bTT \tau^{\red}|_{U_{\cl}} = \bTT \tau|_{U_\cl}$.
\end{corollary}

\begin{proof}
    We have $\bF|_{U_\cl} \simeq 0$. Hence, we see that $\bTT \tau^{\red}|_{U_\cl} = \bTT \tau|_{U_\cl}$ by Lemma \ref{lem:tangent-complex-factor} and Proposition \ref{prop:red-POT-log-GLSM}.
\end{proof}

\begin{definition}\label{def:red-cosection}
    The morphism $w_{\alpha}^\red: X^{\red} \to \mathbb{A}^1_{\mathcal{A}}[-1]$ in Diagram (\ref{diagram:X^red}) associated with $w_{\alpha}$ is a morphism over $\cA$. Hence, $w_{\alpha}^\red: X^{\red} \to \mathbb{A}^1_{\mathcal{A}}[-1]$ is a relative $(-1)$-shifted function on $X^\red$ over $\cA$ (see Definition \ref{rmk:rel-shifted-function}). Following Definition \ref{def:cosection}, the morphism
    $$
    \bTT w_{\alpha}^{\red}|_{X_\cl}: \bTT_{X^{\red}}|_{X_\cl} \to  \bTT_{\bA^1_{\cA}[-1]}|_{X_\cl} \simeq \cO_{X_\cl}[-1]
    $$
    is called the \emph{reduced cosection of $\bTT\tau^{\red}: \bTT_{X_\cl} \to \bTT_{X^\red}|_{X_\cl}$ associated with $w_{\alpha}$}. Moreover, $H^1\left(\bTT_{X^\red}|_{X_{\cl}}\right)$ is called the \emph{reduced obstruction sheaf $\Ob^{\red}_{X_\cl}$}, and the morphism
    $$
    \sigma^\red:= H^1\left(\bTT w_\alpha^\red|_{X_{\cl}}\right): \Ob^\red_{X_\cl} = H^1(\bTT_{X^\red}|_{X_{\cl}}) \to \cO_{X_\cl}
    $$
    is called the \emph{reduced cosection of the reduced obstruction sheaf $\Ob^\red_{X_\cl}$ associated with $w_{\alpha}$}.
\end{definition}

\begin{proposition}
\label{thm:rel-red-cosection-surj-along-divisor}

There is a (co)fiber square in $\QCoh(X^\red)$ as follows.
$$
\begin{tikzcd}[column sep=4em, row sep=2.5em]
	\bTT_{X^\red} & \bTT_{X}|_{X^\red} \\
	\bTT_{\mathbb{A}^1_{\mathcal{A}}[-1]}|_{X^\red} & \bTT_{\cL_{\mathcal{A}}[-1]}|_{X^\red}
	\arrow[from=1-1, to=1-2, "(\bTT\theta)|_{X^\red}"]
	\arrow[from=1-1, to=2-1, "\bTT w_{\alpha}^{\red}|_{X^\red}"']
	\arrow[from=1-2, to=2-2, "\bTT w_{\alpha}|_{X^\red}"]
	\arrow[from=2-1, to=2-2, "(\bTT\tilde{s})|_{X^\red}"']
\end{tikzcd}
$$

Moreover, we have the following results.

$(1)$ The morphism $\Ob_{X_\cl}^\log := H^1(\bTT_{X}|_{X_\cl}) \to H^{1}(\bF)$ induced by Proposition \ref{prop:red-POT-log-GLSM} admits a factorization
$$
\Ob_{X_\cl}^\log \overset{\sigma^\log}{\longrightarrow} \cO_{X_\cl}(\D_{X_\cl}) \longrightarrow H^{1}(\bF),
$$
where the morphism $\cO_{X_\cl}(\D_{X_\cl}) \to H^{1}(\bF)$ is surjective.

$(2)$ If $\sigma^\log: \Ob_{X_\cl}^\log \to \cO_{X_\cl}(\D_{X_\cl})$ is surjective along $(\D_{X_\cl})_\cl$, then $H^2(\bTT_{X^\red}|_{X_\cl}) = 0$.

$(3)$ If $\sigma^\log: \Ob_{X_\cl}^\log \to \cO_{X_\cl}(\D_{X_\cl})$ is surjective along $(\D_{X_\cl})_\cl$, then $\sigma^{\red}: \Ob_{X_{cl}}^{\red} \to \cO_{X_{\cl}}$ is surjective along $(\D_{X_{cl}})_\cl$.

$(4)$ $\sigma^{\red}|_{U_\cl} = \sigma^{\log}|_{U_\cl}$.
\end{proposition}

\begin{proof}
    By Diagram (\ref{diagram:X^red}) and Lemma \ref{lem:cotangent-complex-homotopy-cartesian-diagram}, we obtain the desired (co)fiber square
\begin{equation}\label{diagram:IMP}
\begin{tikzcd}[column sep=4em, row sep=2.5em]
	\bTT_{X^\red} & \bTT_{X}|_{X^\red} \\
	\bTT_{\mathbb{A}^1_{\mathcal{A}}[-1]}|_{X^\red} & \bTT_{\cL_{\mathcal{A}}[-1]}|_{X^\red}
	\arrow[from=1-1, to=1-2, "(\bTT\theta)|_{X^\red}"]
	\arrow[from=1-1, to=2-1, "\bTT w_{\alpha}^{\red}|_{X^\red}"']
	\arrow[from=1-2, to=2-2, "\bTT w_{\alpha}|_{X^\red}"]
	\arrow[from=2-1, to=2-2, "(\bTT\tilde{s})|_{X^\red}"']
\end{tikzcd}
\end{equation}
    in $\QCoh(X^\red)$.
    
    Pulling back Diagram (\ref{diagram:IMP}) to $X_\cl$, by Proposition \ref{prop:red-POT-log-GLSM}, we obtain the following commutative diagram, where rows are (co)fiber sequences.
$$
\begin{tikzcd}[column sep=4em, row sep=2.5em]
	\bTT_{X^\red}|_{X_\cl} & \bTT_{X}|_{X_\cl} & (\bTT_{\theta}[1])|_{X_\cl} \simeq \bF & {} \\
	\bTT_{\mathbb{A}^1_{\mathcal{A}}[-1]}|_{X_\cl} & \bTT_{\cL_{\mathcal{A}}[-1]}|_{X_\cl} & (\bTT_{\tilde{s}}[1])|_{X_\cl} \simeq \bF & {}
	\arrow[from=1-1, to=1-2, "(\bTT\theta)|_{X_\cl}"]
	\arrow[from=1-1, to=2-1, "\bTT w_{\alpha}^{\red}|_{X_{\cl}}"']
	\arrow[from=1-2, to=2-2, "\bTT w_{\alpha}|_{X_{\cl}}"]
	\arrow[from=2-1, to=2-2, "(\bTT\tilde{s})|_{X_\cl}"']
	\arrow[from=1-2, to=1-3]
	\arrow[from=1-3, to=1-4, "+1"]
	\arrow[from=1-3, to=2-3, "\simeq"]
	\arrow[from=2-2, to=2-3]
	\arrow[from=2-3, to=2-4, "+1"]
\end{tikzcd}
$$
    Then, by Proposition \ref{prop:cotangent-complex-obstruction-theory}, we have the following commutative diagram, where rows are exact sequences.
\begin{equation}\label{diag:rel-red-cosection-log-cosection}
    \begin{tikzcd}[column sep=3.6em, row sep=2.5em]
	0 & H^0(\bF) & \Ob^\red_{X_\cl} & \Ob^\log_{X_\cl} &  H^1(\bF) & H^2( \bTT_{X^\red}|_{X_\cl} ) \\
	0 & H^0(\bF) & \cO_{X_\cl} & \cO_{X_\cl}(\D_{X_\cl}) &  H^1(\bF) & 0
	\arrow[from=1-1, to=1-2]
	\arrow[from=1-2, to=2-2, "="]
	\arrow[from=1-3, to=2-3, "\sigma^\red"]
	\arrow[from=2-1, to=2-2]
	\arrow[from=1-3, to=1-4, "H^1\left((\bTT\theta)|_{X_\cl}\right)"]
	\arrow[from=1-3, to=1-4]
	\arrow[from=1-4, to=2-4, "\sigma^\log"]
	\arrow[from=2-3, to=2-4, "s_{X_\cl}^{\vee}"]
	\arrow[from=2-3, to=2-4]
    \arrow[from=1-5, to=2-5, "="]
    \arrow[from=1-2, to=1-3]
    \arrow[from=2-2, to=2-3]
    \arrow[from=1-4, to=1-5]
    \arrow[from=2-4, to=2-5]
    \arrow[from=1-5, to=1-6]
    \arrow[from=1-6, to=2-6]
    \arrow[from=2-5, to=2-6]
    \end{tikzcd}
\end{equation}
    By the second square diagram from the right in Diagram (\ref{diag:rel-red-cosection-log-cosection}), we see that the morphism $\Ob_{X_\cl}^\log \to H^{1}(\bF)$ induced by Proposition \ref{prop:red-POT-log-GLSM} admits a factorization
$$
\Ob_{X_\cl}^\log \overset{\sigma^\log}{\longrightarrow} \cO_{X_\cl}(\D_{X_\cl}) \longrightarrow H^{1}(\bF),
$$
where the morphism $\cO_{X_\cl}(\D_{X_\cl}) \to H^{1}(\bF)$ is surjective. $(1)$ is proved.

    By Proposition \ref{prop:red-POT-log-GLSM} and Proposition \ref{prop:cotangent-complex-obstruction-theory}, there is a long exact sequence
\[\begin{tikzcd}[row sep=1em]
	0 & H^0( \bTT_{X^\red}|_{X_\cl} )& H^0( \bTT_{X}|_{X_\cl} ) & H^0( \bF ) & {} \\
	{} & H^1( \bTT_{X^\red}|_{X_\cl} ) & H^1( \bTT_{X}|_{X_\cl} ) & H^1( \bF ) & H^2( \bTT_{X^\red}|_{X_\cl} ) & 0.
	\arrow[from=1-1, to=1-2]
	\arrow[from=1-2, to=1-3, "="]
	\arrow[from=1-3, to=1-4, "0"]
	\arrow[from=1-4, to=1-5]
	\arrow[from=2-2, to=2-3]
	\arrow[from=2-3, to=2-4]
	\arrow[from=2-4, to=2-5]
	\arrow[from=2-5, to=2-6]
\end{tikzcd}\]
    Note that $H^1(\bF)|_{U_\cl} =0$, so the morphism $H^1( \bTT_{X}|_{X_\cl}) \to H^1(\bF)$ is surjective along $U_\cl$. If $\sigma^\log$ is surjective along $(\D_{X_\cl})_\cl$, then the morphism $H^1( \bTT_{X}|_{X_\cl}) \to H^1(\bF)$ is surjective along $(\D_{X_\cl})_\cl$ by $(1)$. Hence, we see that $H^2( \bTT_{X^\red}|_{X_\cl} )=0$ if $\sigma^\log$ is surjective along $(\D_{X_\cl})_\cl$. $(2)$ is proved.

    Given that Diagram (\ref{diagram:IMP}) is a (co)fiber square, we have the following (co)fiber sequence.
$$
\bTT_{X^\red} \overset{\left( -(\bTT\theta)|_{X^\red}, \bTT w_{\alpha}^{\red}|_{X^\red} \right)}{\longrightarrow} \bTT_{X}|_{X^\red} \oplus \bTT_{\mathbb{A}^1_{\mathcal{A}}[-1]}|_{X^\red} \overset{\bTT w_{\alpha}|_{X^\red} \oplus (\bTT\tilde{s})|_{X^\red}}{\longrightarrow} \bTT_{\cL_{\mathcal{A}}[-1]}|_{X^\red} \overset{+1}{\longrightarrow}
$$
    Pulling back the above (co)fiber sequence to $X_\cl$ and then applying $H^0$, we obtain the following exact sequence
$$
H^0(\bTT_{\cL_{\mathcal{A}}[-1]}|_{X_\cl}) = 0 \longrightarrow \Ob^\red_{X_\cl} \overset{\left( - \bar{\theta} , \sigma^\red \right)}{\longrightarrow} \Ob^\log_{X_\cl} \oplus \cO_{X_\cl} \overset{\sigma^\log \oplus s_{X_\cl}^\vee}{\longrightarrow} \cO_{X_\cl}(\D_{X_\cl}) \longrightarrow H^2(\bTT_{X^\red}|_{X_\cl}).
$$
    Here, $\bar{\theta}$ denotes $H^1 \left( (\bTT\theta)|_{X_\cl} \right)$ for simplicity. Assume that $\sigma^\log$ is surjective along $(\D_{X_\cl})_\cl$. Then $H^2(\bTT_{X^\red}|_{X_\cl}) = 0$ by $(2)$, and hence we have a short exact sequence
$$
0 \longrightarrow \Ob^\red_{X_\cl} \overset{\left( - \bar{\theta} , \sigma^\red \right)}{\longrightarrow} \Ob^\log_{X_\cl} \oplus \cO_{X_\cl} \overset{\sigma^\log \oplus s_{X_\cl}^\vee}{\longrightarrow} \cO_{X_\cl}(\D_{X_\cl}) \longrightarrow 0.
$$
    Consequently, we have $\im \left( (-\bar{\theta}, \sigma^\red) \right) = \ker \left( \sigma^\log \oplus s_{X_\cl}^\vee \right)$, which implies that $\sigma^\red$ is surjective along $(\D_{X_\cl})_\cl$. $(3)$ is proved.

    Pulling back Diagram (\ref{diag:rel-red-cosection-log-cosection}) to $U_\cl$, we obtain the following commutative diagram, where rows are exact sequences.
\begin{equation}\label{diag:rel-red-cosection-log-cosection-along-opem-complement}
\begin{tikzcd}[column sep=3em, row sep=2.5em]
	H^0(\bF)|_{U_\cl} =0 & \Ob^\red_{X_\cl}|_{U_\cl} & \Ob^\log_{X_\cl}|_{U_\cl} &  H^1(\bF)|_{U_\cl} =0 \\
	H^0(\bF)|_{U_\cl} =0 & \cO_{X_\cl}|_{U_\cl} & \cO_{X_\cl}(\D_{X_\cl})|_{U_\cl} &  H^1(\bF)|_{U_\cl} =0
	\arrow[from=1-1, to=2-1, "="]
	\arrow[from=1-2, to=2-2, "\sigma^\red|_{U_\cl}"]
	\arrow[from=1-2, to=1-3, "="]
	\arrow[from=1-2, to=1-3]
	\arrow[from=1-3, to=2-3, "\sigma^\log|_{U_\cl}"]
	\arrow[from=2-2, to=2-3, "="]
	\arrow[from=2-2, to=2-3]
    \arrow[from=1-4, to=2-4, "="]
    \arrow[from=1-1, to=1-2]
    \arrow[from=2-1, to=2-2]
    \arrow[from=1-3, to=1-4]
    \arrow[from=2-3, to=2-4]
\end{tikzcd}
\end{equation}
    By the middle square diagram in Diagram (\ref{diag:rel-red-cosection-log-cosection-along-opem-complement}), we see that $\sigma^{\red}|_{U_\cl} = \sigma^{\log}|_{U_\cl}$. $(4)$ is proved.
\end{proof}

\begin{re}
    See Example \ref{eg:phi-red-POT} for a similar exact sequence appearing in the first row of Diagram (\ref{diag:rel-red-cosection-log-cosection}).
\end{re}

\begin{corollary}
\label{cor:red-POT}
    If the log cosection $\sigma^\log: \Ob_{X_\cl}^\log \to \cO_{X_\cl}(\D_{X_\cl})$ is surjective along $(\D_{X_\cl})_\cl$, then $$\bTT\tau^{\red}: \bTT_{X_\cl} \to \bTT_{X^\red}|_{X_\cl}$$ is a relative perfect obstruction theory on $X_\cl$ in the sense of \cite{BeFa97}.
\end{corollary}

\begin{proof}
    By Remark \ref{rmk:criterion-red-POT-is-POT} and Proposition \ref{thm:rel-red-cosection-surj-along-divisor}$(2)$, $\bTT\tau^{\red}$ is a relative perfect obstruction theory on $X_\cl$ in the sense of \cite{BeFa97} .
\end{proof}

\begin{re}\label{rmk:X-red-being-qsmth}
    By Proposition \ref{thm:rel-red-cosection-surj-along-divisor}$(2)$, if the log cosection $\sigma^\log: \Ob_{X_\cl}^\log \to \cO_{X_\cl}(\D_{X_\cl})$ is surjective along $(\D_{X_\cl})_\cl$, then $\bTT_{X^\red}|_{X_\cl}$ is a perfect complex of Tor-amplitude in $[0,1]$. Since the Tor-amplitude of a perfect complex on a derived stack can be checked on the classical truncation, $\bTT_{X^\red}$ has Tor-amplitude in $[0,1]$, which implies that $X^\red$ is quasi-smooth.
\end{re}

\begin{definition}\label{def:red-vir-cycle}
    If $X^\red$ is quasi-smooth (e.g. if the log cosection $\sigma^\log: \Ob_{X_\cl}^\log \to \cO_{X_\cl}(\D_{X_\cl})$ associated with $w_\alpha$ is surjective along $(\D_{X_\cl})_\cl$), then the \emph{reduced perfect obstruction theory on $X_\cl$ associated with $w_\alpha$} is defined to be $$\bTT\tau^{\red}: \bTT_{X_\cl} \to \bTT_{X^\red}|_{X_\cl}.$$ The \emph{reduced virtual cycle $[X_\cl]^\red$ associated with $w_\alpha$} is then defined to be the virtual cycle induced by $\bTT\tau^{\red}$.
\end{definition}

\begin{re}\label{rmk:red-cosection-applicable}
    If $X^\red$ is quasi-smooth and $X_\cl$ is separated, then Proposition \ref{prop:rel-cosection-rel-POT} applies to the relative $(-1)$-shifted function $w_{\alpha}^\red: X^{\red} \to \mathbb{A}^1_{\mathcal{A}}[-1]$ on $X^\red$ over $\cA$ associated with $w_{\alpha}$.
\end{re}

\begin{corollary}
\label{thm:red-cosection-localized-virt-cycle}

Assume that $X_\cl$ is separated. If the log cosection $\sigma^\log: \Ob^{\log}_{X_\cl} \to \cO_{X_\cl}(\D_{X_\cl})$ is surjective along $(\D_{X_\cl})_\cl$, then the reduced cosection $\sigma^{\red}: \Ob_{X_{cl}}^{\red} \to \cO_{X_{\cl}}$ gives rise to a cosection localized virtual cycle $$ [X_{\cl}]^\vir_{\sigma^{\red}} \in A_*\left(X_{\cl}(\sigma^{\red})\right) $$ such that $\iota^{\red}_*\left([X_{\cl}]^\vir_{\sigma^{\red}}\right)=[X_\cl]^{\red}$. Here, $X_{\cl}(\sigma^\red)$ denotes the classical degeneracy locus of the (reduced) cosection $\sigma^\red$ in the sense of Proposition \ref{prop:rel-cosection-rel-POT}.
\end{corollary}

\begin{proof}
    It follows from Remark \ref{rmk:X-red-being-qsmth}, Remark \ref{rmk:red-cosection-applicable}, Proposition \ref{prop:rel-cosection-rel-POT}, \cite[Theorem~1.1]{KiLi13}, and Proposition \ref{prop:rel-qsmth-rel-POT}.
\end{proof}

\begin{re}\label{rmk:degeneracy-loci-coincide}
    Assume that $X_\cl$ is separated. If the log cosection $\sigma^\log: \Ob^{\log}_{X_\cl} \to \cO_{X_\cl}(\D_{X_\cl})$ is surjective along $(\D_{X_\cl})_\cl$, then we have
    $$
    X_{\cl}(\sigma^{\log}) = X_{\cl}(\sigma^{\red})
    $$
    by Corollary \ref{thm:red-cosection-localized-virt-cycle} and Proposition \ref{thm:rel-red-cosection-surj-along-divisor}$(3)(4)$. As a result, $X_{\cl}(\widetilde{\sigma^\log})$ is then (set-theoretically) the disjoint union of $X_{\cl}(\sigma^{\red})$ and $(\D_{X_{cl}})_\cl$ by Remark \ref{rmk:decompose-virt-cycle}.
\end{re}

\begin{corollary}\label{cor:virt-cycle-red-virt-cycle-coincide-on-open-locus}
If the log cosection $\sigma^\log: \Ob^{\log}_{X_\cl} \to \cO_{X_\cl}(\D_{X_\cl})$ is surjective along $(\D_{X_\cl})_\cl$, then there is a factorization of perfect obstruction theories
$$
\begin{tikzcd}
{\bTT_{X_\cl}} \arrow[rr, "\bTT \tau"] \arrow[rd, "\bTT \tau^{\red}"'] &               & {\bTT_{X}|_{X_\cl}} \\
                         & {\bTT_{X^\red}|_{X_\cl}} \arrow[ru, "(\bTT \theta)|_{X_\cl}"'] &   
\end{tikzcd}
$$
such that $\bTT \tau^{\red}|_{U_\cl} = \bTT \tau|_{U_\cl}.$
\end{corollary}

\begin{proof}
    It directly follows from Corollary \ref{cor:red-POT}, Lemma \ref{lem:tangent-complex-factor}, and Corollary \ref{lem:POT-coindcide-outside-divisor}.
\end{proof}

\begin{re}\label{rmk:reduction-argument}
    Assume that $X_\cl$ is separated. If the log cosection $\sigma^\log: \Ob^{\log}_{X_\cl} \to \cO_{X_\cl}(\D_{X_\cl})$ is surjective along $(\D_{X_\cl})_\cl$, then the cosection localized virtual cycle $[X_{\cl}]^\vir_{\sigma^{\red}}$ coincides with the part of the cosection localized virtual cycle $[X_{\cl}]^\vir_{\widetilde{\sigma^\log}} $ supported on $X_\cl(\sigma^\red)=X_\cl(\sigma^\log)$. This follows from Proposition \ref{prop:log-meromorphic-cosection-localized-virt-cycle}, Corollary \ref{thm:red-cosection-localized-virt-cycle}, Remark \ref{rmk:degeneracy-loci-coincide}, Corollary \ref{cor:virt-cycle-red-virt-cycle-coincide-on-open-locus}, and Proposition \ref{thm:rel-red-cosection-surj-along-divisor}(4).  
\end{re}

\subsection{Reduced Boundary Cycles via Logarithmic Cosections}\label{sec:red-bound-cycle}

Throughout this subsection, we work with the duals of (relative) perfect obstruction theories in the sense of \cite{BeFa97}, and refer to them simply as perfect obstruction theories, to align with the conventions in \cite[Section~6]{CJR21}.

\begin{definition}
    Let $w_{\alpha} : X \to L_{\cA}[-1]$ be a $(-1)$-shifted log function on $(X, \alpha_X)$. The derived stack $X^{\tred}$ is defined via the following homotopy Cartesian diagram in $\LogDStk$.
\begin{equation}\label{diagram:X^tred}
\begin{tikzcd}
X^{\tred} \arrow[d, "\upsilon"'] \arrow[r] \arrow[rd, "\lrcorner_{\text{\tiny h}}", phantom, very near start] & \cA \arrow[d, "0"'] \\
X \arrow[r, "w_{\alpha}"]                 & L_{\cA}[-1]          
\end{tikzcd}
\end{equation}
\end{definition}

\begin{re}
    $L_{\cA}[-1] \to \cA$ is quasi-smooth, hence $0: \cA \to L_{\cA}[-1]$ is homotopically $2$-smooth. As a result, $X^\tred$ is homotopically $2$-smooth.
\end{re}

\begin{re}
    We have the following commutative diagram, where square diagrams are homotopy Cartesian diagrams.
\begin{equation}\label{diagram:shifted-function-on-X^tred}
\begin{tikzcd}
	X^{\tred} & \cA \\
	X^{\red} & \bA^1_{\cA}[-1] \\
	X & L_{\cA}[-1]
	\arrow[from=1-1, to=1-2]
	\arrow[from=1-1, to=2-1, "\xi"]
    \arrow[from=1-1, to=2-2, "\lrcorner_{\text{\tiny h}}", phantom, very near start]
	\arrow[from=1-2, to=2-2, "0"]
	\arrow[from=2-1, to=2-2, "w_{\alpha}^\red"]
	\arrow[from=2-1, to=3-1, "\theta"]
    \arrow[from=2-1, to=3-2, "\lrcorner_{\text{\tiny h}}", phantom, very near start]
	\arrow[from=2-2, to=3-2, "\widetilde{s}"]
	\arrow[from=3-1, to=3-2, "w_{\alpha}"]
    \arrow[from=1-1, to=3-1, "\upsilon"', bend right=50]
\end{tikzcd}
\end{equation}
\end{re}

\begin{re}\label{rmk:class-X-tred-is-class-X}
    By similar arguments in Remark \ref{rmk:class-X-red-is-class-X}, $X^\tred_\cl$ and $X_\cl$ are exactly the same classical Deligne--Mumford stack. For simplicity, we also use $X_\cl$ to denote the classical truncation $X^{\tred}_\cl$ of $X^\tred$.
\end{re}

\begin{re}
By similar arguments in Remark \ref{rmk:X-fact-through-X-red}, we have the following commutative diagram.
\begin{equation}\label{diagram:X^tred-factor}
\begin{tikzcd}
	{} & X^{\tred} \\
	X_\cl & X
	\arrow[from=2-1, to=1-2, hook, "\tau^{\tred}"]
	\arrow[from=1-2, to=2-2, "\upsilon"]
	\arrow[from=2-1, to=2-2, hook, "\tau"]
\end{tikzcd}
\end{equation}
\end{re}

\begin{lemma}\label{lem:tangent-complex-factor-tred}
    There is a commutative diagram as follows.
$$
\begin{tikzcd}
{\bTT_{X_\cl}} \arrow[rr, "\bTT \tau"] \arrow[rd, "\bTT \tau^{\tred}"'] &               & {\bTT_{X}|_{X_\cl}} \\
                         & {\bTT_{X^\tred}|_{X_\cl}} \arrow[ru, "(\bTT \upsilon)|_{X_\cl}"'] &   
\end{tikzcd}
$$
\end{lemma}

\begin{proof}
    It directly follows from applying the log tangent complex functor to Diagram (\ref{diagram:X^tred-factor}).
\end{proof}

\begin{proposition}
\label{prop:tred-POT-log-GLSM}
    There is a (co)fiber sequence as follows.
    $$\bTT_{X^\tred}|_{X_\cl} \overset{(\bTT \upsilon)|_{X_\cl}}{\longrightarrow} \bTT_{X}|_{X_\cl} \overset{(\bTT w_{\alpha})|_{X_\cl}}{\longrightarrow} \bTT_{\cL_\cA[-1]}|_{X_\cl} \simeq \cO_{X_\cl}(\D_{X_\cl})[-1] \overset{+1}{\longrightarrow} $$
\end{proposition}

\begin{proof}
    It directly follows from Diagram (\ref{diagram:X^tred}) and Lemma \ref{lem:cotangent-complex-homotopy-cartesian-diagram}.
\end{proof}

\begin{re}
    Proposition \ref{prop:tred-POT-log-GLSM} may be interpreted as saying that $X^\tred$ is a natural derived geometric enhancement of the totally reduced perfect obstruction theory on $X_\cl$ constructed in \cite[Section~6.5]{CJR21}.
\end{re}

\begin{re}\label{rmk:criterion-tred-POT-is-POT}
    Note that $X^\tred$ is homotopically $2$-smooth. Hence, to show that $$\bTT\tau^{\tred}: \bTT_{X_\cl} \to \bTT_{X^\tred}|_{X_\cl}$$ is a relative perfect obstruction theory on $X_{\cl}$ in the sense of \cite{BeFa97}, it suffices to show that $H^2( \bTT_{X^\tred}|_{X_\cl} ) = 0$ by Proposition \ref{prop:cotangent-complex-obstruction-theory}.
\end{re}

\begin{re}\label{rmk:tred-exact-seq}
    By Proposition \ref{prop:tred-POT-log-GLSM}, there is an exact sequence
    $$
H^1(\bTT_{X}|_{X_\cl}) = \Ob^\log_{X_\cl} \overset{\sigma^\log}{\longrightarrow} \cO_{X_\cl}(\D_{X_\cl}) \longrightarrow H^2(\bTT_{X^\tred}|_{X_\cl}).
    $$
\end{re}

\begin{corollary}
\label{cor:tred-POT}
    If the log cosection $\sigma^\log: \Ob_{X_\cl}^\log \to \cO_{X_\cl}(\D_{X_\cl})$ is surjective everywhere on ${X_\cl}$, then $$\bTT\tau^{\tred}: \bTT_{X_\cl} \to \bTT_{X^\tred}|_{X_\cl}$$ is a relative perfect obstruction theory on $X_\cl$ in the sense of \cite{BeFa97}.
\end{corollary}

\begin{proof}
    It directly follows from Remark \ref{rmk:criterion-tred-POT-is-POT} and Remark \ref{rmk:tred-exact-seq}.
\end{proof}

\begin{re}\label{rmk:X-tred-being-qsmth}
    By Corollary \ref{cor:tred-POT} and similar arguments in Remark \ref{rmk:X-red-being-qsmth}, if the log cosection $\sigma^\log: \Ob_{X_\cl}^\log \to \cO_{X_\cl}(\D_{X_\cl})$ is surjective everywhere on ${X_\cl}$, then $X^\tred$ is in fact quasi-smooth.
\end{re}

\begin{definition}
    If $X^\tred$ is quasi-smooth (e.g. if the log cosection $\sigma^\log: \Ob_{X_\cl}^\log \to \cO_{X_\cl}(\D_{X_\cl})$ associated with $w_\alpha$ is surjective everywhere on ${X_\cl}$), then the \emph{totally reduced perfect obstruction theory on $X_\cl$ associated with $w_\alpha$} is defined to be $$\bTT\tau^{\tred}: \bTT_{X_\cl} \to \bTT_{X^\tred}|_{X_\cl}.$$ The \emph{totally reduced virtual cycle $[X_\cl]^\tred$ associated with $w_\alpha$} is then defined to be the virtual cycle induced by $\bTT\tau^{\tred}$.
\end{definition}

\begin{definition}\label{def:red-bound-POT}
    Given Diagram (\ref{diagram:X^tred}), we have the following commutative diagram, where square diagrams are homotopy Cartesian diagrams.
\begin{equation}
\label{diag:divisor-red}
\begin{tikzcd}[column sep=4em]
(\D_{X_\cl})_\cl \arrow[r, hook] \arrow[dr, hook] \arrow[rrr, bend left, "\tau_\D"] \arrow[rr, bend left, "\tau_{\D}^\red"'] & \D_{X_\cl} \arrow[r, hook] \arrow[d, hook] \arrow[rd, "\lrcorner_{\text{\tiny h}}", phantom, very near start] & \D_{X^{\tred}} \arrow[r, "\upsilon_\D"] \arrow[d, hook] \arrow[rd, "\lrcorner_{\text{\tiny h}}", phantom, very near start] & \D_{X} \arrow[r] \arrow[d, hook] \arrow[rd, "\lrcorner_{\text{\tiny h}}", phantom, very near start] & B\bG_m \arrow[d, hook] \\
& X_\cl \arrow[r, hook, "\tau^\tred"]  & X^{\tred} \arrow[r, "\upsilon"]    & X \arrow[r, "\alpha_X"]     & \cA    
\end{tikzcd}
\end{equation}
Here, $\D^\red_X := \D_{X^\tred}$ is called the $w_\alpha$-reduction of $\D_X$.
\end{definition}

\begin{re}
    Note that we have
    $$
    (\D_{X_\cl})_\cl = (\D^\red_X)_\cl = (\D_X)_\cl.
    $$
\end{re}

\begin{re}
    $\D^\red_X$ is homotopically $2$-smooth, because $X^\tred$ is homotopically $2$-smooth.
\end{re}

\begin{re}
    For simplicity, we define \emph{semi-log tangent complex functor $\bTs_{-}$} to be the relative tangent complex functor $\bT_{-/B\bG_m}$ over $B\bG_m$. Then, by Diagram (\ref{diag:divisor-red}), there are two quasi-isomorphisms
\begin{align}
\label{align:cycle-bound-cycle}
    \bTT_{X}|_{\D_X} &\simeq \bTs_{\D_X}, \\
\label{align:tred-cycle-red-bound-cycle}
    \bTT_{X^\tred}|_{\D^\red_X} &\simeq \bTs_{\D_X^\red}.
\end{align}
\end{re}

\begin{lemma}\label{lem:tangent-complex-factor-Dtred}
    There is a commutative diagram as follows.
$$
\begin{tikzcd}
{\bTs_{(\D_{X_\cl})_\cl}} \arrow[rr, "\bTs \tau_\D"] \arrow[rd, "\bTs \tau^{\red}_\D"'] &               & {\bTs_{\D_X}|_{(\D_{X_\cl})_\cl}} \\
                         & {\bTs_{\D^\red_X}|_{(\D_{X_\cl})_\cl}} \arrow[ru, "(\bTs \upsilon_\D)|_{(\D_{X_\cl})_\cl}"'] &   
\end{tikzcd}
$$
\end{lemma}

\begin{proof}
    It directly follows from applying the semi-log tangent complex functor to the following diagram.
    $$
\begin{tikzcd}
	{} & \D_X^{\red} \\
	(\D_{X_\cl})_\cl & \D_X
	\arrow[from=2-1, to=1-2, hook, "\tau^{\red}_\D"]
	\arrow[from=1-2, to=2-2, "\upsilon_\D"]
	\arrow[from=2-1, to=2-2, hook, "\tau_\D"]
\end{tikzcd}
    $$
\end{proof}

\begin{re}
    Since $\D_X$ is quasi-smooth and $B\bG_m$ is smooth, $\bTs \tau_\D: \bTs_{\D_{X_\cl}} \to \bTs_{\D_X}|_{\D_{X_\cl}}$ is a (relative) perfect obstruction theory on $(\D_{X_\cl})_\cl = (\D_X)_\cl$ in the sense of Behrend--Fantechi \cite[Section~7]{BeFa97} by Proposition \ref{prop:cotangent-complex-obstruction-theory}.
\end{re}

\begin{proposition}
\label{prop:red-divisor-POT-log-GLSM}
    There is a (co)fiber sequence as follows.
    $$\bTs_{\D^\red_X}|_{(\D_{X_\cl})_\cl} \overset{(\bTs \upsilon_\D)|_{(\D_{X_\cl})_\cl}}{\longrightarrow} \bTs_{\D_X}|_{(\D_{X_\cl})_\cl} \overset{(\bTT w_\alpha)|_{(\D_{X_\cl})_\cl}}{\longrightarrow} \cO_{X_\cl}(\D_{X_\cl})|_{(\D_{X_\cl})_\cl}[-1] \overset{+1}{\longrightarrow} $$
\end{proposition}

\begin{proof}
    It follows from Diagram (\ref{diag:divisor-red}), Equation (\ref{align:cycle-bound-cycle}), Equation (\ref{align:tred-cycle-red-bound-cycle}), and Proposition \ref{prop:tred-POT-log-GLSM}.
\end{proof}

\begin{re}
    Proposition \ref{prop:red-divisor-POT-log-GLSM} may be interpreted as saying that $\D_X^\red$ is a natural derived geometric enhancement of the reduced perfect obstruction theory on $(\D_{X_\cl})_\cl$ constructed in \cite[Section~6.5]{CJR21}.
\end{re}

\begin{re}\label{rmk:criterion-red-divisor-POT-is-POT}
    Note that $\D_X^\red$ is homotopically $2$-smooth. Hence, to show that $$\bTs \tau^{\red}_\D: \bTs_{(\D_{X_\cl})_\cl} \to \bTs_{\D^\red_X}|_{(\D_{X_\cl})_\cl}$$ is a relative perfect obstruction theory on $(\D_{X_\cl})_\cl$ in the sense of \cite{BeFa97}, it suffices to show that $H^2( \bTs_{\D^\red_X}|_{(\D_{X_\cl})_\cl} ) = 0$ by Proposition \ref{prop:cotangent-complex-obstruction-theory}.
\end{re}

\begin{corollary}
\label{cor:red-divisor-POT}
    If the morphism $$H^1((\bTT w_{\alpha})|_{(\D_{X_\cl})_\cl}): H^1(\bTs_{\D_X}|_{(\D_{X_\cl})_\cl}) = H^1(\bTT_{X}|_{(\D_{X_\cl})_\cl}) \to \cO_{X_\cl}(\D_{X_\cl})|_{(\D_{X_\cl})_\cl}$$ is surjective, then $$\bTs \tau^{\red}_\D: \bTs_{(\D_{X_\cl})_\cl} \to \bTs_{\D^\red_X}|_{(\D_{X_\cl})_\cl}$$ is a relative perfect obstruction theory on $(\D_{X_\cl})_\cl$ in the sense of \cite{BeFa97}.
\end{corollary}

\begin{proof}
    It directly follows from Proposition \ref{prop:red-divisor-POT-log-GLSM} and Remark \ref{rmk:criterion-red-divisor-POT-is-POT}.
\end{proof}

\begin{re}\label{rmk:DX-red-being-qsmth}
    By Corollary \ref{cor:red-divisor-POT} and similar arguments in Remark \ref{rmk:X-red-being-qsmth}, if the morphism $$H^1((\bTT w_{\alpha})|_{(\D_{X_\cl})_\cl}): H^1(\bTs_{\D_X}|_{(\D_{X_\cl})_\cl}) = H^1(\bTT_{X}|_{(\D_{X_\cl})_\cl}) \to \cO_{X_\cl}(\D_{X_\cl})|_{(\D_{X_\cl})_\cl}$$ is surjective, then $\D_X^\red$ is in fact quasi-smooth.
\end{re}

\begin{re}\label{rmk:log-cosection-surj-along-divisor}
    Assuming the log cosection 
$$\sigma^\log := H^1((\bTT w_{\alpha})|_{X_\cl}): \Ob_{X_\cl}^\log :=H^1(\bTT_{X}|_{X_\cl}) \to \cO_{X_\cl}(\D_{X_\cl})$$
is surjective along $(\D_{X_\cl})_\cl$, i.e. the morphism
$$
H^1((\bTT w_{\alpha})|_{X_\cl})|_{(\D_{X_\cl})_\cl}: H^1(\bTT_{X}|_{X_\cl})|_{(\D_{X_\cl})_\cl} \to \cO_{X_\cl}(\D_{X_\cl})|_{(\D_{X_\cl})_\cl}
$$
    is surjective. Then the morphism 
$$H^1((\bTT w_{\alpha})|_{(\D_{X_\cl})_\cl}): H^1(\bTT_{X}|_{(\D_{X_\cl})_\cl}) \to \cO_{X_\cl}(\D_{X_\cl})|_{(\D_{X_\cl})_\cl}$$
    is surjective.
\end{re}

\begin{re}\label{rmk:what-does-log-cosection-surj-imply}
    By Corollary \ref{cor:red-divisor-POT}, Remark \ref{rmk:DX-red-being-qsmth}, and Remark \ref{rmk:log-cosection-surj-along-divisor}, if the log cosection $\sigma^\log: \Ob^{\log}_{X_\cl} \to \cO_{X_\cl}(\D_{X_\cl})$ is surjective along $(\D_{X_\cl})_\cl$, then $$\bTs \tau^{\red}_\D: \bTs_{(\D_{X_\cl})_\cl} \to \bTs_{\D^\red_X}|_{(\D_{X_\cl})_\cl}$$ is a relative perfect obstruction theory on $(\D_{X_\cl})_\cl$ in the sense of \cite{BeFa97}, and $\D_X^\red$ is in fact quasi-smooth.
\end{re}

\begin{definition}\label{rmk:red-divisor-POT}
    If $\D_X^\red$ is quasi-smooth (e.g. if the log cosection $\sigma^\log: \Ob_{X_\cl}^\log \to \cO_{X_\cl}(\D_{X_\cl})$ associated with $w_\alpha$ is surjective along $(\D_{X_\cl})_\cl$), then the \emph{reduced perfect obstruction theory on $(\D_{X_\cl})_\cl$ associated with $w_\alpha$} is defined to be $$\bTs \tau^{\red}_\D: \bTs_{(\D_{X_\cl})_\cl} \to \bTs_{\D^\red_X}|_{(\D_{X_\cl})_\cl}.$$ The \emph{reduced boundary cycle $[(\D_{X_\cl})_\cl]^\red$ associated with $w_\alpha$} is then defined to be the virtual cycle induced by $\bTs \tau^{\red}_\D$.
\end{definition}

\begin{re}\label{rmk:notation-red-bound-cycle}
    Note that $(\D_{X_\cl})_\cl = (\D_{X})_\cl$. For simplicity, we may also use $[(\D_{X})_\cl]^\red$ to denote the reduced boundary cycle associated with $w_\alpha$.
\end{re}

\subsection{Decomposition Theorem for Virtual Cycles}\label{sec:decompose-vir-cycle}

In this subsection, we prove the decomposition theorem for virtual cycles (Theorem \ref{thm:decompose-virt-cycle}) using results from Section \ref{sec:log-cosection}, Section \ref{sec:red-vir-cycle}, and Section \ref{sec:red-bound-cycle}.

Let $w_{\alpha} : X \to L_{\cA}[-1]$ be a $(-1)$-shifted log function on a separated quasi-smooth derived log Deligne--Mumford stack $(X, \alpha_X)$.

\begin{re}\label{rmk:vdim-X-red-DX-red}
    Let $\vdim\ X$ denote the virtual dimension of $X$, which is the virtual rank of $\bL_X|_{X_\cl}$. Note that $\vdim\ X = \vdim\ X^\red = \vdim\ \D_X^\red = \vdim\ X^\tred -1$.
\end{re}

If the log cosection $\sigma^\log: \Ob^{\log}_{X_\cl} \to \cO_{X_\cl}(\D_{X_\cl})$ associated with $w_{\alpha}$ is surjective along $(\D_{X_\cl})_\cl$, then $X^\red$ and $\D_X^\red$ are quasi-smooth (see Remark \ref{rmk:X-red-being-qsmth} and Remark \ref{rmk:what-does-log-cosection-surj-imply}), and we have
$$
\dim\ [X_\cl]^\vir = \dim\ [X_\cl]^\red = \dim\ [(\D_{X_\cl})_\cl]^\red
$$
by Remark \ref{rmk:vdim-X-red-DX-red}. Moreover, we have the following theorem.

\begin{theorem}[Decomposition Theorem for Virtual Cycles]
\label{thm:decompose-virt-cycle}

    Let $w_{\alpha} : X \to L_{\cA}[-1]$ be a $(-1)$-shifted log function on a separated\footnote{A derived Deligne--Mumford stack is separated if its classical truncation is separated.} quasi-smooth derived log Deligne--Mumford stack $(X, \alpha_X)$. If the log cosection $\sigma^\log: \Ob^{\log}_{X_\cl} \to \cO_{X_\cl}(\D_{X_\cl})$ associated with $w_{\alpha}$ is surjective along $(\D_{X_\cl})_\cl$, then we have a relation of virtual cycles
$$
[X_\cl]^\vir = [X_\cl]^\red + i_*[(\D_{X_\cl})_\cl]^\red.
$$
    Here $i: (\D_{X_\cl})_\cl \hookrightarrow X_\cl$ is the natural closed embedding.
\end{theorem}

\begin{proof}
    There is a right exact sequence of Chow groups of classical Deligne--Mumford stacks
$$
A_{*}\left( X_\cl\left( \sigma^\red \right) \right) \to A_{*}\left( X_\cl \right) \to A_{*}\left( X_\cl \setminus X_\cl\left( \sigma^\red \right) \right) \to 0.
$$
    Replacing $X_\cl$ by $X_\cl \setminus X_\cl\left( \sigma^\red \right)$, we may assume that the log cosection $\sigma^\log$ is surjective everywhere on $X_\cl$ by Remark \ref{rmk:degeneracy-loci-coincide}. By Corollary \ref{thm:red-cosection-localized-virt-cycle}, the cosection localized virtual cycle $[X_\cl]^\vir_{\sigma^\red}$ is supported on $X_\cl(\sigma^\red)$, which is empty by assumption. Hence, we see that $[X_\cl]^\vir_{\sigma^\red}=0$, and thus $[X_\cl]^\red=0$. It remains to show that
$$
[X_\cl]^\vir = i_* [(\D_{X_\cl})_\cl]^\red
$$
    by Remark \ref{rmk:reduction-argument}.
    
    By Corollary \ref{cor:tred-POT} and Remark \ref{rmk:X-tred-being-qsmth}, $\bTT\tau^\tred$ is a perfect obstruction theory and $X^\tred$ is quasi-smooth. By Remark \ref{rmk:what-does-log-cosection-surj-imply}, $\bTs\tau_{\D}^\red$ is a perfect obstruction theory and $\D_X^\red$ is quasi-smooth. Moreover, the natural closed embedding $\D^\red_{X} \hookrightarrow X^\tred$ in Diagram (\ref{diag:divisor-red}) is quasi-smooth, and its classical truncation is $i: (\D_{X_\cl})_\cl \hookrightarrow X_\cl$. Hence, by \cite[Proposition~1.5]{STV15}, we have
$$
i^![X_\cl]^\tred = [(\D_{X_\cl})_\cl]^\red,
$$
    where $i^!: A_{*}\left( X_\cl \right) \to A_{*}\left( (\D_{X_\cl})_\cl \right)$ is the virtual pullback (see \cite{Manolache12} for more details).

    $\fE_{X_\cl} := \V_{X_\cl}(\bLL_{X}|_{X_\cl}[-1])$ and $\fE^\tred_{X_\cl} := \V_{X_\cl}(\bLL_{X^\tred}|_{X_\cl}[-1])$ are vector bundle stacks over $X_\cl$ in the sense of \cite{BeFa97} (see \cite[Proposition~2.4]{BeFa97} and \cite[Remark~3.12]{ArPs24}). Moreover, $\fE^\tred_{X_\cl}$ is the kernel cone stack of $\fE_{X_\cl} \to \V_{X_\cl} \left(\cO_{X_\cl}(-\D_{X_\cl})\right)$ induced by the morphism $\bTT_{X}|_{X_\cl} \to \cO_{X_\cl}(\D_{X_\cl})[-1]$ in Proposition \ref{prop:tred-POT-log-GLSM} (see \cite[Proposition~2.7]{BeFa97}, \cite[Lemma~3.4]{ArPs24}, and \cite[Remark~3.9]{ArPs24}). Let $\fC_{X_\cl}$ be the relative intrinsic normal cone of $(\alpha_X)_\cl: X_\cl \to \cA$ in the sense of \cite{BeFa97}.
    Unwinding the definition of cosection localized virtual cycle in \cite[Definition~3.2]{KiLi13} and {\cite[Definition~2.3]{KiLi13}}, we have 
$$
[X_\cl]^\vir_{\widetilde{\sigma^\log}} = i^! 0^!_{\fE^\tred_{X_\cl}}[\fC_{X_\cl}] = i^![X_\cl]^\tred.
$$
    Note that $[X_\cl]^\vir_{\widetilde{\sigma^\log}} \in A_{*}\left( (\D_{X_\cl})_\cl \right)$, because $X_{\cl}(\widetilde{\sigma^\log}) = (\D_{X_\cl})_\cl$ by Remark \ref{rmk:decompose-virt-cycle}. 
    
    By Proposition \ref{prop:log-meromorphic-cosection-localized-virt-cycle}, we have
$$
[X_\cl]^\vir = i_* [X_\cl]^\vir_{\widetilde{\sigma^\log}}.
$$
    
    Therefore, we see that
$$
[X_\cl]^\vir = i_* [X_\cl]^\vir_{\widetilde{\sigma^\log}}= i_*i^![X_\cl]^\tred = i_* [(\D_{X_\cl})_\cl]^\red.
$$   
\end{proof}

\begin{re}
    $\D^\red_{X} \hookrightarrow X^\tred$ is a virtual Cartier divisor, and hence quasi-smooth, which ensures the existence of the virtual pullback $i^!$ associated with its classical truncation $i: (\D_{X_\cl})_\cl \hookrightarrow X_\cl$ by \cite[Proposition~1.5]{STV15}. Note that the fact that the classical truncation of a virtual Cartier divisor is not necessarily an effective Cartier divisor in the classical sense is not an obstruction to the existence of virtual pullback $i^!$.
\end{re}

\begin{re}
    Note that $\xi: X^\tred \to X^\red$ is homotopically $2$-smooth, so the current technology for virtual pullback does not apply at all. In particular, we do \textbf{NOT} have
    $$
    [X_\cl]^\tred = \xi^![X_\cl]^\red = 0
    $$
    when log cosection $\sigma^\log$ is surjective everywhere on $X_\cl$.
\end{re}

\begin{re}
    By Diagram (\ref{diagram:shifted-function-on-X^tred}), associated with $w_\alpha$, there is a zero relative $(-1)$-shifted function
$$
w_\alpha^\tred: X^\tred \longrightarrow \cA \overset{0}{\longrightarrow} \bA^1_\cA[-1]
$$
    on $X^\tred$ over $\cA$. The degeneracy locus of the (relative) cosection associated with $w_\alpha^\tred$ is $X_\cl$. This implies that the corresponding cosection localized virtual cycle is precisely $[X_\cl]^\tred$ if we assume that the log cosection $\sigma^\log$ associated with $w_\alpha$ is surjective everywhere on $X_\cl$. 
\end{re}

\begin{example}[Toy Model for Theorem \ref{thm:decompose-virt-cycle}]
\label{eg:toy-model-decompose-virt-cycle}

Let $X$ be a smooth (classical) scheme, and $\D_X$ be an effective Cartier divisor on $X$. Let $j$ denote the canonical section $\cO_X \to \cO_X(\D_X)$ defining the divisor $\D_X$.

Assume there is a morphism\footnote{Within this example, vector bundles and locally free sheaves share the same notation, e.g. $E := \Spec_X \Sym_X (E^\vee)$.}
$$
\sigma^\log: E \to \cO_X(\D_X)
$$
    that is surjective along $\D_X$. Here, $E$ is a locally free sheaf of rank $r$. Note that the morphism $$\sigma^\log \oplus j: E \oplus \cO_X \to \cO_X(\D_X)$$ is surjective, and hence $E^\red := \ker \left(\sigma^\log \oplus j\right)$ is a locally free sheaf of rank $r$. Then there is a short exact sequence
\begin{align}
    \label{align:origin-X-red}
    0 \to E^\red \to E \oplus \cO_X \to \cO_X(\D_X) \to 0
\end{align}
    of locally free sheaves. By {\cite[Theorem~3.2(e)]{Fulton98}}, we have
\begin{align*}
    \bfct(E \oplus \cO_X) = \bfct(E^\red)\cdot \bfct(\cO_X(\D_X)),
\end{align*}
    where $\bfct$ denotes the Chern polynomial.
    Thus, by standard computation, we have
\begin{align}
    \label{align:1-decompose-euler-class}
    \bfc_r(E^\red) \bfc_1(\cO_X(\D_X)) &= 0; \\
    \label{align:2-decompose-euler-class}
    \bfc_r(E^\red) + \bfc_{r-1}(E^\red)\bfc_1(\cO_X(\D_X)) &= \bfc_r(E).
\end{align}
    Equation (\ref{align:2-decompose-euler-class}) tells us that there is a decomposition formula
    $$\bfc_r(E) \cap [X] = \bfc_r(E^\red) \cap [X] + \bfc_{r-1}(E^\red) \cap [\D_X]$$
    for cycle $\bfc_r(E) \cap [X]$, where $\bfc_{r-1}(E^\red) \cap [\D_X]$ is supported on $\D_X$. $\blacksquare$
\end{example}

\begin{re}
    Note that Sequence (\ref{align:origin-X-red}) inspires the construction of $X^\red$ (see Definition \ref{def:red-POT}).
\end{re}

\begin{example}[Local Model for Theorem \ref{thm:decompose-virt-cycle}]
\label{eg:local-model-decompose-virt-cycle}

Consider the following diagram.\footnote{Within this example, vector bundles and locally free sheaves share the same notation, e.g. $E := \Spec_X \Sym_X (E^\vee)$.}
$$
\begin{tikzcd}
                   & {E} \arrow[r, "\sigma^\log"] \arrow[d]      & {\cO_X(\D_X)} \\
{Z(s)_\cl} \arrow[r, hook] & {X} \arrow[u, bend right, "s"'] &   
\end{tikzcd}
$$
Here, $X$ is a connected smooth (classical) scheme of dimension $d$; $\cO_X(\D_X)$ is a line bundle associated with an effective Cartier divisor $\D_X$ on $X$; $E$ is a vector bundle of rank $r$; $Z(s)_\cl$ is the classical zero locus of a section $s: X \to E$. We assume that $\sigma^\log \circ s = 0.$

By Example \ref{eg:local-model-POT}, we have
$$
[Z(s)_\cl]^\vir = e(E,s) \in A_{d-r}\left( Z(s)_\cl\right).
$$
Here, $[Z(s)_\cl]^\vir$ is the virtual cycle; $e(E,s)$ is the refined Euler class of $E$ with respect to $s$.

We additionally assume that $\sigma^\log: E \to \cO_X(\D_X)$ is surjective along $\D_X$. Let $j: \cO_X \to \cO_X(\D_X)$ denote the canonical section defining the divisor $\D_X$. Note that the morphism $$\sigma^\log \oplus j: E \oplus \cO_X \to \cO_X(\D_X)$$ is surjective due to the assumption that $\sigma^\log$ is surjective along $\D_X$. 
Then there is a commutative diagram below, where the row is a short exact sequence of locally free sheaves.
$$
\begin{tikzcd}
{0} \arrow[r] & {E^\red} \arrow[r, "(\bar{\theta} \text{,} \sigma^\red)"] & {E \oplus \cO_X} \arrow[r, "\sigma^\log \oplus j"]                               & {\cO_X(\D_X)} \arrow[r] & {0} \\
             &              & {\cO_X} \arrow[lu, dashed, "s"] \arrow[u, "(s \text{,} 0)"'] \arrow[ru, "0"'] &              &   
\end{tikzcd}
$$
We define the reduced virtual cycle $[Z(s)_\cl]^\red$ of $Z(s)_\cl$ to be the refined Euler class $e(E^\red,s)$ of $E^\red$ with respect to $s$.

Note that $\sigma^\red \circ s = 0$. Moreover, $\sigma^\red: E^\red \to \cO_X$ is surjective along $\D_X$, due to the fact that $\ker \left( \sigma^\log \oplus j \right) = \im \left( (\bar{\theta}, \sigma^\red) \right)$. Thus, we see that the classical degeneracy locus of $\sigma^\red$ is contained in $Z(s)_\cl \setminus \D_X$. This means that the reduced virtual cycle $[Z(s)_\cl]^\red$ of $Z(s)_\cl$ is supported outside $\D_X$ by Example \ref{eg:local-model-cosection-localization}.

From now on, we moreover assume that $\sigma^\log: E \to \cO_X(\D_X)$ is a surjective morphism.

Then $\sigma^\red: E^\red \to \cO_X$ is a surjective morphism due to the assumption that $\sigma^\log$ is a surjective morphism and the fact that $\ker \left( \sigma^\log \oplus j \right) = \im \left( (\bar{\theta}, \sigma^\red) \right)$. Thus, we see that the degeneracy locus of $\sigma^\red$ is empty and hence the cosection localized virtual cycle $[Z(s)_\cl]_{\sigma^\red} = 0$ by Example \ref{eg:local-model-cosection-localization}. As a result, we have 
$$[Z(s)_\cl]^\red = 0,$$
which implies that 
$$
\bfc_r(E) \cap [X] = \bfc_{r-1}(E^\red) \cap [\D_X]
$$
by Example \ref{eg:toy-model-decompose-virt-cycle} and {\cite[Proposition~14.1(a)]{Fulton98}}.

There is also a commutative diagram below, where the row is a short exact sequence of locally free sheaves.
\begin{equation}\label{diagram:E-tred}
\begin{tikzcd}
{0} \arrow[r] & {E^\tred} \arrow[r] & {E} \arrow[r, "\sigma^\log"]                               & {\cO_X(\D_X)} \arrow[r] & {0} \\
             &              & {\cO_X} \arrow[lu, dashed, "s"] \arrow[u, "s"'] \arrow[ru, "0"'] &              &   
\end{tikzcd}
\end{equation}
We define the totally reduced virtual cycle $[Z(s)_\cl]^\tred$ of $Z(s)_\cl$ to be the refined Euler class $e(E^\tred,s)$ of $E^\tred$ with respect to $s$, and the reduced boundary cycle to be $$[\D_X \cap Z(s)_\cl]^\red:= \bfc_1(\cO_X(\D_X)) \cap e(E^\tred,s) \in A_{d-r}(\D_X \cap Z(s)_\cl).$$
Note that we have
$$
\bfct(E^\tred)\cdot \bfct(\cO_X(\D_X)) = \bfct(E)
$$
by the above short exact sequence. Hence, we have
$$
\bfc_{r-1}(E^\tred) \cap [\D_X] = \bfc_r (E) \cap [X] = \bfc_{r-1}(E^\red) \cap [\D_X].
$$
Note that $\bfc_{r-1}(E^\tred)$ is the top Chern class (or Euler class) of $E^\tred$, but $\bfc_{r-1}(E^\red)$ is not the top Chern class (or Euler class) of $E^\red$.

Set $U := X \setminus \D_X$. Since $\sigma^\log: E \to \cO_X(\D_X)$ is assumed to be surjective, we can construct a map $\widetilde{\sigma^\log}: E \dashrightarrow \cO_X$ such that it is undefined along $D_X$ and
$$
    \widetilde{\sigma^\log}|_U := \sigma^\log|_U: E|_U \to \cO_U \text{ is surjective}
$$
Note that $\id: X \to X$ is a $\widetilde{\sigma^\log}$-regularizing morphism in the sense of \cite[Definition~2.3]{KiLi13}, since $\widetilde{\sigma^\log}$ extends to $\sigma^\log$ by construction. Hence, following \cite[Basic Construction]{KiLi13}, we have
$$
[Z(s)_\cl]^\vir_{\widetilde{\sigma^\log}} := \bfc_1(\cO_X(\D_X)) \cap e(E^\tred,s)
$$
with $\ker\left( \sigma^\log \right) = E^\tred$.

Then we immediately see that 
$$
[Z(s)_\cl]^\vir_{\widetilde{\sigma^\log}} = \bfc_1(\cO_X(\D_X)) \cap e(E^\tred,s) = [\D_X \cap Z(s)_\cl]^\red.
$$
Therefore, we have
$$
[Z(s)_\cl]^\vir = i_*[Z(s)_\cl]^\vir_{\widetilde{\sigma^\log}} = i_*[\D_X \cap Z(s)_\cl]^\red,
$$
where $i: \D_X \cap Z(s)_\cl \hookrightarrow Z(s)_\cl$ is the closed embedding. $\blacksquare$

\end{example}

\begin{example}[The Derived Self-Intersection of $\bP^1$ in $\bP^2$]\
\label{eg:der-self-intersect-P1-in-P2}

Let $\widetilde{\bP^1}$ denote the derived self-intersection of $\bP^1$ in $\bP^2$, defined via the homotopy Cartesian diagram
\begin{equation}\label{diag:1-tilde-P^1}
\begin{tikzcd}
{\widetilde{\bP^1}} \arrow[r] \arrow[d] \ar[dr, phantom, "\lrcorner_{\text{\tiny h}}", very near start] & {\bP^1} \arrow[d,"0"] \\
{\bP^1} \arrow[r, "0"]           & {\bP^2}       .
\end{tikzcd}
\end{equation}
Note that $(\widetilde{\bP^1})_\cl \cong \bP^1$.

Since there is a natural open embedding $\Tot_{\bP^1}(\cO_{\bP^1}(1)) := \Spec_{\bP^1}\Sym_{\bP^1}(\cO_{\bP^1}(-1)) \hookrightarrow \bP^2$, $\widetilde{\bP^1}$ also fits into the homotopy Cartesian diagram
\begin{equation}\label{diag:2-tilde-P^1}
\begin{tikzcd}
{\widetilde{\bP^1}} \arrow[r] \arrow[d] \ar[dr, phantom, "\lrcorner_{\text{\tiny h}}", very near start] & {\bP^1} \arrow[d,"0"] \\
{\bP^1} \arrow[r, "0"]           & {\Tot_{\bP^1}(\cO_{\bP^1}(1))}       
\end{tikzcd}
\end{equation}
by \cite[Corollary~25.3.6.6]{LurieSAG}\footnote{Also known as Hurewicz theorem.}. Thus, we see that $\widetilde{\bP^1} \cong \Tot_{\bP^1}(\cO_{\bP^1}(1))[-1]$, and hence
$$
\cO_{\widetilde{\bP^1}} \simeq \left[ \cO_{\bP^1}(-1)\langle \nu \rangle \overset{0}{\longrightarrow} \cO_{\bP^1} \right]
$$
with cohomological degree $|\nu|=-1$. We immediately see that $H^{-1}\left(\Gamma\left(\widetilde{\bP^1}, \cO_{\widetilde{\bP^1}} \right) \right)=0$. In other words, there is no non-zero $(-1)$-shifted function on $\widetilde{\bP^1}$.

However, we can still construct a non-trivial $(-1)$-shifted log function on $\widetilde{\bP^1}$.

We use $(\bP^1, \pt \hookrightarrow \bP^1)$ to denote the (derived) log scheme whose log structure is associated with the (virtual) Cartier divisor $\pt \hookrightarrow \bP^1$, and $(\bP^2, \bP^1 \hookrightarrow \bP^2)$ to denote the (derived) log scheme whose log structure is associated with the (virtual) Cartier divisor $\bP^1 \hookrightarrow \bP^2$.\footnote{There are many ways to equip $\bP^1$ and $\bP^2$ with log structures, but the log structures considered here are chosen so that $0: \bP^1 \to \bP^2$ in Diagram (\ref{diag:1-tilde-P^1}) becomes a strict morphism. More explicitly, the desired strict morphism is $$0: (\bP^1_{x,y} , \{x=0\}=\pt \to \bP^1_{x,y} ) \to (\bP^2_{x,y,z} , \{x=0\} = \bP^1_{y,z} \to \bP^2_{x,y,z}).$$} By Diagram (\ref{diag:1-tilde-P^1}), we obtain a derived log scheme $(\widetilde{\bP^1}, \pt \times^h_{\bP^1} \pt \hookrightarrow \widetilde{\bP^1})$ whose log structure is associated with the virtual Cartier divisor $\pt \times^h_{\bP^1} \pt \hookrightarrow \widetilde{\bP^1}$. By Diagram (\ref{diag:2-tilde-P^1}), we obtain a derived log scheme $(\widetilde{\bP^1}, \pt \times^h_{\bA^1} \pt \hookrightarrow \widetilde{\bP^1})$ whose log structure is associated with the virtual Cartier divisor $\pt \times^h_{\bA^1} \pt \hookrightarrow \widetilde{\bP^1}$. Since $\bA^1 \hookrightarrow \bP^1$ is an open embedding, we see that $\pt \times^h_{\bP^1} \pt \cong \pt \times^h_{\bA^1} \pt \cong \bA^1[-1]$ by \cite[Corollary~25.3.6.6]{LurieSAG}. Thus, we conclude that $(\widetilde{\bP^1}, \pt \times^h_{\bP^1} \pt \hookrightarrow \widetilde{\bP^1})$ and $(\widetilde{\bP^1}, \pt \times^h_{\bA^1} \pt \hookrightarrow \widetilde{\bP^1})$ are equivalent derived log schemes.

For simplicity, let $\D_{\widetilde{\bP^1}}$ denote $\bA^1[-1]$, and $(\widetilde{\bP^1}, \D_{\widetilde{\bP^1}} \hookrightarrow \widetilde{\bP^1})$ denote the derived log scheme whose log structure is associated with the virtual Cartier divisor $\D_{\widetilde{\bP^1}} \hookrightarrow \widetilde{\bP^1}$.

By \cite{KhRy25}, there is a generalized Cartier divisor $t_{\text{\tiny $\infty$}}: \cO_{\widetilde{\bP^1}}(-\D_{\widetilde{\bP^1}}) \to \cO_{\widetilde{\bP^1}}$ associated with the virtual Cartier divisor $\D_{\widetilde{\bP^1}} \hookrightarrow \widetilde{\bP^1}$. Hence, we can conclude that $H^{-1}\left( \Gamma\left(\widetilde{\bP^1}, \cO_{\widetilde{\bP^1}}(\D_{\widetilde{\bP^1}}) \right) \right) =\bC \langle \frac{\nu}{t_{\text{\tiny $\infty$}}} \rangle$. Thus, there is a $(-1)$-shifted log function
$$
1 \cdot \frac{\nu}{t_{\text{\tiny $\infty$}}} : \cO_{\widetilde{\bP^1}} \to \cO_{\widetilde{\bP^1}}(\D_{\widetilde{\bP^1}})[-1].
$$
Moreover, we see that the classical degeneracy locus of the log cosection induced by the above $(-1)$-shifted log function is empty (see Example \ref{eg:classical-truncation-affine-derived-log-critical-locus}). Therefore, we have $[(\widetilde{\bP^1})_\cl]^\red = 0$ by Remark \ref{rmk:degeneracy-loci-coincide} and Corollary \ref{thm:red-cosection-localized-virt-cycle}. 

$[(\widetilde{\bP^1})_\cl]^\vir = [\pt]$ follows from Diagram (\ref{diag:2-tilde-P^1}), Example \ref{eg:local-model-POT}, and self-intersection formula \cite[Corollary~6.3]{Fulton98}. Then by Theorem \ref{thm:decompose-virt-cycle}, we have
\begin{align}\label{align:tilde-P^1}
[\pt]=[(\widetilde{\bP^1})_\cl]^\vir = [(\widetilde{\bP^1})_\cl]^\red + i_*[(\D_{\widetilde{\bP^1}})_\cl]^\red = i_*[(\D_{\widetilde{\bP^1}})_\cl]^\red,
\end{align}
where $i: \pt \hookrightarrow \bP^1$ is the natural closed embedding. Note that $(\D_{\widetilde{\bP^1}})_\cl = \pt$.
    
    Equation (\ref{align:tilde-P^1}) is exactly what one would expect in intersection theory, because the intersection product $[\bP^1 \cdot_{\bP^2} \bP^1]$ given by Diagram (\ref{diag:1-tilde-P^1}) is precisely $[(\widetilde{\bP^1})_\cl]^\vir = [\pt]$, which essentially comes from the boundary contribution $[(\D_{\widetilde{\bP^1}})_\cl]^\red$.
    $\blacksquare$
\end{example}

\begin{re}
    By Diagram (\ref{diag:2-tilde-P^1}) and \cite[Example~4.35]{Jiang25}, there is an equivalence of derived schemes $$\widetilde{\bP^1} \cong \Proj_{\bP^1} \Sym_{\bP^1} \left( \left[ \cO_{\bP^1}(-1) \overset{0}{\longrightarrow} \cO_{\bP^1} \right] \right).$$ Recall that in Example \ref{eg:der-self-intersect-P1-in-P2}, we also have an equivalence of derived schemes $$\widetilde{\bP^1} \cong \Spec_{\bP^1}\Sym_{\bP^1}(\cO_{\bP^1}(-1)[1]).$$
    Thus, $\widetilde{\bP^1}$ is somewhat an interesting derived scheme, since it is both a derived vector bundle and a derived projective vector bundle over $\bP^1$.
\end{re}

\section{A Prototype for Derived Logarithmic GLSM}\label{sec:4}

In this section, we construct log compactifications of the derived Hodge bundles and the derived moduli of stable maps with $p$-fields (or, for short, the derived moduli of $p$-fields). We then construct $(-1)$-shifted log functions on these log compactifications, which induce log cosections.

The construction given in Section \ref{sec:log-compact-der-moduli-p-fields} can be made more general. For the sake of clarity, we are content with this level of generality.

The starting point of this section is to recover the work of \cite{CJR21} in a derived geometry setting. However, we do not have a clear approach to constructing a quasi-smooth derived enhancement of the Log GLSM moduli stack in \cite{CJR21}. Nevertheless, thanks to Corollary \ref{cor:log-compact-derived-weil-restriction}, we can log compactify the derived moduli of $p$-fields in a relatively simple way, and hence the title of this section. In addition to Log GLSM \cite{CJR21}, this section is influenced by \cite{ChLi12, CJW21, Oh22, OhTh24, KMMP25}.

\begin{re}\label{rmk:sect-4-from-sect-3}
    Thanks to Proposition \ref{prop:rel-qsmth-rel-POT} and Proposition \ref{prop:rel-cosection-rel-POT}, everything in section  \ref{sec:3} immediately generalizes to the relative setting in this section by replacing $\cA$ with $\cA \times S$, where $S$ is a smooth classical $1$-Artin stack with separated diagonal (e.g.\ the smooth classical $1$-Artin stack $\fMgn$ of prestable curves, see the proof of \cite[Lemma~A.1]{GrSi13}).
\end{re}

In this section, for simplicity, we set $\bG_{m, B} := \bG_{m} \times^h B$, where $B$ is a derived stack.

\subsection{A Logarithmic Compactification of the Derived Hodge Bundles}\label{sec:log-compact-der-hodge-bund}

We begin by constructing a $(-1)$-shifted function on the derived Hodge bundle, and hence a cosection on the Hodge bundle. We then construct a $(-1)$-shifted log function on a log compactified derived Hodge bundle (the derived projective completion of the derived Hodge bundle), and hence a log cosection on the classical projective completion of the Hodge bundle.

\subsubsection{Basic Setup}
Consider the universal curve $\bar{\pi}: \bfCgn \to \bfMgn$ over the smooth classical Deligne--Mumford stack $\bfMgn$ of stable curves of genus $g$ with $n$ marked points. Let $\w$ denote the relative dualizing sheaf of $\bar{\pi}: \bfCgn \to \bfMgn$.

\subsubsection{Derived Hodge Bundles}

The vector bundle 
$$\HB := \Spec_{\bfMgn} \Sym_{\bfMgn} \left((\bar{\pi}_*\w)^\vee\right)$$ 
over $\bfMgn$ is called the \emph{Hodge bundle}. Following this construction, we can construct a derived vector bundle 
$$\RHB := \Spec_{\bfMgn} \Sym_{\bfMgn} \left((\RR\bar{\pi}_*\w)^\vee\right)$$
over $\bfMgn$ such that $(\RHB)_\cl \cong \HB$. We call $\RHB$ the \emph{derived Hodge bundle} (see {\cite[Example~1.2.4]{KMMP25}}).

\begin{re}
    Since $(\RHB)_\cl \cong \HB$ is a classical Deligne--Mumford stack, $\RHB$ is a derived Deligne--Mumford stack.
\end{re}

\begin{re}
    There is a quasi-isomorphism $\bL_{\RHB/\bfMgn} \simeq \left((\RR\bar{\pi}_*\w)^\vee\right) \big|_{\RHB}$. Recall that $\RR\bar{\pi}_*\w$ is a perfect complex of Tor-amplitude in $[0,1]$. Hence, $\RHB$ is quasi-smooth over $\bfMgn$. Since $\bfMgn$ is smooth, we see that $\RHB$ is quasi-smooth.
\end{re}

\begin{re}
\label{rmk:der-hodge-bund-as-weil-res}

    By Proposition \ref{prop:weil-res-vect-bundle} and Proposition \ref{prop:weil-res-base-change}, there is an equivalence of derived stacks
    $$
    \RHB \cong \RW_{\bfMgn}\left( \Spec_{\bfCgn} \Sym_{\bfCgn} (\w^\vee) \ \big/ \ \bfCgn \right).
    $$
\end{re}

Recall that $H^1(\RR\bar{\pi}_*\w) = \cO_{\bfMgn}$. Hence, there is a natural morphism
$$
\kappa: \RR\bar{\pi}_*\w \longrightarrow \cO_{\bfMgn}[-1],
$$
which induces a relative $(-1)$-shifted function
$$
\kappa: \RHB \longrightarrow \bA^1_{\bfMgn}[-1] := \Spec_{\bfMgn} \Sym_{\bfMgn} \left((\cO_{\bfMgn}[-1])^\vee \right)
$$
on $\RHB$ over $\bfMgn$.

\begin{re}\label{rmk:morph-of-der-vect-bund-over-bfMgn}
    Note that $\kappa: \RHB \longrightarrow \bA^1_{\bfMgn}[-1]$ is a morphism of derived vector bundles over $\bfMgn$.
\end{re}

Applying the relative tangent complex functor, we obtain a morphism
$$
\bT \kappa \simeq \kappa\big|_{\RHB}: \bT_{\RHB/\bfMgn} \simeq (\RR\bar{\pi}_*\w) \big|_{\RHB} \longrightarrow \cO_{\RHB}[-1].
$$
Pulling back to $\HB$ and then taking $H^1$, we obtain a (relative) cosection
$$
\sigma_{\HB/\bfMgn}:= H^1(\kappa\big|_{\HB}): H^1(\bT_{\RHB/\bfMgn}\big|_{\HB}) = H^1\left((\RR\bar{\pi}_*\w) \big|_{\HB}\right) = \cO_{\HB} \overset{=}{\longrightarrow} \cO_{\HB}
$$
on $\HB$. Note that $\sigma_{\HB/\bfMgn}$ is surjective everywhere on $\HB$ by construction, so its degeneracy locus is empty.

\subsubsection{A Log Compactification of Derived Hodge Bundles}\label{ss-sec:log-comp-der-hodg-bund}

Thanks to Proposition \ref{prop:log-compact-derived-proj-completion} (or by Remark \ref{rmk:der-hodge-bund-as-weil-res} and Corollary \ref{cor:log-compact-derived-weil-restriction}), the derived Hodge bundle $\RHB$ admits a log compactification 
$$
\bP\left( \RHB \oplus 1 \right) := \Proj_{\bfMgn} \Sym_{\bfMgn} \left((\RR\bar{\pi}_*\w)^\vee \oplus \cO_{\bfMgn} \right),
$$
which is the derived projective completion of $\RHB$ (see Definition \ref{def:der-proj-complete}). 
By {\cite[Proposition~A.6.(v)]{AHPS23}}, the classical truncation $\left(\bP\left( \RHB \oplus 1 \right)\right)_\cl$ of $\bP\left( \RHB \oplus 1 \right)$ is $$\bP\left( \HB \oplus 1 \right) := \Proj_{\bfMgn} \Sym_{\bfMgn} \left((\bar{\pi}_*\w)^\vee \oplus \cO_{\bfMgn} \right),$$ which is the classical projective completion of $\HB$.

Note that the relative $(-1)$-shifted function $$\kappa: \RHB \to \bA^1_{\bfMgn}[-1]$$ over $\bfMgn$ is $\bG_{m}$-equivariant by Remark \ref{rmk:morph-of-der-vect-bund-over-bfMgn} and \cite{Monier21}. Hence, there is a $\bG_{m}$-equivariant morphism
$$
\RHB \times^h_{\bfMgn} \bA^1_{\bfMgn} \longrightarrow \bA^1_{\bfMgn}[-1] \times^h_{\bfMgn} \bA^1_{\bfMgn}
$$
over $\bfMgn$, which induces a morphism
$$
\left[ \RHB \times^h_{\bfMgn} \bA^1_{\bfMgn} \Big/ \bG_{m, \bfMgn} \right] \longrightarrow \left[ \bA^1_{\bfMgn}[-1] \times^h_{\bfMgn} \bA^1_{\bfMgn} \Big/ \bG_{m, \bfMgn} \right] \cong L_\cA[-1] \times^h \bfMgn.
$$
over $\cA \times^h \bfMgn \cong \cA \times \bfMgn$.

\begin{re}
    There is an equivalence of derived stacks $\left[ \bA^1[-1] \times^h \bA^1 \big/ \bG_{m} \right] \cong L_\cA[-1]$.
\end{re}

By Remark \ref{rmk:proj-sym-being-empty}, there is a natural open embedding
$$
\bP\left( \RHB \oplus 1 \right) \longhookrightarrow \left[ \RHB \times^h_{\bfMgn} \bA^1_{\bfMgn} \Big/ \bG_{m, \bfMgn} \right].
$$
Thus, we obtain a relative $(-1)$-shifted log function
$$
\kappa^{\log}: \bP\left( \RHB \oplus 1 \right) \longhookrightarrow \left[ \RHB \times^h_{\bfMgn} \bA^1_{\bfMgn} \Big/ \bG_{m, \bfMgn} \right] \longrightarrow L_\cA[-1] \times^h \bfMgn
$$
on $\bP\left( \RHB \oplus 1 \right)$ over $\bfMgn$. Note that $\kappa^{\log}|_{\RHB} \cong \kappa: \RHB \longrightarrow \bA^1_{\bfMgn}[-1]$ is relative $(-1)$-shifted function over $\bfMgn$ by construction.

\begin{re}\label{rmk:generic-fib-shift-univ-line-bund}
    There is a homotopy Cartesian diagram
$$
\begin{tikzcd}
{\bA^1[-1]} \arrow[r, hook] \arrow[d] \ar[dr, phantom, "\lrcorner_{\text{\tiny h}}", very near start] & {L_\cA[-1]} \arrow[d] \\
{\pt \cong [\bG_m/\bG_m]} \arrow[r, hook]           & {\cA} ,  
\end{tikzcd}
$$
where $\pt \hookrightarrow \cA$ is an open embedding.
\end{re}

Following Section \ref{sec:log-cosection}, $\kappa^{\log}$ induces a (relative) log cosection
$$
\sigma^\log_{\bP\left( \HB \oplus 1 \right)/\bfMgn}:= H^1\left(\bT^{\bfMgn-\log}\kappa^{\log}\big|_{\bP\left( \HB \oplus 1 \right)}\right) : H^1\left(\bT^{\bfMgn-\log}_{\bP\left( \RHB \oplus 1 \right)}\big|_{\bP\left( \HB \oplus 1 \right)}\right) \to  \cO_{\bP\left( \HB \oplus 1 \right)}\Big(\bP\left( \HB \right)\Big)
$$
on $\bP\left( \HB \oplus 1 \right)$ over $\bfMgn$. 

Note that $\sigma^\log_{\bP\left( \HB \oplus 1 \right)/\bfMgn}$ is surjective along $\bP\left( \HB \right)$ by construction. Moreover, we have $$\sigma^\log_{\bP\left( \HB \oplus 1 \right)/\bfMgn}\big|_{ \HB } = \sigma_{\HB/\bfMgn},$$ because $\kappa^{\log}|_{\RHB} \cong \kappa$. This implies that $\sigma^\log_{\bP\left( \HB \oplus 1 \right)/\bfMgn}$ is surjective along $\HB$. Hence, we conclude that $\sigma^\log_{\bP\left( \HB \oplus 1 \right)/\bfMgn}$ is surjective everywhere on $\bP\left( \HB \oplus 1 \right)$.

As a result, by Corollary \ref{thm:red-cosection-localized-virt-cycle} and Remark \ref{rmk:degeneracy-loci-coincide}, we see that $[\bP\left( \HB \oplus 1 \right)]^\red = 0$; by Theorem \ref{thm:decompose-virt-cycle}, we have $[\bP\left( \HB \oplus 1 \right)]^\vir = i_* [\bP\left( \HB\right)]^\red$, where $i: \bP\left( \HB\right) \hookrightarrow \bP\left( \HB \oplus 1 \right)$ is the natural closed embedding.

\begin{re}
    By considering the sequence of morphisms
$$
\bP\left( \RHB \oplus 1 \right) \to \cA \times^h \bfMgn \cong \cA \times \bfMgn \to \bfMgn,
$$
    we obtain a (co)fiber sequence
$$
\bL_{\cA}|_{\bP\left( \RHB \oplus 1 \right)} \simeq \bL_{\cA \times \bfMgn/\bfMgn}|_{\bP\left( \RHB \oplus 1 \right)} \longrightarrow \bL_{\bP\left( \RHB \oplus 1 \right)/\bfMgn} \longrightarrow \bL^{\bfMgn-\log}_{\bP\left( \RHB \oplus 1 \right)} \overset{+1}{\longrightarrow}.
$$
Thus, $\bL^{\bfMgn-\log}_{\bP\left( \RHB \oplus 1 \right)}$ is somewhat tractable due to Remark \ref{rmk:cotangent-complex-cA} and Theorem \ref{thm:derived-Euler-seq}.
\end{re}

\subsection{A Logarithmic Compactification of the Derived Moduli of Stable Maps with $p$-Fields}\label{sec:log-compact-der-moduli-p-fields}

We begin by constructing a $(-1)$-shifted function on the derived moduli of stable maps with $p$-fields (or, for short, the derived moduli of $p$-fields), and hence a cosection on the classical moduli of stable maps with $p$-fields (or, for short, the classical moduli of $p$-fields). We then construct a $(-1)$-shifted log function on a log compactified derived moduli of $p$-fields (the derived projective completion of the derived moduli of $p$-fields), and hence a log cosection on the classical projective completion of the classical moduli of $p$-fields.

\cite{OhTh24} provides an alternative construction of a (relative) $(-1)$-shifted function on the derived moduli of $p$-fields, and hence a cosection on the classical moduli of $p$-fields. See Remark \ref{rmk:alter-construct-shift-fun} for more discussion. Note that \cite{KiPa25} constructs $(-1)$-shifted locked $1$-forms on a derived enhancement of the GLSM moduli stack (see \cite[Lemma~7.8]{KiPa25}).

Note that the classical truncation of the log compactified derived moduli of $p$-fields given here does not agree with the log GLSM moduli stack in \cite[Section~4.1]{CJR21}, as mentioned in Remark \ref{rmk:class-trunc-is-not-log-glsm-moduli}.

\subsubsection{Basic Setup}

Consider the universal curve $\pi: \fCgn \to \fMgn$ over the smooth classical $1$-Artin stack $\fMgn$ of pre-stable curves of genus $g$ with $n$ marked points. Let $\W$ denote the relative dualizing sheaf of $\pi: \fCgn \to \fMgn$.

Let $\cE$ be a locally free sheaf on the $r$-dimensional projective space $\bP^r$. $E := \V_{\bP^r} (\cE^\vee) = \Spec_{\bP^r} \Sym_{\bP^r} (\cE^\vee)$ is a vector bundle over $\bP^r$. Assume that there exists a section $s: \cO_{\bP^r} \to \cE$ whose classical zero locus $Z(s)_\cl$ is a smooth complete intersection subvariety of $\bP^r$ of codimension $\rk\ \cE$.

Let $\cZ := \V_{\fCgn \times \bP^r} \left( \left( \W \boxtimes \cE^\vee \right)^\vee \right) \cong \Spec_{\fCgn \times \bP^r} \Sym_{\fCgn \times \bP^r} \left( \left( \W \boxtimes \cE^\vee \right)^\vee \right)$ denote a vector bundle over $\fCgn \times \bP^r$.

\subsubsection{Derived Moduli of $p$-Fields}

Consider the sequence of morphisms $\cZ \longrightarrow \fCgn \times \bP^r \longrightarrow \fCgn \overset{\pi}{\longrightarrow} \fMgn$. 

By Example \ref{eg:der-map-stk-is-weil-res}, there is an equivalence of derived stacks
$$
R\fW := \RH_{\fMgn}\left( \fCgn , \fMgn \times \bP^r \right) \cong \RW_{\fMgn}\left( \fCgn \times \bP^r /  \fCgn \right).
$$
The universal family of $\RH_{\fMgn}\left( \fCgn , \fMgn \times \bP^r \right)$ is the commutative diagram as follows.
$$
\begin{tikzcd}
{\cC_{R\fW} := \RH_{\fMgn}\left( \fCgn, \fMgn \times \bP^r  \right) \times^h_{\fMgn} \fCgn} \ar[d, "\widetilde{\pi}"] \ar[r, "\ev"] &{\fMgn \times \bP^r}\\
{\RH_{\fMgn}\left( \fCgn, \fMgn \times \bP^r  \right)}
\end{tikzcd}
$$
We define the morphism $f: \cC_{R\fW} \to \bP^r$ to be the composition $\cC_{R\fW} \overset{\ev}{\longrightarrow} \fMgn \times \bP^r \longrightarrow \bP^r$.

By Proposition \ref{prop:two-weil-res-equiv}, there is an equivalence of derived stacks
$$
\RW_{R\fW} \left( \cZ \times^h_{(\fCgn \times \bP^r)} \cC_{R\fW} \ \big/ \ \cC_{R\fW} \right) \cong \RW_{\fMgn}\left(\cZ/\fCgn\right).
$$
Note that the morphism $\cC_{R\fW} \to \fCgn \times \bP^r$ is the universal evaluation map of $\RW_{\fMgn}\left( \fCgn \times \bP^r /  \fCgn \right)$, which can be induced by the universal evaluation map $\ev: \cC_{R\fW} \to \fMgn \times \bP^r$ of $\RH_{\fMgn}\left( \fCgn, \fMgn \times \bP^r  \right)$, due to the following commutative diagram.
$$
\begin{tikzcd}
{\cC_{R\fW}} \ar[rrd, "\ev"] \ar[rdd] \ar[rd, dashed] \\
{} &{\fCgn \times \bP^r} \arrow[r] \arrow[d] \ar[dr, phantom, "\lrcorner", very near start] & {\fMgn \times \bP^r} \arrow[d] \\
{} &{\fCgn} \arrow[r, "\pi"]           & {\fMgn}       
\end{tikzcd}
$$
Moreover, we have equivalences of derived stacks
\begin{align*}
    \cZ \times^h_{(\fCgn \times \bP^r)} \cC_{R\fW} &\cong \V_{\cC_{R\fW}} \left( \left( \W \boxtimes \cE^\vee \right)^\vee \big|_{\cC_{R\fW}} \right) \\
    &\cong \V_{\cC_{R\fW}} \left( \left( \omega_{\cC_{R\fW}/R\fW} \otimes^{\LL}_{\cO_{\cC_{R\fW}}} \LL f^*\cE^\vee \right)^\vee \right)
\end{align*}
As a result, we have equivalences of derived stacks
\begin{align*}
    \RW_{\fMgn}\left(\cZ/\fCgn\right) &\cong \RW_{R\fW} \left( \V_{\cC_{R\fW}} \left( \left( \omega_{\cC_{R\fW}/R\fW} \otimes^{\LL}_{\cO_{\cC_{R\fW}}} \LL f^*\cE^\vee \right)^\vee \right) \Big/\cC_{R\fW} \right) \\
    &\cong \Spec_{R\fW} \Sym_{R\fW} \left( \left(\RR \widetilde{\pi}_* \left( \omega_{\cC_{R\fW}/R\fW} \otimes^{\LL}_{\cO_{\cC_{R\fW}}} \LL f^*\cE^\vee \right)\right)^\vee \right) \\
    &\cong \V_{R\fW} \left( \left(\RR \widetilde{\pi}_* \left( \omega_{\cC_{R\fW}/R\fW} \otimes^{\LL}_{\cO_{\cC_{R\fW}}} \LL f^*\cE^\vee \right)\right)^\vee \right).
\end{align*}
The second equivalence follows from Corollary \ref{cor:log-compact-derived-weil-restriction}.

By Example \ref{eg:Gromov-Witten}, there is a natural open embedding 
$$
R\cW :=\RMP \hookrightarrow R\fW := \RH_{\fMgn}\left( \fCgn, \fMgn \times \bP^r  \right),
$$
where $\RMP$ is the quasi-smooth derived Deligne--Mumford stacks of stable maps to $\bP^r$. Note that we have a commutative diagram as follows.
$$
\begin{tikzcd}
{\cC_{R\cW} := R\cW \times^h_{\fMgn}\fCgn}  \ar[dr, phantom, "\lrcorner_{\text{\tiny h}}", very near start] \arrow[rrr, shift left=1, bend left=15, "f^\st"'] \arrow[r, hook] \arrow[d, "\widetilde{\pi}^\st"] &{\cC_{R\fW}} \arrow[rr, shift right=1, bend right=15, "f"'] \arrow[r, "\ev"] \arrow[d, "\widetilde{\pi}"] & { \fMgn \times \bP^r} \arrow[r] & { \bP^r } \\
{R\cW} \arrow[r, hook] &{R\fW}        
\end{tikzcd}
$$

We then define the \emph{derived moduli $\RMP^E$ of stable maps to $\bP^r$ with $p$-fields} (or \emph{derived moduli $R\cW^E$ of $p$-fields} for simplicity) to be $ R\cW \times^h_{R\fW} \RW_{\fMgn}\left(\cZ/\fCgn\right)$, which is equivalent to
\begin{align*}
    &R\cW \times^h_{R\fW} \Spec_{R\fW} \Sym_{R\fW} \left( \left(\RR \widetilde{\pi}_* \left( \omega_{\cC_{R\fW}/R\fW} \otimes^{\LL}_{\cO_{\cC_{R\fW}}} \LL f^*\cE^\vee \right)\right)^\vee \right) \\
    \cong &\Spec_{R\cW} \Sym_{R\cW} \left( \left(\RR (\widetilde{\pi}^{\st})_* \left( \omega_{\cC_{R\cW}/R\cW} \otimes^{\LL}_{\cO_{\cC_{R\cW}}} \LL (f^{\st})^*\cE^\vee \right)\right)^\vee \right) \\
    \cong &\V_{R\cW} \left( \left(\RR (\widetilde{\pi}^{\st})_* \left( \omega_{\cC_{R\cW}/R\cW} \otimes^{\LL}_{\cO_{\cC_{R\cW}}} \LL (f^{\st})^*\cE^\vee \right)\right)^\vee \right).
\end{align*}
    By construction, there is a natural open embedding 
$$
R\cW^E \hookrightarrow \RW_{\fMgn}\left(\cZ/\fCgn\right).
$$

\begin{re}\label{re:alternative-der-moduli-p-fields}
By Serre duality, there is an equivalence of derived stacks
$$
R\cW^E \cong \V_{R\cW} \left( \left(\RR (\widetilde{\pi}^{\st})_* \left( \omega_{\cC_{R\cW}/R\cW} \otimes^{\LL}_{\cO_{\cC_{R\cW}}} \LL (f^{\st})^*\cE^\vee \right)\right)^\vee \right) \cong \V_{R\cW} \left( \left(\RR (\widetilde{\pi}^{\st})_*  \LL (f^{\st})^*\cE \right)[1] \right).
$$
Thus, our construction of derived moduli of $p$-fields agrees with \cite[Example~1.2.5]{KMMP25}.
\end{re}

\begin{re}
By Remark \ref{rmk:class-truncation-Weil-Restriction} and Example \ref{eg:Gromov-Witten}, the classical truncation $(R\cW^E)_\cl$ of $R\cW^E$ is
\begin{align*}
    \left(R\cW \times^h_{R\fW} \RW_{\fMgn}\left(\cZ/\fCgn\right)\right)_\cl &\cong \left(R\cW\right)_\cl \times_{\left(R\fW\right)_\cl} \left(\RW_{\fMgn}\left(\cZ/\fCgn\right)\right)_\cl \\
    &\cong \MP \times_{\Hom_{\fMgn}\left( \fCgn, \fMgn \times \bP^r  \right)} \Weil_{\fMgn}\left(\cZ/\fCgn\right),
\end{align*}
which agrees with the classical Deligne--Mumford stack of $p$-fields defined in {\cite[Section~1.1]{CJW21}}.  Since $(R\cW^E)_\cl$ is a classical Deligne--Mumford stack, we see that $R\cW^E$ is a derived Deligne--Mumford stack.
\end{re}

\begin{re}
    Since $\widetilde{\pi}^{\st}$ is proper, quasi-smooth, and has finite cohomological dimension at most $1$ (see Example \ref{eg:Gromov-Witten}), we see that
$$
\bL_{R\cW^E/R\cW} \simeq \left(\RR (\widetilde{\pi}^{\st})_* \left( \omega_{\cC_{R\cW}/R\cW} \otimes^{\LL}_{\cO_{\cC_{R\cW}}} \LL (f^{\st})^*\cE^\vee \right)\right)^\vee \Big|_{R\cW^E}
$$
is a perfect complex of Tor-amplitude in $[-1,0]$ by similar arguments in Example \ref{eg:Gromov-Witten}. Thus, $R\cW^E$ is quasi-smooth over $R\cW$. Since $R\cW$ is quasi-smooth by Example \ref{eg:Gromov-Witten}, we conclude that $R\cW^E$ is quasi-smooth.
\end{re}

\begin{re}
    By considering the sequence of morphisms
$$
R\cW^E \to R\cW \to \fMgn,
$$
    we obtain a (co)fiber sequence
$$
\bL_{R\cW/\fMgn}|_{R\cW^E} \longrightarrow \bL_{R\cW^E/\fMgn} \longrightarrow \bL_{R\cW^E/R\cW} \overset{+1}{\longrightarrow}.
$$
As a result, $\bL_{R\cW^E/\fMgn}$ is somewhat tractable due to Example \ref{eg:Gromov-Witten} (or Proposition \ref{prop:weil-res-tangent-complex}).
\end{re}

\subsubsection{Construction of $(-1)$-Shifted Functions}\label{ss-sec:shift-function-on-der-moduli-p-fields}

There is a natural morphism 
$$
\W \boxtimes \cE^\vee \overset{1 \boxtimes s^{\vee}}{\longrightarrow} \W \boxtimes \cO_{\bP^r} \simeq \W \big|_{\fCgn \times \bP^r}
$$
induced by the section $s: \cO_{\bP^r} \to \cE$. Hence, we obtain a morphism
$$
\cZ \cong \V_{\fCgn \times \bP^r} \left( \left( \W \boxtimes \cE^\vee \right)^\vee \right) \overset{\V_{\fCgn \times \bP^r}\left(1 \boxtimes s^{\vee}\right)}{\xrightarrow{\hspace{2cm}}} \V_{\fCgn \times \bP^r} \left( \left( \W \big|_{\fCgn \times \bP^r} \right)^\vee \right) \longrightarrow \V_{\fCgn} \left( \W^\vee \right),
$$
where the second morphism is the projection map. Applying the derived Weil restriction functor, we then obtain a morphism
$$
\RW_{\fMgn}\left(\cZ/\fCgn\right) \overset{\natural}{\longrightarrow} \RW_{\fMgn}\left(\V_{\fCgn \times \bP^r} \left( \left( \W \big|_{\fCgn \times \bP^r} \right)^\vee \right) \Big/ \fCgn\right) \longrightarrow \RW_{\fMgn}\left(\V_{\fCgn} \left( \W^\vee \right) \Big/ \fCgn\right)
$$
over $\fMgn$.

\begin{re}\label{rmk:morph-of-der-vect-bund-over-RfW}
    Note that $\natural:\RW_{\fMgn}\left(\cZ/\fCgn\right) \to \RHBPM$ is a morphism of derived vector bundles over $R\fW$.
\end{re}

\begin{re}
    There is an equivalence of derived stacks
$$
\RW_{\fMgn}\left(\V_{\fCgn} \left( \W^\vee \right) \Big/ \fCgn\right) \cong \Spec_{\fMgn} \Sym_{\fMgn} \left((\RR \pi_*\W)^\vee\right) =: \RHBP
$$
by Proposition \ref{prop:weil-res-vect-bundle}.
\end{re}

\begin{re}\label{rmk:der-Hodge-bund-over-moduli-of-maps}
    We have equivalences of derived stacks
\begin{align*}
    &\RW_{\fMgn}\left(\V_{\fCgn \times \bP^r} \left( \left( \W \big|_{\fCgn \times \bP^r} \right)^\vee \right) \Big/ \fCgn\right) \\ \cong &\RW_{R\fW} \left( \V_{\cC_{R\fW}} \left( \left( \W \big|_{\fCgn \times \bP^r} \right)^\vee \big|_{\cC_{R\fW}} \right) \Big/ \cC_{R\fW} \right) \\
    \cong &\RW_{R\fW} \left( \V_{\cC_{R\fW}} \left( \omega_{\cC_{R\fW}/R\fW}^\vee \right) \Big/ \cC_{R\fW} \right) \\
    \cong &\Spec_{R\fW} \Sym_{R\fW} \left((\RR \widetilde{\pi}_*\omega_{\cC_{R\fW}/R\fW})^\vee\right) =: \RHBPM \\
    \cong &R\fW \times^h_{\fMgn} \Spec_{\fMgn} \Sym_{\fMgn} \left((\RR \pi_*\W)^\vee\right) \\
    \cong &R\fW \times^h_{\fMgn} \RHBP.
\end{align*}
The first equivalence is due to Proposition \ref{prop:two-weil-res-equiv}; the third equivalence is due to Corollary \ref{cor:log-compact-derived-weil-restriction}; the fourth equivalence is due to the base change formula.
\end{re}

Following the construction of the relative $(-1)$-shifted function
$\kappa: \RHB \to \bA^1_{\bfMgn}[-1]$ on $\RHB$ over $\bfMgn$ given in Section \ref{sec:log-compact-der-hodge-bund}, we can construct a relative $(-1)$-shifted function
$$
\varkappa: \RHBP \longrightarrow \bA^1_{\fMgn}[-1] := \Spec_{\fMgn} \Sym_{\fMgn} \left((\cO_{\fMgn}[-1])^\vee \right)
$$
on $\RHBP$ over $\fMgn$.

\begin{re}\label{rmk:morph-of-der-vect-bund-over-fMgn}
    Note that $\varkappa: \RHBP \longrightarrow \bA^1_{\fMgn}[-1]$ is a morphism of derived vector bundles over $\fMgn$.
\end{re}

Thus, there is a relative $(-1)$-shifted function
$$
\RW_{\fMgn}\left(\cZ/\fCgn\right) \overset{\natural}{\longrightarrow} \RHBPM \cong R\fW \times^h_{\fMgn} \RHBP \longrightarrow \RHBP \overset{\varkappa}{\longrightarrow} \bA^1_{\fMgn}[-1]
$$
on $\RW_{\fMgn}\left(\cZ/\fCgn\right)$ over $\fMgn$.

Moreover, there is a natural open embedding $R\cW^E \hookrightarrow \RW_{\fMgn}\left(\cZ/\fCgn\right)$, so we finally obtain a relative $(-1)$-shifted function
$$
\ell_{/\fMgn}: R\cW^E \longhookrightarrow \RW_{\fMgn}\left(\cZ/\fCgn\right) \longrightarrow \bA^1_{\fMgn}[-1]
$$
on $R\cW^E$ over $\fMgn$.

\begin{re}
    Note that $\ell_{/\fMgn}$ is determined by the section $s: \cO_{\bP^r} \to \cE$.
\end{re}

\begin{re}\label{rmk:rel-shift-function-on-RcWE-over-RcW}
    We can construct a relative $(-1)$-shifted function on $R\cW^E$ over $R\cW$ as follows. There is a morphism
$$
R\cW^E \cong R\cW \times^h_{R\fW} \RW_{\fMgn}\left(\cZ/\fCgn\right) \overset{\natural^\st:= R\cW \times^h_{R\fW}\natural}{\xrightarrow{\hspace{1cm}}} R\cW \times^h_{R\fW} \RHBPM \cong \Spec_{R\cW} \Sym_{R\cW} \left((\RR (\widetilde{\pi}^\st)_*\omega_{\cC_{R\cW}/R\cW})^\vee\right) =: \RHBM.
$$
Moreover, there is a relative $(-1)$-shifted function
$$
\RHBM \cong R\cW \times^h_{\fMgn} \RHBP \overset{\varkappa^\st := R\cW \times^h_{\fMgn}\varkappa}{\xrightarrow{\hspace{1cm}}} R\cW \times^h_{\fMgn} \bA^1_{\fMgn}[-1] \cong \Spec_{R\cW} \Sym_{R\cW} \left((\cO_{R\cW}[-1])^\vee \right) =: \bA^1_{R\cW}[-1]
$$
on $\RHBM$ over $R\cW$, following the construction given in Section \ref{sec:log-compact-der-hodge-bund}. Therefore, we obtain a relative $(-1)$-shifted function
$$
\ell_{/R\cW}: R\cW^E \overset{\natural^\st}{\longrightarrow} \RHBM \overset{\varkappa^\st}{\longrightarrow} \bA^1_{R\cW}[-1]
$$
on $R\cW^E$ over $R\cW$. 
\end{re}

\subsubsection{Cosections and Degeneracy Loci}\label{ss-sec:cosection-degeneracy-locus}

Applying the relative tangent complex functor to $\ell_{/\fMgn}$, we obtain a morphism
$$
\bT_{R\cW^E/\fMgn} \to \bT_{\bA^1_{\fMgn}[-1]/\fMgn}|_{R\cW^E} \simeq \cO_{R\cW^E}[-1].
$$
Pulling back to the classical Deligne--Mumford stack $\MPp := \left( R\cW^E \right)_\cl$ of $p$-fields, we obtain a (relative) cosection
$$
\bT_{R\cW^E/\fMgn}|_{\MPp} \to \cO_{\MPp}[-1]
$$
of the (relative) perfect obstruction theory
$$
\bL_{R\cW^E/\fMgn}|_{\MPp} \to \bL_{\MPp/\fMgn}.
$$
Taking $H^1$, we obtain a (relative) cosection
$$
\sigma_{\MPp/\fMgn}: H^1(\bT_{R\cW^E/\fMgn}|_{\MPp}) \longrightarrow \cO_{\MPp}
$$
of the (relative) obstruction sheaf $\Ob_{\MPp/\fMgn}$. 
This (relative) cosection $\sigma_{\MPp/\fMgn}$ is precisely the cosection $\rho$ constructed in {\cite[Section~4.4]{CJW21}}. Therefore, we conclude that the (set-theoretic) degeneracy locus of $\sigma_{\MPp/\fMgn}$ is the classical Deligne--Mumford stack $\overline{\mathcal{M}}_{g,n}(Z(s)_\cl,\beta)$ of stable maps to $Z(s)_\cl$ by {\cite[Lemma~4.4.2]{CJW21}} and {\cite[Proposition~4.3.1]{CJW21}}.

\begin{re}\label{rmk:alter-construct-shift-fun}
    In \cite[Introduction]{OhTh24}, they provide an alternative construction of a (relative) $(-1)$-shifted function $\tilde{s}$ on $\mathbf{F}$ over a quasi-smooth derived Deligne--Mumford stack $\mathbf{P}$; in \cite[Section~2]{OhTh24}, they prove that the degeneracy locus of the (absolute) cosection $h^1(d\tilde{s})$ on $\mathrm{F}$ is precisely $\mathrm{M}$, assuming $\mathbf{M}$ is quasi-smooth.\footnote{We are using notation in \cite{OhTh24}.}
    In particular, the relative $(-1)$-shifted function $\ell_{/\fMgn}$ on derived moduli $R\cW^E$ of $p$-fields constructed in Section \ref{ss-sec:shift-function-on-der-moduli-p-fields} is equivalent, via Serre duality, to the relative $(-1)$-shifted function on $R\cW^E$ given in \cite[Introduction]{OhTh24}.
    Moreover, by taking $\mathbf{P}:= R\cW$, we have $\mathbf{F} \cong R\cW^E$ (see Remark \ref{re:alternative-der-moduli-p-fields}), $\mathrm{F} = \MPp$, $\mathbf{M} \cong R\overline{\mathcal{M}}_{g,n}(Z(s)_\cl,\beta)$ (see Example \ref{eg:Gromov-Witten}), and $\mathrm{M} = \overline{\mathcal{M}}_{g,n}(Z(s)_\cl,\beta)$.
\end{re}

\begin{re}
    If we drop the assumption that $Z(s)_\cl$ is smooth, then $R\overline{\mathcal{M}}_{g,n}(Z(s)_\cl,\beta)$ may fail to be quasi-smooth (see Example \ref{eg:Gromov-Witten}). As a result, the degeneracy locus of the cosection induced by $s$ may not agree with $\overline{\mathcal{M}}_{g,n}(Z(s)_\cl,\beta)$ by \cite[Section~2]{OhTh24}. A specific case is obtained by taking $Z(s)_\cl$ to be a singular quintic $3$-fold, and considering positive genus stable maps. In this case, the degeneracy locus is not even proper, and in particular is not $\overline{\mathcal{M}}_{g,n}(Z(s)_\cl,\beta)$.
\end{re}

\begin{re}($(-2)$-Shifted Symplectic Structures on $\bfMgn$)
    Let $s : \cO_{\bP^r} \to \cE$ be a section of locally free sheaf $\cE$ on $\bP^r$ such that the classical zero locus $Z(s)_\cl$ is $\pt$. Then the degeneracy locus of the (relative) cosection $\sigma_{\overline{\mathcal{M}}_{g,n}(\mathbb{P}^r,0)^E/\fMgn}$ is $\bfMgn$. Thus, by Remark \ref{rmk:virt-Lag-cycle-cosection-localized-virt-cycle}, there exists a $(-2)$-shifted symplectic derived Deligne--Mumford stack whose classical truncation is $\bfMgn$.
\end{re}

\subsubsection{A Log Compactification of Derived Moduli of $p$-Fields}\label{ss-sec:log-compact-der-moduli-p-fields}

Recall that we have
$$
R\cW^E \cong \V_{R\cW} \left( \left(\RR (\widetilde{\pi}^{\st})_* \left( \omega_{\cC_{R\cW}/R\cW} \otimes^{\LL}_{\cO_{\cC_{R\cW}}} \LL (f^{\st})^*E^\vee \right)\right)^\vee \right) \cong \V_{R\cW} \left( \left(\RR (\widetilde{\pi}^{\st})_*  \LL (f^{\st})^*E \right)[1] \right).
$$
In other words, $R\cW^E$ is a derived vector bundle over $R\cW$.
Then we define $\overline{R\cW^E}^\log$ to be
$$
\Proj_{R\cW} \Sym_{R\cW} \left( \left(\RR \left(\widetilde{\pi}^{\st}\right)_*  \LL (f^{\st})^*E \right)[1] \oplus \cO_{R\cW} \right),
$$
which is a log compactification of $R\cW^E$ by Proposition \ref{prop:log-compact-derived-proj-completion} (or Corollary \ref{cor:log-compact-derived-weil-restriction}).

\begin{re}\label{rmk:class-trunc-is-not-log-glsm-moduli}
    The classical truncation of $\overline{R\cW^E}^\log$ does not agree with the log GLSM moduli stack in \cite[Section~4.1]{CJR21}.
\end{re}

\subsubsection{Construction of $(-1)$-Shifted Log Functions}\label{ss-sec:shift-log-function-on-log-compact-der-moduli-p-fields}

Note that the relative $(-1)$-shifted function
$$
\RW_{\fMgn}\left(\cZ/\fCgn\right) \overset{\natural}{\longrightarrow} \RHBPM \cong R\fW \times^h_{\fMgn} \RHBP \longrightarrow \RHBP \overset{\varkappa}{\longrightarrow} \bA^1_{\fMgn}[-1]
$$
on $\RW_{\fMgn}\left(\cZ/\fCgn\right)$ over $\fMgn$ is $\bG_{m}$-equivariant by Remark \ref{rmk:morph-of-der-vect-bund-over-RfW}, Remark \ref{rmk:morph-of-der-vect-bund-over-fMgn}, and \cite{Monier21}. Hence, there is a $\bG_{m}$-equivariant morphism
$$
\RW_{\fMgn}\left(\cZ/\fCgn\right) \times^h_{R\fW} \bA^1_{R\fW} \longrightarrow \RHBPM \times^h_{R\fW} \bA^1_{R\fW} \longrightarrow \RHBP \times^h_{\fMgn} \bA^1_{\fMgn} \longrightarrow \bA^1_{\fMgn}[-1] \times^h_{\fMgn} \bA^1_{\fMgn}
$$
over $\fMgn$.
Consequently, there is a morphism
\begin{align*}
    &\left[ \RW_{\fMgn}\left(\cZ/\fCgn\right) \times^h_{R\fW} \bA^1_{R\fW} \Big/ \bG_{m, R\fW} \right] \longrightarrow \left[ \RHBPM \times^h_{R\fW} \bA^1_{R\fW} \Big/ \bG_{m, R\fW} \right] \longrightarrow \\
    &\left[ \RHBP \times^h_{\fMgn} \bA^1_{\fMgn} \Big/ \bG_{m, \fMgn} \right] \longrightarrow \left[ \bA^1_{\fMgn}[-1] \times^h_{\fMgn} \bA^1_{\fMgn} \Big/ \bG_{m, \fMgn} \right] \cong L_\cA[-1] \times^h \fMgn
\end{align*}
over $\cA \times^h \fMgn \cong \cA \times \fMgn$.

By Remark \ref{rmk:proj-sym-being-empty} and definition of $R\cW^E$, there is an open embedding
$$
\overline{R\cW^E}^\log \longhookrightarrow \left[ R\cW^E \times^h_{R\cW} \bA^1_{R\cW} \Big/ \bG_{m, R\cW} \right] \longhookrightarrow \left[ \RW_{\fMgn}\left(\cZ/\fCgn\right) \times^h_{R\fW} \bA^1_{R\fW} \Big/ \bG_{m, R\fW} \right].
$$
Thus, we obtain a relative $(-1)$-shifted log function
$$
\ell_{/\fMgn}^\log: \overline{R\cW^E}^\log \hookrightarrow \left[ \RW_{\fMgn}\left(\cZ/\fCgn\right) \times^h_{R\fW} \bA^1_{R\fW} \Big/ \bG_{m, R\fW} \right] \to L_\cA[-1] \times^h \fMgn
$$
on $\overline{R\cW^E}^\log$ over $\fMgn$.

\begin{re}
An alternative construction of the relative $(-1)$-shifted log function $\ell_{/\fMgn}^\log$ on $\overline{R\cW^E}^\log$ over $\fMgn$ is as follows. There is a relative $(-1)$-shifted function
$$
\ell_{/R\cW}: R\cW^E \overset{\natural^\st}{\longrightarrow} \RHBM \overset{\varkappa^\st}{\longrightarrow} \bA^1_{R\cW}[-1]
$$
on $R\cW^E$ over $R\cW$ by Remark \ref{rmk:rel-shift-function-on-RcWE-over-RcW}. Note that $\ell_{/R\cW}$ is $\bG_{m}$-equivariant (see \cite{Monier21}). Hence, there is a $\bG_{m}$-equivariant morphism
$$
R\cW^E \times^h_{R\cW} \bA^1_{R\cW} \longrightarrow \RHBM \times^h_{R\cW} \bA^1_{R\cW} \longrightarrow \bA^1_{R\cW}[-1] \times^h_{R\cW} \bA^1_{R\cW}
$$
over $R\cW$, which induces a morphism
$$
\left[ R\cW^E \times^h_{R\cW} \bA^1_{R\cW} \Big/ \bG_{m, R\cW} \right] \longrightarrow \left[ \RHBM \times^h_{R\cW} \bA^1_{R\cW} \Big/ \bG_{m, R\cW} \right] \longrightarrow \left[ \bA^1_{R\cW}[-1] \times^h_{R\cW} \bA^1_{R\cW} \Big/ \bG_{m, R\cW} \right] \cong L_\cA[-1] \times^h R\cW
$$
over $\cA \times^h R\cW$. Thus, we obtain a relative $(-1)$-shifted log function
$$
\ell_{/R\cW}^\log: \overline{R\cW^E}^\log \hookrightarrow \left[ R\cW^E \times^h_{R\cW} \bA^1_{R\cW} \Big/ \bG_{m, R\cW} \right] \to L_\cA[-1] \times^h R\cW
$$
on $\overline{R\cW^E}^\log$ over $R\cW$. Moreover, we have a commutative diagram
$$
\begin{tikzcd}
{\overline{R\cW^E}^\log} \arrow[r, hook] \arrow[rr, bend left=20, "\ell_{/R\cW}^\log"'] & {\left[ R\cW^E \times^h_{R\cW} \bA^1_{R\cW} \Big/ \bG_{m, R\cW} \right]} \arrow[r] \arrow[rd] & {L_\cA[-1] \times^h R\cW} \arrow[r] \arrow[d] \ar[dr, phantom, "\lrcorner_{\text{\tiny h}}", very near start] & {L_\cA[-1] \times^h \fMgn} \arrow[d] \\
             &                         & {\cA \times^h R\cW} \arrow[r]           & {\cA \times^h \fMgn}          .
\end{tikzcd}
$$
Note that the morphism $\left[ R\cW^E \times^h_{R\cW} \bA^1_{R\cW} \Big/ \bG_{m, R\cW} \right] \to L_\cA[-1] \times^h R\cW \to L_\cA[-1] \times^h \bfMgn$ admits a factorization
$$
\begin{tikzcd}
{\left[ R\cW^E \times^h_{R\cW} \bA^1_{R\cW} \Big/ \bG_{m, R\cW} \right]} \arrow[r, hook] \arrow[d] & {\left[ \RW_{\fMgn}\left(\cZ/\fCgn\right) \times^h_{R\fW} \bA^1_{R\fW} \Big/ \bG_{m, R\fW} \right]} \arrow[d] \\
{L_\cA[-1] \times^h R\cW} \arrow[r]           & {L_\cA[-1] \times^h \fMgn}          
\end{tikzcd}
$$
by the construction given in Remark \ref{rmk:rel-shift-function-on-RcWE-over-RcW}. As a result, we conclude that the morphism
$$
\overline{R\cW^E}^\log \overset{\ell_{/R\cW}^\log}{\longrightarrow} L_\cA[-1] \times^h R\cW \longrightarrow L_\cA[-1] \times^h \fMgn
$$
over $\cA \times^h \fMgn \cong \cA \times \fMgn$ is the relative $(-1)$-shifted log function $\ell_{/\fMgn}^\log$ on $\overline{R\cW^E}^\log$ over $\fMgn$.
\end{re}

\subsubsection{Proof of Theorem \ref{thm:intro-prototype-der-log-GLSM}}\label{ss-sec:prf-prototype-der-log-glsm}

We prove Theorem \ref{thm:intro-prototype-der-log-GLSM} by collecting the results in Section \ref{sec:log-compact-der-moduli-p-fields}.

The construction of $\overline{R\cW^E}^\log$ is given in Section \ref{ss-sec:log-compact-der-moduli-p-fields}, while the relative $(-1)$-shifted log function $\ell_{/\fMgn}^\log$ on $\overline{R\cW^E}^\log$ over $\fMgn$ is constructed in Section \ref{ss-sec:shift-log-function-on-log-compact-der-moduli-p-fields}. 

By the construction of $\ell_{/\fMgn}^\log$, its restriction $\ell_{/\fMgn}^\log\big|_{R\cW^E}$ is the relative $(-1)$-shifted function $\ell_{/\fMgn}$ on $R\cW^E$ over $\fMgn$ constructed in Section \ref{ss-sec:shift-function-on-der-moduli-p-fields}.

By Section \ref{ss-sec:cosection-degeneracy-locus} (in particular Remark \ref{rmk:alter-construct-shift-fun}), we conclude that $\ell_{/\fMgn}$ agrees with the $(-1)$-shifted function constructed \cite[Introduction]{OhTh24}, and induces precisely the (relative) cosection $\rho$ constructed in \cite[Section~4.4]{CJW21}.

\appendix
\section{Reduced Perfect Obstruction Theories in General}\label{append:1}

In Appendix \ref{append:red-POT-in-general}, we develop a general construction of reduced perfect obstruction theories via derived zero loci, inspired by \cite[Definition~4.7]{STV15}, \cite[Construction~6.18]{AKLPR22P}\footnote{Note that we refer to a specific arXiv version, not the latest or published version.} and \cite[Construction~4.2.7]{ParkNotes25}, and provide several examples.

For the sake of clarity, we include basic properties that help compute cotangent complexes in Appendix \ref{append:basic-cotang-complex}, and the computation of the cotangent complexes of derived zero loci in Appendix \ref{append:cotang-complex-der-zero-loci}.

\subsection{Basic Properties of Cotangent Complexes}\label{append:basic-cotang-complex}

\begin{lemma}\label{lem:cotangent-complex-homotopy-fib-prod}
Given a homotopy Cartesian diagram of derived Artin stacks
$$
\begin{tikzcd}
{X \times^h_S Y} \ar[dr, phantom, "\lrcorner_{\text{\tiny h}}", very near start] \arrow[r] \arrow[d] & {Y} \arrow[d] \\
{X} \arrow[r]  & {S}
\end{tikzcd}
$$
in $\DStk$, there are canonical quasi-isomorphisms
\begin{align*}
    \bL_{Y/S}|_{X \times^h_S Y} &\simeq \bL_{X \times^h_S Y/X} \\
    \bL_{X/S}|_{X \times^h_S Y} \oplus \bL_{Y/S}|_{X \times^h_S Y} &\simeq \bL_{X \times^h_S Y / S}
\end{align*}
in $\QCoh(X \times^h_S Y)$.
\end{lemma}

\begin{re}
    A classical version of Lemma \ref{lem:cotangent-complex-homotopy-fib-prod} is \cite[Theorem~8.1.(ii)]{Olssonsheaves07}.
\end{re}

\begin{lemma}\label{lem:cotangent-complex-cofiber-seq}
    Let $S$ be a derived Artin stack. Given a sequence $X \to Y \to Z$ of morphisms of derived Artin stacks in $\DStk_{/S}$, there is a (co)fiber sequence
$$
\bL_{Y/Z}|_X \longrightarrow \bL_{X/Z} \longrightarrow \bL_{X/Y} \overset{+1}{\longrightarrow}
$$
    in $\QCoh(X)$.
\end{lemma}

\begin{re}
    A classical version of Lemma \ref{lem:cotangent-complex-cofiber-seq} is \cite[Theorem~8.1.(iii)]{Olssonsheaves07}.
\end{re}

\begin{lemma}\label{lem:cotangent-complex-homotopy-cartesian-diagram}
Let $S$ be a derived Artin stack. Given a homotopy Cartesian diagram of derived Artin stacks
$$
\begin{tikzcd}
{W} \ar[dr, phantom, "\lrcorner_{\text{\tiny h}}", very near start] \arrow[r] \arrow[d] & {Y} \arrow[d] \\
{X} \arrow[r]  & {Z}
\end{tikzcd}
$$
in $\DStk_{/S}$, there is a (co)fiber square
$$
\begin{tikzcd}
{\bL_{Z/S}|_W} \arrow[r] \arrow[d] & {\bL_{X/S}|_W} \arrow[d] \\
{\bL_{Y/S}|_W} \arrow[r]           & {\bL_{W/S}}         
\end{tikzcd}
$$
in $\QCoh(W)$.
\end{lemma}

\begin{proof}
    Consider the following homotopy Cartesian diagram\footnote{Following Vakil, we may refer to such a diagram as a homotopy magic diagram.}
$$
\begin{tikzcd}
{W \cong X \times^h_Z Y} \ar[dr, phantom, "\lrcorner_{\text{\tiny h}}", very near start] \arrow[r] \arrow[d, shift left=1] & {X \times^h_S Y} \arrow[d] \\
{Z} \arrow[r]           & {Z \times^h_S Z},
\end{tikzcd}
$$
    where $Z \to Z \times^h_S Z$ is the relative diagonal over $S$.
    
    By this homotopy Cartesian diagram, there is a canonical quasi-isomorphism
$$ \bL_{W / X \times^h_S Y} \simeq \bL_{Z / Z \times^h_S Z}|_{W}.$$
    Moreover, for the morphism $W \to X \times^h_S Y$, there is a (co)fiber sequence
$$
\bL_{X \times^h_S Y / S}|_W \longrightarrow \bL_{W/S} \longrightarrow \bL_{W/X \times^h_S Y} \overset{+1}{\longrightarrow}.
$$

    Now consider the following commutative diagram, where the square diagram is a homotopy Cartesian diagram.
$$
\begin{tikzcd}
{Z} \arrow[r]    \arrow[rr, bend left, "id_Z"]       & {Z \times^h_S Z} \ar[dr, phantom, "\lrcorner_{\text{\tiny h}}", very near start] \arrow[r] \arrow[d]                   &   {Z} \arrow[d] \\
{} &{Z} \arrow[r] &{S} 
\end{tikzcd}
$$
    Here, $Z \to Z \times^h_S Z$ is the relative diagonal over $S$.
    By the homotopy Cartesian diagram above, there is a canonical quasi-isomorphism
    $$ \bL_{Z \times^h_S Z / Z} \simeq \bL_{Z / S}|_{Z \times^h_S Z}.$$
    Moreover, for the morphism $id_Z :Z \to Z \times^h_S Z \to Z$, there is a canonical quasi-isomorphism
    $$ \bL_{Z / Z \times^h_S Z} \simeq \bL_{Z \times^h_S Z / Z}|_{Z}[1].$$

    Putting previous results together, we obtain a (co)fiber sequence
$$
\bL_{X \times^h_S Y / S}|_W \longrightarrow \bL_{W/S} \longrightarrow \bL_{W/X \times^h_S Y}\simeq \bL_{Z / Z \times^h_S Z}|_{W} \simeq \bL_{Z \times^h_S Z / Z}|_{W}[1] \simeq \bL_{Z / S}|_{W}[1] \overset{+1}{\longrightarrow} .
$$
    Note that there is a canonical quasi-isomorphism
    $$\bL_{X \times^h_S Y / S} \simeq \bL_{X/S}|_{X \times^h_S Y} \oplus \bL_{Y/S}|_{X \times^h_S Y}.$$
    Thus, we obtain a (co)fiber sequence
$$
\bL_{X/S}|_{W} \oplus \bL_{Y/S}|_{W} \longrightarrow \bL_{W/S} \longrightarrow \bL_{Z / S}|_{W}[1] \overset{+1}{\longrightarrow},
$$
    which is equivalent to a (co)fiber sequence
$$
\bL_{Z / S}|_{W} \longrightarrow \bL_{X/S}|_{W} \oplus \bL_{Y/S}|_{W} \longrightarrow \bL_{W/S} \overset{+1}{\longrightarrow}.
$$

    Note that the (co)fiber sequence above is literally the (co)fiber square
$$
\begin{tikzcd}
{\bL_{Z/S}|_W} \arrow[r] \arrow[d] & {\bL_{X/S}|_{W} \oplus \bL_{Y/S}|_{W}} \arrow[d] \\
{0} \arrow[r]           & {\bL_{W/S}}   ,      
\end{tikzcd}
$$    
    which, by simple observation, is equivalent to the (co)fiber square
$$
\begin{tikzcd}
{\bL_{Z/S}|_W} \arrow[r] \arrow[d] & {\bL_{X/S}|_W} \arrow[d] \\
{\bL_{Y/S}|_W} \arrow[r]           & {\bL_{W/S}}        . 
\end{tikzcd}
$$
\end{proof}

\begin{re}
    An affine version of Lemma \ref{lem:cotangent-complex-homotopy-cartesian-diagram} is \cite[Proposition~1.2.1.16(3)]{ToVe08}; a local version of Lemma \ref{lem:cotangent-complex-homotopy-cartesian-diagram} is \cite[Lemma~1.4.1.12]{ToVe08}; an absolute version of Lemma \ref{lem:cotangent-complex-homotopy-cartesian-diagram} is implicitly used in the proof of \cite[Theorem~2.9]{PTVV13}.
\end{re}

\subsection{Cotangent Complexes of Derived Zero Loci}\label{append:cotang-complex-der-zero-loci}

\begin{definition}\label{def:def-linear-der-stk}
    Let $X$ be a derived stack, and $\cE$ be a perfect complex on $X$. The derived stack
\begin{align*}
    \V_X(\cE) : \DSchm_{/X}^{\op} &\longrightarrow \mathcal{S} \\
    (u: T \to X) &\longmapsto \map_{\QCoh(T)}\left(\LL u^* \cE, \cO_{T} \right)
\end{align*}
over $X$ is called the \emph{linear derived stack over $X$ associated to $\cE$} (see \cite[Section~4.3.(3)]{Toen09}, \cite[Section~3.3]{ToenSurvey14} and \cite{Monier21}).
\end{definition}

\begin{re}
    For simplicity, we use $E$ to denote $\V_X(\cE^\vee)$. There is a natural projection morphism $E \to X$. Moreover, the cotangent complex $\bL_{E/X}$ is quasi-isomorphic to $\cE^\vee|_E$.
\end{re}

\begin{re}\label{rmk:der-vect-bund}
    If $\cE^{\vee}$ is a perfect complex on $X$ that is of Tor-amplitude\footnote{Recall that we are using cohomological conventions.} $\le 0$, then the natural morphism $E \to X$ is representable by derived schemes. In this case, $E$ can be written as $\Spec_X \Sym_X (\cE^{\vee})$, called the \emph{derived vector bundle over $X$}. Moreover, if $\cE^{\vee}$ is of Tor-amplitude in $[-n, 0]$, then $E \to X$ is homotopically $n$-smooth.
\end{re}

\begin{definition}
    Let $X$ be a derived stack, and $\cE$ be a perfect complex on $X$. Let $s$ denote a section $\cO_X \to \cE$. Abusing notation, we also use $s$ to denote the morphism $X \to E$ induced by $s: \cO_X \to \cE$. The \emph{derived zero locus $X(s)$ of a section $s: \cO_X \to \cE$} is defined via the following homotopy Cartesian diagram.
$$
\begin{tikzcd}
{X(s)} \arrow[r] \arrow[d] \ar[dr, phantom, "\lrcorner_{\text{\tiny h}}", very near start] & {X} \arrow[d,"s"] \\
{X} \arrow[r, "0"]           & {E}          
\end{tikzcd}
$$
\end{definition}

\begin{re}
    $X(s)$ is precisely the derived zero locus $X(s^\vee)$ of the cosection $s^\vee: \cE^\vee \to \cO_X$ (see \cite[Definition~4.13]{Jiang25}).
\end{re}

\begin{proposition}\label{prop:cotangent-complex-derived-zero-loci}
    Let $X$ be a derived Artin stack, and $\cE$ be a perfect complex on $X$. Then $\bL_{X(s)}$ fits into the (co)fiber sequence
$$
\cE^{\vee}|_{X(s)} \overset{d\bar{s}}{\longrightarrow} \bL_{X}|_{X(s)} \longrightarrow \bL_{X(s)} \overset{+1}{\longrightarrow}
$$
    in $\QCoh(X(s))$. Here, the morphism $$d\bar{s}: \cE^{\vee}|_{X(s)} \to \bL_{X}|_{X(s)}$$ is induced by the morphism $$\bL s|_{X(s)}: \bL_{E}|_{X(s)} \to \bL_{X}|_{X(s)}.$$
\end{proposition}

\begin{proof}    
    There is a (co)fiber sequence
$$
\bL_{E}|_{X(s)} \simeq \cE^{\vee}|_{X(s)} \oplus \bL_X |_{X(s)} \overset{\bL s|_{X(s)}}{\longrightarrow} \bL_{X}|_{X(s)} \longrightarrow \bL_{X/E}|_{X(s)} \simeq \cE^{\vee}|_{X(s)}[1] \overset{+1}{\longrightarrow},
$$
    which is equivalent to the (co)fiber sequence
$$
\cE^{\vee}|_{X(s)} \overset{+1[-1]}{\longrightarrow} \cE^{\vee}|_{X(s)} \oplus \bL_X|_{X(s)} \overset{\bL s}{\longrightarrow} \bL_{X}|_{X(s)} \longrightarrow .
$$
    The cofiber sequence above is literally the (co)fiber square
$$
\begin{tikzcd}
{\cE^{\vee}|_{X(s)}} \arrow[r, "+1{[-1]}"] \arrow[d] & {\cE^{\vee}|_{X(s)} \oplus \bL_X|_{X(s)}} \arrow[d, "\bL s|_{X(s)}"] \\
{0} \arrow[r]           & {\bL_X|_{X(s)}}     ,
\end{tikzcd}
$$
    which is equivalent to the (co)fiber square
$$
\begin{tikzcd}
{\cE^{\vee}|_{X(s)}} \arrow[r, "d\bar{s}"] \arrow[d, "id"] & {\bL_X|_{X(s)}} \arrow[d, "id"] \\
{\cE^{\vee}|_{X(s)}} \arrow[r, "d\bar{s}"]           & {\bL_X|_{X(s)}}   
\end{tikzcd}
$$
such that 
\begin{align}
    \label{align:1-cotangent-complex-derived-zero-loci}
    \bL s|_{X(s)} \simeq d\bar{s} \oplus id &: \cE^{\vee}|_{X(s)} \oplus \bL_X|_{X(s)} \to \bL_X|_{X(s)}; \\
    \label{align:2-cotangent-complex-derived-zero-loci}
    +1[-1] \simeq (-id , d\bar{s}) &: \cE^{\vee}|_{X(s)} \to \cE^{\vee}|_{X(s)} \oplus \bL_X|_{X(s)}.
\end{align}

    Applying Lemma \ref{lem:cotangent-complex-homotopy-cartesian-diagram} to the homotopy Cartesian diagram
$$
\begin{tikzcd}
{X(s)} \arrow[r, "i_0"] \arrow[d, "i_s"] \ar[dr, phantom, "\lrcorner_{\text{\tiny h}}", very near start] & {X} \arrow[d,"s"] \\
{X} \arrow[r, "0"]           & {E}          ,
\end{tikzcd}
$$    
    we obtain the following commutative diagram, where rows and columns are (co)fiber sequences, and the bottom right square is a (co)fiber square.
\[\begin{tikzcd}
	{0} & {\left( \bL_{X / E}[-1] \right) |_{X(s)}} & {\bL_{X(s) / X}[-1]} & {} \\
	{\left( \bL_{X / E}[-1] \right) |_{X(s)}} & {\bL_{E}|_{X(s)}} & {\bL_{X}|_{X(s)}} & {} \\
	{\bL_{X(s) / X}[-1]} & {\bL_{X}|_{X(s)}} & {\bL_{X(s)}} & {} \\
	{} & {} & {} & {}
	\arrow[from=1-1, to=1-2]
	\arrow[from=1-1, to=2-1]
	\arrow[from=1-2, to=1-3, "\simeq"]
	\arrow[from=1-2, to=2-2, "+1{[-1]}"]
	\arrow[from=1-3, to=1-4]
	\arrow[from=1-3, to=2-3]
	\arrow[from=2-1, to=2-2, "+1{[-1]}"']
	\arrow[from=2-1, to=3-1, "\simeq"]
	\arrow[from=2-2, to=2-3, "\bL s"']
	\arrow[from=2-2, to=3-2, "d0"]
	\arrow[from=2-3, to=2-4]
	\arrow[from=2-3, to=3-3, "di_0"']
	\arrow[from=3-1, to=3-2]
	\arrow[from=3-1, to=4-1]
	\arrow[from=3-2, to=3-3, "di_s"]
	\arrow[from=3-2, to=4-2]
	\arrow[from=3-3, to=3-4]
	\arrow[from=3-3, to=4-3]
\end{tikzcd}\]
    By Equation (\ref{align:1-cotangent-complex-derived-zero-loci}) and Equation (\ref{align:2-cotangent-complex-derived-zero-loci}), we in fact obtain the following commutative diagram,  where rows and columns are (co)fiber sequences, and the bottom right square is a (co)fiber square.
\[\begin{tikzcd}
	{0} & {\cE^{\vee} |_{X(s)}} & {\cE^{\vee} |_{X(s)}} & {} \\
	{\cE^{\vee} |_{X(s)}} & { \cE^{\vee}|_{X(s)} \oplus \bL_{X}|_{X(s)} } & {\bL_{X}|_{X(s)}} & {} \\
	{\cE^{\vee} |_{X(s)}} & {\bL_{X}|_{X(s)}} & {\bL_{X(s)}} & {} \\
	{} & {} & {} & {}
	\arrow[from=1-1, to=1-2]
	\arrow[from=1-1, to=2-1]
	\arrow[from=1-2, to=1-3, "\simeq"]
	\arrow[from=1-2, to=2-2, "(-\id\text{,}0)"]
	\arrow[from=1-3, to=1-4]
	\arrow[from=1-3, to=2-3, "-d\bar{s}"']
	\arrow[from=2-1, to=2-2, "(-id\text{,}d\bar{s})"']
	\arrow[from=2-1, to=3-1, "\simeq"]
	\arrow[from=2-2, to=2-3, "d\bar{s} \oplus \id"']
	\arrow[from=2-2, to=3-2, "0 \oplus \id"]
	\arrow[from=2-3, to=2-4]
	\arrow[from=2-3, to=3-3, "di_0"']
	\arrow[from=3-1, to=3-2, "d\bar{s}"]
	\arrow[from=3-1, to=4-1]
	\arrow[from=3-2, to=3-3, "di_s"]
	\arrow[from=3-2, to=4-2]
	\arrow[from=3-3, to=3-4]
	\arrow[from=3-3, to=4-3]
\end{tikzcd}\]
    Therefore, we conclude that there is a quasi-isomorphism
    $$\bL_{X(s)} \simeq \text{cofib} \left(d\bar{s}: \cE^{\vee}|_{X(s)} \to \bL_{X}|_{X(s)} \right)$$
    in $\QCoh(X(s))$.
\end{proof}

\begin{re}
    See Example \ref{eg:aff-derived-zero-locus} for an explicit description of the cotangent complex of a derived zero locus.
\end{re}

\begin{re}\label{rmk:cotangent-complx-der-quot-stk}
    Following {\cite[Example~8.6.5]{KhanNotes23}}, together with similar arguments in the proof of Proposition \ref{prop:cotangent-complex-derived-zero-loci}, one could similarly give an explicit computation of (derived) quotient stacks, assuming the group scheme is linearly reductive (see Remark \ref{rmk:cotangent-complex-cA}, Example \ref{eg:cotangent-cmplx-der-proj-space}, Example \ref{eg:1st-cotangent-complex-genuine-der-quot-stack}, and Example \ref{eg:2nd-cotangent-complex-genuine-der-quot-stack}). 
\end{re}

We list examples of derived zero loci that are related to virtual techniques in enumerative geometry (also known as virtual geometry).

\begin{example}
    
    The most important example is a local model for perfect obstruction theory appeared in \cite[Section~2]{GrPa99}. See Example \ref{eg:local-model-POT} for more details. $\blacksquare$
\end{example}

\begin{example}\label{eg:degeneracy-loci}

    The next important example is the degeneracy locus of a everywhere defined cosection introduced in \cite{KiLi13}, or more generally in \cite[Notation~6.1(i)]{AKLPR22P}. See Section \ref{sec:cosection-localization} for more details. $\blacksquare$
\end{example}

\begin{example}

    Following the Example \ref{eg:degeneracy-loci}, we construct the degeneracy locus of a log cosection as a derived zero locus. See Section \ref{sec:log-cosection} for more details. $\blacksquare$
\end{example}

\subsection{Reduced Perfect Obstruction theories via Derived Zero Loci}\label{append:red-POT-in-general}

For the rest of the appendix, $X$ is assumed to be a quasi-smooth derived Deligne--Mumford stack. 

The definition given below is inspired by \cite[Construction~6.18]{AKLPR22P}\footnote{Note that we refer to a specific arXiv version, not the latest or published version.} and \cite[Construction~4.2.7]{ParkNotes25}. 

\begin{definition}
    Let $\cE^\vee$ be a perfect complex on $X$ that is of Tor-amplitude in $[-1,0]$, and $\phi$ be a section $\cO_X \to \cE$. If the derived zero locus $X(\phi)$ is quasi-smooth, and its classical truncation $X(\phi)_\cl$ coincides with $X_\cl$, then we call $X(\phi)$ the \emph{$\phi$-reduction of $X$}, denoted as $X^{\phi-\red}$. The \emph{$\phi$-reduced perfect obstruction theory on $X_\cl$} is thus defined to be the perfect obstruction theory induced by $X^{\phi-\red}$ (see Corollary \ref{cor:qsmooth-POT}).
\end{definition}

\begin{re}\label{rmk:cofib-phi-red-POT}
    For the $\phi$-reduction $X^{\phi-\red}$ of $X$, there is a (co)fiber sequence
    $$
    \bT_{X^{\phi-\red}}|_{X_{\cl}} \longrightarrow \bT_{X}|_{X_{\cl}} \overset{(d\bar{\phi})^\vee|_{X_\cl}}{\longrightarrow} \cE|_{X_{\cl}} \overset{+1}{\longrightarrow}
    $$
    by Proposition \ref{prop:cotangent-complex-derived-zero-loci}. The \emph{$\phi$-reduced obstruction sheaf $\Ob^{\phi-\red}_{X_{\cl}}$} is defined to be $H^1(\bT_{X^{\phi-\red}}|_{X_{\cl}})$.
\end{re}

\begin{re}
    In general, without additional assumptions, $X(\phi)$ is homotopically $2$-smooth, and its classical truncation $X(\phi)_\cl$ does not coincide with $X_\cl$.
\end{re}

\begin{re}
    By definition of derived zero locus, the classical truncation $X(\phi)_\cl$ of $X(\phi)$ fits into the following classical Cartesian diagram.
$$
\begin{tikzcd}
{X(\phi)_\cl} \arrow[r] \arrow[d] \ar[dr, phantom, "\lrcorner", very near start] & {X_\cl} \arrow[d,"\phi_\cl"] \\
{X_\cl} \arrow[r, "0"]           & {E_\cl}          
\end{tikzcd}
$$
    Hence, to show that $X(\phi)_{\cl} = X_{\cl}$, it suffices to show that $\phi_{\cl} = 0: X_{\cl} \to E_{\cl}$.
\end{re}

\begin{re}\label{rmk:X-phi-red-being-qsmth}
    Assuming $X(\phi)_{\cl} = X_{\cl}$, there is a (co)fiber sequence
    $$
    \bT_{X(\phi)}|_{X_{\cl}} \longrightarrow \bT_{X}|_{X_{\cl}} \overset{(d\bar{\phi})^\vee|_{X_\cl}}{\longrightarrow} \cE|_{X_{\cl}} \overset{+1}{\longrightarrow}
    $$
    by Proposition \ref{prop:cotangent-complex-derived-zero-loci}. Hence, to show that $X(\phi)$ is quasi-smooth, by similar arguments in Remark \ref{rmk:X-red-being-qsmth}, it suffices to show that the morphism $H^1\left((d\bar{\phi})^\vee|_{X_\cl}\right): H^1(\bT_X|_{X_{\cl}}) \to H^1(\cE|_{X_{\cl}})$ is surjective so that $H^2(\bT_{X(\phi)}|_{X_\cl}) = 0$.
\end{re}

Some existing reduced perfect obstruction theories are special cases of this general construction.

\begin{example}[\cite{OkPa10, MaPa13, MPT10}; {\cite[Definition~4.7]{STV15}}{\cite[Construction~6.18]{AKLPR22P}}{\cite[Construction~4.2.7]{ParkNotes25}}]
\label{eg:origin-red-POT}

    Let $w : X \to \bA^1[-1]$ be a $(-1)$-shifted function, which is equivalent to a section $\phi: \cO_X \to \cO_X[-1] =: \cE$. Abusing notation, we also use $\phi$ to denote the morphism.
    $$X \to \bA_X^1[-1] := \Spec_X \Sym_X \left((\cO_X[-1])^{\vee}\right)$$
    induced by $\phi: \cO_X \to \cO_X[-1]$.

    Let $X(\phi)$ be the derived zero locus of the section $\phi: \cO_X \to \cO_X[-1]$, defined via the following homotopy Cartesian diagram.
$$
    \begin{tikzcd}
{X(\phi)} \arrow[r] \arrow[d] \ar[dr, phantom, "\lrcorner_{\text{\tiny h}}", very near start] & {X} \arrow[d,"\phi"] \\
{X} \arrow[r, "0"]           & {\bA_X^1[-1]}         
    \end{tikzcd}
$$
    Note that the classical truncation $X(\phi)_{\cl}$ of $X(\phi)$ coincides with $X_{\cl}$, because the classical truncation of $\phi: X \to \bA_X^1[-1]$ is the identity morphism $\id: X_\cl \to X_\cl$.

    Note that the cosection $\sigma:= H^1\left(\bT w|_{X_\cl}\right) : \Ob_{X_{\cl}} := H^1(\bT_X|_{X_{\cl}}) \to \cO_{X_{\cl}}$ associated with $w : X \to \bA^1[-1]$ is precisely the morphism $H^1\left((d\bar{\phi})^\vee|_{X_\cl}\right): H^1(\bT_X|_{X_{\cl}}) \to H^1(\cE|_{X_{\cl}}) = \cO_{X_{\cl}}$. Thus, if the cosection $\sigma: \Ob_{X_{\cl}} \to \cO_{X_{\cl}}$ is surjective everywhere on $X_{\cl}$, then $X(\phi)$ is quasi-smooth by Remark \ref{rmk:X-phi-red-being-qsmth}, and hence $X(\phi)$ induces the desired reduced perfect obstruction theory on $X_{\cl}$ by Remark \ref{rmk:cofib-phi-red-POT}.
    $\blacksquare$
\end{example}

\begin{example}[{\cite[Section~6]{CJR21}}; Section \ref{sec:red-vir-cycle}]
\label{eg:alternate-construct-X-red}

Assume that $X$ is equipped with a morphism $\alpha_X: X \to \cA$. In other words, $(X, \alpha_X)$ is a quasi-smooth derived log Deligne--Mumford stack. 

$T[-1]\cA := \Spec_\cA \Sym_\cA \left( \left( \bT_\cA[-1] \right)^\vee \right)$ is the $(-1)$-shifted tangent bundle over $\cA$. We construct the derived vector bundle $E$ over $X$ via the following homotopy Cartesian diagram.
$$
    \begin{tikzcd}
{E} \arrow[r] \arrow[d] \ar[dr, phantom, "\lrcorner_{\text{\tiny h}}", very near start] & {T[-1]\cA} \arrow[d] \\
{X} \arrow[r, "\alpha_X"]           & {\cA}       
    \end{tikzcd}
$$
Thus, we have $$E \cong \Spec_X \Sym_X \left( \left( \bT_\cA[-1]|_X \right)^\vee \right),$$ where $\bT_\cA[-1]|_X$ is quasi-isomorphic to $$\left[s_{X}^\vee:= s_\cA^\vee|_{X}: \cO_{X} \to \cO_{X}(\D_{X}) \right],$$ which is a perfect complex of Tor-amplitude in $[0,1]$.

Let $w_\alpha : X \to L_{\cA}[-1]$ be a $(-1)$-shifted log function on $(X, \alpha_X)$, which is equivalent to a section $\cO_X \to \cO_X(\D_X)[-1]$. Hence, $w_\alpha$ induces a section $$\phi: \cO_X \to \cO_X(\D_X)[-1] \to \bT_\cA[-1]|_X =: \cE.$$ Abusing notation, we also use $\phi$ to denote the morphism $X \to E$ induced by $\phi: \cO_X \to \cE$.

Let $X(\phi)$ be the derived zero locus of the section $\phi: \cO_X \to \cE$, defined via the following homotopy Cartesian diagram.
$$
    \begin{tikzcd}
{X(\phi)} \arrow[r] \arrow[d] \ar[dr, phantom, "\lrcorner_{\text{\tiny h}}", very near start] & {X} \arrow[d,"\phi"] \\
{X} \arrow[r, "0"]           & {E}       
    \end{tikzcd}
$$
     Note that $\phi_\cl=0:X_\cl \to E_\cl$, so that the classical truncation $X(\phi)_\cl$ of $X(\phi)$ coincides with $X_\cl$.  By Remark \ref{rmk:X-phi-red-being-qsmth}, there is a (co)fiber sequence
$$
\bT_{X(\phi)}|_{X_{\cl}} \longrightarrow \bT_{X}|_{X_{\cl}} \overset{(d\bar{\phi})^{\vee}|_{X_{\cl}}}{\longrightarrow} (\bT_\cA[-1])|_{X_{\cl}} \overset{+1}{\longrightarrow}.
$$
    Note that $\cE|_{X_\cl} := \bT_\cA[-1]|_{X_{\cl}}$ is quasi-isomorphic to $\bF$ in Proposition \ref{prop:red-POT-log-GLSM}.
    
    Note that the morphism $$H^1\left((d\bar{\phi})^{\vee}|_{X_{\cl}}\right): H^1\left(\bT_X|_{X_\cl}\right) \to H^1\left((\bT_\cA[-1])|_{X_{\cl}}\right) = H^1(\bF)$$ factors through the morphism $$\sigma^\tw := H^1\left(\bT w_\alpha|_{X_\cl}\right) : H^1\left(\bT_X|_{X_\cl}\right) \to \cO_{X_\cl}(\D_{X_\cl}),$$ because $\phi: \cO_X \to \bT_\cA[-1]|_X$ is defined to factor through the section $\cO_X \to \cO_X(\D_X)[-1]$ induced by $w_\alpha: X \to L_{\cA}[-1]$.
    Thus, if the morphism $\sigma^\tw :  H^1\left(\bT_X|_{X_\cl}\right) \to \cO_{X_\cl}(\D_{X_\cl})$ is surjective along $(\D_{X_{cl}})_\cl$, then $H^1\left((d\bar{\phi})^{\vee}|_{X_{\cl}}\right): H^1\left(\bT_X|_{X_\cl}\right) \to H^1\left((\bT_\cA[-1])|_{X_{\cl}}\right)$ is surjective everywhere on $X_{\cl}$. Thus, $X(\phi)$ is quasi-smooth by Remark \ref{rmk:X-phi-red-being-qsmth}, and hence $X(\phi)$ induces the desired reduced perfect obstruction theory on $X_{\cl}$ by Remark \ref{rmk:cofib-phi-red-POT}.
    $\blacksquare$
\end{example}

\begin{example}[{\cite[Section~6]{CJR21}}; Section \ref{sec:red-bound-cycle}]

    Assume that $X$ is equipped with a morphism $\alpha_X: X \to \cA$. In other words, $(X, \alpha_X)$ is a quasi-smooth derived log Deligne--Mumford stack. 
    
    Let $w_\alpha : X \to \cL_{\cA}[-1]$ be a $(-1)$-shifted log function on $(X, \alpha_X)$, which is equivalent to a section $\phi: \cO_X \to \cO_X(\D_X)[-1] =: \cE$. Abusing notation, we also use $\phi$ to denote the morphism.
    $$X \to \Tot_X\left(\cO_X(\D_X)\right)[-1] := \Spec_X \Sym_X \left(\cO_X(-\D_X)[1]\right)$$
    induced by $\phi: \cO_X \to \cO_X(\D_X)[-1]$.
    
    Let $X(\phi)$ be the derived zero locus of the section $\phi: \cO_X \to \cO_X(\D_X)[-1]$, defined via the following homotopy Cartesian diagram.
$$
    \begin{tikzcd}
{X(\phi)} \arrow[r] \arrow[d] \ar[dr, phantom, "\lrcorner_{\text{\tiny h}}", very near start] & {X} \arrow[d,"\phi"] \\
{X} \arrow[r, "0"]           & {\Tot_X\left(\cO_X(\D_X)\right)[-1]}       .  
    \end{tikzcd}
$$
    The rest of the arguments are similar to those in Example \ref{eg:origin-red-POT}.
    $\blacksquare$
\end{example}

\begin{example}\label{eg:phi-red-POT}
    Let $\phi$ be a section $\cO_X \to \cE := \cO_X^{\oplus r_0} \oplus \left(\cO_X^{\oplus r_1}[-1]\right)$. Assume the $\phi$-reduction $X^{\phi-\red}$ of $X$ exists. Applying $H^0$ to the (co)fiber sequence in Remark \ref{rmk:cofib-phi-red-POT}, we obtain a long exact sequence
\[\begin{tikzcd}[column sep=4.5em, row sep=1em]
	0 & H^0(\bT_{X^{\phi-\red}}|_{X_{\cl}}) & H^0(\bT_{X}|_{X_{\cl}}) & \cO_{X_{\cl}}^{\oplus r_0} &{}\\
	{}& H^1(\bT_{X^{\phi-\red}}|_{X_{\cl}}) & H^1(\bT_{X}|_{X_{\cl}}) & \cO_{X_{\cl}}^{\oplus r_1} & 0    .
	\arrow[from=1-1, to=1-2]
	\arrow[from=1-2, to=1-3, "="]
	\arrow[from=1-3, to=1-4, "0"]
	\arrow[from=2-1, to=2-2]
	\arrow[from=2-2, to=2-3]
	\arrow[from=2-3, to=2-4]
    \arrow[from=2-4, to=2-5]
\end{tikzcd}\]
    Note that the isomorphism $H^0(\bT_{X^{\phi-\red}}|_{X_{\cl}}) = H^0(\bT_{X}|_{X_{\cl}})$ follows from Proposition \ref{prop:cotangent-complex-obstruction-theory}. We then see that $\phi$-reduced obstruction sheaf $\Ob^{\phi-\red}_{X_{\cl}} := H^1(\bT_{X^{\phi-\red}}|_{X_{\cl}})$ depends on the obstruction sheaf $\Ob_{X_{\cl}} := H^1(\bT_{X}|_{X_{\cl}})$ and $\cE|_{X_\cl}$. Note that there is an upper bound for $r_1$.
    
    If the virtual rank of $\cE|_{X_\cl}$ is negative, i.e. $r_0 < r_1$, then the virtual dimension of $X^{\phi-\red}$ is greater than the virtual dimension of $X$. This could be helpful when the virtual dimension of $X$ is negative (the virtual cycle is then immediately trivial), such as the quasi-smooth derived Deligne--Mumford stacks of higher genus ($\ge 2$) stable maps to higher dimensional ($\ge 4$) smooth projective Calabi-Yau varieties.
    $\blacksquare$
\end{example}

\section{Computation of Dual Logarithmic Obstruction Sheaves}\label{append:2}

For a quasi-smooth derived log Deligne--Mumford stack $(X, \alpha_X)$, we have 
\begin{align*}
    H^{-1}(\bLL_X|_{X_\cl}) = &\left(H^1(\bTT_X|_{X_\cl})\right)^\vee = \left(Ob^\log_{X_\cl}\right)^\vee, \\
    H^{-1}(\bL_X|_{X_\cl}) = &\left(H^1(\bT_X|_{X_\cl})\right)^\vee = \left(Ob_{X_\cl}\right)^\vee.
\end{align*}
The purpose of this appendix is to compare $\left(Ob^\log_{X_\cl}\right)^\vee$ and $\left(Ob_{X_\cl}\right)^\vee$. Although they agree in a special case, they differ in general.

For simplicity, we work with Example \ref{eg:log-cotangent-complex-aff-derived-log-schm} throughout the appendix. Note that Example \ref{eg:log-cotangent-complex-aff-derived-log-schm} is the continuation of Example \ref{eg:aff-derived-zero-locus}.

By Example \ref{eg:aff-derived-zero-locus}, $\bL_{Z(s)}|_{Z(s)_\cl}$ is quasi-isomorphic to
$$
\cO_{Z(s)_\cl}\langle d\nu_1,d\nu_2\rangle \overset{d\nu_j \mapsto df_j}{\longrightarrow} \cO_{Z(s)_\cl}\langle dx_1,dx_2,dx_3\rangle,
$$
which is a perfect complex of Tor-amplitude in $[-1,0]$.

By Example \ref{eg:log-cotangent-complex-aff-derived-log-schm}, $\bLL_{Z(s)}|_{Z(s)_\cl}$ is quasi-isomorphic to the total complex of the bi-complex
$$
\begin{tikzcd}
{\cO_{Z(s)_\cl}\langle d\lambda \rangle} \arrow[r]           & {0}           \\
{\cO_{Z(s)_\cl}\langle d\mu \rangle} \arrow[u, "d\mu \mapsto x_3d\lambda"] \arrow[r, "d\mu \mapsto dx_3"] & {\cO_{Z(s)_\cl}\langle dx_1,dx_2,dx_3\rangle} \arrow[u] \\
{0} \arrow[u] \arrow[r] & {\cO_{Z(s)_\cl}\langle d\nu_1,d\nu_2\rangle} \arrow[u,"d\nu_j \mapsto df_j"]     ,
\end{tikzcd}
$$
which is a perfect complex of Tor-amplitude in $[-1,0]$.

By direct computation, we have
\begin{align*}
H^{-1}(\bL_{Z(s)}|_{Z(s)_\cl}) &=
\left\{
b_1d\nu_1+b_2d\nu_2 \; \middle| \;
\begin{array}{l}
b_1,b_2 \in \cO_{Z(s)_\cl},  \\
b_1df_1+b_2df_2 = 0 \in \cO_{Z(s)_\cl}\langle dx_1,dx_2,dx_3\rangle
\end{array}
\right\}, \\[1ex]
H^{-1}(\bLL_{Z(s)}|_{Z(s)_\cl}) &=
\left\{
(ad\mu, b_1d\nu_1+b_2d\nu_2) \; \middle| \;
\begin{array}{l}
a,b_1,b_2 \in \cO_{Z(s)_\cl},  \\
ax_3d\lambda = 0 \in \cO_{Z(s)_\cl}\langle d\lambda \rangle,  \\
-adx_3+(b_1df_1+b_2df_2) = 0 \in \cO_{Z(s)_\cl}\langle dx_1,dx_2,dx_3\rangle
\end{array}
\right\}.
\end{align*}

\textbf{Case $1$}. If $x_3$ is a non-zero-divisor in $\cO_{Z(s)_\cl}$, then $ax_3d\lambda=0$ implies that $a = 0$. Thus, we see that $$H^{-1}(\bLL_{Z(s)}|_{Z(s)_\cl}) = H^{-1}(\bL_{Z(s)}|_{Z(s)_\cl}).$$

\begin{re}
    If $x_3$ is a non-zero-divisor in $\cO_{Z(s)_\cl}$ (i.e. $\D_{Z(s)_{cl}}$ is necessarily a classical scheme), then $\bLL_{Z(s)}|_{Z(s)_\cl}$ is quasi-isomorphic to the perfect complex
$$
\cO_{Z(s)_\cl}\langle d\nu_1,d\nu_2\rangle \overset{d\nu_j \mapsto df_j}{\longrightarrow} \cO_{Z(s)_\cl}\langle dx_1,dx_2,d\log x_3\rangle,
$$
    that is of Tor-amplitude in $[-1,0]$.
\end{re}

\textbf{Case $2$}. If $f_1=f_2=x_3^2$ (hence $x_3$ is a zero-divisor in $\cO_{Z(s)_\cl}$), then we have
$$
H^{-1}(\bLL_{Z(s)}|_{Z(s)_\cl}) = 
\left\{
(ad\mu, b_1d\nu_1+b_2d\nu_2) \; \middle| \;
a = 2(b_1+b_2)x_3 \in \cO_{Z(s)_\cl}
\right\}
\ne
H^{-1}(\bL_{Z(s)}|_{Z(s)_\cl}).
$$

\textbf{Case $3$}. If $f_1=x_2x_3$ and $f_2=x_2^3$ (hence $x_3$ is a zero-divisor in $\cO_{Z(s)_\cl}$), then we have
$$
H^{-1}(\bLL_{Z(s)}|_{Z(s)_\cl}) = 
\left\{
(ad\mu, b_1d\nu_1+b_2d\nu_2) \; \middle| \;
\begin{array}{l}
a = b_1x_2 \in \cO_{Z(s)_\cl},   \\
b_1x_3+3b_2x_2^2 =0 \in \cO_{Z(s)_\cl}
\end{array}
\right\}
\ne
H^{-1}(\bL_{Z(s)}|_{Z(s)_\cl}).
$$

\bibliographystyle{alpha}
\bibliography{biblio}

\end{document}